\documentclass{amsart}
\usepackage{a4wide}
\usepackage{amssymb}
\usepackage{amsmath}
\usepackage{pifont}
\usepackage{amsfonts}
\usepackage{latexsym}
\usepackage{epsfig}
\usepackage{graphicx}
\usepackage{verbatim}
\usepackage{float}
\usepackage{exscale}
\usepackage{color}
\usepackage[titletoc]{appendix}

\newcommand{\BIGOP}[1]{\mathop{\mathchoice%
{\raise-0.22em\hbox{\huge $#1$}}%
{\raise-0.05em\hbox{\Large $#1$}}{\hbox{\large $#1$}}{#1}}}
\newcommand{\bigtimes}{\BIGOP{\times}}
\newcommand{\BIGboxplus}{\mathop{\mathchoice%
{\raise-0.35em\hbox{\huge $\boxplus$}}%
{\raise-0.15em\hbox{\Large $\boxplus$}}{\hbox{\large $\boxplus$}}{\boxplus}}}

\renewcommand{\ldots}{\ensuremath{\dotsc}}

\newcommand{\mint}{\textstyle\mints\displaystyle}
\newcommand{\mints}{\int\!\!\!\!\!{\rm-}\,}

\newcommand{\Rplus}{{\mathbb R}_{>0}}

\newcommand{\Yrho}{Y}
\newcommand{\Yrhod}{\Upsilon}
\def\epsilon{\varepsilon}
\def\hat{\widehat}

\def\undertilde#1{{\baselineskip=0pt\vtop
{\hbox{$#1$}\hbox{$\scriptscriptstyle\sim$}}}{}}
\def\underdtilde#1{{\baselineskip=0pt\vtop
{\hbox{$#1$}\hbox{$\scriptscriptstyle\approx$}}}{}}
\def\epsilon{\varepsilon}
\def\hat{\widehat}

\def\nabq{\undertilde{\nabla}_{q}}
\def\nabqi{\undertilde{\nabla}_{q_i}}
\def\nabqj{\undertilde{\nabla}_{q_j}}
\def\nabx{\undertilde{\nabla}_{x}}
\def\nabttx{\underdtilde{\nabla}_{x}}
\def\nabr1{\undertilde{\nabla}_{r_1}}
\def\nabr2{\undertilde{\nabla}_{r_2}}

\def\lae{\ell_{a}}

\def\ut{\undertilde{u}}
\def\utk{\undertilde{u}_{\kappa}}
\def\utka{\undertilde{u}_{\kappa,\alpha}}
\def\utkaL{\undertilde{u}_{\kappa,\alpha,L}}
\def\utkaLd{\undertilde{u}_{\kappa,\alpha,L,\delta}}
\def\utkaLD{\undertilde{u}_{\kappa,\alpha,L}^{\Delta t}}
\def\utkaLDp{\undertilde{u}_{\kappa,\alpha,L}^{\Delta t,+}}
\def\utkaLDm{\undertilde{u}_{\kappa,\alpha,L}^{\Delta t,-}}
\def\utkaLDwpm{\undertilde{u}_{\kappa,\alpha,L}^{\Delta t,\pm}}
\def\utkaLDpm{\undertilde{u}_{\kappa,\alpha,L}^{\Delta t(,\pm)}}
\def\mt{\undertilde{m}}

\def\mtka{\undertilde{m}_{\kappa,\alpha}}
\def\mtkaL{\undertilde{m}_{\kappa,\alpha,L}}
\def\mtkaLD{\undertilde{m}_{\kappa,\alpha,L}^{\Delta t}}
\def\mtkaLDp{\undertilde{m}_{\kappa,\alpha,L}^{\Delta t,+}}
\def\mtkaLDm{\undertilde{m}_{\kappa,\alpha,L}^{\Delta t,-}}
\def\mtkaLDpm{\undertilde{m}_{\kappa,\alpha,L}^{\Delta t(,\pm)}}
\def\mtkaLDwpm{\undertilde{m}_{\kappa,\alpha,L}^{\Delta t,\pm}}
\def\utae{\undertilde{u}_{L}}
\def\utaed{\undertilde{u}_{L,\delta}}
\def\uta{\undertilde{u}_{L}}

\def\utaeDm{\utae^{\Delta t,-}}
\def\utaeDp{\utae^{\Delta t,+}}

\def\vt{\undertilde{v}}
\def\wt{\undertilde{w}}
\def\xt{\undertilde{x}}
\def\qt{\undertilde{q}}

\def\bt{\undertilde{b}}
\def\ft{\undertilde{f}}
\def\gt{\undertilde{g}}

\def\nt{\undertilde{n}}
\def\yt{\undertilde{y}}

\def\lambdat{\undertilde{\lambda}}
\def\varsigmat{\undertilde{\varsigma}}
\def\Dlxn{\partial^{|\lambdat|}}
\def\Dlxd{\partial^{\lambda_1}_{x_1}\cdots\partial^{\lambda_d}_{x_d}}

\def\calBt{\undertilde{\mathcal{B}}}
\def\Ct{\undertilde{C}}
\def\Et{\undertilde{E}}
\def\Ft{\undertilde{F}}

\def\Ht{\undertilde{H}}

\def\Lt{\undertilde{L}}

\def\Vt{\undertilde{V}}
\def\Wt{\undertilde{W}}

\def\cVt{\undertilde{\mathcal{V}}}
\def\dq{\,{\rm d}\undertilde{q}}
\def\dx{\,{\rm d}\undertilde{x}}

\def\dy{\,{\rm d}\undertilde{y}}
\def\dt{\,{\rm d}t}
\def\zerot{\undertilde{0}}
\def\zerott{\underdtilde{0}}
\def\tautt{\underdtilde{\tau}}

\def\Att{\underdtilde{A}}
\def\Btt{\underdtilde{B}}
\def\Ctt{\underdtilde{C}}

\def\Dtt{\underdtilde{D}}

\def\Htt{\underdtilde{H}}
\def\Itt{\underdtilde{I}}

\def\Ltt{\underdtilde{L}}
\def\Stt{\underdtilde{S}}

\def\unabtt{\underdtilde{\nabla}_{x}\,\ut}

\def\vnabtt{\underdtilde{\nabla}_{x}\,\vt}

\def\sigtt{\underdtilde{\sigma}}

\def\pk{p_{\kappa}}
\def\Pk{P_{\kappa}}
\def\hpsi{\hat \psi}
\def\tpsi{\widetilde{\psi}}
\def\tut{\widetilde{\ut}}

\def\trho{\widetilde{\rho}}

\def\hpsiaet{\widetilde{\psi}_{L}}
\def\psikaL{\psi_{\kappa,\alpha,L}}
\def\hpsikaL{\hat \psi_{\kappa,\alpha,L}}

\def\hpsikaLd{\hat \psi_{\kappa,\alpha,L,\delta}}
\def\hpsikaLD{\hat \psi_{\kappa,\alpha,L}^{\Delta t}}
\def\hpsikaLDp{\hat \psi_{\kappa,\alpha,L}^{\Delta t,+}}
\def\hpsikaLDm{\hat \psi_{\kappa,\alpha,L}^{\Delta t,-}}
\def\hpsikaLDpm{\hat \psi_{\kappa,\alpha,L}^{\Delta t(,\pm)}}
\def\hpsikaLDwpm{\hat \psi_{\kappa,\alpha,L}^{\Delta t,\pm}}

\def\hpsika{\hat \psi_{\kappa,\alpha}}
\def\hpsik{\hat \psi_{\kappa}}

\def\hpsiaedt{\widetilde{\psi}_{L,\delta}}
\def\psiae{\psi_{L}}
\def\psia{\widetilde \psi_{L}}
\def\pkaL{p_{\kappa,\alpha,L}}
\def\rhok{\rho_{\kappa}}
\def\rhoka{\rho_{\kappa,\alpha}}
\def\rhokaL{\rho_{\kappa,\alpha,L}}
\def\rhokaLd{\rho_{\kappa,\alpha,L,\delta}}
\def\rhokaLD{\rho_{\kappa,\alpha,L}^{[\Delta t]}}
\def\rhokaLDp{\rho_{\kappa,\alpha,L}^{\Delta t,+}}
\def\rhokaLDm{\rho_{\kappa,\alpha,L}^{\Delta t,-}}
\def\rhokaLDpm{\rho_{\kappa,\alpha,L}^{\Delta t(,\pm)}}
\def\rhokaLDwpm{\rho_{\kappa,\alpha,L}^{\Delta t,\pm}}
\def\vrho{\varrho}
\def\vrhok{\varrho_{\kappa}}
\def\vrhoka{\varrho_{\kappa,\alpha}}
\def\vrhokaL{\varrho_{\kappa,\alpha,L}}
\def\vrhokaLd{\varrho_{\kappa,\alpha,L,\delta}}

\def\vrhokaLDp{\varrho_{\kappa,\alpha,L}^{\Delta t,+}}
\def\vrhokaLDm{\varrho_{\kappa,\alpha,L}^{\Delta t,-}}
\def\tvrho{\widetilde{\varrho}}

\def\dd {{\,\rm d}}

\newcommand{\nabxtt}{\underdtilde{\nabla}_{x}\,}

\newcommand{\bet}{\noalign{\vskip6pt plus 3pt minus 1pt}}

\def\Xint#1{\mathchoice
{\XXint\displaystyle\textstyle{#1}}%
{\XXint\textstyle\scriptstyle{#1}}%
{\XXint\scriptstyle\scriptscriptstyle{#1}}%
{\XXint\scriptscriptstyle\scriptscriptstyle{#1}}%
\!\int}
\def\XXint#1#2#3{{\setbox0=\hbox{$#1{#2#3}{\int}$}
\vcenter{\hbox{$#2#3$}}\kern-.5\wd0}}

\def\dashint{\Xint-}

\newtheorem{example}{Example}[section]
\newtheorem{definition}{Definition}[section]
\newtheorem{lemma}{Lemma}[section]
\newtheorem{theorem}{Theorem}[section]
\newtheorem{remark}{Remark}[section]
\newtheorem{corollary}{Corollary}[section]

\renewcommand{\theequation}{\arabic{section}.\arabic{equation}}

\newcounter{ind}
\def\eqlabstart{%
 \setcounter{ind}{\value{equation}}\addtocounter{ind}{1}%
 \setcounter{equation}{0}%
 \renewcommand{\theequation}{\arabic{section}.\arabic{ind}\alph{equation}}%
}
\def\eqlabend{%
 \renewcommand{\theequation}{\arabic{section}.\arabic{equation}}%
 \setcounter{equation}{\value{ind}}%
}

\newcommand{\red}[1]{\textcolor{black}{#1}}
\newcommand{\arxiv}[1]{\textcolor{black}{#1}}

\newcounter{appendix}
\renewcommand\appendix{\par
        \refstepcounter{appendix}
           \setcounter{section}{0}
       \setcounter{theorem}{0}
           \setcounter{equation}{0}
\renewcommand\thesection{\appendixname ~\Alph{section}}
\renewcommand\thesubsection{\Alph{section}.\arabic{subsection}}%
\renewcommand\theequation{\Alph{section}.\arabic{equation}}}%

\begin{document}

\markboth{John W. Barrett and Endre S\"{u}li}
{Existence 
of Global Weak Solutions for Compressible Dilute Polymers}

%
%

\title[Existence 
of Global Weak Solutions for Compressible Dilute Polymers]
{Existence of global weak solutions to compressible isentropic finitely extensible nonlinear\\
bead-spring chain models for dilute polymers} 

\author{JOHN W. BARRETT}

\address{\footnotesize Department of Mathematics, Imperial College London\\
London SW7 2AZ, UK\\
{\tt jwb@imperial.ac.uk}}

\author{ENDRE S\"ULI}

\address{
Mathematical Institute, University of Oxford\\ Oxford OX2 6GG, UK\\
{\tt endre.suli@maths.ox.ac.uk}}

\begin{abstract}
We prove the existence of global-in-time weak solutions to a general class of models that arise
from the kinetic theory of dilute solutions of nonhomogeneous polymeric liquids, where the polymer
molecules are idealized as bead-spring chains with finitely extensible nonlinear elastic (FENE)
type spring potentials.
The class of models under consideration involves the unsteady, compressible, isentropic,
isothermal Navier--Stokes system in a bounded domain $\Omega$ in $\mathbb{R}^d$, $d = 2$ or $3$,
for the
density $\rho$, the velocity $\ut$ and the pressure $p$ of the fluid, with an equation of state
of the form $p(\rho) = c_p \rho^\gamma$, where $c_p$ is a positive constant and $\gamma>\frac{3}{2}$.
The right-hand side of the Navier--Stokes momentum equation includes an elastic extra-stress
tensor,  \red{ which is the sum of the classical Kramers expression and a quadratic interaction term.
The elastic extra-stress tensor} stems from the random movement of the polymer chains and is defined through the
associated probability density function that satisfies a Fokker--Planck-type parabolic equation,
a crucial feature of which is the presence of a centre-of-mass diffusion term.
We require no structural assumptions
on the drag term in the Fokker--Planck equation; in particular, the drag term need not
be corotational. With a nonnegative initial density
$\rho_0 \in L^\infty(\Omega)$ for the continuity equation;
a square-integrable initial velocity datum $\undertilde{u}_0$ for
the Navier--Stokes momentum equation; and
a nonnegative initial probability density function $\psi_0$
for the Fokker--Planck equation, which has finite
relative entropy with respect to the Maxwellian $M$ associated with the spring potential
in the model, we prove, {\em via} a limiting procedure on certain discretization and regularization parameters,
the existence of a global-in-time bounded-energy weak solution
$t \mapsto (\rho(t),\undertilde{u}(t), \psi(t))$
to the coupled Navier--Stokes--Fokker--Planck system, satisfying the initial condition
$(\rho(0),\undertilde{u}(0), \psi(0)) = (\rho_0,\undertilde{u}_0, \psi_0)$.

\bigskip


\bigskip


\end{abstract}

\maketitle

\section{Introduction}
\label{sec:1}

This paper establishes the existence of global-in-time weak solutions
to a large class of bead-spring chain models with finitely
extensible nonlinear elastic (FENE) 
type spring potentials, ---
a system of nonlinear partial differential equations that arises
from the kinetic theory of dilute polymer solutions. The solvent is a
compressible, isentropic, viscous, isothermal Newtonian fluid
confined to a bounded Lipschitz domain $\Omega \subset \mathbb{R}^d$, $d=2$ or $3$, with
boundary $\partial \Omega$. For the sake of simplicity of presentation,
we shall suppose that $\Omega$ has a `solid boundary' $\partial \Omega$;
the velocity field $\ut$ will then satisfy the no-slip boundary condition
$\ut=\zerot$ on $\partial \Omega$.
The equations of continuity and balance of linear momentum
have the form of the compressible Navier--Stokes equations
\red{(cf.    
Lions \cite{Lions2}, Feireisl \cite{Feir1}, Novotn\'{y} \& Stra\v{s}kraba \cite{NovStras}, or  Feireisl \& Novotn\'y \cite{FN})}  
in which the  {\em elastic extra-stress} tensor $\tautt$ (i.e., the
polymeric part of the Cauchy stress tensor) appears as a source
term in the conservation of momentum equation:

Given $T \in \mathbb{R}_{>0}$, find
$ \rho\,:\,(\xt,t)
\in \Omega \times [0,T] \mapsto
\rho(\xt,t) \in {\mathbb R}$ and
$\ut\,:\,(\xt,t)\in
\overline\Omega \times [0,T] \mapsto
\ut(\xt,t) \in {\mathbb R}^d$ 
such that
\begin{subequations}
\begin{alignat}{2}
\frac{\partial \rho}{\partial t}
+ \nabx\cdot(\ut \,\rho)
&= 0 \qquad &&\mbox{in } \Omega \times (0,T],\label{ns0a}\\
\rho(\xt,0)&=\rho_{0}(\xt)    \qquad &&\forall \xt \in \Omega,\label{ns00a} \\
\frac{\partial(\rho\,\ut)}{\partial t}
+ \nabx\cdot(\rho\,\ut \otimes \ut) - \nabx \cdot \Stt(\ut,\rho) + \nabx \,p(\rho)
&= \rho\,\ft + \nabx \cdot \tautt \qquad &&\mbox{in } \Omega \times (0,T],\label{ns1a}\\
\ut &= \zerot               \qquad &&\mbox{on } \partial \Omega \times (0,T],\label{ns3a}\\
(\rho\,\ut)(\xt,0)&=(\rho_0\,\ut_{0})(\xt)    \qquad &&\forall \xt \in \Omega.\label{ns4a}
\end{alignat}
\end{subequations}
It is assumed that each of the equations above has been written in its nondimensional form;
$\rho$ denotes a nondimensional solvent density,
$\ut$ is a nondimensional solvent velocity, defined as the velocity field
scaled by the characteristic flow speed $U_0$.
Here $\Stt(\ut,\rho)$ is the Newtonian part of the viscous stress tensor defined by
\begin{align}
\Stt(\ut,\rho) :=  \mu^S(\rho)\,[\Dtt(\ut) - \frac{1}{d} \,(\nabx \cdot \ut)\, \Itt] +
\mu^B(\rho)
\,(\nabx \cdot \ut) \,\Itt,
\label{Stt}
\end{align}
where $\Itt$ is the $d \times d$ identity tensor,
$\Dtt(\vt) := \frac{1}{2}\,(\nabxtt \vt + (\nabxtt \vt)^{\rm T})$
is the rate of strain tensor,
with  $(\vnabtt)(\xt,t) \in {\mathbb
R}^{d \times d}$ and $\big(\vnabtt\big)_{ij} = \textstyle
\frac{\partial v_i}{\partial x_j}$.
The shear viscosity,
$\mu^S(\cdot) \in \mathbb{R}_{>0}$,
and the bulk viscosity, $\mu^B(\cdot) \in \mathbb{R}_{\geq0}$, of the solvent
are both scaled and, generally, density-dependent.
In addition, $p$ is
the nondimensional pressure satisfying the isentropic equation of state
\begin{align}
p(\rho) = c_p\,\rho^\gamma,
\label{pgamma}
\end{align}
where $c_p \in {\mathbb R}_{>0}$ and the constant $\gamma$ is such that $\gamma > \frac{3}{2}$.

\begin{remark}\label{tait}
For the sake of simplicity of the exposition we focus here on the classical isentropic equation of
state \eqref{pgamma}. Our analysis applies, without alterations, to some other familiar equations
of state, such as the (Kirkwood-modified)
Tait equation of state
\[ p(\rho) = A_0 \left(\frac{\rho}{\rho_\ast}\right)^\gamma - A_1,\]
where $\gamma>\frac{3}{2}$, $A_0$ and $A_1$ are constants, $A_0 - A_1 = p_*$ is the equilibrium reference
pressure, and $\rho_*$ is the
equilibrium reference density. For distilled water, $\gamma \in [5.16,7.11]$
(the actual value being dependent on the
ambient temperature); for carbon tetrachloride (at $30^\circ{\rm C}$) $\gamma = 12.54$,
for glycerine (at $20^\circ{\rm C}$) $\gamma=9.80$; cf. Table 1.1 on p.4 of \cite{Naug}
for values of
$\gamma$, $A_0$ and $A_1$ for other substances.
\end{remark}

On the right-hand side of \eqref{ns1a}, $\ft$ is the nondimensional density of body forces
and $\tautt$ denotes the elastic extra-stress tensor.
In a {\em bead-spring chain model}, consisting of $K+1$ beads coupled with $K$ elastic
springs to represent a polymer chain, 
$\tautt$ is defined by a version of the 
\textit{Kramers expression}
depending on 
the probability
density function $\psi$ of the (random) conformation vector
$\qt := (\qt_1^{\rm T},\dots, \qt_K^{\rm T})^{\rm T}
\in \mathbb{R}^{Kd}$ of the chain (see equation (\ref{tau1}) below), with $\qt_i$
representing the $d$-component conformation/orientation vector of the $i$th spring.
The Kolmogorov equation satisfied by $\psi$ is a second-order parabolic equation,
the Fokker--Planck equation, whose transport coefficients depend
on the velocity field $\ut$, and the hydrodynamic drag coefficient appearing in the Fokker--Planck
equation is, generally, a nonlinear function of the density $\rho$.
The domain $D$ of admissible conformation
vectors $D \subset \mathbb{R}^{Kd}$ is a $K$-fold
Cartesian product $D_1 \times \cdots \times D_K$ of balanced convex
open sets $D_i \subset \mathbb{R}^d$, $i=1,\dots, K$; the term
{\em balanced} means that $\qt_i \in D_i$ if, and only if, $-\qt_i \in D_i$.
Hence, in particular, $\undertilde{0} \in D_i$, $i=1,\dots,K$.
Typically $D_i$ is the whole of $\mathbb{R}^d$ or a bounded open $d$-dimensional ball
centred at the origin $\zerot \in \mathbb{R}^d$ for each $i=1,\dots,K$.
When $K=1$, the model is referred to as the {\em dumbbell model.}

Let $\mathcal{O}_i\subset [0,\infty)$ denote the image of $D_i$
under the mapping $\qt_i \in D_i \mapsto
\frac{1}{2}|\qt_i|^2$, and consider the {\em spring potential}~$U_i
\!\in\! C^1(\mathcal{O}_i; \mathbb{R}_{\geq 0})$,
$i=1,\dots, K$.
Clearly, $0 \in \mathcal{O}_i$. We shall suppose that $U_i(0)=0$
and that $U_i$ is unbounded on
$\mathcal{O}_i$ for each $i=1,\dots, K$.
The elastic spring-force $\Ft_i\,:\, D_i \subseteq \mathbb{R}^d \rightarrow \mathbb{R}^d$
of the $i$th spring in the chain is defined by
\begin{equation}\label{eqF}
\Ft_i(\qt_i) := U_i'(\textstyle{\frac{1}{2}}|\qt_i|^2)\,\qt_i, \qquad i=1,\dots,K.
\end{equation}

The partial Maxwellian $M_i$, associated with the spring potential $U_i$, is defined by
\[
M_i(\qt_i) := \frac{1}{\mathcal{Z}_i} {\rm e}^{-U_i(\frac{1}{2}\,|\qt_i|^2)}, \qquad
\mathcal{Z}_i:= {\displaystyle \int_{D_i} {\rm e}^{-U_i(\frac{1}{2}|\qt_i|^2)} \dq_i},\qquad
i=1,\dots,K.
\]
The (total) Maxwellian in the model is then
\begin{align}
M(\qt) := \prod_{i=1}^K M_i(\qt_i) \qquad \forall
\qt:=(\qt_1^{\rm T},\ldots,\qt_K^{\rm T})^{\rm T} \in D
:= \bigtimes_{i=1}^K D_i.
\label{MN}
\end{align}
Observe that, for $i=1, \dots, K$,
\begin{subequations}
\begin{equation}
M(\qt)\,\nabqi [M(\qt)]^{-1} = - [M(\qt)]^{-1}\,\nabqi M(\qt) =
\nabqi \red{\big(}U_i(\textstyle{\frac{1}{2}}|\qt_i|^2)\red{\big)}
=U_i'(\textstyle{\frac{1}{2}}|\qt_i|^2)\,\qt_i, \label{eqM}
\end{equation}
and, by definition,
\begin{align}
\int_D M(q) \dq = 1.
\label{Mint1}
\end{align}
\end{subequations}

\begin{example}
\label{ex1.1} \em
In the Hookean dumbbell model $K=1$, and the spring force
is defined by ${\Ft}({\qt}) = {\qt}$, with ${\qt} \in {D}=\mathbb{R}^d$,
corresponding to ${U}(s)= s$, $s \in \mathcal{O} = [0,\infty)$.
More generally, in a Hookean bead-spring chain model, $K \geq 1$, $\Ft_i(\qt_i) = \qt_i$,
corresponding to $U_i(s) = s$, $i=1,\dots,K$, and $D$ is
the Cartesian product of $K$ copies of $\mathbb{R}^d$. The associated Maxwellian is
\[ M(\qt) = M_1(\qt_1)\cdots M_K(\qt_K) = \frac{1}{\mathcal{Z}} {\rm e}^{-\frac{1}{2}|\qt|^2},\]
with $|\qt|^2:= |\qt_1|^2 + \cdots +|\qt_K|^2$ 
and $\mathcal{Z} := \mathcal{Z}_1 \cdots \mathcal{Z}_K = (2\pi)^{{Kd}/{2}}$.
Hookean dumbbell and Hookean bead-spring chain models are physically unrealistic
as they admit arbitrarily large extensions.
$\quad\diamond$
\end{example}

A more realistic class of models assumes that the springs in the bead-spring chain have
finite extension: the domain $D$ is then taken to be a Cartesian product of $K$
{\em bounded} open balls $D_i \subset \mathbb{R}^d$, centred at the
origin $\zerot \in \mathbb{R}^d$, $i=1,\dots, K,$ with $K \geq 1$. The spring potentials
$U_i\, :\, s \in [0,\frac{b_i}{2}) \mapsto U_i(s) \in [0,\infty)$,
with $b_i>0$, $i=1,\dots, K$, are in that case
nonlinear and unbounded functions, and the associated bead-spring chain model is referred
to as a FENE (finitely extensible nonlinear elastic) model; in the case of $K=1$, the
corresponding model is called a FENE dumbbell model.

Here we shall be concerned with finitely extensible nonlinear 
bead-spring chain
models, with
$D:= B(\zerot,b_1^{\frac{1}{2}}) \times \cdots \times
B(\zerot, b_K^{\frac{1}{2}})$, where $b_i > 0$, $i=1,\dots,K$, $K \geq 1$, and $B(\zerot, b_i^{\frac{1}{2}})$
is a bounded open ball in $\mathbb{R}^d$ of radius $b_i^{\frac{1}{2}}$,
centred at $\zerot \in \mathbb{R}^d$. We shall adopt the following
structural hypotheses on the spring potentials
$U_i$ and the associated partial Maxwellians $M_i$, $i=1,\dots, K$.


We shall assume that $D_i = B(0,b^{\frac{1}{2}}_i)$ with $b_i>0$ for $i=1,\dots, K$, and that
for $i=1,\dots, K$ there exist constants $c_{ij}>0$,
$j=1, 2, 3, 4$, and $\theta_i > 1$
such that the spring potential $U_i \in C^1[0,\frac{b_i}{2})$ and the associated partial Maxwellian $M_i$ satisfy
\eqlabstart
\begin{eqnarray}
&& ~\hspace{6mm}c_{i1}\,[\mbox{dist}(\qt_i, \,\partial D_i)]^{\theta_i}
\leq  M_i(\qt_i) \,
\leq
c_{i2}\,[\mbox{dist}(\qt_i, \,\partial D_i)]^{\theta_i} \qquad \forall \qt_i \in
D_i, \label{growth1}\\
&& ~\hspace{15mm}c_{i3} \leq \mbox{dist}(\qt_i,\,\partial D_i) \,U_i'
(\textstyle{\frac{1}{2}}|\qt_i|^2)
\leq c_{i4}\qquad \forall \qt_i \in D_i. \label{growth2}
\end{eqnarray}
\eqlabend
%
It follows from \red{(\ref{growth1},b)} that (if $\theta_i >1$,
as has been assumed here,)
\begin{equation}\label{additional-1}
\int_{D_i} \left[1 + [U_i(\textstyle{\frac{1}{2}}|\qt_i|^2)]^2
+ [U_i'(\textstyle{\frac{1}{2}}|\qt_i|^2)]^2\right] M_i(\qt_i) \, \dd
\qt_i < \infty, \qquad i=1, \dots, K.
\end{equation}

\medskip

\begin{example}
\label{ex1.2} \em
In the FENE (finitely extensible nonlinear elastic)
dumbbell model, introduced by Warner \cite{Warner:1972}, $K=1$ and the spring force is given by
$
{\Ft}({\qt}) = (1 - |{\qt}|^2/b)^{-1}\,{\qt}$,
$\qt \in D = B(\zerot,b^{\frac{1}{2}})$,
corresponding to ${U}(s) = - \frac{b}{2}\log \left(1-\frac{2s}{b}\right)$,
$s \in \mathcal{O} = [0,\frac{b}{2})$, $b>2$.
%
More generally, in a FENE bead spring chain, one considers $K+1$ beads linearly
coupled with $K$ springs, each with a FENE spring potential.
Direct calculations show that the partial Maxwellians $M_i$ and the
elastic potentials $U_i$, $i=1,\dots, K$, of the FENE bead spring chain satisfy
the conditions (\ref{growth1},b) with $\theta_i:= \frac{b_i}{2}$,
provided that $b_i>2$, $i=1,\dots, K$. Thus, (\ref{additional-1}) also holds
and $b_i>2$, $i=1,\dots, K$. Note, however, that \eqref{additional-1}
fails for $b_i \in (0,2]$, i.e., for $\theta_i \in (0,1]$,
which is why we have assumed in the statement of (\ref{growth1},b)
that $\theta_i>1$ for $i=1,\dots,K$.

It is interesting to note that in the (equivalent)
stochastic version of the FENE dumbbell model ($K=1$)
a solution to the system of
stochastic differential equations
associated with the Fokker--Planck equation exists and has
trajectorial uniqueness if, and only if, $\frac{b}{2}\geq 1$;
(cf. Jourdain, Leli\`evre \& Le Bris \cite{JLL2} for details).
Thus, in the general class of FENE-type bead-spring chain models considered here,
the assumption $\theta_i > 1$, $i=1,\dots, K$, is the weakest reasonable
requirement on the decay-rate of $M_i$  in (\ref{growth1}) as
$\mbox{dist}(\qt_i,\partial D_i) \rightarrow 0$.
\end{example}

The governing equations of the general nonhomogeneous bead-spring chain
models with centre-of-mass diffusion considered in this paper are
(\ref{ns0a}--e), where the extra-stress tensor $\tautt$ is defined by
\begin{equation}\label{tautot0}
\tautt(\psi)(\xt,t)
:= \tautt_1(\psi)(\xt,t) - \left(\int_{D\times D}\gamma(\qt,\qt')\,
\psi(\xt,\qt,t) \,\psi(\xt,\qt',t)\, \dq \,\dd \qt'\right) \;\Itt,
\end{equation}
$\gamma\,:\, D \times D \rightarrow \mathbb{R}_{\geq 0}$ is a smooth, time-independent,
$\xt$-independent and $\psi$-independent interaction kernel, which we shall henceforth consider to be
\[ \gamma(\qt,\qt') \equiv \mathfrak{z},\]
where $\mathfrak{z} \in {\mathbb R}_{>0}$; thus,
\begin{equation}\label{tautot}
\tautt(\psi)
:= \tautt_1(\psi)
-
\mathfrak{z}\left(\int_{D}\psi
\dq\right)^2 \Itt.
\end{equation}
Here, $\tautt_1(\psi)$ is the \textit{Kramers expression}; that is,
\begin{equation}
\label{tau1} \tautt_1(\psi):=
k\, \left[\left(\sum_{i=1}^K
\Ctt_i(\psi) \right) - (K+1)\int_D \psi \dq \;\Itt
\right],
\end{equation}
where $k \in {\mathbb R}_{>0}$, \red{with the first term in the square brackets being due to the $K$ springs and the second to the $K+1$ beads
in the bead-spring chain representing the polymer molecule. Further,}
\begin{align}
\Ctt_i(\psi)(\xt,t) &:= \int_{D} \psi(\xt,\qt,t)\, U_i'({\textstyle
\frac{1}{2}}|\qt_i|^2)\,\qt_i\,\qt_i^{\rm T} \dq, \qquad
i=1,\ldots,K. 
\label{eqCtt}
\end{align}

The probability density function $\psi$ is a solution of the
Fokker--Planck (forward Kolmogorov) equation
\begin{align}
\label{fp0}
&\frac{\partial \psi}{\partial t} + \nabx\,\cdot\,(\ut\,\psi) +
\sum_{i=1}^K \nabqi
\cdot \left(\sigtt(\ut) \, \qt_i\,\psi \right)
\nonumber
\\
\bet
&\hspace{0.1in} =
\epsilon\,\Delta_x\left(\frac{\psi}{\zeta(\rho)}\right) +
\frac{1}{4 \,\lambda}\,
\sum_{i=1}^K \sum_{j=1}^K
A_{ij}\,\nabqi \cdot \left(
M
\,\nabqj \left(\frac{\psi}{\zeta(\rho)\, M}\right)\right) \quad \mbox{in } \Omega \times D \times
(0,T],
\end{align}
with $\sigtt(\vt) \equiv \nabxtt \vt$ and a, generally,
density-dependent scaled drag coefficient $\zeta(\cdot)\in \Rplus$.
A concise derivation of the Fokker--Planck equation \eqref{fp0} can be found
in Section 1 of Barrett and S\"{u}li \cite{BS-2011-density-arxiv}.
Let
$\partial\overline{D}_i := D_1 \times \cdots \times D_{i-1}
\times \partial D_i \times D_{i+1} \times \cdots \times D_K$.
We impose the following boundary and initial conditions on solutions of \eqref{fp0}:
\begin{subequations}
\begin{align}
&\left[\frac{1}{4\,\lambda} \sum_{j=1}^K A_{ij}\,
M\,\nabqj\!\left(\frac{\psi}{\zeta(\rho)\,M}\right)
- \sigtt(\ut) \,\qt_i\,
\psi
\right]\! \cdot \frac{\qt_i}{|\qt_i|}
=0\qquad ~~\nonumber\\
&~ \hspace{1.49in}\mbox{on }
\Omega \times \partial \overline{D}_i\times (0,T],
\mbox{~~for $i=1,\dots, K$,} \label{eqpsi2aa-x}\\
&\epsilon\,\nabx \left(\frac{\psi}{\zeta(\rho)}\right)\,\cdot\,\nt =0
\qquad\mbox{on }
\partial \Omega \times D\times (0,T],\label{eqpsi2ab-x}\\
&\psi(\cdot,\cdot,0)=\psi_{0} (\cdot,\cdot) \geq 0
\qquad \!\!\mbox{on $\Omega\times D$},\label{eqpsi3ac-x}
\end{align}
\end{subequations}
where $\qt_i$ is normal to $\partial D_i$, as $D_i$
is a bounded ball centred at the origin,
and $\nt$ is normal to $\partial \Omega$.

The first term in (\ref{tau1}) is due to the $K$ springs, whilst the second is due to
the $K+1$ beads; see Chapter 15 in \cite{BCAH}.
The nondimensional constant $k>0$ featuring in \eqref{tau1} is a constant multiple
of the product of
the Boltzmann constant $k_B$ and the absolute temperature $\mathtt{T}$.
In \eqref{fp0}, $\varepsilon>0$ is the centre-of-mass diffusion coefficient defined
as $\varepsilon := (\ell_0/L_0)^2/(4(K+1)\lambda)$ with
$L_0$ a characteristic length-scale  of the solvent flow,
$\ell_0:=\sqrt{k_B \mathtt{T}/\mathtt{H}}$ signifying the characteristic microscopic length-scale
and $\lambda :=\frac{\zeta_0 U_0}{4\mathtt{H}L_0 }$,
where $\zeta_0>0$ is a characteristic drag coefficient and $\mathtt{H}>0$ is a spring-constant.
The nondimensional parameter $\lambda \in \Rplus$, called the Deborah number
(and usually denoted by $\mathsf{De}$), characterizes
the elastic relaxation property of the fluid, and
$\Att=(A_{ij})_{i,j=1}^K$ is the symmetric positive definite
\textit{Rouse matrix}, or connectivity matrix;
for example, $\Att = {\tt tridiag}\left[-1, 2, -1\right]$ in the case
of a (topologically) linear chain, depicted in Fig.
\ref{modelPoly}; see, Nitta \cite{Nitta}. Concerning these scalings and notational
conventions, we remark that the factor $\frac{1}{4\lambda}$ in equation \eqref{fp0} above
appears as a factor $\frac{1}{2\lambda}$ in the Fokker--Planck equation in our earlier papers
\cite{BS2011-fene,BS2010-hookean,BS2011-feafene,BS2010}.

Two remarks are in order at this point concerning the extra-stress tensor \eqref{tautot}.

\begin{remark}\label{rem-1.2}
By comparing the last term in the expression (A) in Table 15.2-1 on p.156 in
\cite{BCAH} with the second term in \eqref{tau1}, we observe that Bird et al. write $K$
(i.e., $(N-1)$ in their notation)
instead of our prefactor $(K+1)$ (i.e., $N$ in terms of their notation),
in the context of incompressible spatially homogeneous flows (i.e., the velocity field is assumed to be independent of $\xt$ there). As we shall explain in equations \eqref{energy-id1}--\eqref{energy-id4} below, the analysis in the present paper applies to a general class of extra stress tensors \eqref{tautot}, with Kramers type expressions of the form
\begin{equation}
\label{tau1a} \tautt_1(\psi):=
k\, \left[\left(\sum_{i=1}^K
\Ctt_i(\psi) \right) - \mathfrak{k}\int_D \psi \dq \;\Itt
\right],
\end{equation}
with $\mathfrak{k} \in \mathbb{R}$ and $\mathfrak{z}>0$. Since the actual value of $\mathfrak{k}$ is of no particular relevance in our analysis, we took $\mathfrak{k}=K+1$ in the second term in \eqref{tau1a}, yielding \eqref{tau1}, as this choice simplifies the expressions that arise in the course of the proof. As will be shown below, if $\mathfrak{k} \geq K+1$ and $\mathfrak{z}\geq 0$ a formal energy identity holds, and if $\mathfrak{k} < K+1$ and $\mathfrak{z}>0$, one can still prove a formal energy inequality. In contrast with this, in the case of \textit{incompressible} flows, $\mathfrak{k}$ and $\mathfrak{z}$ are of no relevance in the analysis and can be any two real numbers; indeed, in \cite{BS2011-fene,BS2010-hookean,BS-2012-density-JDE}, $\mathfrak{k}=K$ and $\mathfrak{z}=0$ were the values used in \eqref{tau1a} and
\eqref{tautot}, respectively, resulting in $\tautt(\psi)$ that is the classical Kramers expression for the extra stress tensor.
\end{remark}

\begin{remark}\label{rem-1.3}
Our second remark concerns the quadratic modification to $\tautt_1(\psi)$
appearing as the second term in {\rm(\ref{tautot})}. By defining the \textit{polymer number density}
\begin{equation}\label{pdn}
\vrho(\xt,t):= \int_D \psi(\xt,\qt,t) \dd \qt,\qquad (\xt,t) \in \Omega \times [0,T],
\end{equation}
formally integrating \eqref{fp0} over $D$ and using the boundary condition \eqref{eqpsi2aa-x},
we deduce that
\begin{subequations}
\begin{equation}\label{pdnumber1}
\frac{\partial \vrho}{\partial t} + \nabx \cdot (\ut\, \vrho) =
\varepsilon \,\Delta_x \left(\frac{\vrho}{\zeta(\rho)}\right)
\quad \mbox{on $\Omega \times (0,T]$},
\end{equation}
together with the boundary 
and initial conditions (which result from integrating (\ref{eqpsi2ab-x},c) over $D$)
\begin{equation}\label{pdnumber2}
\epsilon\,\nabx \left(\frac{\vrho}{\zeta(\rho)}\right)\,\cdot\,\nt =0
\quad\mbox{on }
\partial \Omega \times (0,T]
\qquad\mbox{and}\qquad
\vrho(\xt,0)=\int_D \psi_0(\xt,\qt) \dd \qt\quad \mbox{for $\xt \in \Omega$}.
\end{equation}
\end{subequations}
In order to avoid potential confusion concerning our notational conventions, we draw the reader's attention to the fact that, here and throughout the rest of the paper, the symbol $\rho$ signifies the density of the solvent, while the symbol $\varrho$ denotes the polymer number density, as defined in \eqref{pdn}.

If $\nabx \cdot \ut \equiv 0$
and $\vrho(\cdot,0)$ is constant, and either $\varepsilon=0$
or $\zeta(\rho)$ is independent of $\rho$
(and therefore identically
equal to a constant), then $\vrho(\xt,t)$ is constant ($\equiv \vrho(\cdot,0)$),
for all $(\xt,t) \in \Omega \times (0,T]$,
and hence
\[\int_D \psi(\xt,\qt,t) \dd \qt = \int_D \psi_0 (\xt,\qt) \dd \qt \in {\mathbb R}_{>0} \qquad
\mbox{for all $(\xt,t) \in \Omega \times (0,T]$};\]
in other words, the polymer number density is constant.
This conservation property then guarantees complete control of
$\vrho(\xt,t) = \int_D \psi(\xt,\qt,t)
\dd \qt$
(when $\mathfrak{z}=0$, and a fortiori, for $\mathfrak{z}>0$) in terms of the initial
probability density function $\psi_0$ in the course of the
weak compactness argument upon which the proof of existence of weak solutions rests (cf. \cite{BS2011-fene}
for the analysis in the case of $\varepsilon>0$, constant $\zeta$ and
 $\mathfrak{z}=0$). In particular,
the time derivative $\frac{\partial \psi}{\partial t}$
can be bounded in a sufficiently strong norm to enable the application of an
Aubin--Lions--Simon type compactness
theorem that ensures strong convergence of the sequence of approximations to the
probability density function $\psi$.

In the case of nonhomogeneous flows with density-dependent drag, the situation is
more complicated because of the consequential weaker control (in the sense of norms) on
the function $\vrho$: while, thanks to {\rm(\ref{pdnumber1},b)} 
and the assumed homogeneous Dirichlet boundary condition on $\ut$, it is still true that
%
\[ \int_{\Omega} \vrho(\xt,t)\dd \xt = \int_{\Omega \times D} \psi(\xt,\qt,t)\dd \xt \dd \qt =
\int_{\Omega \times D}\psi_0(\xt,\qt) \dd \xt \dd \qt = \int_{\Omega} \vrho(\xt,0) \dd \xt\]
for all $t \in (0,T]$, the availability of such a weak conservation property only
further complicates the analysis.
In
\cite{BS-2012-density-JDE} (in the case of $\nabx \cdot \ut \equiv 0$ and $\mathfrak{z}=0$)
a more involved argument, based on a combination of
compensated compactness and Vitali's theorem, was therefore
used in order to deduce strong convergence of the sequence of approximations to $\psi$.

\red{In the present paper, in order to focus on the essential new difficulty ---
 the lack of the divergence-free property of $\ut$ ---
we shall suppose that the drag coefficient in the Fokker--Planck equation is identically
equal to a constant, which we shall henceforth, without loss of generality, assume to be equal
to $1$, i.e., $\zeta(\rho)\equiv 1$. As in the majority of
contributions to the mathematical analysis of the compressible Navier--Stokes equations to date,
we shall also assume that the shear and bulk viscosity coefficients
are independent of the density $\rho$, i.e., that
$\mu^S \in \mathbb{R}_{>0}$ and $\mu^B \in \mathbb{R}_{\geq 0}$. For the analysis of the compressible Navier--Stokes equations in the
case of a particular family of density-dependent shear and bulk viscosities, we refer the reader to \cite{Bre1,Bre2}.} Setting
$\mathfrak{z}=0$ in this context results in the loss of a bound on the $L^1(0,T;\mathfrak{X}')$ norm of the time derivative
of the probability density function $\psi$, for any reasonable choice of the function space $\mathfrak{X}$.
Failure to control $\frac{\partial \psi}{\partial t}$ \red{or a time-difference of $\psi$}
in even such a weak sense brings into question the
meaningfulness of the model in the case of compressible flows for solutions of as low a degree of
regularity as is guaranteed by the formal energy bound in the case of $\mathfrak{z}=0$,
see \red{Remarks \ref{z0rem} and \ref{rem-z}} below.
Motivated by the papers of Constantin \cite{CON}, Constantin et al. \cite{CFTZ} and
Bae \& Trivisa \cite{BaeKon}, we have therefore included the quadratic term in \eqref{tautot},
with $\mathfrak{z}>0$. As we shall show later on, inclusion of the
quadratic term into {\rm(\ref{tautot})} does not destroy energy balance thanks to the fact that
the polymer number
density function $\vrho$ satisfies the initial-boundary-value problem
{\rm(\ref{pdnumber1},b)}, 
and has the
beneficial effect of guaranteeing $L^\infty(0,T;L^2(\Omega))\cap L^2(0,T;H^1(\Omega))$
norm control for $\vrho$, rendering the time derivative of the probability density function finite
in the norm of $L^2(0,T;\mathfrak{X}')$, with a suitable choice of $\mathfrak{X}$
as a Maxwellian-weighted Sobolev
space of sufficiently high order.
This then enables the application of Dubinsk\u{\i}i's extension of the
Aubin--Lions--Simon compactness theorem (cf. \cite{Dub} and Barrett \& S\"{u}li \cite{BS-DUB}).

From the physical point of view \eqref{Stt}, \eqref{tautot}--
\eqref{eqCtt}
can be seen as a decomposition of the Cauchy stress $\pi$ as the sum of a contribution from the
solvent, $\pi_s$, and the polymeric extra stress, $\pi_p$,
resulting from the presence of the polymer molecules,
which are idealized here as bead-spring-chains (cf. eq. {\rm (13.3-1)} in \cite{BCAH}):
\begin{eqnarray*}
\underdtilde{\pi} &=& \underdtilde{\pi_s} + \underdtilde{\pi_p}
= (\underdtilde{S}(\ut,\rho) - p_s \,\underdtilde{I})
+ \left(k \sum_{i=1}^K \underdtilde{C}_i(\psi) - p_p\, \underdtilde{I}\right),
\end{eqnarray*}
where
%
$p_s = p =c_p \,\rho^\gamma$ 
(with $c_p>0$, $\gamma>\frac{3}{2}$, and $\rho$ denoting the density) is the fluid pressure, and
$p_p 
= k\,(K+1)\, \vrho 
+ \mathfrak{z}\,\vrho^2$  
(with $k$ denoting a constant multiple of the product of the Boltzmann constant and the
absolute temperature, $\mathfrak{z}>0$ and $\vrho$
signifying the polymer number density) is the polymeric contribution to the total
pressure, defined as $p_s + p_p$.
\red{
The expression $k\,(K+1)\, \vrho + \mathfrak{z}\,\vrho^2$ can be viewed as a quadratic truncation of
a virial expansion of the polymeric pressure $p_p$ in terms of the polymer number density $\varrho$.}
\end{remark}

\begin{definition}\label{model-p}
The collection of equations and structural hypotheses
{\rm(\ref{ns0a}--e)}--
{\rm(\ref{eqpsi2aa-x}--c)} 
together with the assumption that the Rouse matrix $A$ is symmetric and positive definite  (as is always the case, by definition,)
will be referred to throughout the paper as model
$({\rm P})$, or as the {\em compressible FENE-type bead-spring chain model with
centre-of-mass diffusion}.
It will be assumed throughout the paper that the shear viscosity, $\mu^S \in \mathbb{R}_{>0}$,
the bulk viscosity,
$\mu^B \in \mathbb{R}_{\geq 0}$, and the drag coefficient, $\zeta \in \mathbb{R}_{>0}$,
are independent of the density $\rho$. For the ease of exposition we shall set $\zeta \equiv 1$.
\end{definition}
%

%
\begin{figure}[t]\label{modelPoly}
    \includegraphics[width=0.55\textwidth]{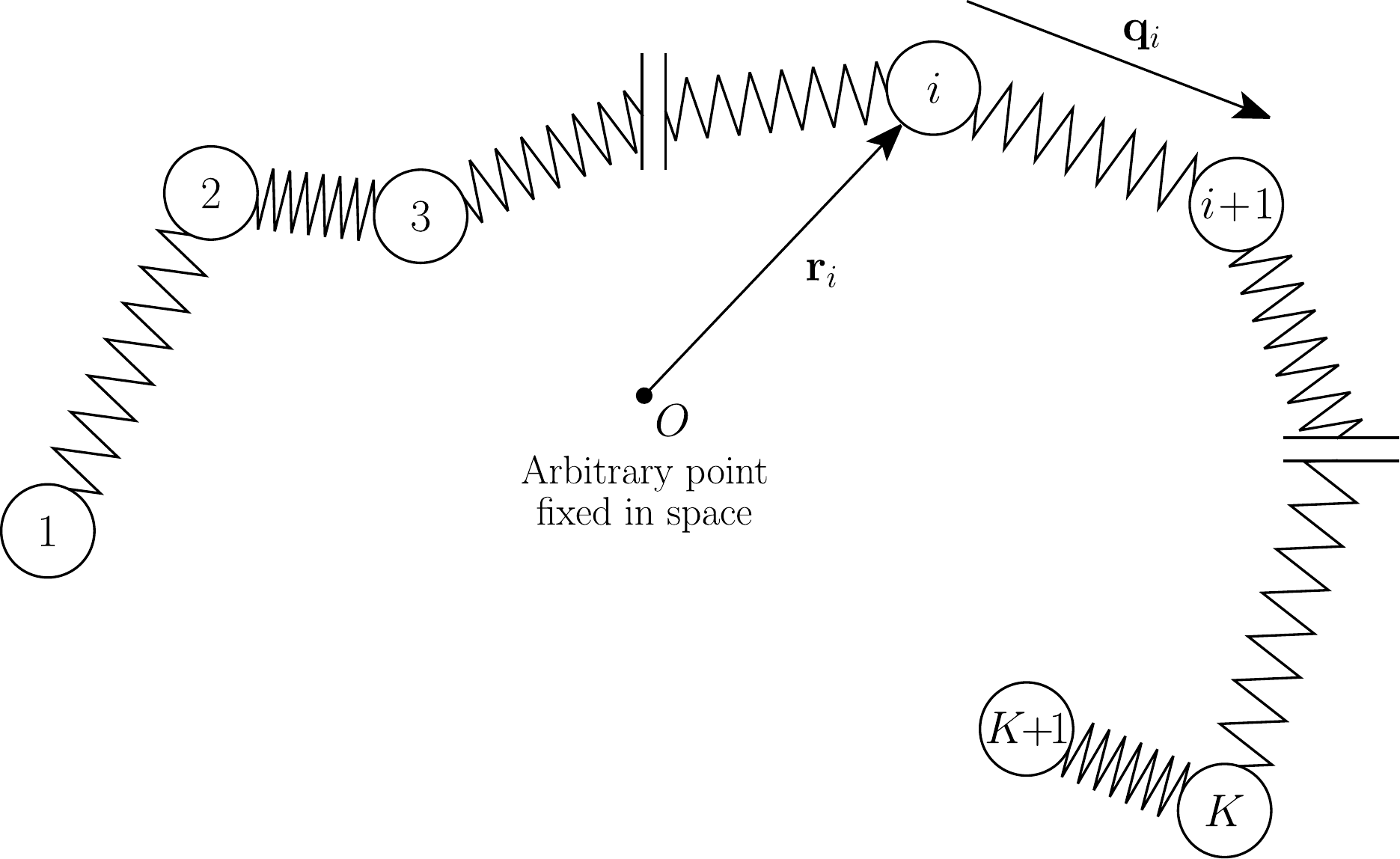}
    \caption{\vspace{0in}A (topologically) linear bead-spring chain with $K$ springs and $K+1$ beads.}
\end{figure}
%

For a survey of recent developments concerning the mathematical analysis of
Navier--Stokes--Fokker--Planck systems in the case of \textit{incompressible} polymeric flows
the reader is referred to the introductory sections in our papers
\cite{BSS,BS,BS2,BS2011-fene,BS-2011-density-arxiv,BS2010-hookean,BS-2012-density-JDE}.
The convergence of finite element approximations of Navier--Stokes--Fokker--Planck systems
is discussed in
\cite{Barrett-boyaval-09,BS4,BS2011-feafene}, where we also
review the literature on the mathematical
theory of numerical algorithms for these equations.
For further details concerning the mathematical analysis of kinetic models of polymers in the case of
incompressible fluids we refer the reader to
Renardy \cite{R}, Lions \& Masmoudi \cite{LM},
E, Li \& Zhang \cite{E} and Li, Zhang \& Zhang \cite{LZZ},
Jourdain, Leli\`evre \& Le Bris \cite{JLL2},
Constantin \cite{CON},
Du, Yu \& Liu \cite{DU},
Yu, Du \& Liu \cite{YU},
Zhang \& Zhang \cite{ZZ},
Lions \& Masmoudi \cite{LM2},
Masmoudi \cite{M}, \cite{M10},
Otto \& Tzavaras \cite{OT}, and the references therein.

A noteworthy feature of equation (\ref{fp0}) in the model $({\rm P})$
compared to classical Fokker--Planck equations for bead-spring-chain models
for dilute polymers appearing in the literature is the presence of the
$\xt$-dissipative centre-of-mass diffusion term $\varepsilon
\,\Delta_x \psi$ on the right-hand side of the Fokker--Planck equation (\ref{fp0}).
We refer to Barrett \& S\"uli \cite{BS} for the derivation of (\ref{fp0})
in the case of $K=1$ and constant $\rho$;
see also the article by Schieber \cite{SCHI}
concerning generalized dumbbell models with centre-of-mass diffusion,
and the recent paper of Degond \& Liu \cite{DegLiu} for a careful justification
of the presence of the centre-of-mass diffusion term through asymptotic analysis.
In the case of variable density, viscosity and drag, the derivation of the Fokker--Planck
equation with centre-of-mass diffusion appears in \cite{BS-2011-density-arxiv}
(cf. also \cite{BS2010-hookean,BS-2012-density-JDE}).

For a survey of macroscopic models of \textit{compressible} viscoelastic flow,
the reader is referred to the paper
by Bollada \& Phillips \cite{Boll-Phil-2012}. The existence and uniqueness of local
strong solutions and the existence of global solutions near equilibrium for macroscopic models of
three-dimensional compressible viscoelastic fluids was considered in \cite{Hu-Wang1,Qian-Zhang,
Qian,Hu-Wang2,Hu-Wang3,Hu-Wu}. The existence of measure-valued solutions to
non-{N}ewtonian compressible, isothermal, monopolar fluid flow models was studied
by Ne\v{c}asov\'{a} in \cite{Necasova1,Necasova2};
for bipolar isothermal non-{N}ewtonian compressible
fluids related analysis was pursued in \cite{Necasova3}.
In a series of papers (cf. \cite{Mamontov1,Mamontov2,Mamontov3})
Mamontov developed a priori estimates for two- and three-dimensional
compressible nonlinear viscoelastic flow problems and studied the existence of solutions.
Zhikov \& Pastukhova \cite{Zhik-Past} proved the existence of global weak solutions to a
class of
compressible viscoelastic flow models with $p$-Laplacian structure. There is also a substantial literature in chemical
engineering on the use of the compressible Oldroyd-B system in modelling bubble dynamics in compressible viscoelastic
liquids (cf., for example, \cite{brujan}). Closer to the subject of the present paper,
Bae \& Trivisa \cite{BaeKon} have established the existence of global weak solutions to
Doi's rod-model in three-dimensional bounded domains. The model concerns suspensions of
rod-like molecules in compressible fluids and involves the coupling of a
Fokker--Planck type equation with the compressible Navier--Stokes system. Also,
Jiang, Jiang \& Wang \cite{Jiang-Jiang-Wang} have studied the existence of
global weak solutions to the equations of compressible flow of nematic liquid crystals in two dimensions.

Our objective in the present paper is to
prove the existence of global-in-time weak solutions to the general class of models of compressible
viscoelastic flow, labelled above as Problem (P) (cf. Definition \ref{model-p}), that arise from
the kinetic theory of dilute solutions of nonhomogeneous polymeric liquids where polymer molecules
are idealized as bead-spring chains, with finitely extensible nonlinear elastic (FENE) type
spring potentials, involving the coupling of the
unsteady, isentropic, compressible Navier--Stokes equations with the Fokker--Planck equation.
Despite their importance, in the present paper we shall, for the sake of simplicity,
neglect all thermal effects and will focus instead on mechanical properties of the fluid in the isothermal setting.
With a nonnegative initial density $\rho_0 \in L^\infty(\Omega)$ for the continuity equation;
a square-integrable initial velocity datum $\undertilde{u}_0$ for the Navier--Stokes momentum
equation; and
a nonnegative initial probability density function $\psi_0$ for the Fokker--Planck equation,
which has finite
relative entropy with respect to the Maxwellian $M$ associated with the spring potential
in the model, we prove, {\em via} a limiting procedure on certain discretization and regularization parameters,
the existence of a global-in-time bounded-energy weak solution $t \mapsto
(\rho(t),\undertilde{u}(t), \psi(t))$
to the coupled Navier--Stokes--Fokker--Planck system, satisfying the initial condition
$(\rho(0),\undertilde{u}(0), \psi(0)) = (\rho_0,\undertilde{u}_0, \psi_0)$.
Since the argument is long and technical, we give a brief overview of the main steps of the proof.

At the heart of the proof is a formal energy identity, which we shall now derive
under the assumption that $\ut$, $\rho$, $\psi$ and $\varrho$ are sufficiently smooth, and, at least for
our purposes in this introductory section, $\rho$ is nonnegative, and
$\psi$ and $\varrho$ are positive. Instead of \eqref{tau1}, used in the rest of the paper, we shall make use of the more general formula \eqref{tau1a}, in order to explain the admissible range of $\mathfrak{k}$ alluded to in Remark \ref{rem-1.2} as well as our reasons for choosing $\mathfrak{k}=K+1$ in \eqref{tau1}.
By taking the $L^2(\Omega)$
inner product of equation \eqref{ns0a} first with $\frac{1}{2}|\ut|^2$
and then with $P'(\rho)$, where $P(\rho)=\frac{p(\rho)}{\gamma-1}$ and so,
on noting (\ref{pgamma}), $\rho\,P'(\rho)-P(\rho)=p(\rho)$, and then taking the
$\undertilde{L}^2(\Omega)$ inner product
of equation \eqref{ns1a} with $\ut$, we deduce upon partial integration and noting the homogeneous Dirichlet
boundary condition on $\ut$ and (\ref{tautot}), (\ref{eqCtt}) and (\ref{tau1a})  that
%
\begin{eqnarray}\label{energy-id1}
&&\frac{\dd}{\dd t} \int_\Omega \left[\frac{1}{2}\,\!\rho\, |\ut|^2 + P(\rho)\right]{\rm d} \xt
+  \mu^{S} \int_\Omega \left|\underdtilde{D}(\ut)
- \frac{1}{d}\,(\nabx \cdot \ut)\,\Itt\right|^2{\rm d} \xt
+ \mu^B \int_\Omega |\nabx \cdot \ut|^2 \dd \xt\nonumber\\
&&\qquad\qquad = \int_\Omega \rho \,\ft \cdot \ut \,\dd \xt - \int_\Omega \tautt :
\nabxtt \ut\, \dd \xt
\nonumber\\
&&\qquad\qquad = \int_\Omega \rho\, \ft \cdot \ut\, \dd \xt - k \sum_{i=1}^K \int_{\Omega \times D}
\psi   
\,\, U'_i({\textstyle \frac{1}{2}}|\qt_i|^2) \,\qt_i \,\qt_i^{\rm T} :
\nabxtt \ut \, \dd \qt \,\dd \xt
\nonumber\\
&&\qquad\qquad \quad + k \,\mathfrak{k} \int_{\Omega \times D} \psi 
\,(\nabx \cdot \ut)
\,\dd \qt \,\dd \xt
+ \mathfrak{z} \int_{\Omega} \left[\int_{D} \psi 
\,\dd \qt\right]^2  (\nabx \cdot \ut)
\dd \xt.
\end{eqnarray}
Recalling our convention that $\zeta \equiv 1$ (cf. Definition
\ref{model-p}), it follows by taking the $L^2(\Omega)$ inner product of \eqref{pdnumber1} with $\vrho$,
partial integration, noting the boundary condition \eqref{pdnumber2} and that $\ut$ satisfies a
homogeneous Dirichlet boundary condition on $\partial\Omega$, that
\begin{equation}\label{energy-id2}
\frac{\dd}{\dd t}\int_\Omega \vrho^2 
\dd \xt + 2\varepsilon \int_\Omega |\nabx \vrho 
|^2 \dd \xt
= - \int_\Omega \vrho^2
(\nabx \cdot \ut) \dd \xt.
\end{equation}
By noting (\ref{pdn}) and substituting \eqref{energy-id2} into the last term on
the right-hand side of \eqref{energy-id1},
we deduce that
\begin{eqnarray}
&&\frac{\dd}{\dd t} \int_\Omega \left[\frac{1}{2}\,\!\rho\, |\ut|^2 + P(\rho)
+ \mathfrak{z} \,\vrho^2\right]{\rm d} \xt\nonumber\\
&&\qquad\qquad\quad+\,  \mu^{S} \int_\Omega \left|\underdtilde{D}(\ut)
- \frac{1}{d}\,(\nabx \cdot \ut)\,\Itt\right|^2{\rm d} \xt
+ \mu^B \int_\Omega |\nabx \cdot \ut|^2 \dd \xt +
2\varepsilon \,\mathfrak{z}\int_\Omega |\nabx \vrho 
|^2 \dd \xt\nonumber\\
\label{energy-id3}
&& \qquad\qquad = \int_\Omega \rho \,\ft \cdot \ut \,\dd \xt
- k \sum_{i=1}^K \int_{\Omega \times D}
\psi  
\,U'_i({\textstyle \frac{1}{2}}|\qt_i|^2) \,\qt_i \,\qt_i^{\rm T} :
\nabxtt \ut  \,\dd \qt \dd \xt
\nonumber\\
&&\qquad\qquad \quad +\, k \,\mathfrak{k} \int_{\Omega \times D} \psi 
\,(\nabx \cdot \ut)
\dd \qt \dd \xt.
\end{eqnarray}
Next, the Fokker--Planck equation \eqref{fp0} (with $\zeta \equiv 1$)
is multiplied by $\log\frac{\psi}{M}$, integrated over $\Omega \times D$,
and integrations by parts are performed, using the boundary conditions \eqref{ns3a},
\eqref{eqpsi2aa-x} and \eqref{eqpsi2ab-x},
in the second and third term on the left-hand side and the two terms on the right-hand side; hence, by
letting $\mathcal{F}(s) = s (\log s - 1) + 1$ for $s>0$ and $\mathcal{F}(0):=\lim_{s \rightarrow 0_+}\mathcal{F}(s)=1$,
\begin{eqnarray}\label{energy-id3a}
&&\frac{\dd}{\dd t} \int_{\Omega \times D} M \,\mathcal{F}\left(\frac{\psi}{M}\right) \dd \qt \dd \xt
+ \int_{\Omega \times D} \psi 
\,(\nabx \cdot \ut) \dd \qt \dd \xt\nonumber\\
&&\quad- \sum_{i=1}^K \int_{\Omega \times D}
\psi 
\,U'_i({\textstyle \frac{1}{2}}|\qt_i|^2)\, \qt_i \,\qt_i^{\rm T} :
\nabxtt \ut  \,\dd \qt \dd \xt
+ K \int_{\Omega \times D} \psi 
\,(\nabx \cdot \ut) \,\dd \qt \dd \xt \nonumber\\
&&\quad +\, 4\varepsilon \int_{\Omega \times D}\! M \left|\nabx\sqrt{\frac{\psi}{M}}\right|^2
\dd \qt \dd \xt
+ \frac{1}{\lambda} \sum_{i=1}^K \sum_{j=1}^K A_{ij}\int_{\Omega \times D}\! M
\,\nabqj \sqrt{\frac{\psi}{M}}
\cdot \nabqi \sqrt{\frac{\psi}{M}} \dd \qt \dd \xt = 0.
\end{eqnarray}
We now multiply \eqref{energy-id3a} by $k$ and add the resulting identity to \eqref{energy-id3} to deduce that
\begin{eqnarray}\label{energy-id4pre}
&&\frac{\dd}{\dd t} \int_\Omega \left[\frac{1}{2}\,\!\rho \,|\ut|^2 + P(\rho)
+ \mathfrak{z} \,\vrho^2 + k \int_{D} M
\,\mathcal{F}\left(\frac{\psi}{M}\right) \dd \qt\right]{\rm d} \xt
+ k\,(K+1-\mathfrak{k})\int_\Omega \varrho \,(\nabx \cdot \ut) \dx \nonumber\\
&&\qquad+\,  \mu^{S} \int_\Omega \left|\underdtilde{D}(\ut) -
\frac{1}{d}\,(\nabx \cdot \ut)\,\Itt\right|^2{\rm d} \xt
+ \mu^B \int_\Omega |\nabx \cdot \ut|^2 \dd \xt
+ 2\varepsilon\, \mathfrak{z} \int_\Omega |\nabx \vrho
|^2 \dd \xt\nonumber\\
&&\qquad +\, 4\varepsilon\, k \int_{\Omega \times D} M \left|\nabx\sqrt{\frac{\psi}{M}}\right|^2
\dd \qt \dd \xt
+ \frac{k}{\lambda} \sum_{i=1}^K \sum_{j=1}^K A_{ij} \int_{\Omega \times D}
M \,\nabqj \sqrt{\frac{\psi}{M}}
\cdot \nabqi \sqrt{\frac{\psi}{M}} \dd \qt \dd \xt\nonumber\\
&& \qquad\qquad = \int_\Omega \rho\, \ft \cdot \ut \dd \xt.
\end{eqnarray}
Unless $\mathfrak{k}=K+1$, one further step is necessary, in order to deal with the second term on the left-hand side. To this end, we return to \eqref{pdnumber1}, with
$\zeta(\rho)\equiv 1$, take the $L^2(\Omega)$ inner product with $\mathcal{F}'(\varrho)$, and integrate by parts in the second and third term, analogously to \eqref{energy-id2}, to deduce that
\[ \frac{\dd}{\dd t} \int_\Omega \mathcal{F}({\varrho}) \dx  + 4 \varepsilon \int_\Omega |\nabx \sqrt{\varrho}|^2 \dx
= - \int_\Omega \varrho\, (\nabx \cdot \ut)  \dx.\]
Substitution of this into the second term in \eqref{energy-id4pre} then yields
\begin{eqnarray}\label{energy-id4}
&&\frac{\dd}{\dd t} \int_\Omega \left[\frac{1}{2}\,\!\rho \,|\ut|^2 + P(\rho)
+ \mathfrak{z} \,\vrho^2 + k\,(\mathfrak{k}-(K+1))\, \mathcal{F}({\varrho}) + k \int_{D} M
\,\mathcal{F}\left(\frac{\psi}{M}\right) \dd \qt\right]{\rm d} \xt
 \nonumber\\
&&\qquad+\,  \mu^{S} \int_\Omega \left|\underdtilde{D}(\ut) -
\frac{1}{d}\,(\nabx \cdot \ut)\,\Itt\right|^2{\rm d} \xt
+ \mu^B \int_\Omega |\nabx \cdot \ut|^2 \dd \xt\nonumber\\
&&\qquad + 2\varepsilon\, \mathfrak{z} \int_\Omega |\nabx \vrho 
|^2 \dd \xt + 4 \varepsilon\, k\,(\mathfrak{k}-(K+1)) \int_\Omega |\nabx \sqrt{\varrho}|^2 \dx \nonumber\\
&&\qquad +\, 4\varepsilon\, k \int_{\Omega \times D} M \left|\nabx\sqrt{\frac{\psi}{M}}\right|^2
\dd \qt \dd \xt
+ \frac{k}{\lambda} \sum_{i=1}^K \sum_{j=1}^K A_{ij} \int_{\Omega \times D}
M \,\nabqj \sqrt{\frac{\psi}{M}}
\cdot \nabqi \sqrt{\frac{\psi}{M}} \dd \qt \dd \xt\nonumber\\
&& \qquad\qquad = \int_\Omega \rho\, \ft \cdot \ut \dd \xt,
\end{eqnarray}
which is the desired (formal) energy identity that represents the starting point for our proof of
existence of global weak solutions to problem (P). The integral over $\Omega$ of the expression in
the square brackets in the first line of \eqref{energy-id4} is the total energy, and the sum of the
terms on the second, third and fourth line of \eqref{energy-id4} is the dissipation of the total
energy; recall that the Rouse matrix $A=(A_{ij})_{i,j=1}^K$ is, by definition, symmetric and
positive definite. While, as an energy identity, \eqref{energy-id4} is meaningful for all $\mathfrak{z}\geq 0$
and all $\mathfrak{k} \geq K+1$, it will transpire in the course of the
proof
that $\mathfrak{z}>0$ is
necessary in order to ensure control of $\frac{\partial \psi}{\partial t}$ or of a time-difference of $\psi$ (cf. Remarks \ref{rem-1.3} and \ref{rem-z}). Once $\mathfrak{z}$ has been chosen to be positive, the two terms in \eqref{energy-id4} that include the factor $(\mathfrak{k}-(K+1))$ are of lower order and contribute no additional information; we have therefore, for the sake of simplicity, set $\mathfrak{k}=K+1$ in \eqref{tau1a}, yielding \eqref{tau1}. If $\mathfrak{k}<K+1$ and $\mathfrak{z}>0$, then upon moving the second term on the left-hand side of \eqref{energy-id4pre} to the right, Gronwall's inequality yields a formal energy \textit{inequality}; since the ultimate outcome is no different from the one with $\mathfrak{k}=K+1$ and $\mathfrak{z}>0$, we shall not discuss this case further.

The main idea of the proof is to construct a sequence of approximating solutions,
whose existence one can prove. The sequence of approximating solutions to the solution of problem (P)
will be defined through suitable truncations (cut-off), regularizations and temporal
semidiscretization. After deriving (truncation-, regularization-, discretization-) parameter
independent bounds on the sequence of approximating solutions, we shall use a variety of
compactness arguments to pass to limits in the parameters. The derivations of the various bounds on
sequences of approximating solutions mimic the derivation of the formal energy identity
\eqref{energy-id4} outlined above. \red{The present paper combines a number of techniques developed
in our earlier work on the existence of global weak solutions to Navier--Stokes--Fokker--Planck systems,
in the case of incompressible flows of dilute polymers, with techniques that were introduced in the
groundbreaking contributions  of Lions \cite{Lions2} and Feireisl \cite{Feir1} to the mathematical
theory of compressible Navier--Stokes equations. For a detailed overview of the latter, we point the
reader to the, more recent, monograph of Novotn\'{y} \& Stra\v{s}kraba \cite{NovStras}, which will be
our main source of reference for technical results from the mathematical theory of compressible
Navier--Stokes equations;} \arxiv{for the convenience of the reader we have included the most
relevant ones of these in Appendices A--G, at the end of the paper.}

\red{Available
results in the literature concerning the existence of weak solutions to the compressible Navier--Stokes equations
are based on considering sequences of approximating problems that are defined by spatial Galerkin discretization, using bases
of smooth functions. Since, as in our previous work on incompressible Navier--Stokes--Fokker--Planck systems,
we are motivated by constructive considerations, ultimately aimed at developing convergent numerical
algorithms for the problem, we shall adopt a different approach. Our construction of
a sequence of approximating problems for the compressible Navier--Stokes--Fokker--Planck system will be based on
discretizing the problem with respect to the temporal variable. For a similar approach, in the case of a coupled compressible Navier--Stokes--Cahn--Hillard system, we refer the reader to the work of Abels and Feireisl \cite{AF-2008} and to Remark \ref{rem-AF} below, which explains the aspects in which our temporal approximation differs from the one in \cite{AF-2008}.
The inclusion of the Fokker--Planck equation into the analysis is nontrivial,
the main hurdle being to ensure that the presence of the extra stress term $\tautt$ (cf. \eqref{tautot0})
on the right-hand side of the
Navier--Stokes momentum equation does not destroy the Lions--Feireisl compactness argument for the compressible Navier--Stokes
equations. We have only been able to achieve
this for $\mathfrak{z}>0$. As the same requirement on $\mathfrak{z}$ has been found to be necessary in the, related,
Doi model for suspensions of rod-like molecules in a compressible fluid, considered by Bae and Trivisa
\cite{BaeKon}, we are confident that the condition $\mathfrak{z}>0$ is not a byproduct of our time-discrete
approach to the proof of existence of weak solutions. It is more likely that the classical
Kramers expression $\tautt_1$, appearing as the first term in \eqref{tautot0}, which was originally derived in the case of
a homogeneous incompressible solvent, needs to be supplemented  in the case of a nonhomogeneous
compressible solvent by additional correction terms (e.g. to account for the fact that compressibility may impact on the size of the
excluded volume surrounding a polymer molecule immersed in the solvent), such as the `interaction term' appearing as the second term in \eqref{tautot0}.}

The proof of the central theorem in the paper,
Theorem \ref{Pexistslemfinal}, stating the existence of global bounded-energy weak solutions to
problem (P), consists of six steps, which are outlined below.

{\em Step 1.} Following the approach in Barrett \& S\"uli
\cite{BS2,BS2011-fene,BS-2011-density-arxiv,BS2010-hookean,BS-2012-density-JDE}
and motivated by recent papers of Jourdain, Leli\`evre, Le Bris \& Otto \cite{JLLO} and
Lin, Liu \& Zhang \cite{LinLZ}
(see also  Arnold, Markowich, Toscani \& Unterreiter \cite{AMTU},
and Desvillettes \& Villani \cite{DV})
concerning the convergence of the probability density function $\psi$ to its equilibrium
value $\psi_{\infty}(\xt,\qt):=M(\qt)$
(corresponding to the equilibrium value $\ut_\infty(\xt) :=\zerot$
of the velocity field in the case of constant density) in the absence of body forces $\ft$,
we observe that if $\frac{\psi}{M}$ is bounded above then, for $L \in \mathbb{R}_{>0}$
sufficiently large, the third term in \eqref{fp0}, referred to as the \textit{drag term} is equal to
\begin{equation}\label{cut1}
\sum_{i=1}^K \nabqi \cdot \left(\sigtt(\ut) \, \qt_i\,M \,
\beta^L\!\left(\frac{\psi}{M}\right)\right)
%
\end{equation}
(recall that, by hypothesis, $\zeta \equiv 1$),
where $\beta^L \in C({\mathbb R})$ is a cut-off function defined as
\begin{align}
\beta^L(s) := \min\{s,L\}.
\label{betaLa}
\end{align}
%
%
It then follows that, for $L\gg 1$, any solution $\psi$ of (\ref{fp0}),
such that $\frac{\psi}{M}$ is bounded above by $L$, also satisfies
\begin{eqnarray}
\label{eqpsi1aa}
&&\hspace{-6.5mm}\frac{\partial \psi}{\partial t} + \nabx \cdot
(\ut \,\psi)
+ \sum_{i=1}^K \nabqi \cdot \left(\sigtt(\ut) \, \qt_i\,M \,
\beta^L\!\left(\frac{\psi}{M}\right)\right)
\nonumber
\\
\bet
&&=
\epsilon\,\Delta_x \psi +
\frac{1}{4 \,\lambda}\,
\sum_{i=1}^K \sum_{j=1}^K
A_{ij}\,\nabqi \cdot \left(
M
\,\nabqj \left(\frac{\psi}{M}\right)\right)
\quad \mbox{in } \Omega \times D \times
(0,T]. ~~~
\end{eqnarray}
%
%
%
We impose the following boundary and initial conditions:
\begin{subequations}
\begin{align}
&\left[\frac{1}{4\,\lambda} \sum_{j=1}^K A_{ij}\,
M\,\nabqj\!\left(\frac{\psi}{M}\right)
- \sigtt(\ut) \,\qt_i\,M\,\beta^L\!\left(\frac{\psi}{M}\right)
\right]\! \cdot \frac{\qt_i}{|\qt_i|}
=0\qquad ~~\nonumber\\
&~ \hspace{1.49in}\mbox{on }
\Omega \times \partial \overline{D}_i\times (0,T],
\mbox{~~for $i=1,\dots, K$,} \label{eqpsi2aa}\\
&\epsilon\,\nabx \psi\,\cdot\,\nt =0
\qquad\mbox{on }
\partial \Omega \times D\times (0,T],\label{eqpsi2ab}\\
&\psi(\cdot,\cdot,0)=M(\cdot)\,\beta^L\!\left({\psi_{0}
(\cdot,\cdot)}/{M(\cdot)}\right) \geq 0
\qquad \mbox{on $\Omega\times D$}.\label{eqpsi3ac}
\end{align}
\end{subequations}
%
The initial datum $\psi_0$ for the Fokker--Planck equation is nonnegative,
defined on $\Omega\times D$, with
\[\int_{D} \psi_0(\xt,\qt) \dd \qt \in L^\infty(\Omega), 
\]
and is assumed to have finite relative entropy with
respect to the Maxwellian $M$; i.e.
\[\int_{\Omega \times D} \psi_0(\xt,\qt)
\log \frac{\psi_0(\xt,\qt)}{M(\qt)} \dq \dx < \infty.\]
The model with cut-off parameter $L>1$ is further regularized, by introducing a dissipation term
of the form $- \alpha \,\Delta \rho$, with $\alpha>0$, into the continuity equation \eqref{ns0a}
and supplementing the resulting parabolic equation with a homogeneous Neumann boundary condition
on $\partial\Omega \times (0,T]$. In addition, the equation of state \eqref{pgamma} is replaced
by a regularized equation of state, $p_{\kappa}(\rho)=p(\rho) + \kappa (\rho^4 +
\rho^\Gamma)$, where $\kappa \in \mathbb{R}_{>0}$ and $\Gamma = \max\{\gamma,8\}$. The resulting
problem is denoted by $({\rm P}_{\kappa,\alpha,L})$.

\textit{Step 2.} Ideally, one would like to pass to the limits $\kappa \rightarrow 0_{+}$,
$\alpha \rightarrow 0_{+}$, $L\rightarrow +\infty$ to deduce the existence of solutions to
$({\rm P})$. Unfortunately,
such a direct attack at the problem is
fraught with technical difficulties. Instead,
we shall first (semi)discretize the problem $({\rm P}_{\kappa, \alpha, L})$
by an implicit Euler type scheme with respect to $t$, with step size $\Delta t$.
This then results in a time-discrete version $({\rm P}^{\Delta t}_{\kappa, \alpha, L})$ of
$({\rm P}_{\kappa, \alpha, L})$.

\textit{Step 3.} By using Schauder's fixed point theorem, we will show in
Section \ref{sec:existence-cut-off} the existence of solutions to
$({\rm P}^{\Delta t}_{\kappa, \alpha, L})$.
In the course of the proof, for technical reasons, a further cut-off,
now from below, with a cut-off parameter $\delta \in (0,1)$, is required. In addition, a fourth-order
hyperviscosity term is added to the Navier--Stokes momentum equation \eqref{ns1a}.
We shall let $\delta$ pass to $0$ to
complete the proof of existence of solutions to $({\rm P}^{\Delta t}_{\kappa, \alpha, L})$ in the limit
of $\delta \rightarrow 0_+$ in Section \ref{sec:existence-cut-off}; cf. Lemma \ref{conv}.

\textit{Step 4.} In Section \ref{sec:entropy} we then go on to derive bounds on the sequence of solutions
to problem $({\rm P}^{\Delta t}_{\kappa, \alpha, L})$; in particular, we develop various bounds on the
sequence of weak solutions to $({\rm P}^{\Delta t}_{\kappa, \alpha, L})$
that are uniform in the time step
$\Delta t$ and the cut-off parameter $L$, and thus permit the extraction of weakly convergent subsequences,
as $L\rightarrow +\infty$ and $\Delta t \rightarrow 0_{+}$, with $\Delta t = o(L^{-1})$, when $L \rightarrow
+\infty$. The weakly convergent subsequences will then be shown to converge strongly in
suitable norms. This then allows us to pass to the limit as $L \rightarrow +\infty$,
with $\Delta t = o(L^{-1})$.
The main result of Section \ref{sec:entropy} is Theorem
\ref{5-convfinal}, which summarizes the outcome of this limiting process.

\textit{Step 5.} Section \ref{sec:Pkappa} is concerned with passage to the limit
$\alpha \rightarrow 0_{+}$ with the parabolic regularization parameter that was introduced into the
continuity equation in Step 1. The main result of Section \ref{sec:Pkappa} is Theorem
\ref{Pkexistslemfinal}, which summarizes the outcome of this limiting process.

\textit{Step 6.} Finally, in Section \ref{sec:P}
we pass to the limit $\kappa \rightarrow 0_{+}$ with the
regularization parameter that was introduced into the equation of state in Step 1, which then
leads to our main result,
Theorem \ref{Pexistslemfinal},
stating the existence of global bounded-energy weak solutions to problem (P).


\section{The polymer model $({\rm P}_{\kappa,\alpha,L})$}
\label{sec:2}
\setcounter{equation}{0}

Let $\Omega \subset {\mathbb R}^d$ be a bounded open set with a
Lipschitz-continuous boundary $\partial \Omega$, and suppose that
the set $D:= D_1\times \cdots \times D_K$ of admissible
conformation vectors $\qt := (\qt_1^{\rm T}, \ldots ,
\qt_K^{\rm T})^{\rm T}$ in (\ref{fp0}) is
such that $D_i$, $i=1, \dots, K$, is an open ball
in ${\mathbb R}^d$, $d=2$ or $3$, centred at the origin\red{,} with boundary $\partial D_i$
and radius $\sqrt{b_i}$, $b_i>2$; let
\begin{align}
\partial D := \bigcup_{i=1}^K \partial \overline D_i,\qquad \mbox{where}\qquad \partial
\overline D_i:= D_1 \times \cdots \times D_{i-1} \times \partial D_i \times D_{i+1}
\times \cdots \times D_K.
\label{dD}
\end{align}
Collecting (\ref{ns0a}--e), (\ref{tautot}), (\ref{tau1}), (\ref{eqCtt}), (\ref{eqpsi1aa}) and (\ref{eqpsi2aa}--c),
we then consider the following regularized initial-boundary-value problem,
dependent on the following given regularization parameters $\kappa>0$, $\alpha>0$ and $L > 1$.
As has been already emphasized in the
Introduction, the centre-of-mass diffusion coefficient $\varepsilon>0$ is a
physical parameter and is regarded as being fixed throughout.

{\boldmath (${\rm P}_{\kappa,\alpha,L}$)}
Find
$ \rhokaL\,:\,(\xt,t)
\in \Omega \times [0,T] \mapsto
\rhokaL(\xt,t) \in {\mathbb R}$ and
$\utkaL\,:\,(\xt,t)\in \overline{\Omega} \times [0,T]
\mapsto \utkaL(\xt,t) \in {\mathbb R}^d$ 
such that
\begin{subequations}
\begin{alignat}{2}
\frac{\partial \rhokaL}{\partial t}
+ \nabx\cdot(\utkaL \,\rhokaL) - \alpha \, \Delta_x \rhokaL
&= 0 \qquad &&\mbox{in } \Omega \times (0,T],\label{equ0}
\\
\alpha
\,\nabx\, \rhokaL\,\cdot\,\nt &=0 \qquad &&\mbox{on }
\partial \Omega \times (0,T],\label{equ0b}\\
\rhokaL(\xt,0)&=\rho^0(\xt)    \qquad &&\forall \xt \in \Omega,\label{equ00}
\end{alignat}
\begin{align}
&\frac{\partial (\rhokaL\,\utkaL)}{\partial t} - \frac{\alpha}{2}\,(\Delta_x \rhokaL)\,\utkaL
+ \nabx\cdot(\rhokaL\,\utkaL \otimes \utkaL)
\nonumber \\
& 
- 
\nabx \cdot \Stt(\utkaL,\rhokaL)
+ \nabx \,\pk(\rhokaL)= \rhokaL\,\ft + \nabx \cdot \tautt(\psikaL)
\qquad \mbox{in } \Omega \times (0,T],\label{equ1}
\end{align}
\begin{alignat}{2}
\utkaL &= \zerot \qquad &&\mbox{on } \partial \Omega \times (0,T],
\label{equ3}\\
(\rhokaL\,\utkaL)(\xt,0)&=(\rho^0\ut_{0})(\xt) \qquad &&\forall \xt \in \Omega,
\label{equ4}
\end{alignat}
\end{subequations}
where $\psikaL\,:\,(\xt,\qt,t)\in
\overline{\Omega} \times \overline{D} \times [0,T]
\mapsto \psikaL(\xt,\qt,t)
\in {\mathbb R}$.
Here $\Stt(\cdot,\cdot)$ and $\tautt(\cdot,\cdot)$
are given by (\ref{Stt}) and  (\ref{tautot}),
and $\rho^0(\alpha)$ is a regularization of $\rho_0$,
see (\ref{projrho0}) below.
In addition, $\pk(\cdot)$ is a regularization of $p(\cdot)$, (\ref{pgamma}),
defined by
\begin{align}
\pk(s) := p(s) + \kappa\,(s^4+ s^{\Gamma}),
\qquad \mbox{where } \kappa \in \mathbb{R}_{>0}
\mbox{ and } \Gamma = \max\{\gamma,8\}.
\label{pkdef}
\end{align}
The Fokker--Planck equation with microscopic cut-off satisfied by $\psikaL$ is:
\begin{align}
\label{eqpsi1a}
&\hspace{-2mm}\frac{\partial \psikaL}{\partial t} +
\nabx \cdot
\left(\utkaL \,M\,\beta^L\!\left(\frac{\psikaL}{M}\right)\right)
+
\sum_{i=1}^K \nabqi
\cdot \left[\sigtt(\utkaL) \, \qt_i\,M\,
\beta^L
\left( \frac{
\psikaL}{
M}\right)\right]
\nonumber
\\
&\hspace{0.01in} =
\epsilon\,\Delta_x\, \psikaL
+
\frac{1}{4 \,\lambda}\,
\sum_{i=1}^K \sum_{j=1}^K
A_{ij}\,\nabqi \cdot \left(
M \,\nabqj \left(\frac{\psikaL}{
M}\right)\right)
\qquad \mbox{in } \Omega \times D \times
(0,T].
\end{align}
Here, for a given $L > 1$, $\beta^L \in C({\mathbb R})$ is defined by (\ref{betaLa}),
$\sigtt(\vt) \equiv \nabxtt \vt$, and
\begin{align}
\!\!\!\!\!\Att \in {\mathbb R}^{K \times K} \mbox{ is symmetric positive definite
with smallest eigenvalue $a_0 \in {\mathbb R}_{>0}$.}
\label{A}
\end{align}
%
We impose the following boundary and initial conditions:
\begin{subequations}
\begin{align}
&\left[\frac{1}{4\,\lambda} \sum_{j=1}^K A_{ij}\,
M\,
\nabqj \left(\frac{\psikaL}{
M}\right)
- \sigtt(\utkaL) \,\qt_i\,M\,
\beta^L\left(\frac{\psikaL}{
M}\right)
\right]\cdot \frac{\qt_i}{|\qt_i|}
=0 \nonumber \\
&\hspace{2.815in}\mbox{on }
\Omega \times \partial \overline{D}_i \times (0,T],
%
\quad i=1, \dots, K,
\label{eqpsi2}\\
&\epsilon
\,\nabx 
\,\psikaL 
\,\cdot\,\nt =0 \hspace{0.815in} \qquad \qquad\qquad\, \mbox{on }
\partial \Omega \times D\times (0,T],\label{eqpsi2a}\\
&\psikaL(\cdot,\cdot,0)=M(\cdot)\,
\beta^L(\psi_{0}(\cdot,\cdot)/ 
M(\cdot)) \geq 0 \qquad \mbox{on  $\Omega\times D$},\label{eqpsi3}
\end{align}
\end{subequations}
where $\nt$ is the unit outward normal to $\partial \Omega$.
The boundary conditions for $\psikaL$
on $\partial\Omega\times D\times(0,T]$ and $\Omega \times \partial D
\times (0,T]$ have been chosen so as to ensure that
\begin{align}
\int_{\Omega \times D}\psikaL(\xt,\qt,t) \dq \dx  =
\int_{\Omega \times D} \psikaL(\xt,\qt,0) \dq \dx  \qquad \forall t \in (0,T].
\label{intDcon}
\end{align}
Henceforth, we shall write
\[\hpsikaL = \frac{\psikaL}{
M},\quad
  \hpsi_0 = \frac{\psi_0}{
  M}.\]
Thus, for example, \eqref{eqpsi3} in terms of this compact notation becomes:
$\hpsikaL(\cdot,\cdot,0) = \beta^L(\hpsi_0(\cdot,\cdot))$ on $\Omega \times D$.

\section{Existence of a solution to (P$_{\kappa,\alpha,L}^{\Delta t}$),
a discrete-in-time approximation of (P$_{\kappa,\alpha,L}$)}
\label{sec:existence-cut-off}
\setcounter{equation}{0}

For later purposes, we recall the following
Lebesgue interpolation result and
the Gagliardo--Nirenberg inequality.
Let $1 \leq r \leq \upsilon \leq s < \infty$, then\red{,} for any bounded Lipschitz domain $\mathcal{O}$\red{,}
\begin{align}
\|\eta\|_{L^\upsilon(\mathcal{O})} \leq
\|\eta\|_{L^r(\mathcal{O})}^{1-\vartheta}
\,\|\eta\|_{L^s(\mathcal{O})}^{\vartheta}
 \qquad \forall \eta \in L^s(\mathcal{O}),
\label{eqLinterp}
\end{align}
where $\vartheta = \frac{ s(\upsilon-r)}{\upsilon(s-r)}$.
Let $r \in [2,\infty)$ if $d=2$,
and $r \in [2,6]$ if $d=3$ and $\vartheta = d \,\left(\frac12-\frac
1r\right)$. Then, there is a constant $C=C(\Omega,r,d)$,
such that
\begin{equation}\label{eqinterp}
\|\eta\|_{L^r(\Omega)}
\leq C\,
\|\eta\|_{L^2(\Omega)}^{1-\vartheta}
\,\|\eta\|_{H^{1}(\Omega)}^\vartheta
\qquad  \forall \eta \in H^{1}(\Omega).
\end{equation}
We note also the generalised Korn's inequality
\begin{align}
\int_{\Omega} \left[ |\Dtt(\wt)|^2 -\frac{1}{d}\, |\nabx \cdot \wt|^2\right]
\dx =
\int_{\Omega} |\Dtt(\wt)-\frac{1}{d}\, (\nabx \cdot \wt)\,
\Itt|^2 \dx \geq c_0 \,\|\wt\|_{H^1(\Omega)}^2
\qquad \forall \wt \in
\Ht_0^{1}(\Omega),
\label{Korn}
\end{align}
where $c_0>0$, see Dain \cite{Dain}.
We remark that the notation $|\cdot|$ will be used to signify one of the following.
When applied to a real number $x$,
$|x|$ will denote the absolute value of the number $x$; when applied to a vector
$\vt$,  $|\vt|$ will
stand for the Euclidean norm of the vector $\vt$; and, when applied to a square matrix
$\Att$, $|\Att|$ will
signify the Frobenius norm, $[\mathfrak{tr}(\Att^{\rm T}\Att)]^{\frac{1}{2}}$, of the matrix
$\Att$, where, for a square matrix
$\Btt$, $\mathfrak{tr}(\Btt)$ denotes the trace of $\Btt$.

Let $\mathcal{F}\in C(\mathbb{R}_{>0})$ be defined by
$\mathcal{F}(s):= s\,(\log s -1) + 1$, $s>0$.
As $\lim_{s \rightarrow 0_+} \mathcal{F}(s) = 1$,
the function $\mathcal{F}$ can be considered
to be defined and continuous on $[0,\infty)$,
where it is a nonnegative, strictly convex function
with $\mathcal{F}(1)=0$.

We assume the following:
\begin{align}\nonumber
&\partial \Omega \in C^{2,\theta},\; \theta \in (0,1); \qquad
\rho_0 \in L^\infty_{\geq 0}(\Omega); 
\qquad \ut_0 \in \Lt^2(\Omega);\\
&\psi_0 \geq 0 \ {\rm ~a.e.\ on}\ \Omega \times D
\mbox{ with }
\mathcal{F}(\hpsi_0)
\in L^1_M(\Omega \times D) 
\mbox{ and}
\int_{D} \psi_0(\cdot,\qt)\,\dq \in L^\infty_{\geq 0}(\Omega);
\nonumber\\
& \mu^S \in {\mathbb R}_{>0}, \ \mu^B \in {\mathbb R}_{\geq 0}; \qquad
\mbox{the Rouse matrix $\Att\in \mathbb{R}^{K \times K}$ satisfies (\ref{A})};
\nonumber \\
&\mbox{$p, \,\pk \in C^1({\mathbb R}_{\geq 0},{\mathbb R}_{\geq 0})$ are defined by (\ref{pgamma})
and (\ref{pkdef});}
\nonumber \\
&\ft \in L^{2}(0,T;\Lt^\infty(\Omega))
\qquad \mbox{and} \qquad
D_i = B(\zerot,b^{\frac{1}{2}}_i), \quad \theta_i>1,\quad i= 1, \dots,
K, \quad  \mbox{in (\ref{growth1},b)}.
\label{inidata}
\end{align}
We introduce $P,\, \Pk \in C^1({\mathbb R}_{\geq 0},{\mathbb R}_{\geq 0})$, for $\kappa >0$,
 such that
\begin{align}
&s \,P_{(\kappa)}'(s)-P_{(\kappa)}(s) = p_{(\kappa)}(s) \quad \mbox{and} \quad P_{(\kappa)}(0)=
P_{(\kappa)}'(0)=0\nonumber \\
\qquad &\Rightarrow \qquad P(s)= \frac{p(s)}{\gamma-1} = \frac{c_p}{\gamma-1}\,s^{\gamma}
\qquad \mbox{and} \qquad
\Pk(s) = P(s) + \kappa\left(\frac{s^4}{3}+\frac{s^{\Gamma}}{\Gamma-1}\right).
\label{Pdef}
\end{align}
Here, and throughout, the subscript ``$(\cdot)$'' means with and without the subscript ``$\,\cdot\,$''.
We adopt a similar notation for superscripts.

In (\ref{inidata}), $L^r_M(\Omega \times D)$, for $r\in [1,\infty)$,
denotes the Maxwellian-weighted $L^r$ space over $\Omega \times D$ with norm
\[
\| \varphi\|_{L^{r}_M(\Omega\times D)} :=
\left\{ \int_{\Omega \times D} \!\!M\,
|\varphi|^r \dq \dx
\right\}^{\frac{1}{r}}.
\]
Similarly, we introduce $L^r_M(D)$,
the Maxwellian-weighted $L^r$ space over $D$. Letting
\begin{eqnarray}
\| \varphi\|_{H^{1}_M(\Omega\times D)} &:=&
\left\{ \int_{\Omega \times D} \!\!M\, \left[
|\varphi|^2 + \left|\nabx \varphi \right|^2 + \left|\nabq
\varphi \right|^2 \,\right] \dq \dx
\right\}^{\frac{1}{2}}\!\!, \label{H1Mnorm}
\end{eqnarray}
we then set
\begin{eqnarray}
\quad X \equiv H^{1}_M(\Omega \times D)
&:=& \left\{ \varphi \in L^1_{\rm loc}(\Omega\times D): \|
\varphi\|_{H^{1}_M(\Omega\times D)} < \infty \right\}. \label{H1M}
\end{eqnarray}
It is shown in Appendix C of 
\cite{BS2010} (with the set $X$ denoted by $\widehat X$ there)
that
\begin{align}
C^{\infty}(\overline{\Omega \times D})
\mbox{ is dense in } X.
\label{cal K}
\end{align}

We have from Sobolev embedding that
\begin{equation}
H^1(\Omega;L^2_M(D)) \hookrightarrow L^s(\Omega;L^2_M(D)),
\label{embed}
\end{equation}
where $s \in [1,\infty)$ if $d=2$ or $s \in [1,6]$ if $d=3$.
In addition, we note that the embeddings
\begin{subequations}
\begin{align}
 H^1_M(D) &\hookrightarrow L^2_M(D) ,\label{wcomp1}\\
H^1_M(\Omega \times D) \equiv
L^2(\Omega;H^1_M(D)) \cap H^1(\Omega;L^2_M(D))
&\hookrightarrow
L^2_M(\Omega \times D) \equiv L^2(\Omega;L^2_M(D))
\label{wcomp2}
\end{align}
\end{subequations}
are compact if $\theta_i \geq 1$, $i=1, \dots,  K$, in (\ref{growth1},b);
see Appendix D
of \cite{BS2010}.

We recall the Aubin--Lions--Simon compactness theorem, see, e.g.,
Simon \cite
{Simon}. Let $\mathfrak{X}_0$, $\mathfrak{X}$ and
$\mathfrak{X}_1$ be Banach
spaces 
with a compact embedding $\mathfrak{X}_0
\hookrightarrow \mathfrak{X}$ and a continuous embedding $\mathfrak{X} \hookrightarrow
\mathfrak{X}_1$. Then, for $\varsigma_i\in[1,\infty)$, $i=0,1$,
the embedding
\begin{align}
&\{\,\eta \in L^{\varsigma_0}(0,T;\mathfrak{X}_0): \frac{\partial \eta}{\partial t}
\in L^{\varsigma_1}(0,T;\mathfrak{X}_1)\,\} \hookrightarrow L^{\varsigma_0}(0,T;\mathfrak{X})
\label{compact1}
\end{align}
is compact. 
We recall also a generalization of the Aubin--Lions--Simon compactness theorem
due to Dubinski\u{\i} \cite{Dub}, see also Barrett \& S\"{u}li \cite{BS-DUB}.
{\color{black}Prior to stating Dubinski\u{\i}'s theorem we introduce the necessary prerequisites.

Let $\mathfrak{X}$ be a linear space over the field $\mathbb{R}$ of real numbers, and suppose that $\mathfrak{M}$ is a subset of $\mathfrak{X}$ such that
\begin{equation}
\label{eq:property}
\lambda \eta \in \mathfrak{M}\qquad \mbox{$\forall \lambda \in \mathbb{R}_{\geq 0}$}, ~~ \mbox{$\forall \eta \in \mathfrak{M}$}.
\end{equation}
In other words, whenever $\eta$ is contained in $\mathfrak{M}$, the ray through $\eta$ from the origin of the linear space $\mathfrak{X}$ is also contained in $\mathfrak{M}$.  Note in particular that while any set $\mathfrak{M}$ with property \eqref{eq:property} must contain the zero element of the linear space $\mathfrak{X}$, the set
$\mathfrak{M}$ need not be closed under summation. The linear space $\mathfrak{X}$ will be referred to as the {\em ambient space} for $\mathfrak{M}$. Suppose further that each element $\eta$ of a set $\mathfrak{M}$ with property \eqref{eq:property}
is assigned a certain real number, denoted by $[\eta]_{\mathfrak M}$, such that:
\begin{enumerate}
\item[(i)] $[\eta]_{\mathfrak M} \geq 0$; and $[\eta]_{\mathfrak M} = 0$ if, and only if, $\eta=0$; and
\item[(ii)] $[\lambda \eta]_{\mathfrak M} = \lambda [\eta]_{\mathfrak M}$ for all $\lambda \in \mathbb{R}_{\geq 0}$ and all $\eta \in \mathfrak{X}$.
\end{enumerate}
We shall then say that $\mathfrak{M}$ is a {\em seminormed set}. A subset $\mathfrak{B}$ of a seminormed
set $\mathfrak{M}$ is said to be {\em bounded} if there exists a positive constant $K_0$ such that $[\eta]_{\mathfrak M} \leq K_0$ for all $\eta \in \mathfrak{B}$.  A seminormed set $\mathfrak{M}$ contained in a normed linear space $\mathfrak{X}$ with norm $\|\cdot\|_{\mathfrak X}$ is said to be {\em embedded in $\mathfrak{X}$}, and we write $\mathfrak{M} \subset \mathfrak{X}$, if there exists a $K_0 \in \mathbb{R}_{>0}$ such that
\[  \|\eta \|_{\mathfrak X} \leq K_0 [\eta]_{\mathfrak M}\qquad \forall \eta \in \mathfrak{M}.\]
Thus, bounded subsets of a seminormed set are also bounded subsets of the ambient normed linear space the seminormed set is embedded in.
The embedding of a seminormed set $\mathfrak{M}$ into a normed linear space $\mathfrak{X}$ is said to be {\em compact} if from any bounded, infinite set of elements of $\mathfrak{M}$ one can extract a subsequence that converges in $\mathfrak{X}$.

\begin{theorem}[Dubinski\u{\i}'s compactness theorem]
Suppose that $\mathfrak{M}$ is
a semi-normed set that is compactly embedded  into a Banach space $\mathfrak{X}$, which is, in turn, continuously embedded into a Banach space $\mathfrak{X}_1$.
Then, for $\varsigma_i\in[1,\infty)$, $i=0,1$,
the embedding
\begin{align}
&\{\,\eta \in L^{\varsigma_0}(0,T;\mathfrak{M}): \frac{\partial \eta}{\partial t}
\in L^{\varsigma_1}(0,T;\mathfrak{X}_1)\,\} \hookrightarrow L^{\varsigma_0}(0,T;\mathfrak{X})
\label{Dubinskii}
\end{align}
is compact.
\end{theorem}
}

Let $\mathfrak{X}$ be a Banach space.
We shall denote by $C_{w}([0,T];\mathfrak{X})$ the set of
all functions $\eta\in L^\infty(0,T;\mathfrak{X})$ such that $t \in [0,T] \mapsto
\langle \varphi, \eta(t) \rangle_{\mathfrak{X}} \in \mathbb{R}$ is continuous on $[0,T]$
for all $\varphi \in \mathfrak{X}'$, the dual space of $\mathfrak{X}$.
Here, and throughout,  $\langle \cdot,\cdot \rangle_{\mathfrak{X}}$
denotes the duality pairing between $\mathfrak{X}'$ and $\mathfrak{X}$.
Whenever $\mathfrak{X}$ has a predual, $\mathfrak{E}$, say, (viz. $\mathfrak{E}'=\mathfrak{X}$),
we shall denote by $C_{w\ast}([0,T];\mathfrak{X})$ the set of
all functions $\eta \in L^\infty(0,T;\mathfrak{X})$ such that $t \in [0,T] \mapsto \langle \eta(t),
\zeta \rangle_{\mathfrak{E}} \in \mathbb{R}$ is continuous on $[0,T]$ for all $\zeta \in \mathfrak{E}$.
We note the following result.

\begin{lemma}\label{lemma-strauss}
Let $\mathfrak{X}$ and $\mathfrak{Y}$ be Banach spaces.

\begin{itemize}
\item[(a)] Assume that the space $\mathfrak{X}$ is reflexive and is continuously embedded
in the space $\mathfrak{Y}$; then,
$L^\infty(0,T; \mathfrak{X}) \cap C_{w}([0,T]; \mathfrak{Y}) = C_{w}([0,T];\mathfrak{X})$.
\item[(b)] Assume that $\mathfrak{X}$ has a separable predual $\mathfrak{E}$ and $\mathfrak{Y}$
has a predual $\mathfrak{F}$ such that
$\mathfrak{F}$ is continuously embedded in $\mathfrak{E}$; then,
$L^\infty(0,T;\mathfrak{X}) \cap C_{w\ast}([0,T]; \mathfrak{Y}) = C_{w\ast}([0,T];\mathfrak{X})$.
\end{itemize}
\end{lemma}
\begin{proof}
Part (a) is due to Strauss \cite{Strauss} (cf. Lions \& Magenes \cite{Lions-Magenes},
Lemma 8.1, Ch. 3, Sec. 8.4); part (b) is proved analogously, {\em via}
the sequential Banach--Alaoglu theorem.
\end{proof}

We note from Lemma \ref{lemma-strauss}(a) above and Lemma 6.2
in Novotn\'{y} \& Stra\v{s}kraba \cite{NovStras} \arxiv{(or Lemma \ref{Le-E-1} in Appendix \ref{sec:App-E})} that if
$\{\eta_n\}_{n \in \mathbb N}$ is such that
\begin{subequations}
\begin{align}
\|\eta_n\|_{L^\infty(0,T;L^r(\Omega))} +
\left\|\frac{\partial \eta_n}{\partial t}\right\|_{L^\varsigma(0,T;W^{1,\upsilon}_0(\Omega)')} \leq C,
\qquad r,\varsigma, \upsilon \in (1,\infty),
\label{Cwcoma}
\end{align}
then there exists a subsequence (not indicated) of $\{\eta_n\}_{n \in \mathbb N}$
and an $\eta \in C_w([0,T];L^r(\Omega))$ such that
\begin{align}
\eta_n \rightarrow \eta \qquad \mbox{in } C_w([0,T];L^r(\Omega)).
\label{Cwcomb}
\end{align}
\end{subequations}

Throughout we will assume that
(\ref{inidata}) hold,
so that (\ref{additional-1}) and (\ref{wcomp1},b) hold.
We note for future reference that (\ref{eqCtt}) and
(\ref{additional-1}) yield that, for
$\varphi \in L^2_M(\Omega \times D)$,
\begin{align}
\label{eqCttbd}
\int_{\Omega} |\Ctt_i( M\,\varphi)|^2\,\dx & =
\int_{\Omega} 
\left|
\int_{D} M\,\varphi \,U_i'\,\qt_{i}\,\qt_{i}^{\rm T} \dq \right|^2 \dx
\nonumber
\\
&\leq
\left(\int_{D} M\,(U_i')^2 \,|\qt_i|^4 \,\dq\right)
\left(\int_{\Omega \times D} M\,|\varphi|^2 \dq \dx\right)
\nonumber \\
& \leq C
\left(\int_{\Omega \times D} M\,|\varphi|^2 \dq \dx\right),
\qquad i=1, \dots,  K,
\end{align}
where $C$ is a positive constant.

We state a simple integration-by-parts formula.

\begin{lemma} Let $\varphi \in H^1_M(D)$ and suppose that $\Btt \in
\mathbb{R}^{d\times d}$; then,
\begin{equation}\label{intbyparts}
\int_D M\, \sum_{i=1}^K (\Btt \qt_i) \cdot \nabqi \varphi \dd \qt = \int_D M\,
\varphi \left[ \left(\sum_{i=1}^KU_i'(\textstyle{\frac{1}{2}|\qt_i|^2})\,  \qt_i  \qt_i^{\rm T}\right)
- K\,\Itt
\right]
: \Btt \dq.
\end{equation}
\end{lemma}
\begin{proof}
By Theorem C.1 in Appendix C of Barrett \& S\"uli \cite{BS2010},
the set  $C^\infty(\overline D)$ is dense in $H^1_M(D)$; hence,
for any $\hat\varphi \in H^1_M(D)$ there exists a sequence
$\{\hat\varphi_n\}_{n \geq 0} \subset C^\infty(\overline D)$
converging to $\hat\varphi$ in $H^1_M(D)$. As $M \in C^1(\overline{D})$
and vanishes on $\partial D$, the same is true of each of the functions
$M\hat\varphi_n$, $n \geq 1$. By replacing $\hat\varphi$ by $\hat\varphi_n$ on both sides of
\eqref{intbyparts}, the resulting identity is easily verified by using
the classical divergence theorem
for smooth functions, noting \eqref{eqM} and that $M\hat\varphi_n$
vanishes on $\partial D$.
Then, \eqref{intbyparts} itself follows by
letting $n \rightarrow \infty$, recalling the definition of
the norm in $H^1_M(D)$ and hypothesis \eqref{additional-1}.
\end{proof}

We now formulate our discrete-in-time approximation of problem
(P$_{\kappa,\alpha,L})$ for fixed parameters $\kappa,\,\alpha \in (0,1]$
and $L > 1$. For any $T>0$ and $N \geq 1$, let
$N \,\Delta t=T$ and $t_n = n \, \Delta t$, $n=0, \dots,  N$.
To prove existence of a solution under minimal
smoothness requirements on the initial data, recall (\ref{inidata}),
we regularize the initial data in terms of the parameters $\alpha$, $\Delta t$ and $L$.
For $\rho_0 \in L^\infty_{\geq 0}(\Omega)$, we assign to it the function
$\rho^0(\alpha) \in H^1(\Omega)$, appearing in (\ref{equ00},f),
defined as the unique solution of
\begin{alignat}{2}
\int_{\Omega} \left[ \rho^0\eta + \alpha\,
\nabx \rho^0 \cdot \nabx \eta \right] \dx
&=   \int_{\Omega} \rho_0 \,\eta \dx \qquad
&&
\red{\forall \eta \in H^1(\Omega).}
\label{projrho0}
\end{alignat}
Hence,
\begin{subequations}
\begin{align}
&\rho^0(\cdot) \in [0,\|\rho_0\|_{L^\infty(\Omega)}],
\label{rho0conv}\\
\mbox{and} \qquad & \rho^0 \rightarrow \rho_0 \quad
\mbox{weakly-$\star$ in } L^\infty(\Omega), \quad
\mbox{strongly in }L^2(\Omega),\quad \mbox{as }
\alpha \rightarrow 0_+.
\label{rho0convL2}
\end{align}
\red{Therefore, by \eqref{eqLinterp} and (\ref{rho0conv},b), also
\begin{align}
\qquad\quad \rho^0 \rightarrow \rho_0\quad  \mbox{strongly in $L^p(\Omega)$}, \quad \mbox{as $\alpha \rightarrow 0_+,\quad p \in [1,\infty)$.}
\label{rho0convLp}
\end{align}
}
\end{subequations}

Similarly, for $\ut_0 \in \Lt^2(\Omega)$
we assign to it the function $\ut^0 = \ut^0(\alpha,\Delta t) \in \Ht^1_0(\Omega)$,
defined as the unique solution of
\begin{alignat}{2}
\int_{\Omega} \left[ \rho^0\ut^0 \cdot \vt + \Delta t\,
\nabxtt \ut^0 : \nabxtt \vt \right] \dx
&=   \int_{\Omega} \rho^0 \ut_0 \cdot \vt \dx \qquad
&&\forall \vt \in \Ht^1_0(\Omega).
\label{proju0}
\end{alignat}
Hence, it follows from (\ref{proju0}) and (\ref{rho0conv}) that
there exists a $C \in {\mathbb R}_{>0}$, independent of $\Delta t$, $L$, $\alpha$ and $\kappa$,
such that
\begin{subequations}
\begin{align}
&\int_{\Omega} \left[\rho^0 |\ut^0|^2 + \Delta t \,|\unabtt^0|^2 \right]\dx 
\leq
\int_{\Omega} \rho^0|\ut_0|^2 \dx
\leq C,
\label{idatabd}\\
\mbox{and}\qquad &\int_\Omega \rho^0(\ut^0-\ut_0)\cdot\vt \,\dx \rightarrow 0\quad \forall \vt \in\Lt^2(\Omega),
\quad \mbox{as $\Delta t \rightarrow 0_+$}.
\label{ut0conv}
\end{align}
\end{subequations}
Analogously, 
we shall assign a certain `smoothed' initial
datum, $\hpsi^0 = \hpsi^0(L,\Delta t) \in H^1_M(\Omega \times D)$,
to the given initial datum $\hpsi_0 = \frac{\psi_0}{M}$ such that
\begin{align}
&\int_{\Omega \times D} M \left[ 
\hpsi^0\, \varphi +
\Delta t\, \left( \nabx \hpsi^0 \cdot \nabx \varphi +
\nabq \hpsi^0 \cdot \nabq \varphi
\right) \right] \dq \dx \nonumber\\
&\hspace{2in}= \int_{\Omega \times D} M \,
\beta^{L}(\hpsi_0)\, \varphi \dq \dx
\qquad \forall \varphi \in
H^1_M(\Omega \times D).
\label{psi0}
\end{align}
For $r\in [1,\infty)$, let
\begin{align}
Z_r &:= \left\{ \varphi \in L^r_M(\Omega \times D) :
\varphi
\geq 0 \mbox{ a.e.\ on } \Omega \times D
\right\}.
\label{hatZ}
\end{align}
It is proved 
in the Appendix of \cite{BS2011-feafene} that there exists a unique
$\hpsi^0 \in H^1_M(\Omega \times D)$ satisfying \eqref{psi0}; furthermore,
$\hpsi^0 \in Z_2$,
\begin{subequations}
\begin{align}
&\int_{\Omega \times D} M\,
\mathcal{F}(\hpsi^0) \dq \dx
+ 4\,\Delta t \int_{\Omega \times D} M\,\left[
\big|\nabx \sqrt{\hpsi^0} \big|^2 + \big|\nabq \sqrt{\hpsi^{0}}\big|^2
\right]\!\dq \dx
\leq \int_{\Omega \times D} M \,
\mathcal{F}(\hpsi_0) \dq \dx
\label{inidata-1}
\end{align}
and
\begin{align}
\hpsi^0 = \beta^L(\hpsi^0) \rightarrow \hpsi_0 \quad
\mbox{weakly in }L_M^1(\Omega \times D), \quad \mbox{as} \quad L \rightarrow \infty, \quad
\Delta t \rightarrow 0_+.
\label{psi0conv}
\end{align}
\end{subequations}
Finally, by choosing $\varphi(\xt,\qt) =
\varrho^0(\xt) \otimes 1(\qt)$
in \eqref{psi0}, where $\varrho^{0}(\xt) := \int_D M(\qt)\, \hpsi^{0}(\xt, \qt) \dq \dx$,
$\xt \in \Omega$,
yields, on noting (\ref{betaLa}) and (\ref{inidata}), 
that
\begin{align}
&\frac{1}{2} \left[ \|\varrho^0\|_{L^2(\Omega)}^2 +  \int_{\Omega} \left( \int_D M \left( \hpsi^0-\beta^L(\hpsi_0)
\right) \dq \right)^2 \dx \right]
+ \Delta t \,\|\nabx \varrho^0\|_{L^2(\Omega)}^2
\nonumber \\
& \hspace{0.5in}
\red{
=  \frac{1}{2}\, \int_{\Omega} \left(\int_D M \,\beta^L(\hpsi_0)\,\dq\right)^2 \dx
\leq \frac{1}{2}\, \int_{\Omega} \left(\int_D M \, \hpsi_0\,\dq\right)^2 \dx
\leq C.
}
\label{vrho0}
\end{align}

Next, we define
\begin{align}
\Vt := \{ \wt \in \Ht^1_0(\Omega)
: \wt \in \Lt^\infty(\Omega),\;\nabx \cdot \wt \in L^\infty(\Omega) \}
\label{Vdef}
\end{align}
and
\begin{align}
\Yrho^n &:= L^2(t_{n-1},t_n;H^1(\Omega))
\cap
H^1(t_{n-1},t_n;H^1(\Omega)') \cap L^\infty(t_{n-1},t_n;L^{\frac{\Gamma}{2}}(\Omega))
\cap L^{\frac{4\Gamma}{3}}_{\geq 0}(\Omega\times(t_{n-1},t_n)).
\label{denspace}
\end{align}
%
We recall also that,  
for all 
$\vt,\,\wt \in \Ht^1(\Omega)$,
\begin{align}
\red{(}\vt \otimes \vt\red{)} : \nabxtt \wt  &=
[(\vt \cdot \nabx) \wt] \cdot \vt
=  - [(\vt \cdot \nabx) \vt] \cdot \wt
+ (\vt \cdot \nabx)(\vt \cdot \wt)  \qquad \mbox{a.e.\ in } \Omega.
\label{ruwiden}
\end{align}
Noting the above, our discrete-in-time  approximation of (P$_{\kappa,\alpha,L}$)
is then defined as follows.

{\boldmath $({\rm P}_{\kappa,\alpha,L}^{\Delta t})$} Let $N \in \mathbb{N}_{\geq 1}$
and set $\Delta t := T/N$; let, further,
$\rho^0_{\kappa,\alpha,L} := \rho^0 \in L^{\infty}_{\geq 0}(\Omega)$,
$\utkaL^0 := \ut^0 \in \Ht^1_0(\Omega)$
and $\hpsikaL^0 := \hpsi^0 \in Z_2$.
For $n = 1,\dots, N$,  and given $(\rhokaL^{n-1},\utkaL^{n-1},\hpsikaL^{n-1}) \in
L^\Gamma_{\geq 0}(\Omega)\times \Ht^1_0(\Omega) \times Z_2$, find
\begin{align}
&\rhokaL
^{[\Delta t],n} 
\in \Yrho^n 
\mbox{ with } \rhokaL^{n}(\cdot):=\rhokaL^{[\Delta t],n}
(\cdot,t_n) \in L^\Gamma_{\geq 0}(\Omega),\;
\utkaL^n \in \Ht^1_0(\Omega)  \mbox{ and } \hpsikaL^{n} \in X \cap Z_2,
\label{rho-reg}
\end{align}
%
such that
$\rhokaL^{[\Delta t],n}(\cdot,t_{n-1})= \rhokaL^{n-1}(\cdot)$,
\begin{subequations}
\begin{align}
& \int_{t_{n-1}}^{t_n}\left[ \left\langle\frac{\partial\rhokaL^{[\Delta t],n}}{\partial t} ,
\eta \right\rangle_{H^{1}(\Omega)}
+
\int_\Omega \left( \alpha \, \nabx \rhokaL^{[\Delta t],n}
 -  \rhokaL^{[\Delta t],n}
 \, \utkaL^n \right) \cdot \nabx \eta
\,\dx \right] \dd t
= 0
\nonumber\\ & \hspace{3.3in}
\qquad
\forall \eta
\in L^2(t_{n-1},t_n;H^{1}(\Omega)),
\label{rhonL}
\end{align}
\begin{align}
&\int_{\Omega} \left[
\frac{\rhokaL^n\,\utkaL^{n}-\rhokaL^{n-1}\,\utkaL^{n-1}}{\Delta t}
- \tfrac{1}{2}\,\frac{\rhokaL^{n}-\rhokaL^{n-1}}{\Delta t}
\,\utkaL^{n}
\right]
\cdot \wt \dx
\nonumber
\\
& \quad 
+  \tfrac{1}{2} \int_{\Omega}  
\rhokaL^{n-1}
\left[
[(\utkaL^{n-1} \cdot \nabx) \utkaL^{n}] \cdot \wt
-[(\utkaL^{n-1} \cdot \nabx) \wt] \cdot \utkaL^n
\right]
\dx \nonumber \\
& \quad
+\int_{\Omega}
\Stt(\utkaL^n)
: \nabxtt \wt \dx
- \int_{\Omega}  \left(\frac{1}{\Delta t}\int_{t_{n-1}}^{t_n}
\pk(\rhokaL^{[\Delta t],n}) \dd t\right) \nabx\cdot \wt \dx
\nonumber \\
&\qquad   
= \int_{\Omega} \rhokaL^{n}\,\ft^n \cdot \wt \dx 
- \int_{\Omega} \tautt_1(M\,
\hpsikaL^n)
:
\nabxtt \wt \dx \nonumber \\
& \qquad \qquad
- 2 \,\mathfrak{z}
\int_{\Omega}
\left( \int_D M\,\beta^L(\hpsikaL^n) \,\dq \right)
\nabx \left( \int_D M\,\hpsikaL^n \,\dq \right)
\cdot \wt \,\dx
\qquad \forall \wt \in \Vt,
\label{Gequn}
\end{align}
\begin{align}
&\int_{\Omega \times D} M\,
\frac{
\hpsikaL^n
- 
\hpsikaL^{n-1}}
{\Delta t}
\,\varphi \dq \dx
\nonumber
\\
&\hspace{0.3in}  +  \sum_{i=1}^K \int_{\Omega \times D}
 M \left[\, \frac{1}{4\, \lambda}\,
\sum_{j=1}^K A_{ij}\,\nabqj \hpsikaL^n
-[\,\sigtt(
\utkaL^n) \,\qt_i\,]\,
\beta^L(\hpsikaL^{n})\right]\,\cdot\, \nabqi
\varphi \dq \dx\,
\nonumber \\
&\hspace{0.3in}
+ \int_{\Omega \times D} M
\left[ \varepsilon\,\nabx \hpsikaL^n
-
\,\beta_L(\hpsikaL^n)\, 
\utkaL^{n} \right]\cdot 
\nabx\varphi
\dq \dx
=0
\qquad \forall \varphi \in
X;
\label{psiG}
\end{align}
\end{subequations}
where, 
%
%
for  $t \in (t_{n-1}, t_n]$ and $n=1, \dots,  N$,
\begin{align}
\ft^{\{\Delta t\}}(\cdot,t) =
\ft^n(\cdot) := \frac{1}{\Delta t}\,\int_{t_{n-1}}^{t_n}
\ft(\cdot,t) \dt \in \Lt^{\infty}(\Omega). 
\label{fn}
\end{align}
It follows from (\ref{inidata}) and (\ref{fn}) that
\begin{subequations}
\begin{align}
&\int_{t_{n-1}}^{t_n}\|\ft^{\{\Delta t\}}\|_{L^\infty(\Omega)}^2\,\dt
\leq \int_{t_{n-1}}^{t_n}\|\ft\|_{L^\infty(\Omega)}^2\,\dt, \qquad n=1,\ldots,N,
\label{fnbd}\\
&\ft^{\{\Delta t\}} \rightarrow \ft \qquad \mbox{strongly in } L^{2}(0,T;\Lt^r(\Omega)),\quad \mbox {as } \Delta t \rightarrow 0_{+},
\label{fncon}
\end{align}
\end{subequations}
where \red{$r \in [1,\infty]$.}

\begin{remark}\label{rem-AF}
A possible alternative to our temporal approximation scheme \eqref{rhonL} for \eqref{equ0}, 
which is the weak formulation of a parabolic initial boundary-value problem posed over the time slab $\Omega \times [t_{n-1},t_n)$, $n=1,\dots, N$, would have been to proceed as in the work of
Abels and Feireisl \cite{AF-2008} and approximate \eqref{equ0} by an implicit finite difference scheme with respect to $t$.
That would have avoided the use of $\rho^{[\Delta t],n}_{\kappa,\alpha,L}$ here, but would have had the disadvantage, from the
point of view of constructive considerations in numerical analysis at least, that
nonnegativity of $\rho^{n}_{\kappa,\alpha,L}$ will have been guaranteed for $\Delta t \leq \Delta t_0$ only, where
$\Delta t_0 \in (0,T]$ is sufficiently small, with the value of $\Delta t_0$ not being easily quantifiable in terms of the data and
 its independence of $\kappa$, $\alpha$ and $L$ being less than obvious. 
In contrast with that, our $\rho^{[\Delta t],n}_{\kappa,\alpha,L}$, and thereby also $\rho^{n}_{\kappa,\alpha,L}$, will be shown to be nonnegative for all $\Delta t=\frac{T}{N}$ and all $n=1,\dots,N$, regardless of the choice of $\kappa$, $\alpha$, $L$ and $N$.
\end{remark}

We rewrite (\ref{rhonL}) as
\begin{align}
& \int_{t_{n-1}}^{t_n} \left[ \left\langle\frac{\partial\rhokaL^{[\Delta t],n}}{\partial t} ,
\eta \right\rangle_{H^1(\Omega)}
+ c(\utkaL^n)(\rhokaL^{[\Delta t],n},\eta)
\right]
\dd t
= 0
\qquad
\forall \eta
\in L^2(t_{n-1},t_n;H^{1}(\Omega)),
\label{c}
\end{align}
where, for all $\vt \in \Ht^1_0(\Omega)$
and $\eta_i \in H^1(\Omega)$, $i=1,\,2$,
\begin{align}
c(\vt)(\eta_1,\eta_2) :=
\int_\Omega \left( \alpha \,\nabx \eta_1
 -  \eta_1
 \, \vt\right) \cdot \nabx \eta_2
\,\dx.
\label{cdef}
\end{align}
Similarly, on noting (\ref{Stt}) and (\ref{tau1}), we rewrite (\ref{Gequn}) as
\begin{equation}
b
(\rhokaL^n)(\utkaL^n,\wt) =
\ell_b(
\rhokaL^{[\Delta t],n},\hpsikaL^n)(\wt)
\qquad \forall \wt \in \Vt;
\label{bLM}
\end{equation}
where, for all $\eta \in  L^2_{\geq 0}(\Omega)$
and
$\wt_i \in \Ht^{1}_{0}(\Omega)$, $i=1,2$,
\begin{subequations}
\begin{align}
b
(\eta)(\wt_1,\wt_2) &:=
\tfrac{1}{2}\, \int_{\Omega} ( 
\eta + \rhokaL^{n-1})\,\wt_1\cdot \wt_2 \dx
+ \Delta t\, \mu^S\int_{\Omega} 
\,\Dtt(\wt_1)
:\Dtt(\wt_2) \dx
\nonumber \\
&\hspace{1cm}
+ 
\Delta t
\left(\mu^B 
- \frac{\mu^S 
}{d}\right)
\int_{\Omega}
 (\nabx \cdot \wt_1)\,(\nabx \cdot \wt_2)
\dx
\nonumber \\
& \hspace{1cm} + \tfrac{1}{2} 
\int_{\Omega} 
\rhokaL^{n-1}
\left[
[(\utkaL^{n-1} \cdot \nabx) \wt_1] \cdot \wt_2
-[(\utkaL^{n-1} \cdot \nabx) \wt_2] \cdot \wt_1
\right]
\dx
\label{bgen}
\end{align}
and, for all $\eta \in \Yrho^n$ with $\eta(\cdot,t_n) \in L^2_{\geq 0}(\Omega)$,
$\varphi \in X$ 
and $\wt \in \Vt$,
\begin{align}
\ell_b(
\eta,\varphi)(\wt) &:=
\int_\Omega \left[
\rhokaL^{n-1}\,
\utkaL^{n-1} \cdot \wt
+ \Delta t \, \eta(\cdot,t_n)\, \ft^n \cdot \wt
-\Delta t \,k\,\sum_{i=1}^K
\Ctt_i(M\,
\varphi) : \nabxtt
\wt
\right] \dx
\nonumber \\
& \hspace{1cm} + 
\int_{\Omega} \left(
\int_{t_{n-1}}^{t_n} 
\pk(\eta) \dd t + \Delta t\,k\,(K+1) \int_D M\,\varphi \dq\right)
\nabx \cdot \wt \dx
\nonumber \\
& \hspace{1cm} - 2\,\Delta t\,\mathfrak{z}\,
\int_{\Omega} \left( \int_D M\,\beta^L(\varphi) \,\dq\right)
\nabx\left( \int_D M\,\varphi \,\dq\right)
\cdot \wt \,\dx.
\label{lbgen}
\end{align}
\end{subequations}
%
%
It follows
for fixed $\utkaL^{n-1} \in \Ht^1_0(\Omega)$, $\rhokaL^{n-1} \in L^\Gamma_{\geq 0}(\Omega)$
and $\eta \in L^2_{\geq 0}(\Omega)$,
and the generalised Korn's inequality, (\ref{Korn}),
that the nonsymmetric bilinear functional
$b(\eta)(\cdot,\cdot)$ is a nonsymmetric continuous coercive bilinear
functional on
$\Ht^1_0(\Omega) \times \Ht^1_0(\Omega)$.
In addition, for fixed
$\utkaL^{n-1} \in \Ht^1_0(\Omega)$, $\rhokaL^{n-1} \in L^\Gamma_{\geq 0}(\Omega)$,
$\eta \in \Yrho^n$ with $\eta(\cdot,t_n) \in  L^{2}_{\geq 0}(\Omega)$
and $\varphi \in X$, 
it follows, on recalling \eqref{eqCttbd}, (\ref{pkdef}) and (\ref{denspace}),
that $\ell_b(\eta,\varphi)(\cdot)$ is a continuous linear
functional on $\Vt$.

It is also convenient to rewrite (\ref{psiG}) as
\begin{align}
a(\hpsikaL^n,\varphi) = \lae(\utkaL^n,\beta^L(\hpsikaL^n))
(\varphi) \qquad \forall \varphi \in X,
\label{genLM}
\end{align}
where, for all $\varphi_i \in X$, $i=1,\,2$,
\begin{subequations}
\begin{align}
a(\varphi_1,\varphi_2)
&:=\int_{\Omega \times D} M\, \bigg[ 
\varphi_1\,\varphi_2  + \Delta t\, \epsilon\,\nabx
\varphi_1\,\cdot\, \nabx \varphi_2
\label{agen}
+\, \frac{\Delta t}{4\,\lambda} \,
\sum_{i=1}^K \sum_{j=1}^K A_{ij}\,
\nabqj \varphi_1 \, \cdot\, \nabqi
\varphi_2 \biggr] \dq \dx
\end{align}
and, for all $\vt \in \Ht^1(\Omega)$, $\xi \in L^\infty(\Omega\times D)$
and $\varphi \in X$,
\begin{align}
\lae(\vt,\xi)(\varphi) &:=
\int_{\Omega \times D}
M\,\left[ 
\hpsikaL^{n-1}
\,\varphi
+ \Delta t\,\xi\,\left(\sum_{i=1}^K [\,\sigtt(\vt)
\,\qt_i\,]\,
\cdot\, \nabqi
\varphi
+ \vt \,\cdot\, \nabx \varphi \right)
\right]\!\dq \dx.
\label{lgen}
\end{align}
\end{subequations}
Clearly, $a(\cdot,\cdot)$ is a symmetric continuous
coercive bilinear functional on $X \times X$.
In addition, it is easily deduced for fixed
$\vt \in \Ht^1(\Omega)$ and $\xi \in L^\infty(\Omega\times D)$, on noting (\ref{eqinterp}),
that
 $\lae(\vt,\xi)(\cdot)$
is a continuous linear functional on $X$.

In order to prove existence of a solution, (\ref{rho-reg}), to
(P$_{\kappa,\alpha,L}^{\Delta t}$),
$\rhokaL^{[\Delta t],n}(\cdot,t_{n-1})= \rhokaL^{n-1}(\cdot)$
and (\ref{rhonL}--c),
which is equivalent to (\ref{c}), (\ref{bLM}) and (\ref{genLM}),
{\color{black} we require two convex regularizations of the entropy function $\mathcal{F}\,: s \in \mathbb{R}_{\geq 0} \mapsto \mathcal{F}(s) = s (\log s - 1) + 1 \in \mathbb{R}_{\geq 0}$, denoted by $\mathcal{F}^L$ and $\mathcal{F}^L_\delta$.

For any $L>1$, we define $\mathcal{F}^L \in C(\mathbb{R}_{\geq 0})\cap C^{2,1}_{\rm loc}(\mathbb{R}_{>0})$ by
\begin{equation}\label{eq:FL}
\mathcal{F}^L(s):= \left\{\begin{array}{ll}
s(\log s - 1) + 1,   &  ~~0 \leq s \leq L,\\
\frac{s^2 - L^2}{2L} + s(\log L - 1) + 1,  &  ~~L \leq s.
\end{array} \right.
\end{equation}
Note that
\begin{subequations}
\begin{equation}\label{eq:FL1}
[\mathcal{F}^L]'(s) = \left\{\begin{array}{ll}
\log s,   &  ~~0 < s \leq L,\\
\frac{s}{L} + \log L - 1,  &  ~~L \leq s,
\end{array} \right.
\end{equation}
and
\begin{equation}\label{eq:FL2}
[\mathcal{F}^L]''(s) = \left\{\begin{array}{ll}
\frac{1}{s},   &  ~~0 < s \leq L,\\
\frac{1}{L},  &  ~~L \leq s.
\end{array} \right.
\end{equation}
\end{subequations}
Hence, on noting the definition \eqref{betaLa} of $\beta^L$, we have that
\begin{subequations}
\begin{equation}\label{eq:FL2a}
\beta^L(s) = \min\{s,L\} = \left([\mathcal{F}^L]''(s)\right)^{-1},\quad s \in \mathbb{R}_{\geq 0},
\end{equation}
with the convention $\frac{1}{\infty}:=0$ when $s=0$, and

\begin{equation}\label{eq:FL2b}
[\mathcal{F}^L]''(s) \geq \mathcal{F}''(s) = \frac{1}{s},\quad s \in \mathbb{R}_{> 0}.
\end{equation}
\end{subequations}
We shall also require the following inequality, relating $\mathcal{F}^L$ to $\mathcal{F}$:
\begin{equation}\label{eq:FL2c}
\mathcal{F}^L(s) \geq \mathcal{F}(s),\quad s \in \mathbb{R}_{\geq 0}.
\end{equation}
For $0\leq s \leq 1$, \eqref{eq:FL2c} trivially holds, with equality. For $s\geq 1$, it
follows from \eqref{eq:FL2b}, with $s$ replaced by a dummy variable $\sigma$, after integrating
twice over $\sigma \in [1,s]$, and noting that $[\mathcal{F}^L]'(1)= \mathcal{F}'(1)$ and $\mathcal{F}^L(1)=\mathcal{F}(1)$.

For $L>1$ and $\delta \in (0,1)$, the function $\mathcal{F}_{\delta}^L \in C^{2,1}({\mathbb R})$
is defined by
\begin{align}
 &\mathcal{F}_{\delta}^L(s) := \left\{
 \begin{array}{ll}
 \textstyle\frac{s^2 - \delta^2}{2\,\delta}
 + s\,(\log \delta - 1) + 1,
 \quad & \mbox{$s \le \delta$}, \\
\mathcal{F}^L(s),
 & \mbox{$\delta \le s$}.
 \end{array} \right. \label{GLd}
\end{align}
Hence,
\begin{subequations}
\begin{align}
\quad &[\mathcal{F}_{\delta}^{L}]'(s) = \left\{
 \begin{array}{ll}
 \textstyle \frac{s}{\delta} + \log \delta - 1,
 \quad & \mbox{$s \le \delta$}, \\
 \textstyle [\mathcal{F}^L]'(s),
 & \mbox{$\delta \le s$},
 \end{array} \right. \label{GLdp}\\
\quad
 &[\mathcal{F}_{\delta}^{L}]''(s) = \left\{
 \begin{array}{ll}
 \frac{1}{\delta}, \quad & \mbox{$s \le \delta$}, \\
 \textstyle [\mathcal{F}^L]''(s), & \mbox{$\delta \le s$}.
\end{array} \right. \label{Gdlpp}
\end{align}
\end{subequations}

We note that
\begin{subequations}
\begin{align}
{\mathcal F}^L_\delta(s) &\leq {\mathcal F}^L(s) \qquad \forall s \geq 0,
\label{cFabove} \\
\mathcal{F}^L_\delta(s) &\geq \left\{
\begin{array}{ll}
\frac{s^2}{2\,\delta}, &\quad \mbox{$s \leq 0$},
\\
\frac{s^2}{4\,L} - C(L),&\quad \mbox{$s \geq 0$};
\end{array}
\right.
\label{cFbelow}
\end{align}
\end{subequations}
and} that
$[\mathcal{F}_{\delta}^{L}]''(s)$
is bounded below by $\frac{1}{L}$ for all $s \in \mathbb{R}$.
Finally, we set
\begin{align}
\beta^L_\delta(s) := ([\mathcal{F}_{\delta}^{L}]'')^{-1}(s)
= \max \{\beta^L(s),\delta\},
\label{betaLd}
\end{align}
and observe that $\beta^L_\delta(s)$
is bounded above by $L$ and bounded below by $\delta$ for all $s \in \mathbb{R}$.
Note also that both $\beta^L$ and $\beta^L_\delta$ are Lipschitz continuous on
$\mathbb{R}$, with Lipschitz constants equal to $1$.

In addition, we
regularize the bilinear functional $b(\eta)(\cdot,\cdot)$, (\ref{bgen}), on $X \times X$,
by introducing the Banach space
\begin{align}
\cVt := \{ \wt \in \Ht^2(\Omega) \cap \Ht^1_0(\Omega)
: \nabx \cdot \wt \in H^2(\Omega) \},
\label{calVdef}
\end{align}
which is compactly embedded in $\Vt$, (\ref{Vdef}).
Then we define for $\delta \in {\mathbb R}_{>0}$, $\eta \in L^2_{\geq 0}(\Omega)$
and $\wt_i \in \cVt$, $i=1,\,2$,
\begin{align}
b_\delta(\eta)(\wt_1,\wt_2) &:=
b(\eta)(\wt_1,\wt_2)
\nonumber \\
& \qquad  + \Delta t \,\delta  \sum_{|\lambdat|=2}
\int_{\Omega}
\left[
\frac{\Dlxn \wt_1}{\Dlxd}\cdot\frac{\Dlxn \wt_2}{\Dlxd}
+
\frac{\Dlxn (\nabx \cdot \wt_1)}{\Dlxd}\,\frac{\Dlxn (\nabx \cdot \wt_2)}{\Dlxd}
\right] \dx.
\label{bddef}
\end{align}
It follows that
$b_\delta(\eta)(\cdot,\cdot)$ is a nonsymmetric continuous coercive bilinear functional on
$\cVt\times \cVt$ for fixed $\eta \in L^2_{\geq 0}(\Omega)$ and $\delta \in (0,1)$.
We also replace $\ell_{b}(\eta,\varphi)(\cdot)$ by $\ell_{b,\delta}(\eta,\varphi)(\cdot)$,
where $\beta^L(\varphi)$ in $\ell_{b}(\eta,\varphi)(\cdot)$ is replaced by
$\beta^L_\delta(\varphi)$; that is, for $\delta \in (0,1)$,
\begin{align}
\ell_{b,\delta}(\eta,\varphi)(\wt):=
\ell_{b}(\eta,\varphi)(\wt)
+ 2\,\Delta t\,\mathfrak{z}\,
\int_{\Omega} \left( \int_D M\,[\beta^L(\varphi)-\beta^L_\delta(\varphi)] \,\dq\right)
\nabx\left( \int_D M\,\varphi \,\dq\right)
\cdot \wt \,\dx.
\label{lbdgen}
\end{align}
As for fixed $\eta \in \Yrho^n$ with $\eta(\cdot,t_n) \in  L^{2}_{\geq 0}(\Omega)$
and $\varphi \in X$, $\ell_b(\eta,\varphi)(\cdot)$ is a continuous linear
functional on $\Vt$,
it follows
that $\ell_{b,\delta}(\eta,\varphi)(\cdot)$, for $\delta \in (0,1)$, is a continuous linear
functional on $\cVt$.

Next, we introduce
\begin{align}
\Yrhod^n &:= L^2(t_{n-1},t_n;H^1(\Omega))
\cap
H^1(t_{n-1},t_n;H^{1}(\Omega)')\cap
L^{\infty}_{\geq 0}(\Omega\times(t_{n-1},t_n)).
\label{denspacereg}
\end{align}
We note that $\Yrhod^n \hookrightarrow C([t_{n-1},t_n];L^2_{\geq 0}(\Omega))$
and $\Yrhod^n \subset \Yrho^n$, \red{(cf. \eqref{denspace})}.
Finally, we regularize the initial data and set
\begin{align}
\rhokaLd^{n-1}= \beta^{\delta^{-1}}(\rhokaL^{n-1}),
\label{regrid}
\end{align}
where $\beta^{\delta^{-1}}$ is given by (\ref{betaLa}) with $L={\delta}^{-1}$.

We now consider the following regularized version
of the coupled system (\ref{c}), (\ref{bLM}) and
(\ref{genLM}) for a given $\delta \in (0,1)$:

Given $(\rhokaL^{n-1},\utkaL^{n-1},\hpsikaL^{n-1}) \in
L^\Gamma_{\geq 0}(\Omega)\times \Ht^1_0(\Omega) \times Z_2$, find
$(\rhokaLd^{[\Delta t],n},\,\utkaLd^n$, $\hpsikaLd^{n})
\in \Yrhod^n 
\times \cVt \times X$
such that $\rhokaLd^{[\Delta t],n}(\cdot,t_{n-1})= \rhokaLd^{n-1}(\cdot)$,
\begin{subequations}
\begin{align}
\int_{t_{n-1}}^{t_n} \left[ \left\langle\frac{\partial\rhokaLd^{[\Delta t],n}}{\partial t} ,
\eta \right\rangle_{H^{1}(\Omega)}
+ c(\utkaLd^n)(\rhokaLd^{[\Delta t],n},\eta)
\right]
\dd t
= 0
\qquad
\forall \eta
\in L^2(t_{n-1},t_n;H^1(\Omega)),
\label{cd}
\end{align}
\begin{alignat}{2}
b_\delta(\rhokaLd^{[\Delta t],n}(\cdot,t_n))(\utkaLd^n,\wt) &= \ell_{b,\delta}
(\rhokaLd^{[\Delta t],n},\hpsikaLd^n)(\wt)
\qquad &&\forall \wt \in \cVt,
\label{bLMd} \\
a(\hpsikaLd^n,\varphi) &= \lae(\utkaLd^n,\beta^L_\delta(\hpsikaLd^n))
(\varphi) \qquad &&\forall \varphi \in X.
\label{genLMd}
\end{alignat}
\end{subequations}

The existence of a solution to (\ref{cd}--c) will be proved by
using a fixed-point argument.
Given $(\tut,\tpsi)
\in \Vt\times L^2_M(\Omega \times D)$,
let $(\rho^{\star},\ut^{\star}, \psi^\star) \in
\Yrhod^n 
\times \cVt \times X$ be such that
$\rho^\star(\cdot,t_{n-1})= \rhokaLd^{n-1}(\cdot)$,
\begin{subequations}
\begin{align}
\int_{t_{n-1}}^{t_n} \left[ \left\langle\frac{\partial\rho^\star}{\partial t} ,
\eta \right\rangle_{H^{1}(\Omega)}
+ c(\tut)(\rho^{\star},\eta)
\right]
\dd t
&= 0
\qquad
\forall \eta
\in L^2(t_{n-1},t_n;H^1(\Omega)),
\label{fixrho}
\end{align}
\begin{alignat}{2}
a(\psi^{\star},\varphi) &=
\lae(\tut,\beta^L_\delta(\tpsi))(\varphi) \qquad
&&\forall \varphi \in X, \label{fix3} \\
b_\delta(\rho^\star(\cdot,t_n))(\ut^{\star},\wt) &= \ell_{b,\delta}(\rho^\star,\psi^\star)(\wt)
\qquad &&\forall \wt \in \cVt.
\label{fix4}
\end{alignat}
\end{subequations}
%


For fixed $\vt \in \Vt$,
it follows that
$c(\vt)(\cdot,\cdot)$, (\ref{cdef}),
is a nonsymmetric continuous bilinear functional on $H^1(\Omega) \times H^1(\Omega)$,
and, moreover,
for all $\eta \in H^1(\Omega)$,
\begin{align}
c(\vt)(\eta,\eta) = \alpha\,\|\nabx \eta\|_{L^2(\Omega)}^2 + \textstyle
\frac{1}{2} \displaystyle
\int_{\Omega} (\nabx \cdot \vt)\, \eta^2 \dx
\geq
\alpha\,\|\nabx \eta\|_{L^2(\Omega)}^2 - \textstyle \frac{1}{2}\,
\|\nabx \cdot \vt\|_{L^\infty(\Omega)}\,\|\eta\|_{L^2(\Omega)}^2.
\label{cbel}
\end{align}
Hence for any fixed $\tut \in \Vt$,
the existence of a unique weak solution
\begin{align}
\rho^\star \in
L^2(t_{n-1},t_n;H^1(\Omega))\cap H^1(t_{n-1},{t_n}; H^1(\Omega)')
\hookrightarrow C([t_{n-1},{t_n}]; L^2(\Omega))
\label{rhoL2H1}
\end{align}
satisfying $\rho^\star(\cdot,t_{n-1})= \rhokaLd^{n-1}(\cdot)$
and (\ref{fixrho})
is immediate; see, for example, Wloka \cite{Wloka}, Thm. 26.1.
Further, on choosing, for $s \in (t_{n-1},t_n]$,
$\eta(\cdot,t) = \chi_{[t_{n-1},s]}\, {\rm e}^{-\|\nabx \cdot \tut\|_{L^\infty(\Omega)}\,(t-t_{n-1})}\,
[\rho^\star(\cdot,t)]_{-}$ in
\eqref{fixrho}, where, for a set $S \subset \mathbb{R}$, $\chi_S$
denotes the characteristic function of  $S$, and
recalling (\ref{cbel}),
we obtain that 
\begin{align}
&{\rm e}^{-\|\nabx \cdot \tut\|_{L^\infty(\Omega)}\,(s-t_{n-1})}\,
\|[\rho^{\star}(\cdot,s)]_{-}\|^2_{L^2(\Omega)}
\nonumber \\
& \hspace{0.5in}
+ 2 \alpha \int_{t_{n-1}}^s \int_\Omega
{\rm e}^{-\|\nabx \cdot \tut\|_{L^\infty(\Omega)}\,(t-t_{n-1})}
|\nabx [\rho^\star(\xt,t)]_{-} |^2 \dx \dd t \leq 0,
\qquad s \in (t_{n-1},t_n].
\label{rhomin}
\end{align}
Next, we set
\begin{align}
R(t) := {\rm e}^{\|\nabx \cdot \tut\|_{L^\infty(\Omega)}\,(t-t_{n-1})} \|\rhokaLd^{n-1}\|_{L^\infty(\Omega)},
\qquad t\in [t_{n-1},t_n],
\label{Rdef}
\end{align}
so that
\begin{align}
&\int_{t_{n-1}}^{t_n} \left[ \left\langle\frac{\partial(\rho^\star-R)}{\partial t} ,
\eta \right\rangle_{H^{1}(\Omega)}
+ c(\tut)(\rho^{\star}-R,\eta)
\right]
\dd t
\nonumber \\
&\qquad =
- \int_{t_{n-1}}^{t_n}
R \int_\Omega \left(
\nabx \cdot \tut +  \|\nabx \cdot \tut\|_{L^\infty(\Omega)}
\right) \eta \dx \dt
\qquad
\forall \eta
\in L^2(t_{n-1},t_n;H^1(\Omega)).
\label{fixrhoR}
\end{align}
Then, similarly to (\ref{rhomin}),
on choosing
$\eta(\cdot,t) = \chi_{[t_{n-1},s]}\, {\rm e}^{-\|\nabx \cdot \tut \|_{L^\infty(\Omega)}\,(t-t_{n-1})}\,
[\rho^\star(\cdot,t)-R(t)]_{+}$ in
\eqref{fixrhoR} for $s \in (t_{n-1},t_n]$,
we obtain that
\begin{align}
&{\rm e}^{-\|\nabx \cdot \tut\|_{L^\infty(\Omega)}\,(s-t_{n-1})}\,
\|[\rho^{\star}(\cdot,s)-R(s)]_{+}\|^2_{L^2(\Omega)}
\nonumber \\
& \hspace{0.4in}
+ 2 \alpha \int_{t_{n-1}}^s \int_\Omega
{\rm e}^{-\|\nabx \cdot \tut\|_{L^\infty(\Omega)}\,(t-t_{n-1})}
|\nabx [\rho^\star(\xt,t)-R(t)]_{+} |^2 \dx \dd t \leq 0,
\qquad s \in (t_{n-1},t_n].
\label{rhomax}
\end{align}
On noting (\ref{rhoL2H1}), extending \eqref{rhomin} and \eqref{rhomax} from the interval $(t_{n-1},t_n]$ to $[t_{n-1},t_n]$ by letting $s \rightarrow t_{n-1}$ in \eqref{rhomin} and \eqref{rhomax}, and combining the resulting inequalities, we deduce that
\begin{align}
\rho^*(\cdot, t) \in [0,R(t)]\qquad \mbox{for $t \in [t_{n-1},t_n]$,} \qquad \mbox{ and so }
\rho^\star \in \Yrhod^n. 
\label{rhoLinf}
\end{align}


As $a(\cdot,\cdot)$ is a symmetric continuous
coercive bilinear functional on $X \times X$
and $\lae(\vt,\xi)(\cdot)$
is a continuous linear functional on $X$
for fixed
$\vt \in \Ht^1(\Omega)$ and $\xi \in L^\infty(\Omega\times D)$,
the Lax--Milgram theorem yields the existence of a unique solution $\psi^\star \in
X$ to (\ref{fix3}).
Similarly, for $\delta \in (0,1)$, as
$b_\delta(\eta)(\cdot,\cdot)$ is a nonsymmetric continuous coercive bilinear functional on
$\cVt\times \cVt$ for fixed $\eta \in L^2_{\geq 0}(\Omega)$,
and $\ell_{b,\delta}(\eta,\varphi)(\cdot)$ is a continuous linear
functional on $\cVt$
for fixed $\eta \in \Yrhod^n$ with $\eta(\cdot,t_n) \in L^{2}_{\geq 0}(\Omega)$
and $\varphi \in X$, 
the Lax--Milgram theorem yields the existence of a unique solution $\ut^\star \in \cVt$ to (\ref{fix4}).
Therefore
the overall procedure (\ref{fixrho}--c)
that, for $\rho^{n-1}_{\kappa, \alpha, L} \in L^\Gamma_{\geq 0}(\Omega)$
     fixed,
     maps $(\tut,\tpsi) \in \Vt \times L^2_M(\Omega \times D)$
to
$(\rho^{\star},\ut^{\star}, \psi^\star) \in
\Yrhod^n \times \cVt \times X$, with $\rho^\ast(\cdot,t_{n-1}) = \rho^{n-1}_{\kappa, \alpha, L}$,
is well defined.

\begin{lemma}
\label{fixlem} Let ${\mathcal T}: \Vt \times L^2_M(\Omega \times D)
\rightarrow \cVt \times X$ denote the
nonlinear map that takes the functions $(\tut,\tpsi)$ to
$(\ut^{\star}, \psi^\star)={\mathcal T}(\tut,\tpsi)$
{\em via} the procedure
{\rm (\ref{fixrho}--c)}. Then, the mapping ${\mathcal T}$ has a fixed point. Hence, there exists a solution
$(\rhokaLd^{[\Delta t],n},\utkaLd^n,\hpsikaLd^n)
\in \Yrhod^n 
\times \cVt \times X$
to {\rm (\ref{cd}--c)}.
\end{lemma}
\begin{proof}
Clearly, a fixed point of ${\mathcal T}$ yields a
solution of (\ref{cd}--c).
In order to show that ${\mathcal T}$ has a fixed point, we apply
Schauder's fixed-point theorem; that is, we need to show that:
(i)~${\mathcal T}:
\Vt \times L^2_M(\Omega \times D) \rightarrow
\Vt \times  L^2_M(\Omega \times D)$ is continuous;
(ii)~${\mathcal T}$ is compact; and
(iii)~there exists a $C_{\star} \in {\mathbb R}_{>0}$ such that
\begin{eqnarray}
\|\tut\|_{H^1(\Omega)}+
\|\tut\|_{L^\infty(\Omega)}
+\|\nabx\cdot\tut\|_{L^\infty(\Omega)}
+\|\tpsi\|_{L^{2}_M(\Omega\times D)} \leq C_{\star}
\label{fixbound}
\end{eqnarray}
for every $(\tut,\tpsi) \in \Vt \times L^2_M(\Omega \times D)$
and $\varkappa \in (0,1]$
satisfying $(\tut,\tpsi) = \varkappa\, {\mathcal T}(\tut,\tpsi)$.

(i) Let $\{\tut^{(m)},\tpsi^{(m)}\}_{m \in {\mathbb N}}$ be such that,
as  $m \rightarrow \infty$,
%
\begin{alignat}{2}
\tut^{(m)}
&\rightarrow \tut \quad \mbox{strongly in }
\Ht^1_0(\Omega), \qquad &&\tut^{(m)}
\rightarrow \tut \quad \mbox{strongly in }
\Lt^\infty(\Omega),
\nonumber \\
\nabx\cdot \tut^{(m)}
&\rightarrow \nabx \cdot \tut \quad \mbox{strongly in }
L^\infty(\Omega),
\qquad
&&\tpsi^{(m)}
\rightarrow \tpsi \quad \mbox{strongly in }
L^{2}_M(\Omega\times D).
\label{Gcont1}
\end{alignat}
%
It follows immediately from (\ref{Gcont1}), (\ref{betaLd}), (\ref{growth1})
and (\ref{eqLinterp}) 
that
\begin{alignat}{1}
M^{\frac{1}{2}}\,\beta^L_\delta(\tpsi^{(m)})
\rightarrow M^{\frac{1}{2}}\,
\beta^L_\delta(\tpsi) \qquad
\mbox{strongly in }
L^r(\Omega\times D)\quad
\mbox{as } m \rightarrow \infty
\label{betacon}
\end{alignat}
for all $r \in [1,\infty)$.
In order to prove that ${\mathcal T}: \Vt \times L^2_M(\Omega \times D) \rightarrow
\Vt \times L^2_M(\Omega \times D)$ is continuous, we need to show that
$(\vt^{(m)},\xi^{(m)}):={\mathcal T}(\tut^{(m)},\tpsi^{(m)})$
is such that, as $m \rightarrow \infty$,
\begin{alignat}{2}
\vt^{(m)}
&\rightarrow \vt \quad \mbox{strongly in }
\Ht^1_0(\Omega), \qquad &&\vt^{(m)}
\rightarrow \vt \quad \mbox{strongly in }
\Lt^\infty(\Omega),
\nonumber \\
\nabx\cdot \vt^{(m)}
&\rightarrow \nabx \cdot \vt \quad \mbox{strongly in }
L^\infty(\Omega),
\qquad
&&\xi^{(m)}
\rightarrow \xi \quad \mbox{strongly in }
L^{2}_M(\Omega\times D),
\label{Gcont2}
\end{alignat}
where
$(\vt,\xi):={\mathcal T}(\tut,\tpsi)$.
We have from the definition of ${\mathcal T}$, recall (\ref{fixrho}--c), that, for all
$m \in {\mathbb N}$,
$(\vt^{(m)},\xi^{(m)}) \in \cVt \times X$ is the unique solution to
\begin{subequations}
\begin{alignat}{2}
a( \xi^{(m)}, \varphi)&= \ell_a(\tut ^{(m)},
\beta^L_\delta(\tpsi^{(m)}))(\varphi)
&&\qquad \forall \varphi \in X,
\label{Gcont5}\\
b_\delta(\trho^{(m)}(\cdot,t_n))(\vt^{(m)},\wt) &= \ell_{b,\delta}(\trho^{(m)},\xi^{(m)})(\wt)
&&\qquad \forall \wt \in \cVt,
\label{Gcont5a}
\end{alignat}
\end{subequations}
where
$\trho^{(m)} \in
\Yrhod^n$ 
is the unique solution to
$\trho^{(m)}(\cdot,t_{n-1})= \rhokaLd^{n-1}(\cdot)$ and
\begin{align}
\int_{t_{n-1}}^{t_n} \left[ \left\langle\frac{\partial\trho^{(m)}}{\partial t} ,
\eta \right\rangle_{H^{1}(\Omega)}
+ c(\tut^{(m)})(\trho^{(m)},\eta)
\right]
\dd t
&= 0
\qquad
\forall \eta
\in L^2(t_{n-1},t_n;H^1(\Omega)).
\label{fixrhom}
\end{align}

It follows from (\ref{Gcont1}) that $\|\nabx \cdot \tut^{(m)}\|_{L^\infty(\Omega)} \leq C_\star$
for all $m \in {\mathbb N}$.
On choosing, for $s \in (t_{n-1},t_n]$,
$\eta(\cdot,t) = \chi_{[t_{n-1},s]}\,{\rm e}^{-C_\star\,(t-t_{n-1})} \trho^{(m)}(\cdot,t)$ in
\eqref{fixrhom} yields, on noting (\ref{cbel}), the first two bounds in
\begin{align}
&\|\trho^{(m)}\|_{C([t_{n-1},t_n];L^2(\Omega))}^2
+ \alpha \, \|\trho^{(m)}\|_{L^2(t_{n-1},t_n;H^1(\Omega))}^2
\nonumber \\
& \qquad \qquad
+\|\trho^{(m)}\|_{L^\infty(t_{n-1},t_n;L^\infty(\Omega))}^2
+ \left\|\frac{\partial \trho^{(m)}}{\partial t}\right\|_{L^2(t_{n-1},t_n;H^1(\Omega)')}^2
\leq C,
\label{rhombds}
\end{align}
where, here and below, $C$ is independent of $m$.
The third bound in (\ref{rhombds}) follows from applying the bound (\ref{rhoLinf})
to (\ref{fixrhom}) with $\tut$ in
(\ref{Rdef}) replaced by $\tut^{(m)}$ and noting (\ref{Gcont1}).
The fourth bound in (\ref{rhombds}) follows immediately from the first two bounds in
(\ref{rhombds}), (\ref{Gcont1})
and (\ref{fixrhom}).
Choosing $\varphi = \xi^{(m)}$ in (\ref{Gcont5}) yields, on noting (\ref{agen},b)
and (\ref{Gcont1}), that
\begin{align}
\|\xi^{(m)}\|_{H^1_M(\Omega\times D)}^2 \leq C.
\label{chimbds}
\end{align}
Choosing $\wt = \vt^{(m)}$ in (\ref{Gcont5a}) yields, on noting (\ref{bddef}), (\ref{lbdgen}), (\ref{bgen},b),
$\trho^{(m)}(\cdot,t_n) \in L^2_{\geq 0}(\Omega)$, $\rhokaL^{n-1} \in L^\Gamma_{\geq 0}(\Omega)$, (\ref{Korn}),
(\ref{Mint1}), (\ref{eqCttbd}),
(\ref{chimbds}),
(\ref{pkdef}) and (\ref{rhombds}), that
\begin{align}
\|\vt^{(m)}\|_{H^2(\Omega)}^2
+\|\nabx \cdot \vt^{(m)}\|_{H^2(\Omega)}^2 
\leq C.
\label{vtmbds}
\end{align}

It follows from
(\ref{rhombds}),  (\ref{chimbds}), (\ref{vtmbds}),
(\ref{compact1}), 
(\ref{betaLd}),
(\ref{growth1}),
(\ref{eqLinterp}), 
(\ref{eqCttbd})
and
the compactness of $\cVt$ in $\Vt$ and $X \equiv H^1_M(\Omega \times D)$ in $L^2_M(\Omega \times D)$
that there exists a subsequence
$\{(\trho^{(m_k)},\vt^{(m_k)},\xi^{(m_k)})\}_{m_k \in {\mathbb N}}$ and
functions $(\trho,\,\vt,\,\xi) \in \Yrhod^n 
\times \cVt \times X$ such that, as $m_k \rightarrow \infty$, for any $r \in [1,\infty)$,
\begin{subequations}
\begin{alignat}{3}
\trho^{(m_k)} & \rightarrow \trho
\quad &&\mbox{weakly in }
L^{2}(t_{n-1},t_n;H^1(\Omega)), \quad &&\mbox{
strongly in  } L^r(t_{n-1},t_n;L^r(\Omega)),
\label{Gcont8rhoa}\\
\frac{\partial \trho^{(m_k)}}{\partial t} & \rightarrow \frac{\partial \trho}{\partial t}
\quad &&\mbox{weakly in }
L^{2}(t_{n-1},t_n;H^1(\Omega)'), &&
\label{Gcont8rhob}\\
\trho^{(m_k)}(\cdot,t_n) & \rightarrow \trho(\cdot,t_n)
\quad &&\mbox{weakly in }
L^{2}(\Omega), &&
\label{Gcont8rhoc}\\
\vt^{(m_k)} &\rightarrow \vt \quad
&&\mbox{weakly in }
\Ht^{2}(\Omega), \quad &&\mbox{strongly in }
\Ht^{1}_0(\Omega)\cap\Lt^{\infty}(\Omega),
\label{Gcont8vta}\\
\nabx \cdot \vt^{(m_k)} &\rightarrow \nabx \cdot \vt \quad
&&\mbox{weakly in }
H^{2}(\Omega), \quad &&\mbox{strongly in }
L^{\infty}(\Omega),
\label{Gcont8vtb}\\
\xi^{(m_k)} &\rightarrow \xi \quad
&&\mbox{weakly in }
H^{1}_M(\Omega\times D), \quad &&\mbox{strongly in }
L^2_M(\Omega \times D),
\label{Gcont8chi}\\
M^{\frac{1}{2}}\,\beta^L_\delta(\xi^{(m_k)}) &\rightarrow
M^{\frac{1}{2}}\,\beta^L_\delta(\xi) \quad
&&\mbox{strongly in }
L^r(\Omega \times D),&&
\label{Gcont8chid}
\\
\Ctt_i(\xi^{(m_k)}) &\rightarrow \Ctt_i(\xi) \quad
&&\mbox{strongly in }
L^2(\Omega), \quad&&i=1,\ldots,K.
\label{Gcont8Ctt}
\end{alignat}
\end{subequations}

We deduce from (\ref{fixrhom}), (\ref{Gcont8rhoa},b) and (\ref{Gcont1}) that
$\trho \in
\Upsilon^n$ 
is the unique solution to
$\trho(\cdot,t_{n-1})= \rhokaLd^{n-1}(\cdot)$ and
\begin{align}
\int_{t_{n-1}}^{t_n} \left[ \left\langle\frac{\partial\trho}{\partial t} ,
\eta \right\rangle_{H^{1}(\Omega)}
+ c(\tut)(\trho,\eta)
\right]
\dd t
&= 0
\qquad
\forall \eta
\in L^2(t_{n-1},t_n;H^1(\Omega)).
\label{fixtrho}
\end{align}
Choosing $\eta=1$ in (\ref{fixtrho}), on noting (\ref{cdef}) and (\ref{regrid}), yields that
\begin{align}
\int_{\Omega} \trho(\cdot,t_n) \,\dx = \int_{\Omega} \rhokaLd^{n-1} \,\dx \leq
 \int_{\Omega} \rhokaL^{n-1} \,\dx.
\label{trhoin}
\end{align}
It follows from (\ref{Gcont5}), (\ref{agen},b), (\ref{Gcont8chi}), (\ref{Gcont1}) and
(\ref{betacon}) that $\xi,\,\tpsi \in X$
and $\tut \in \Vt$ satisfy
\begin{align}
a(\xi,\varphi)
&= \lae(\tut,\beta^L_\delta(\tpsi))(\varphi)
\qquad \forall \varphi \in C^\infty(\overline{\Omega \times D}).
\label{Gcont11}
\end{align}
Then, noting that $a(\cdot,\cdot)$
is a continuous bilinear functional on $X \times X$,
that $\lae(\vt,\beta^L_\delta(\tpsi))(\cdot)$ is a continuous linear functional on $X$,
and recalling (\ref{cal K}), we deduce that
$\xi \in X$ is the unique solution of
(\ref{Gcont11})
for all $\varphi \in X$.
It further follows from (\ref{Gcont5a}), (\ref{bddef}), (\ref{lbdgen}), (\ref{bgen},b) and (\ref{Gcont8rhoa},c,d,e,f,g,h)
that $\vt \in \cVt$ is the unique solution to
\begin{eqnarray}\qquad
b_\delta(\trho(\cdot,t_n))(\vt,\wt) =\ell_{b,\delta}(\trho,\xi)(\wt)
\qquad \forall \wt \in \cVt.
\label{Gcont9}
\end{eqnarray}

Combining (\ref{Gcont11}) with $\varphi \in X$ with (\ref{Gcont9}) and (\ref{fixtrho}), we have that
$(\vt,\xi) = {\mathcal T}(\tut,\tpsi)\in \cVt \times X$. As $(\vt,\xi)$ is unique for fixed $(\ut,\tpsi)$,
the whole sequence converges in (\ref{Gcont8rhoa}--h), and so (\ref{Gcont2}) holds.
Therefore the mapping ${\mathcal T}: \Vt \times L^2_M(\Omega \times D) \rightarrow
\Vt \times L^2_M(\Omega \times D)$ is continuous.

(ii) Since the embeddings $\cVt \hookrightarrow \Vt$ and $X \hookrightarrow L^{2}_M(\Omega \times D)$
are compact, we directly deduce that the mapping ${\mathcal T}: \Vt \times L^2_M(\Omega \times D) \rightarrow
\Vt \times L^2_M(\Omega \times D)$ is compact. It therefore remains to show that (iii) holds.

(iii) Let us suppose that $(\tut,\tpsi) =
\varkappa \, {\mathcal T}(\tut,\tpsi)$; then, $(\trho,\tut,\tpsi) \in
\Yrhod^n 
\times \cVt \times X$ satisfies
$\trho(\cdot,t_{n-1})= \rhokaLd^{n-1}(\cdot)$ and
\begin{subequations}
\begin{align}
\int_{t_{n-1}}^{t_n} \left[ \left\langle\frac{\partial\trho}{\partial t} ,
\eta \right\rangle_{H^{1}(\Omega)}
+ c(\tut)(\trho,\eta)
\right]
\dd t
&= 0
\qquad
\forall \eta
\in L^2(t_{n-1},t_n;H^1(\Omega)),
\label{fixtrhoa}
\end{align}
\begin{alignat}{2}
b_\delta(\trho(\cdot,t_n))(\tut,\wt) &= \varkappa\,\ell_{b,\delta}(\trho,\tpsi)(\wt)
\qquad &&\forall \wt \in \cVt,
\label{fix4sig}
\\
a(\tpsi,\varphi)
&= \varkappa\,\lae(\tut,\beta^L_\delta(\tpsi))(\varphi)
\qquad &&\forall \varphi \in X. \label{fix3sig}
\end{alignat}
\end{subequations}
Choosing $\wt = \tut$ in (\ref{fix4sig})
yields, as $\varkappa \in (0,1]$,  that
\begin{align}
&\frac{\varkappa}{2}\,\displaystyle
\int_{\Omega} \left[ \trho(\cdot,t_n)\,|\tut|^2 + \rhokaL^{n-1}\,
|\tut-\utkaL^{n-1}|^2 - \rhokaL^{n-1}\,|\utkaL^{n-1}|^2
\right]
\dx \nonumber\\
&\hspace{0.5in}
+ \Delta t\,\mu^S
\int_{\Omega} |\Dtt(\tut)|^2 \dx
+ \Delta t\,\left(\mu^B -\frac{\mu^S}{d}\right)
\int_{\Omega} |\nabx \cdot \tut|^2 \dx
\nonumber \\
&\hspace{0.5in}
+ \Delta t\,
\delta  \sum_{|\lambdat|=2}
\int_{\Omega}
\left[\,
\left|\frac{\Dlxn \tut}{\Dlxd}\right|^2
+ \left|
\frac{\Dlxn (\nabx \cdot \tut)}{\Dlxd}\right|^2 \,\right]
\dx
\nonumber \\
&\hspace{1in}
\leq \varkappa \,\Delta t \left[
\int_{\Omega}\trho(\cdot,t_n) \,\ft^n \cdot \tut \dx
- k\,\sum_{i=1}^K
\int_{\Omega}
 \Ctt_i(M 
 \,\tpsi): \nabxtt \tut \dx \right]
\nonumber\\ & \hspace{1.5in}
+ \varkappa\int_{\Omega} \left(\int_{t_{n-1}}^{t_n} \pk(\trho) \dt
+ \Delta t\,k\,(K+1) \int_D M\,\tpsi \dq
\right)
\nabx \cdot \tut \,\dx
\nonumber\\ & \hspace{1.5in}
- 2\,\varkappa \,\Delta t\,\mathfrak{z}
\int_{\Omega}
\left( \int_D M\,\beta^L_\delta(\tpsi) \,\dq \right)
\nabx \left( \int_D M\,\tpsi \,\dq \right)
\cdot \tut \,\dx.
\label{Gequnbhat}
\end{align}
On recalling (\ref{Pdef}),  we choose $\eta(\cdot,t) = \chi_{[t_{n-1},s]}\,
\Pk'(\trho(\cdot,t)+\varsigma)$ in (\ref{fixtrhoa}),
for any $s\in (t_{n-1},t_n]$
and any fixed $\varsigma \in {\mathbb R}_{>0}$,
 to obtain,
on noting (\ref{cdef}), that
\begin{align}
&\int_{\Omega} \Pk(\trho(\cdot,s)) \dx
+ \alpha\,\kappa
\int_{t_{n-1}}^{s} \int_\Omega (4\,\trho^{2}
+\Gamma\,\trho^{\Gamma-2})\,|\nabx \trho|^2 \dx \,\dt
\nonumber \\
& \hspace{0.6in} \leq \int_{\Omega} \Pk(\trho(\cdot,s)+\varsigma)  \dx
+ \alpha\,\int_{t_{n-1}}^{s} \int_\Omega \Pk''(\trho+\varsigma)\,|\nabx \trho|^2 \dx \,\dt
\nonumber \\
& \hspace{0.6in} =
\int_{\Omega} \Pk(\rhokaLd^{n-1}+\varsigma) \dx
+
\int_{t_{n-1}}^{s} \int_{\Omega} \trho\,\tut\cdot
\nabx \Pk'(\trho+\varsigma) \dx \,\dt
\nonumber \\
& \hspace{0.6in} =
\int_{\Omega} \Pk(\rhokaLd^{n-1}+\varsigma) \dx
+  \int_{\Omega}
\left(\int_{t_{n-1}}^{s} \left[ \Pk(\trho+\varsigma)-\trho\,\Pk'(\trho+\varsigma) \right]
\dt \right)  \nabx \cdot \tut\,\dx.
\label{Pkentreg}
\end{align}
As $\rhokaLd^{n-1}\in L^\infty_{\geq 0}(\Omega)$,
$\trho\in L^\infty(t_{n-1},t_n;L^\infty_{\geq 0}(\Omega))$ and $\tut \in \cVt$,
one can pass to the limit $\varsigma \rightarrow 0_{+}$ in (\ref{Pkentreg})
using Lebesgue's dominated convergence theorem to obtain, for any $s \in (t_{n-1},t_n]$, that
\begin{align}
&\int_{\Omega} \Pk(\trho(\cdot,s)) \dx
+ \alpha\,\kappa
\int_{t_{n-1}}^{s} 
\left[ \|\nabx (\trho^2)\|_{L^2(\Omega)}^2 + \frac{4}{\Gamma}\,\|\nabx
(\trho^{\frac{\Gamma}{2}})\|_{L^2(\Omega)}^2 \right]
\dt
\nonumber \\
& \hspace{0.6in}\leq \int_{\Omega} \Pk(\rhokaLd^{n-1}) \dx
 +  \int_{\Omega}
\left(\int_{t_{n-1}}^{s} \left[ \Pk(\trho)-\trho\,\Pk'(\trho) \right]
\dt \right)  \nabx \cdot \tut\,\dx.
\nonumber \\
& \hspace{0.6in} \leq \int_{\Omega} \Pk(\rhokaL^{n-1}) \dx
 - \int_{\Omega}
\left( \int_{t_{n-1}}^{s} \pk(\trho) \,\dt
\right)
\nabx \cdot \tut
\,\dx,
\label{Pkent}
\end{align}
where we have noted (\ref{regrid}) and (\ref{Pdef}) for the final inequality.
We remark that we needed to choose $\Pk'(\trho(\cdot,t)+\varsigma)$, as opposed
to $\Pk'(\trho(\cdot,t))$, in the testing procedure as $\Pk''(\trho)$, that would appear in
(\ref{Pkentreg}), may not be well-defined for $\gamma \in  (\frac{3}{2},2)$,
as we only know that $\trho$ is nonnegative as opposed to being strictly positive.

For a.e.\ $\xt \in \Omega$, let
\red{
\begin{equation}\label{tvrho}
\tvrho(\xt) := \int_D M(\qt)\, \tpsi(\xt, \qt) \dq
\quad \mbox{and}
\quad
\vrhokaL^{n-1}(\xt) := \int_D M(\qt)\, \hpsikaL^{n-1}(\xt, \qt) \dq.
\end{equation}
}
Choosing $\varphi(\xt,\qt) =
\tvrho(\xt) \otimes 1(\qt)$
in \eqref{fix3sig} yields that
\begin{align}
&\frac{\varkappa}{2} \left[ \|\tvrho\|_{L^2(\Omega)}^2 +
\|\tvrho-\vrhokaL^{n-1}\|_{L^2(\Omega)}^2\right]
+ \Delta t \,\varepsilon \,\|\nabx \tvrho\|_{L^2(\Omega)}^2
\nonumber \\
& \hspace{1.6in} =  \frac{\varkappa}{2}\, \|\vrhokaL^{n-1}\|_{L^2(\Omega)}^2
+ \varkappa\,\Delta t\,\int_{\Omega}
\left(\int_D M\,\beta^L_\delta(\tpsi)\,\dq\right)\,\tut \cdot \nabx \tvrho\,\dx.
\label{vrhokaLbd}
\end{align}

Combining (\ref{Gequnbhat}), (\ref{Pkent}) for $s=t_n$ and (\ref{vrhokaLbd}) yields,
on noting (\ref{Korn}), (\ref{inidata}) and (\ref{tvrho}),
that, for all $\varkappa \in (0,1]$,
\begin{align}
&\frac{\varkappa}{2}\,\displaystyle
\int_{\Omega} \left[ \trho(\cdot,t_n)\,|\tut|^2 + \rhokaL^{n-1}\,
|\tut-\utkaL^{n-1}|^2 \right]
\dx + \varkappa \int_{\Omega} \Pk(\trho(\cdot,t_n)) \dx
\nonumber\\
&\hspace{0.1in}
+ \varkappa\,\alpha\,\kappa
\int_{t_{n-1}}^{t_n}
\left[ \|\nabx (\trho^2)\|_{L^2(\Omega)}^2 + \frac{4}{\Gamma}\,\|\nabx
(\trho^{\frac{\Gamma}{2}})\|_{L^2(\Omega)}^2 \right]
\dt
+ \Delta t\,\mu^S\,c_0\, \|\tut\|_{H^1(\Omega)}^2
\nonumber \\
&\hspace{0.1in}
+ \Delta t\,
\delta  \sum_{|\lambdat|=2}
\int_{\Omega}
\left[\,
\left|\frac{\Dlxn \tut}{\Dlxd}\right|^2
+ \left|
\frac{\Dlxn (\nabx \cdot \tut)}{\Dlxd}\right|^2 \,\right]
\dx
\nonumber \\
&\hspace{0.1in} + \varkappa\,\mathfrak{z}\,\left[ \|\tvrho\|_{L^2(\Omega)}^2 +
\|\tvrho-\vrhokaL^{n-1}\|_{L^2(\Omega)}^2\right]
+ 2\,\Delta t  \,\mathfrak{z}\,\varepsilon\,\|\nabx \tvrho\|_{L^2(\Omega)}^2
\nonumber
\\
&\hspace{0.2in}
\leq
\frac{\varkappa}{2}\,\displaystyle
\int_{\Omega} \rhokaL^{n-1}\,|\utkaL^{n-1}|^2 \dx
+\varkappa \int_{\Omega} \Pk(\rhokaL^{n-1})  \dx
+ \varkappa \,\Delta t
\int_{\Omega}\trho(\cdot,t_n) \,\ft^n \cdot \tut \dx
\nonumber \\
&\hspace{0.3in}
+\varkappa\,\mathfrak{z}\,\|\vrhokaL^{n-1}\|_{L^2(\Omega)}^2
+\varkappa \,k\,\Delta t \left[
(K+1) \int_\Omega \tvrho\,
\nabx \cdot \tut \,\dx
- \sum_{i=1}^K
\int_{\Omega}
 \Ctt_i(M
 \,\tpsi): \nabxtt \tut \dx \right].
\label{Ekt}
\end{align}

Choosing $\varphi = [\mathcal{F}_\delta^L]'(\widetilde{\psi})$ in (\ref{fix3sig})
and noting 
(\ref{betaLd})
implies that
\begin{align}
&
\int_{\Omega \times D} M \left( 
\mathcal{F}_\delta^L (\tpsi)
- \mathcal{F}_\delta^L (\varkappa \,\hpsikaL^{n-1})
+ \frac{1}{2\,L}\,|\tpsi-\varkappa\, \hpsikaL^{n-1}|^2
\right) \dq \dx
\nonumber \\
& \qquad
+ \frac{\Delta t}{4\,\lambda}\,
\sum_{i=1}^K \sum_{j=1}^K A_{ij}
\int_{\Omega \times D} M\,
\nabqj
\tpsi \cdot \nabqi ([\mathcal{F}_\delta^L]'(\tpsi))
\dq \dx
\nonumber\\&\qquad\qquad
+ \Delta t \, \varepsilon
\int_{\Omega \times D}
 M\,  \nabx \tpsi \cdot \nabx ([\mathcal{ F}_{\delta}^L]'(\tpsi))
\dq \dx
\nonumber \\
&
\hspace{1in}
\leq \varkappa\,\Delta t
\int_{\Omega \times D}
M\,\left[ 
\left(\sum_{i=1}^K \sigtt(\tut) \,\qt_i \right) \cdot
\nabqi \tpsi + \tut \cdot \nabx \tpsi \right]
\dq \dx
\nonumber
\\
&\hspace{1in}
= \varkappa\,\Delta t
\int_{\Omega}
\left[\sum_{i=1}^K
\Ctt_i(M\,
\tpsi) : \sigtt(\tut)
-(K+1)\, (\nabx \cdot \tut) \int_D M \tpsi \dq
\right]
\dx,
\label{acorstab1}
\end{align}
where in the transition to the final line we applied \eqref{intbyparts} with
$\Btt = \sigtt(\tut)$ (on account of it being independent of the variable $\qt$),
and recalled \eqref{eqCtt}.
Combining 
(\ref{Ekt}) and (\ref{acorstab1}),
and noting 
(\ref{A}) and (\ref{betaLd}), 
yields, for all $\varkappa \in (0,1]$ and $\varsigma \in {\mathbb R}_{>0}$, that
\begin{align}
&\frac{\varkappa}{2}\,\displaystyle
\int_{\Omega} \left[ \trho(\cdot,t_n)\,|\tut|^2 + \rhokaL^{n-1}\,
|\tut-\utkaL^{n-1}|^2 \right]
\dx
+ \varkappa \int_{\Omega} \Pk(\trho(\cdot,t_n)) \dx
\nonumber\\
&\hspace{0.1in}
+ \varkappa\,\alpha\,\kappa
\int_{t_{n-1}}^{t_n} 
\left[ \|\nabx (\trho^2)\|_{L^2(\Omega)}^2 + \frac{4}{\Gamma}\,\|\nabx
(\trho^{\frac{\Gamma}{2}})\|_{L^2(\Omega)}^2 \right]
\dt
\nonumber
\\
&\hspace{0.1in}+
k \int_{\Omega \times D} M \left(
\mathcal{F}_\delta^L (\tpsi)
+ \frac{1}{2\,L}\,|\tpsi-\varkappa\, \hpsikaL^{n-1}|^2
\right) \dq \dx
+ \Delta t\,\mu^S\,c_0\, \|\tut\|_{H^1(\Omega)}^2
\nonumber 
\end{align}
\begin{align}
&\hspace{0.1in}
+ \Delta t\,
\delta  \sum_{|\lambdat|=2}
\int_{\Omega}
\left[\,
\left|\frac{\Dlxn \tut}{\Dlxd}\right|^2
+ \left|
\frac{\Dlxn (\nabx \cdot \tut)}{\Dlxd}\right|^2 \,\right]
\dx
\nonumber 
\\
& \hspace{0.1in}
+ \frac{\Delta t\,k\,a_0}{4\,\lambda\,L}\,
\sum_{i=1}^K 
\int_{\Omega \times D} 
M\,|\nabqi \tpsi|^2 
\dq \dx
+ \frac{\Delta t\, k\,\varepsilon}{L}
\int_{\Omega \times D} 
 M\,  |\nabx \tpsi|^2 
\dq \dx
\nonumber 
\\
& \hspace{0.1in}
+ \varkappa\,\mathfrak{z}\,\left[ \|\tvrho\|_{L^2(\Omega)}^2 +
\|\tvrho-\vrhokaL^{n-1}\|_{L^2(\Omega)}^2\right]
+ 2\,\Delta t \, \mathfrak{z}\,\varepsilon\,\|\nabx \tvrho\|_{L^2(\Omega)}^2
\nonumber 
\\
&\hspace{0.2in}
\leq
\frac{\varkappa}{2}\,\displaystyle
\int_{\Omega} \rhokaL^{n-1}\,|\utkaL^{n-1}|^2 \dx
+\varkappa \int_{\Omega} \Pk(\rhokaL^{n-1})  \dx
+
k \int_{\Omega \times D} M \mathcal{F}_\delta^L (\varkappa \,\hpsikaL^{n-1})
\dq \dx
\nonumber \\
& \hspace{0.3in}
+ \varkappa\,\mathfrak{z}\,\|\vrhokaL^{n-1}\|_{L^2(\Omega)}^2
+ \varkappa \,\Delta t
\int_{\Omega}\trho(\cdot,t_n) \,\ft^n \cdot \tut \dx
\nonumber \\
& \hspace{0.2in}
\leq
\frac{\varkappa}{2}\,\displaystyle
\int_{\Omega} \rhokaL^{n-1}\,|\utkaL^{n-1}|^2 \dx
+\varkappa \int_{\Omega} \Pk(\rhokaL^{n-1})  \dx
+
k \int_{\Omega \times D} M \mathcal{F}_\delta^L (\varkappa \,\hpsikaL^{n-1})
\dq \dx
\nonumber \\
& \hspace{0.3in}
+ \varkappa\,\mathfrak{z}\,\|\vrhokaL^{n-1}\|_{L^2(\Omega)}^2
+ \frac{\varkappa\,\Delta t}{2}\,\left[
\varsigma\,\int_{\Omega}\trho(\cdot,t_n)\,|\tut|^2 \,\dx
+ \frac{1}{\varsigma}\,\|\ft^n\|_{L^\infty(\Omega)}^2\,\int_{\Omega}
\rhokaL^{n-1} \,\dx\right],
\label{Ek}
\end{align}
where, in deriving the final inequality, we have noted 
(\ref{trhoin}).
It is easy to see that $\mathcal{F}^L_\delta(s)$ is nonnegative for all
$s \in \mathbb{R}$, with  $\mathcal{F}^L_\delta(1)=0$.
Furthermore, for any $\varkappa \in (0,1]$,
$\mathcal{F}^L_\delta(\varkappa\, s) \leq \mathcal{F}^L_\delta(s)$
if $s<0$ or $1 \leq \varkappa\, s$, and also
$\mathcal{F}^L_\delta(\varkappa\, s) \leq \mathcal{F}^L_\delta(0) \leq 1$
if $0 \leq \varkappa\, s \leq 1$.
Thus we deduce that
\begin{equation}\label{deltaL}
\mathcal{F}_\delta^L(\varkappa\, s)
\leq \mathcal{F}_\delta^L(s)+ 1\qquad \forall s \in {\mathbb R},\quad
\forall \varkappa \in (0,1].
\end{equation}
Hence, the bounds (\ref{Ek}) and (\ref{deltaL}), on noting (\ref{cFbelow}) and, from
(\ref{betaLd}) and (\ref{betaLa}), that $\beta^L_\delta(\cdot) \leq L$,
give rise, for $\varsigma$ sufficiently small,
to the desired bound (\ref{fixbound}) with $C_*$ dependent only on
$\delta$, $L$, $\Delta t$, $M$, $k$, $\mu^S$, $c_0$, $a_0$, $\ft$,
$\rhokaL^{n-1}$, $\utkaL^{n-1}$ and $\hpsikaL^{n-1}$.
Therefore (iii) holds, and so ${\mathcal T}$ has a fixed point, proving
existence of a solution to (\ref{cd}--c).
\qquad\end{proof}

Similarly to (\ref{vrhokaLbd}), choosing $\varphi(\xt,\qt) =
\vrhokaLd^n(\xt) \otimes 1(\qt)$
in \eqref{genLMd}, where $\vrhokaLd^{n}(\xt) := \int_D M(\qt)\, \hpsikaLd^{n}(\xt, \qt) \dq \dx$,
yields that
\begin{align}
&\frac{1}{2} \left[ \|\vrhokaLd^n\|_{L^2(\Omega)}^2 +
\|\vrhokaLd^n-\vrhokaL^{n-1}\|_{L^2(\Omega)}^2\right]
+ \Delta t \,\varepsilon \,\|\nabx \vrhokaLd^n\|_{L^2(\Omega)}^2
\nonumber \\
& \hspace{0.5in} =  \frac{1}{2}\, \|\vrhokaL^{n-1}\|_{L^2(\Omega)}^2
+ \Delta t \int_{\Omega}
\left(\int_D M\,\beta^L_\delta(\vrhokaLd^n)\,\dq\right)\,\utkaLd^n \cdot \nabx \vrhokaLd^n\,\dx.
\label{vrhokaLdbd}
\end{align}
Choosing $\wt = \utkaLd^{n}$ in (\ref{bLMd}) and $\eta = \Pk'(\rhokaLd^{[\Delta t],n})$
in (\ref{cd}),
and combining with (\ref{vrhokaLdbd}), yields, similarly to (\ref{Ekt}), that
\begin{align}
&\frac{1}{2}\,\displaystyle
\int_{\Omega} \left[ \rhokaLd^{[\Delta t],n}(\cdot,t_n)\,|\utkaLd^n|^2 + \rhokaL^{n-1}\,
|\utkaLd^n-\utkaL^{n-1}|^2 \right]
\dx
+ \int_{\Omega} \Pk(\rhokaLd^{[\Delta t],n}(\cdot,t_n)) \dx
\nonumber\\
&\;
+ \alpha\,\kappa
\int_{t_{n-1}}^{t_n}
\left[ \|\nabx[(\rhokaLd^{[\Delta t],n})^2]\|_{L^2(\Omega)}^2 + \frac{4}{\Gamma}\,\|\nabx
[(\rhokaLd^{[\Delta t],n})^{\frac{\Gamma}{2}}]\|_{L^2(\Omega)}^2 \right]
\dt
\nonumber \\
&\;
+ \Delta t\,\mu^S\,c_0\, \|\utkaLd^n\|_{H^1(\Omega)}^2\nonumber\\
&\;+ \Delta t\,
\delta  \sum_{|\lambdat|=2}
\int_{\Omega}
\left[\,
\left|\frac{\Dlxn \utkaLd^n}{\Dlxd}\right|^2
+ \left|
\frac{\Dlxn (\nabx \cdot \utkaLd^n)}{\Dlxd}\right|^2 \,\right]
\dx
\nonumber
\\
&\;
+ \mathfrak{z}\,\left[ \|\vrhokaLd^n\|_{L^2(\Omega)}^2 +
\|\vrhokaLd^n-\vrhokaL^{n-1}\|_{L^2(\Omega)}^2\right]
+ 2\,\Delta t \,\mathfrak{z}\,\varepsilon\,\|\nabx \vrhokaLd^n\|_{L^2(\Omega)}^2
\nonumber 
\end{align}
\begin{align}
&\hspace{0.2in}
\leq
\frac{1}{2}\,\displaystyle
\int_{\Omega} \rhokaL^{n-1}\,|\utkaL^{n-1}|^2 \dx
+ \int_{\Omega} \Pk(\rhokaL^{n-1})  \dx
+ \Delta t
\int_{\Omega}\rhokaLd^{[\Delta t],n}(\cdot,t_n) \,\ft^n \cdot \utkaLd^n \dx
\nonumber \\
& \hspace{0.4in}
+ \mathfrak{z}\,\|\vrhokaL^{n-1}\|_{L^2(\Omega)}^2
+ k\,\Delta t \,
(K+1) \int_\Omega 
\vrhokaLd^n\,
\nabx \cdot \utkaLd^n \,\dx
\nonumber \\
&\hspace{0.4in}
- k\,\Delta t \sum_{i=1}^K
\int_{\Omega}
 \Ctt_i(M
 \,\hpsikaLd^n): \nabxtt \utkaLd^n \dx.
\label{Ektrhod}
\end{align}
Choosing 
$\varphi = [{\mathcal F}^{L}_\delta]'(\hpsikaLd^{n})$ in (\ref{genLMd}),
combining with (\ref{Ektrhod})
and noting (\ref{cFabove}), yields, similarly to (\ref{Ek}), that, for
$\varsigma \in {\mathbb R}_{>0}$
sufficiently small,
the solution
$(\rhokaLd^{[\Delta t],n},\,\utkaLd^n$ $\hpsikaLd^{n})
\in \Yrhod^n
\times \cVt \times X$
of (\ref{cd}--c) satisfies
\begin{align}
&\frac{1}{2}\,\displaystyle
\int_{\Omega} \left[ \rhokaLd^{[\Delta t],n}(\cdot,t_n)\,|\utkaLd^n|^2 + \rhokaL^{n-1}\,
|\utkaLd^n-\utkaL^{n-1}|^2 \right]
\dx
+ \int_{\Omega} \Pk(\rhokaLd^{[\Delta t],n}(\cdot,t_n)) \dx
\nonumber \\
&\hspace{0.25in}
+ \alpha\,\kappa
\int_{t_{n-1}}^{t_n} 
\left[ \|\nabx [(\rhokaLd^{[\Delta t],n})^2]\|_{L^2(\Omega)}^2 + \frac{4}{\Gamma}\,\|\nabx
[(\rhokaLd^{[\Delta t],n})^{\frac{\Gamma}{2}}]\|_{L^2(\Omega)}^2 \right]
\dt
\nonumber \\
& \hspace{0.25in}
+k \int_{\Omega \times D} M \left(
\mathcal{F}_\delta^L (\hpsikaLd^n)
+ \frac{1}{2\,L}\,|\hpsikaLd^n-\hpsikaL^{n-1}|^2
\right) \dq \dx
\nonumber \\
&\hspace{0.25in}
+ \Delta t\,\mu^S\,c_0\, \|\utkaLd^n\|_{H^1(\Omega)}^2
+ \Delta t\,
\delta  \sum_{|\lambdat|=2}
\int_{\Omega}
\left[\,
\left|\frac{\Dlxn \utkaLd^n}{\Dlxd}\right|^2
+ \left|
\frac{\Dlxn (\nabx \cdot \utkaLd^n)}{\Dlxd}\right|^2 \,\right]
\dx
\nonumber
\\
& \hspace{0.25in}
+ \frac{\Delta t\,k\,a_0}{4\,\lambda\,L}\,
\sum_{i=1}^K 
\int_{\Omega \times D} 
M\,|\nabqi \hpsikaLd^n|^2 
\dq \dx
+ \frac{\Delta t\, k\,\varepsilon}{L}
\int_{\Omega \times D} 
M\,   |\nabx \hpsikaLd^n|^2 
\dq \dx
\nonumber \\
&\hspace{0.25in}
+ \mathfrak{z}\,\left[ \|\vrhokaLd^n\|_{L^2(\Omega)}^2 +
\|\vrhokaLd^n-\vrhokaL^{n-1}\|_{L^2(\Omega)}^2\right]
+ 2\,\Delta t \,\mathfrak{z}\,\epsilon\,\|\nabx \vrhokaLd^n\|_{L^2(\Omega)}^2
\nonumber \\
&\hspace{0.1in}
\leq
\frac{1}{2}\,\displaystyle
\int_{\Omega} \rhokaL^{n-1}\,|\utkaL^{n-1}|^2 \dx
+\int_{\Omega} \Pk(\rhokaLd^{n-1})  \dx
+
k \int_{\Omega \times D} M \mathcal{F}_\delta^L (\hpsikaL^{n-1})
\dq \dx
\nonumber \\
& \hspace{0.2in}
+ \mathfrak{z}\,\|\vrhokaL^{n-1}\|_{L^2(\Omega)}^2
+ \Delta t
\int_{\Omega}\rhokaLd^{[\Delta t],n}(\cdot,t_n) \,\ft^n \cdot \utkaLd^n \dx
\nonumber
\\
&\hspace{0.1in}
\leq
\frac{1}{2}\,\displaystyle
\int_{\Omega} \rhokaL^{n-1}\,|\utkaL^{n-1}|^2 \dx
+\int_{\Omega} \Pk(\rhokaL^{n-1})  \dx
+
k \int_{\Omega \times D} M \mathcal{F}^L (\hpsikaL^{n-1})
\dq \dx
\nonumber \\
& \hspace{0.2in}
+ \mathfrak{z}\,\|\vrhokaL^{n-1}\|_{L^2(\Omega)}^2
+ \frac{\Delta t}{2}\left[
\varsigma \int_{\Omega}\rhokaLd^{[\Delta t],n}(\cdot,t_n)\,|\utkaLd^n|^2 \,\dx
+ \frac{1}{\varsigma}\,\|\ft^n\|_{L^\infty(\Omega)}^2\,\int_{\Omega}
\rhokaL^{n-1} \,\dx\right]
\nonumber \\
&\hspace{0.1in}
\leq C, \nonumber \\
\label{E1}
\end{align}
where
$C$ is independent of $\delta$ and $\Delta t$.

On choosing, for any $s\in (t_{n-1},t_n]$, $\eta(\cdot,t) = \chi_{[t_{n-1},s]}\,
[\rhokaLd^{[\Delta t],n}(\cdot,t)]^{\vartheta-1}$, for $\vartheta=2$ and $\frac{\Gamma}{2}$,
in (\ref{cd}), we obtain, on noting (\ref{cdef}),
(\ref{regrid}) 
and (\ref{E1}), that
\begin{align}
&\frac{1}{\vartheta}\,\|\rhokaLd^{[\Delta t],n}(\cdot,s)\|_{L^\vartheta(\Omega)}^\vartheta
+ \frac{4\alpha(\vartheta-1)}{\vartheta^2}\,
\int_{t_{n-1}}^{s} \|\nabx[ (\rhokaLd^{[\Delta t],n})^{\frac{\vartheta}{2}}]\|_{L^2(\Omega)}^2 \,\dt
\nonumber \\
& \quad = \frac{1}{\vartheta}\left[ \|\rhokaLd^{n-1}\|_{L^\vartheta(\Omega)}^\vartheta
+ (\vartheta-1)\,\int_{t_{n-1}}^{s} \int_{\Omega} 
\utkaLd^n \cdot
\nabx [(\rhokaLd^{[\Delta t],n})^\vartheta] \dx\, \dt \right]
\nonumber \\
& \quad \leq \frac{1}{\vartheta}\left[\|\rhokaL^{n-1}\|_{L^\vartheta(\Omega)}^\vartheta
+ \Delta t \,\|\utkaLd^n\|_{L^2(\Omega)}^2 +
\frac{(\vartheta-1)^2}{4}\,\int_{t_{n-1}}^{t_n} 
\|\nabx [(\rhokaLd^{[\Delta t],n})^\vartheta]\|_{L^2(\Omega)}^2 \dt
\right]
\nonumber \\
& \quad \leq C,
\label{rhodbd}
\end{align}
where $C$ is independent of $\delta$ and $\Delta t$.
On denoting by $\mint\eta$, the mean value of the function $\eta$ over $\Omega$, it follows from
a  Poincar\'e inequality,  
(\ref{E1}) and (\ref{rhodbd}) for
$\vartheta=\frac{\Gamma}{2}$ that
\begin{align}
&\|\rhokaLd^{[\Delta t],n}\|_{L^{\Gamma}(t_{n-1},t_n;L^\Gamma(\Omega))}^\Gamma
= \|(\rhokaLd^{[\Delta t],n})^{\frac{\Gamma}{2}}\|_{L^{2}(t_{n-1},t_n;L^2(\Omega))}^2
\nonumber \\[2mm]
& \hspace{0.6in} \leq
2 \|(I-\mint)(\rhokaLd^{[\Delta t],n})^{\frac{\Gamma}{2}}\|_{L^{2}(t_{n-1},t_n;L^2(\Omega))}^2
+ 2\, \|\mint(\rhokaLd^{[\Delta t],n})^{\frac{\Gamma}{2}}\|_{L^{2}(t_{n-1},t_n;L^2(\Omega))}^2
\nonumber \\[2mm]
& \hspace{0.6in} \leq C\,
\|\nabx[(\rhokaLd^{[\Delta t],n})^{\frac{\Gamma}{2}}]\|_{L^{2}(t_{n-1},t_n;L^2(\Omega))}^2
+ C\,\Delta t \,\|\rhokaLd^{[\Delta t],n}\|_{L^{\infty}(t_{n-1},t_n;L^{\frac{\Gamma}{2}}(\Omega))}^\Gamma
\leq C.
\label{PkrhokaLdc}
\end{align}
Next, we obtain from (\ref{cdef}), (\ref{eqinterp}), (\ref{E1}) and (\ref{rhodbd})
for $\vartheta=\frac{\Gamma}{2}$, on recalling that $\Gamma \geq 8$,
that
\begin{align}
\left| \int_{t_{n-1}}^{t_n} c(\utkaLd^n)(\rhokaLd^{[\Delta t],n},\eta) \dt\right|
&\leq
\alpha \,\|\rhokaLd^{[\Delta t],n}\|_{L^2(t_{n-1},t_n;H^{1}(\Omega))}\, \|\eta\|_{L^2(t_{n-1},t_n;H^{1}(\Omega))}
\nonumber \\
& \hspace{0.5in}
+ \left| \int_{t_{n-1}}^{t_n} \|\rhokaLd^{[\Delta t],n}\|_{L^3(\Omega)}\,\|\utkaLd^n\|_{L^6(\Omega)}\,
\|\nabx \eta\|_{L^2(\Omega)} \,\dt \right|
\nonumber \\
& 
\leq C \,\|\eta\|_{L^2(t_{n-1},t_n;H^{1}(\Omega))}
+ C \left| \int_{t_{n-1}}^{t_n} \|\utkaLd^n\|_{H^1(\Omega)}\,
\|\eta\|_{H^{1}(\Omega)} \,\dt \right|
\nonumber \\
& 
\leq C \left[ 1 + \left(\Delta t\,\|\utkaLd^n\|_{H^1(\Omega)}^2 \right)^{\frac{1}{2}}\right]
\,\|\eta\|_{L^2(t_{n-1},t_n;H^{1}(\Omega))}
\nonumber \\
& 
\leq C\,
\,\|\eta\|_{L^2(t_{n-1},t_n;H^{1}(\Omega))}
\qquad \forall \eta \in L^2(t_{n-1},t_n;H^{1}(\Omega)).
\label{rhodH-1a}
\end{align}
Hence, we deduce from  
(\ref{rhodbd}) for $\vartheta=2$ and $\frac{\Gamma}{2}$,
(\ref{rhodH-1a}),
on noting
(\ref{cd}),
(\ref{E1}), (\ref{Pdef}) and
(\ref{PkrhokaLdc}) that $\rhokaLd^{[\Delta t],n} \in \Yrhod^n$ is such that
\begin{align}
&
\|\rhokaLd^{[\Delta t],n}\|_{L^\infty(t_{n-1},t_n;L^{\frac{\Gamma}{2}}(\Omega))}
+
\|\rhokaLd^{[\Delta t],n}\|_{L^2(t_{n-1},t_n;H^1(\Omega))}^2
+
\|\rhokaLd^{[\Delta t],n}\|_{H^1(t_{n-1},t_n;H^{1}(\Omega)')}^2
\nonumber \\
& \hspace{1in}
+\|\rhokaLd^{[\Delta t],n}(\cdot,t_n)\|_{L^\Gamma(\Omega)}^\Gamma
+\|(\rhokaLd^{[\Delta t],n})^{\frac{\Gamma}{2}}\|_{L^2(t_{n-1},t_n;H^1(\Omega))}^2
\leq C,
\label{rhodall}
\end{align}
where $C$ is independent of $\delta$ and $\Delta t$.
Furthermore, we deduce from  (\ref{eqinterp}) and the last bound in (\ref{rhodall}) that
\begin{align}
\|\rhokaLd^{[\Delta t],n}\|_{L^\Gamma(t_{n-1},t_n;L^{3\Gamma}(\Omega))}^\Gamma =
\|(\rhokaLd^{[\Delta t],n})^{\frac{\Gamma}{2}}\|_{L^2(t_{n-1},t_n;L^6(\Omega))}^2
\leq C\,\|(\rhokaLd^{[\Delta t],n})^{\frac{\Gamma}{2}}\|_{L^2(t_{n-1},t_n;H^1(\Omega))}^2\leq C.
\label{rhodL3}
\end{align}
Finally, it follows from (\ref{eqLinterp}) with $\upsilon=\frac{4\Gamma}{3}$, $r=\frac{\Gamma}{2}$
and $s=3 \Gamma$ yielding $\vartheta =\frac{3}{4}$, the first bound in (\ref{rhodall})
and (\ref{rhodL3}) that
\begin{align}
\|\rhokaLd^{[\Delta t],n}\|_{L^{\frac{4\Gamma}{3}}(t_{n-1},t_n;L^{\frac{4\Gamma}{3}}(\Omega))}^{\frac{4\Gamma}{3}}
\leq
\|\rhokaLd^{[\Delta t],n}\|_{L^{\infty}(t_{n-1},t_n;L^{\frac{\Gamma}{2}}(\Omega))}^{\frac{\Gamma}{3}}
\|\rhokaLd^{[\Delta t],n}\|_{L^{\Gamma}(t_{n-1},t_n;L^{3\Gamma}(\Omega))}^{\Gamma}
\leq C,
\label{rhodL43}
\end{align}
where $C$ is independent of $\delta$ and $\Delta t$.

As the bounds (\ref{E1}), (\ref{rhodL3}) and (\ref{rhodL43}) are independent of $\delta$, we
are now ready to pass to the limit $\delta \rightarrow 0_+$ in (\ref{cd}--c),
to deduce the existence of
a solution $\{(\rhokaL^{[\Delta t],n},\utkaL^n,\hpsikaL^n)\}_{n=1}^N$ to
({\rm P}$^{\Delta t}_{\kappa,\alpha,L}$).

\begin{lemma}
\label{conv}
There exists a subsequence (not indicated) of $\{(\rhokaLd^{[\Delta t],n},\utkaLd^{n},
\hpsikaLd^{n})\}_{\delta >0}$, and functions
$\rhokaL^{[\Delta t],n} \in \Yrho^n$ with $\rhokaL^n(\cdot) = \rhokaL^{[\Delta t],n}(\cdot,t_n)
\in L^\Gamma_{\geq 0}(\Omega)$,
$\utkaL^{n} \in \Ht^1_0(\Omega)$
and $\hpsikaL^{n} \in X \cap Z_2$, $n = 1,\ldots, N$,
with
\begin{align}
\vrhokaL^n(\cdot) := \int_D M(\qt)\,\hpsikaL^n(\cdot,\qt)\,\dq \in H^1(\Omega),
\qquad n=1,\ldots,N,
\label{vrhokaLn}
\end{align}
such that, as $\delta \rightarrow 0_+$,
\begin{subequations}
\begin{alignat}{3}
\rhokaLd^{[\Delta t],n} &\rightarrow  \rhokaL^{[\Delta t],n}
\quad &&\mbox{weakly in } L^2(t_{n-1},t_n;H^1(\Omega)), \quad
&&\mbox{weakly in } H^1(t_{n-1},t_n;H^{1}(\Omega)'),
\label{rhodwcon}\\
\rhokaLd^{[\Delta t],n} &\rightarrow  \rhokaL^{[\Delta t],n}\quad
&&\mbox{strongly in } L^2(t_{n-1},t_n;L^r(\Omega)),\quad
&&\mbox{strongly in } L^{\upsilon}(\Omega \times(t_{n-1},t_n)),
\label{rhodscon} \\
\rhokaLd^{[\Delta t],n}(\cdot,t_n) &\rightarrow  \rhokaL^n(\cdot)\quad
&&\mbox{weakly in } L^\Gamma(\Omega),
\label{rhodwconG}\\
(\rhokaLd^{[\Delta t],n})^\vartheta &\rightarrow  (\rhokaL^{[\Delta t],n})^\vartheta
\quad &&\mbox{weakly in } L^2(t_{n-1},t_n;H^1(\Omega)),
\quad &&\vartheta = 2 \mbox{ and } \frac{\Gamma}{2},
\label{rhodwconth}
\end{alignat}
\end{subequations}
\begin{subequations}
\begin{alignat}{2}
\utkaLd^{n} &\rightarrow \utkaL^{n} \qquad &&\mbox{weakly in } \Ht^1_0(\Omega),
\qquad \mbox{strongly in }
\Lt^{r}(\Omega),
\label{uwconH1}\\
\delta \, \frac{\Dlxn \utkaLd^n}{\Dlxd}
& \rightarrow \zerot \qquad &&\mbox{strongly in } \Lt^2(\Omega), \qquad \forall |\lambdat|=2,
\label{delcon1} \\
\delta \, \frac{\Dlxn (\nabx \cdot \utkaLd^n)}{\Dlxd}
& \rightarrow 0 \qquad &&\mbox{strongly in } L^2(\Omega), \qquad \forall |\lambdat|=2,
\label{delcon12}
\end{alignat}
\end{subequations}
where $r \in [1,\infty)$ if $d=2$ and $r \in [1,6)$ if $d=3$, and $\upsilon \in [1,\frac{4 \Gamma}{3})$;
and
\begin{subequations}
\begin{alignat}{2}
M^{\frac{1}{2}}\,\nabq \hpsikaLd^{n}
&\rightarrow M^{\frac{1}{2}}\,\nabq \hpsikaL^{n}
&&\qquad \mbox{weakly in }
\Lt^2(\Omega\times D), \label{psiwconH1}\\
\bet
M^{\frac{1}{2}}\,\nabx \hpsikaLd^{n}
&\rightarrow M^{\frac{1}{2}}\,\nabx \hpsikaL^{n}
&&\qquad \mbox{weakly in }
\Lt^2(\Omega\times D), \label{psiwconH1x}\\
\bet
M^{\frac{1}{2}}\,\hpsikaLd^{n} &\rightarrow
M^{\frac{1}{2}}\,\hpsikaL^{n}
&&\qquad \mbox{strongly in }
L^{2}(\Omega\times D),\label{psisconL2}
\\
\bet
M^{\frac{1}{2}}\,\beta_\delta^L(\hpsikaLd^{n}) &\rightarrow
M^{\frac{1}{2}}\,\beta^L(\hpsikaL^{n})
&&\qquad \mbox{strongly in }
L^s(\Omega\times D),\label{betaLdsconL2}\\
\Ctt_i(M\,
\hpsikaLd^n) & \rightarrow \Ctt_i(M\,
\hpsikaL^n)
&&\qquad \mbox{strongly in }
\Ltt^{2}(\Omega), \qquad i=1, \ldots,  K,
\label{CwconL2}
\\
\vrhokaLd^n &\rightarrow \vrhokaL^n&&\qquad \mbox{weakly in }
H^1(\Omega), \qquad \mbox{strongly in } L^r(\Omega),
\label{vrhodcon}
\end{alignat}
\end{subequations}
where $s \in [1,\infty)$.
Furthermore, $(\rhokaL^{[\Delta t],n},\utkaL^n, \hpsikaL^n)$ solves (\ref{rhonL}--c)
for $n=1,\ldots, N$. Hence, there exists a solution $\{(\rhokaL^{[\Delta t],n},\utkaL^n,
\hpsikaL^n)\}_{n=1}^N$ to
{\em ({\rm P}$^{\Delta t}_{\kappa,\alpha,L}$)}.
\end{lemma}
\begin{proof}
The weak convergence results
(\ref{rhodwcon},c,d) follow immediately from (\ref{rhodall}).
The strong convergence results (\ref{rhodscon}) follow from (\ref{rhodwcon}), (\ref{compact1}),
(\ref{rhodL43}) and the interpolation result (\ref{eqLinterp}).
Hence $\rhokaL^{[\Delta t],n} \in Y^n$
with $\rhokaL^{n}(\cdot)=\rhokaL^{[\Delta t],n}(\cdot,t_n) \in L^{\Gamma}_{\geq 0}(\Omega)$
as $\rhokaLd^{[\Delta t],n} \in \Upsilon^n$.
The weak convergence result (\ref{uwconH1}) and the strong convergence results (\ref{delcon1},c)
follow immediately from (\ref{E1}),
and hence $\utkaL^n \in \Ht^1_0(\Omega)$ as $\utkaLd^n \in \cVt$.
The strong convergence result (\ref{uwconH1})
follows as $\Ht^{1}(\Omega)$ is compactly embedded in $\Lt^r(\Omega)$ for
the stated values of $r$.

The weak convergence results (\ref{psiwconH1},b)
follow from (\ref{E1}); the strong convergence result (\ref{psisconL2})
and the fact that $\hpsikaL^n \geq 0$
a.e.\ on $\Omega \times D$ follow from the fourth bound in (\ref{E1}),
(\ref{cFbelow}) and (\ref{wcomp2}).
Hence $\hpsikaL^n \in X \cap Z_2$.
The desired results
(\ref{betaLdsconL2},e) follow from (\ref{psisconL2}), (\ref{betaLd}),
(\ref{eqCtt}) and (\ref{eqCttbd}).
See the proof of Lemma 3.3 in \cite{BS2011-fene} for details of the
results (\ref{psiwconH1}--e).
Finally, (\ref{vrhodcon}) follows from (\ref{E1}) and (\ref{psisconL2}).

It follows from
(\ref{rhodwcon}--c),
(\ref{uwconH1}--c), (\ref{psiwconH1}--f), (\ref{cdef}),
(\ref{bddef}), (\ref{lbdgen}), (\ref{bgen},b), (\ref{agen},b)
and
(\ref{cal K})
that we may pass to the limit $\delta \rightarrow 0_+$ in
(\ref{cd}--c) to obtain that $(\rhokaL^{[\Delta t,n\red{]}},\utkaL^n,\hpsikaL^n)$
solves (\ref{c}), (\ref{bLM}),
and (\ref{genLM}); that is, (\ref{rhonL}--c).

Finally, as $(\rhokaL^0,\utkaL^0,\hpsikaL^0) \in L^{\Gamma}_{\geq 0}(\Omega) \times
\Ht^1_0(\Omega) \times Z_2$, performing
the above existence proof at each time level $t_n$, $n=1,\ldots,N$,
yields a solution  $\{(\rhokaL^{[\Delta t],n},\utkaL^n,\hpsikaL^n)\}_{n=1}^N$
to (P$^{\Delta t}_{\kappa,\alpha,L}$) with $\rhokaL^n(\cdot) = \rhokaL^{[\Delta t],n}(\cdot,t_n)$, $n=1,\dots,N$,
by noting that $\rhokaL^{[\Delta t],n}$ thus constructed is an element of $C([t_{n-1},t_n];L^2(\Omega))$,
$n=1,\dots,N$.
\end{proof}

\section{Existence of a solution to (P$_{\kappa,\alpha}$)}
\label{sec:entropy}
\setcounter{equation}{0}

Next, we derive bounds on the solution of $({\rm P}^{\Delta t}_{\kappa,\alpha,L})$,
independent of $\Delta t$ and $L$.
Our starting point is Lemma \ref{conv}, concerning the existence of a solution to the problem
$({\rm P}^{\Delta t}_{\kappa,\alpha,L})$. The model $({\rm P}^{\Delta t}_{\kappa,\alpha,L})$
includes `microscopic cut-off' in the drag and convective terms of the Fokker--Planck equation,
where $L>1$ is a
(fixed, but otherwise arbitrary,) cut-off parameter.
Our next objective is to pass to the limits $L \rightarrow \infty$
and $\Delta t \rightarrow 0_+$ in the model $({\rm P}^{\Delta t}_{\kappa,\alpha,L})$,
with $L$ and $\Delta t$ linked by the condition
$\Delta t = o(L^{-1})$,
as $L \rightarrow \infty$.
To that end, we need to develop various bounds on sequences of weak solutions
of $({\rm P}^{\Delta t}_{\kappa,\alpha,L})$ that are uniform in the time step $\Delta
t$ and the cut-off parameter
$L$, and thus permit the extraction of weakly convergent subsequences,
as $L \rightarrow \infty$, through the use of a weak compactness argument.
The derivation of such bounds, based on the use of the relative entropy associated
with the Maxwellian $M$, is our
main task in this section.


We define
\begin{subequations}
\begin{align}
&\rhokaLD := \rhokaL^{[\Delta t],n},
\quad t\in [t_{n-1},t_n], \quad n=1,\dots,N,
\quad \Rightarrow
\quad \rhokaLD(\cdot,t_n)=\rhokaL^n(\cdot), \quad n=0,\dots,N,
\label{rhokaLD}\\
&\rhokaL^{\{\Delta t\}}(\cdot,t) :=
\frac{1}{\Delta t}\int_{t_{n-1}}^{t_n}\rhokaL^{[\Delta t]}(\cdot,s)\,{\rm d}s,
\quad t\in (t_{n-1},t_n], \quad n=1,\dots,N.
\label{rhokaLDav}
\end{align}
\end{subequations}
Further, we define the pressure variable
\begin{align}
\pkaL^{\{\Delta t\}}(\cdot,t) := \frac{1}{\Delta t}
\int_{t_{n-1}}^{t_n} \pk(\rhokaL^{[\Delta t]}(\cdot,s)) \,{\rm d}s,
\qquad t\in(t_{n-1},t_n], \quad n=1,\dots,N,
\label{pkaL}
\end{align}
and the momentum variable
\begin{align}
\mtkaL^n:= \rhokaL^n\, \utkaL^n, \qquad n=0,\ldots,N.
\label{mtkaLn}
\end{align}

We then introduce the following definitions:
\begin{subequations}
\begin{alignat}{2}
\utkaLD(\cdot,t)&:=\,\frac{t-t_{n-1}}{\Delta t}\,
\utkaL^n(\cdot)+
\frac{t_n-t}{\Delta t}\,\utkaL^{n-1}(\cdot), &&\quad t\in [t_{n-1},t_n], \quad n=1,\dots,N, \label{ulin}\\
\utkaLDp(\cdot,t)&:=\utkaL^n(\cdot),\qquad
\utkaLDm(\cdot,t):=\utkaL^{n-1}(\cdot),
&&\quad t\in(t_{n-1},t_n], \quad n=1,\dots,N. \label{upm}
\end{alignat}
\end{subequations}
We shall adopt $\utkaL^{\Delta t (,\pm)}$ as a collective symbol for $\utkaL^{\Delta t}$,
$\utkaL^{\Delta t,\pm}$.
The corresponding notations
$\hpsikaL^{\Delta t}$, $\hpsikaL^{\Delta t,\pm}$ and $\hpsikaL^{\Delta t (,\pm)}$;
$\rhokaL^{\Delta t}$, $\rhokaL^{\Delta t,\pm}$ and $\rhokaL^{\Delta t (,\pm)}$;
$\mtkaL^{\Delta t}$, $\mtkaL^{\Delta t,\pm}$ and $\mtkaL^{\Delta t (,\pm)}$,
and $\vrhokaL^{\Delta t}$, $\vrhokaL^{\Delta t,\pm}$ and $\vrhokaL^{\Delta t (,\pm)}$
are defined analogously.
The notation $\rhokaL^{\Delta t}$ signifying the
piecewise linear interpolant of $\rhokaL^{[\Delta t]}$
with respect to the variable $t$ is not to be confused with
$\rhokaL^{[\Delta t]}$, itself,
which denotes the function defined piecewise, over the union of
time slabs $\Omega \times [t_{n-1},t_n]$, $n=1,\dots,N$, solving
\eqref{rhonL} subject to the initial condition $\rhokaL^{[\Delta t]}(\cdot,t_{n-1}) = \rhokaL^{n-1}(\cdot)$,
$n=1,\dots, N$, with $\rhokaL^0:= \rho^0$.

%
%


Using the above notation,
(\ref{rhonL}--c) summed for $n=1, \dots,  N$ can be restated in the form:
find $(\rhokaLD(\cdot,t),\utkaLD(\cdot,t),
\hpsikaLD(\cdot,t))
\in H^1(\Omega) \cap L^{\frac{4\Gamma}{3}}_{\geq 0}(\Omega) \times \Ht^1_0(\Omega) \times (X \cap Z_2)$,
with $\frac{\partial \rhokaLD}{\partial t}(\cdot,t) \in H^1(\Omega)'$,
a.e.\ $t \in (0,T)$, such that $\mtkaLD$ is defined via (\ref{mtkaLn}) and
%
%
\begin{subequations}
\begin{align}
&\displaystyle\int_{0}^{T}\left\langle \frac{\partial \rhokaL^{[\Delta t]}}{\partial t}\,,\eta
\right\rangle_{H^1(\Omega)} \dd t
+ \int_0^T \int_\Omega \left( \alpha \,\nabx \rhokaLD -\rhokaLD \,\utkaLDp \right)
\cdot \nabx \eta
\,\dx \,\dt =0
\nonumber \\
& \hspace{3.4in}
\qquad
\forall \eta \in L^2(0,T;H^1(\Omega)),
\label{eqrhocon}
\end{align}
\begin{align}
&\displaystyle\int_{0}^{T}\!\! \int_\Omega \left[
\frac{\partial \mtkaLD}{\partial t} 
- \frac{1}{2}
\frac{\partial \rhokaL^{\Delta t}}{\partial t}\, \utkaLDp
\right]
\cdot
\wt \,\dx\, \dt
+
\displaystyle\int_{0}^{T}\!\! \int_\Omega
\Stt(\utkaLDp)
:\nabxtt \wt
\,\dx\, \dt
\nonumber \\
&\qquad \qquad +
\frac{1}{2} \int_{0}^T\!\! \int_{\Omega}
\left[ \left[ (\mtkaLDm \cdot \nabx) \utkaLDp \right]\cdot\,\wt
- \left[ (\mtkaLDm \cdot \nabx) \wt  \right]\cdot\,\utkaLDp
\right]\!\dx\, \dt
\nonumber \\ & \qquad \qquad
- \int_{0}^T \!\!\int_{\Omega}
\pkaL^{\{\Delta t\}} \,\nabx \cdot \wt \,\dx \,\dt
\nonumber \\
&\qquad =\int_{0}^T
 \int_{\Omega} \left[ \rhokaL^{\Delta t,+}\,
\ft^{\{\Delta t\}} \cdot \wt  
- 
\tautt_1 (M\,\hpsikaLDp)
: \nabxtt
\wt \right] \dx \, \dt
\nonumber \\
&\qquad \qquad
- 2 \,\mathfrak{z}
\int_0^T\!\! \int_{\Omega}
\left( \int_D M\,\beta^L(\hpsikaLDp) \,\dq \right)
\nabx 
\vrhokaLDp
\cdot \wt \,\dx\,\dt
\qquad
\forall \wt \in L^2(0,T;\Vt),
\label{equncon}
\end{align}
%
%
%
\begin{align}
\label{eqpsincon}
&\int_{0}^T \!\!\int_{\Omega \times D}
M\,\frac{ \partial \hpsikaLD}{\partial t} 
\varphi \dq \dx \dt
+
\frac{1}{4\,\lambda}
\,\sum_{i=1}^K
 \,\sum_{j=1}^K A_{ij}
\int_{0}^T \!\!\int_{\Omega \times D} \,
M\,
 \nabqj \hpsikaLDp
\cdot\, \nabqi
\varphi
\dq \dx \dt
\nonumber \\
& \qquad 
+ \int_{0}^T \!\!\int_{\Omega \times D} M \left[
\epsilon\,
\nabx \hpsikaLDp
- \utkaLDp\,
\beta^L(\hpsikaLDp) \right]\cdot\, \nabx
\varphi
\dq \dx \dt
\nonumber \\
&
\qquad 
- \int_{0}^T \!\!\int_{\Omega \times D} M\,\sum_{i=1}^K
\left[\sigtt(\utkaLDp)
\,\qt_i\right]
\beta^L(\hpsikaLDp) \,\cdot\, \nabqi
\varphi
\,\dq \dx \dt = 0
\quad
\forall \varphi \in L^2(0,T;X);
\end{align}
\end{subequations}
%
subject to the initial conditions
$\rhokaL^{\Delta t}(0)=\rho^0 \in L^\infty_{\geq 0}(\Omega)$, $\utkaLD(0) = \ut^0
\in \Ht^1_0(\Omega)$ and
$\hpsikaLD(0) = \hpsi^0
\in X \cap Z_2$, where we recall (\ref{proju0}) and (\ref{psi0}).
We emphasize that (\ref{eqrhocon}--c) 
is an equivalent restatement of problem (${\rm P}^{\Delta t}_{\kappa,\alpha,L}$),
for which existence of a solution has been established (cf. Lemma \ref{conv}).

We are now ready to embark on the derivation of the required bounds,
uniform in the time step $\Delta t$ and the cut-off parameter $L$,
on norms of $\rhokaL^{[\Delta t]}(t) \in H^1_0(\Omega)\cap L^{\frac{4 \Gamma}{3}}_{\geq 0}(\Omega)$,
$\utkaLDp(t)\in \Ht^{1}_0(\Omega)$, $\hpsikaLDp(t) \in X \cap Z_2$
and $\vrhokaLDp(t) \in H^1(\Omega)$, $t \in (0,T]$.

\subsection{$L$, $\Delta t$-independent bounds on the spatial derivatives of $\utkaLD$ and $\hpsikaLD$}
\label{Lindep-space}
%
We note that it is \textit{not} possible to pass to the limit $\delta \rightarrow 0_+$ in (\ref{E1})
to obtain strong enough $L$-independent bounds due to the fourth, seventh and eighth
of the ten terms on the left-hand side.
Similarly, it is \textit{not} possible to pass to the limit in these terms
even before we use the bound $[{\mathcal F}^L_{\delta}]''(\cdot) \geq \frac{1}{L}$;
recall its use in (\ref{acorstab1}) to obtain (\ref{Ek}), and hence (\ref{E1}).
However, it is a simple matter to pass to the
limit $\delta \rightarrow 0_+$ in (\ref{Ektrhod}).
Noting (\ref{rhodwconG},d), (\ref{uwconH1}), (\ref{CwconL2},f) and
the convexity of $P_\kappa(\cdot)$, we may pass to the limit
$\delta \rightarrow 0_+$ in (\ref{Ektrhod}) to obtain for $n=1,\ldots,N$ that
\begin{align}
&\frac{1}{2}\,\displaystyle
\int_{\Omega} \left[ \rhokaL^{n} 
\,|\utkaL^n|^2 + \rhokaL^{n-1}\,
|\utkaL^n-\utkaL^{n-1}|^2 \right]
\dx
+ \int_{\Omega} \Pk(\rhokaL^{n}) 
\dx
\nonumber\\
&\;
+ \alpha\,\kappa
\int_{t_{n-1}}^{t_n}\!\!
\left[ \|\nabx[(\rhokaL^{[\Delta t],n})^2]\|_{L^2(\Omega)}^2 + \frac{4}{\Gamma}\,\|\nabx
[(\rhokaL^{[\Delta t],n})^{\frac{\Gamma}{2}}]\|_{L^2(\Omega)}^2 \right]\!
\dt
+ \Delta t\,\mu^S\,c_0\, \|\utkaL^n\|_{H^1(\Omega)}^2
\nonumber \\
&\;+ \mathfrak{z}\,\left[ \|\vrhokaL^n\|_{L^2(\Omega)}^2 +
\|\vrhokaL^n-\vrhokaL^{n-1}\|_{L^2(\Omega)}^2\right]
+ 2\,\Delta t \,\mathfrak{z}\,\varepsilon\,\|\nabx \vrhokaL^n\|_{L^2(\Omega)}^2
\nonumber
\\
&\hspace{0.2in}
\leq
\frac{1}{2}\,\displaystyle
\int_{\Omega} \rhokaL^{n-1}\,|\utkaL^{n-1}|^2 \dx
+ \int_{\Omega} \Pk(\rhokaL^{n-1})  \dx
+ \Delta t
\int_{\Omega}\rhokaL^{n} 
\,\ft^n \cdot \utkaL^n \dx
\nonumber \\
& \hspace{0.5in}
+ \mathfrak{z}\,\|\vrhokaL^{n-1}\|_{L^2(\Omega)}^2
+ k\,\Delta t \,
(K+1) \int_\Omega 
\vrhokaL^n\,
\nabx \cdot \utkaL^n \,\dx
\nonumber \\
&\hspace{0.5in}
- k\,\Delta t \sum_{i=1}^K
\int_{\Omega}
 \Ctt_i(M
 \,\hpsikaL^n): \nabxtt \utkaL^n \dx.
\label{Ektrho}
\end{align}
Summing the above over $n$, and adopting the notation
(\ref{rhokaLD}), (\ref{ulin},b) and (\ref{vrho0}), we obtain
for $n=1,\ldots N$ that
\begin{align}
&\frac{1}{2}\,\displaystyle
\int_{\Omega} \rhokaLDp(t_n)
\,|\utkaLDp(t_n)|^2 \,\dx
+ \frac{1}{2\Delta t}\,\int_{0}^{t_n} \int_{\Omega} \rhokaLDm\,
|\utkaLDp-\utkaLDm|^2 \,
\dx\,\dt
\nonumber \\
& \qquad
+ \int_{\Omega} \Pk(\rhokaLDp(t_n))  \dx
+  \mu^S c_0\, \int_0^{t_n} \|\utkaLDp \|_{H^1(\Omega)}^2 \dt
\nonumber\\
&\qquad
+ \alpha\,\kappa
\int_{0}^{t_n}
\left[ \|\nabx[(\rhokaL^{[\Delta t]})^2]\|_{L^2(\Omega)}^2 + \frac{4}{\Gamma}\,\|\nabx
[(\rhokaL^{[\Delta t]})^{\frac{\Gamma}{2}}]\|_{L^2(\Omega)}^2 \right]
\dt
+ \mathfrak{z}\,\|\vrhokaLDp(t_n)\|_{L^2(\Omega)}^2\nonumber\\
&\qquad+ \mathfrak{z}\int_0^{t_n} \left[ \|\vrhokaLDp-\vrhokaLDm\|_{L^2(\Omega)}^2
+ 2\,\varepsilon
\,\|\nabx \vrhokaLDp\|_{L^2(\Omega)}^2 \right] \dt
\nonumber
\\
&\hspace{0.2in}
\leq
\frac{1}{2}\,\displaystyle
\int_{\Omega} \rho^0\,|\ut^{0}|^2 \dx
+ \int_{\Omega} \Pk(\rho^0)  \dx
+ \mathfrak{z}\,\|\varrho^0\|_{L^2(\Omega)}^2
\nonumber \\
& \hspace{0.5in}
+ \int_0^{t_n}
\int_{\Omega}\rhokaLDp
\,\ft \cdot \utkaLDp \dx \,\dt
+ k\,(K+1) \int_0^{t_n} \,
 \int_\Omega 
\vrhokaLDp\,
\nabx \cdot \utkaLDp \,\dx\,\dt
\nonumber \\
&\hspace{0.5in}
- k \, \sum_{i=1}^K \int_0^{t_n}
\int_{\Omega}
 \Ctt_i(M
 \,\hpsikaLDp): \nabxtt \utkaLDp \dx \,\dt .
\label{energy-uL}
\end{align}

We now require the appropriate $\hpsikaL^n$ analogue of (\ref{acorstab1}).
The \red{appropriate} choice of test function in
(\ref{psiG})
for this purpose
is $\varphi = 
[\mathcal{F}^L]'(\hpsikaL^n)$. 
While Lemma \ref{conv} guarantees that $\hpsikaL^n$ belongs to $Z_2$, 
and is therefore nonnegative a.e. on $\Omega \times D$, 
there is unfortunately
no reason why $\hpsikaL^n$ should be strictly positive on $\Omega\times D$, 
and therefore the
expression $[\mathcal{F}^L]'(\hpsikaL^n)$ may in general
be undefined. 
Similarly to (\ref{Pkentreg}), we shall circumvent this problem
by choosing
$\varphi=[\mathcal{F}^L]'(\hpsikaL^n + \varsigma)$ in
(\ref{psiG}), which leads, for any fixed $\varsigma \in {\mathbb R}_{>0}$, to
%

\begin{align}\label{z-terms}
0&=\int_{\Omega \times D} M\, \frac{
\hpsikaL^n - 
\hpsikaL^{n-1}}{\Delta t} \, [[\mathcal{F}^L]'(\hpsikaL^n + \varsigma)]\dq \dx\nonumber\\
&\hspace{0.2in}- \int_{\Omega \times D} M\, 
\utkaL^n \cdot \left(\nabx [[\mathcal{F}^L]'(\hpsikaL^n + \varsigma)]\right)\,\beta^L(\hpsikaL^n)
\,\dq \dx\nonumber\\
&\hspace{0.2in} + \frac{1}{4\lambda}\,\sum_{i=1}^K \sum_{j=1}^K
\int_{\Omega\times D}  A_{ij}\, M\,
\nabqj \hpsikaL^n \cdot \nabqi [[\mathcal{F}^L]'(\hpsikaL^n + \varsigma)]\dq \dx\nonumber\\
&\hspace{0.2in} + \varepsilon\int_{\Omega \times D}  M\, \nabx \hpsikaL^n \cdot \nabx [[\mathcal{F}^L]'
(\hpsikaL^n + \varsigma)]
\dq \dx
\nonumber \\
&\hspace{0.2in} -
\sum_{i=1}^K
\int_{\Omega\times D} M\,
\beta^L(\hpsikaL^n)
[\sigtt(\utkaL^n) \, \qt_i] \cdot \nabqi [[\mathcal{F}^L]'(\hpsikaL^n + \varsigma)]\dq \dx 
\nonumber\\
&\hspace{0.2in}=:\sum_{i=1}^5 {\tt T}_i.
\end{align}

It follows from (\ref{eq:FL2a}) that
%
\begin{align}\label{t1t2}
{\tt T}_1 
&\geq
\frac{1}{\Delta t}\int_{\Omega \times D} M\, 
\left[\mathcal{F}^L(\hpsikaL^n+\varsigma) 
-\mathcal{F}^L(\hpsikaL^{n-1}+\varsigma)\right]\dq \dx
\nonumber\\&\hspace{1in}
+\frac{1}{2\,\Delta t\,L} \int_{\Omega \times D} M\, 
(\hpsikaL^n -\hpsikaL^{n-1})^2 \dq \dx.
\end{align}
In addition, it follows from (\ref{eq:FL2a}) that
\begin{align}
{\tt T}_2 &= - \int_{\Omega \times D} M \,\frac{\beta^L(\hpsikaL^n)}{\beta^L(\hpsikaL^n+\varsigma)}
\, \utkaL^n
\cdot \nabx \hpsikaL^n \,\dq \dx
\nonumber \\
&=
\int_{\Omega} \left(\int_{D} M \,\hpsikaL^n\,\dq\right)
\nabx \cdot \utkaL^n \,\dx
\nonumber \\
& \hspace{1in} + \int_{\Omega \times D} M
\,\left[1-\frac{\beta^L(\hpsikaL^n)}{\beta^L(\hpsikaL^n+\varsigma)}\right]
\utkaL^n \cdot \nabx \hpsikaL^n \,\dq \dx.
\label{t2}
\end{align}
Thanks to \eqref{A}, we have that
\begin{subequations}
\begin{align}\label{t3}
{\tt T}_3  
&\geq \frac{a_0}{4\lambda}\int_{\Omega \times D}
M \,  [[\mathcal{F}^L]''(\hpsikaL^n + \varsigma)] \,
|\nabq \hpsikaL^n|^2 \dq \dx,
\\
{\tt T}_4
&\geq \varepsilon\int_{\Omega \times D} M \,
[[\mathcal{F}^L]''(\hpsikaL^n + \varsigma)] \,|\nabx \hpsikaL^n|^2
\,\dq \dx.
\label{t4}
\end{align}
\end{subequations}
It is tempting to bound $ [\mathcal{F}^L]''(\hpsikaL^n + \varsigma)$ below further by
$(\hpsikaL^n + \varsigma)^{-1}$ using \eqref{eq:FL2b}. We have refrained from doing so as the
precise form of \eqref{t4} will be required to absorb the extraneous term that the process of shifting
$\hpsikaL^n$ by the addition of $\varsigma>0$ generates in the last term 
in (\ref{t2}). Similarly, (\ref{t3}) is required for the last
line in \eqref{t5} below.
%
%
%
Finally, it follows from (\ref{eq:FL2a}) and (\ref{eqM}) that
\begin{align}
{\tt T}_5 
&= - \red{\sum_{i=1}^K}\int_\Omega 
\left[\int_D M [(\nabxtt\utkaL^n)\, \qt_i] \cdot \nabqi\hpsikaL^n \dq \right] \dx\nonumber\\
&\qquad + \int_{\Omega \times D} M\, 
\left[1 - \frac{\beta^L(\hpsikaL^n)}{\beta^L(\hpsikaL^n+\varsigma)} \right]
\sum_{i=1}^K\,[(\nabxtt \utkaL^n)\,\qt_i]
\cdot \nabqi \hpsikaL^n \dq \dx
\nonumber
%
%
\nonumber\\
&=- \int_{\Omega\times D} M \, \sum_{i=1}^K U'(\textstyle{\frac{1}{2}}|\qt_i|^2)\, 
\hpsikaL^n \, (\qt_i\, \qt_i^{\tt T}):\nabxtt \utkaL^n \dq\dx
\nonumber \\ & \qquad
+K\int_{\Omega} \left(\int_{D} M \,\hpsikaL^n\,\dq\right)
\nabx \cdot \utkaL^n \,\dx
\nonumber\\
&\qquad + \int_{\Omega \times D} M\, 
\left[1 - \frac{\beta^L(\hpsikaL^n)}{\beta^L(\hpsikaL^n+\varsigma)} \right]
\sum_{i=1}^K\,[(\nabxtt \utkaL^n)\,\qt_i]
\cdot \nabqi \hpsikaL^n \dq \dx.
\label{t5}
\end{align}
%
Substituting (\ref{t1t2})--(\ref{t5}) into (\ref{z-terms}), multiplying by $\Delta t$,
summing over $n$ and adopting the notation (\ref{ulin},b) yields, for $n=1,\ldots,N$, that
%
\begin{align}
&\int_{\Omega \times D} M\, 
\mathcal{F}^L(\hpsikaLDp(t_n) + \varsigma) \dq \,\dx
+\frac{1}{2 \,\Delta t\, L}\int_0^{t_n}\int_{\Omega \times D}
M\,
(\hpsikaLDp - \hpsikaLDm)^2 \dq\, \dx\, \dt
\nonumber\\
&\qquad +\int_0^{t_n} \int_{\Omega \times D}
M\,
[[\mathcal{F}^L]''(\hpsikaLDp + \varsigma)]\,
\left[\frac{a_0}{4\,\lambda}\,
|\nabq \hpsikaLDp |^2
+ \varepsilon\, |\nabx \hpsikaLDp |^2
\right]
\,\dq \,\dx\, \dt
\nonumber
\end{align}
\begin{align}
&\quad\leq \int_{\Omega \times D} M\, 
\mathcal{F}^L(\beta^L(\hpsi^0) + \varsigma)
\dq \,\dx
\nonumber \\
& \quad \qquad -(K+1) \int_0^{t_n}\int_{\Omega} \left(\int_{D} M \,\hpsikaLDp\,\dq\right)
\nabx \cdot \utkaLDp \,\dx \,\dt
\nonumber\\
&\quad \qquad+ \int_0^{t_n} \int_{\Omega \times D} M\,\sum_{i=1}^K
\,U_i'(\textstyle{\frac{1}{2}|\qt|^2})\,
\hpsikaLDp \,(\qt_i\,\qt_i^{\rm T}):
\nabxtt \utkaLDp \dq\, \dx\, \dt
\nonumber\\
&\quad \qquad- \int_0^{t_n} \int_{\Omega \times D} M\,
\left[1 -\frac{\beta^L(\hpsikaLDp)}{\beta^L(\hpsikaLDp
+ \varsigma)}\right] \utkaLDp \cdot \nabx \hpsikaLDp \dq \,\dx\, \dt
\nonumber \\
&\quad \qquad- \int_0^{t_n} \int_{\Omega \times D} M\,
\left[1 -\frac{\beta^L(\hpsikaLDp)}{\beta^L(\hpsikaLDp
+ \varsigma)}\right] \sum_{i=1}^K
\left[(\nabxtt \utkaLDp)\,\qt_i\right]
\cdot \nabqi \hpsikaLDp\, \dq \,\dx\, \dt,
\label{eq:energy-psi-summ1}
\end{align}
where we have noted (\ref{psi0conv}).
%
The denominator in the prefactor of the second integral on the left-hand side motivates us to link $\Delta t$ to $L$
so that $\Delta t\, L = o(1)$,
as $\Delta t \!\rightarrow\! 0_{+}$ (or, equivalently, $\Delta t = o(L^{-1})$, as
$L \rightarrow \infty$), in order to drive the integral multiplied by the prefactor to $0$ in
the limit of $\Delta t \rightarrow 0_+$,
once the product of the two has been bounded above by a constant, independent of
$\Delta t$ and $L$.

Comparing \eqref{eq:energy-psi-summ1} with \eqref{energy-uL}, and noting (\ref{eqCtt}),
we see that after multiplying
\eqref{eq:energy-psi-summ1} by $k$ and adding the resulting inequality to \eqref{energy-uL}
the last two terms on the right-hand side of \eqref{energy-uL} are cancelled by $k$ times the
second and third terms on the right-hand side of \eqref{eq:energy-psi-summ1}.
Hence, for 
$n 
=1,\ldots,N$, 
we deduce that
\begin{align}
&\frac{1}{2}\,\displaystyle
\int_{\Omega} \rhokaLDp(t_n)
\,|\utkaLDp(t_n)|^2 \,\dx
+ \frac{1}{2\Delta t}\,\int_{0}^{t_n} \int_{\Omega} \rhokaLDm\,
|\utkaLDp-\utkaLDm|^2 \,
\dx\,\dt
\nonumber \\
& \qquad
+ \int_{\Omega} \Pk(\rhokaLDp(t_n))  \dx
+k\,\int_{\Omega \times D} M\,
\mathcal{F}^L(\hpsikaLDp(t_n) + \varsigma) \dq \,\dx
\nonumber\\
&\qquad
+ \alpha\,\kappa
\int_{0}^{t_n}
\left[ \|\nabx[(\rhokaL^{[\Delta t]})^2]\|_{L^2(\Omega)}^2 + \frac{4}{\Gamma}\,\|\nabx
[(\rhokaL^{[\Delta t]})^{\frac{\Gamma}{2}}]\|_{L^2(\Omega)}^2 \right]
\dt
\nonumber
\\
&\qquad
+  \mu^S c_0\, \int_0^{t_n} \|\utkaLDp \|_{H^1(\Omega)}^2 \dt
+\frac{k}{2 \,\Delta t\, L}\int_0^{t_n}\int_{\Omega \times D}
M\,
(\hpsikaLDp - \hpsikaLDm)^2 \dq\, \dx\, \dt
\nonumber
\\
&\qquad + k\,\int_0^{t_n} \int_{\Omega \times D}
M\,
[[\mathcal{F}^L]''(\hpsikaLDp + \varsigma)]\,
\left[\frac{a_0}{4\,\lambda}\,
|\nabq \hpsikaLDp |^2
+ \varepsilon\, |\nabx \hpsikaLDp |^2
\right]
\,\dq \,\dx\, \dt
\nonumber \\
& \qquad + \mathfrak{z}\,\|\vrhokaLDp(t_n)\|_{L^2(\Omega)}^2
+ \mathfrak{z}\int_0^{t_n} \left[ \|\vrhokaLDp-\vrhokaLDm\|_{L^2(\Omega)}^2
+ 2\,\varepsilon
\,\|\nabx \vrhokaLDp\|_{L^2(\Omega)}^2 \right] \dt
\nonumber
\\
&\hspace{0.2in}
\leq
\frac{1}{2}\,\displaystyle
\int_{\Omega} \rho^0\,|\ut^{0}|^2 \dx
+ \int_{\Omega} \Pk(\rho^0)  \dx
+ \mathfrak{z}\,
\|\varrho^0\|_{L^2(\Omega)}^2
\nonumber \\
&\hspace{0.3in}
+ \int_0^{t_n}
\int_{\Omega}\rhokaLDp
\,\ft \cdot \utkaLDp \dx \,\dt
+k\,\int_{\Omega \times D} M\, 
\mathcal{F}^L(\beta^L(\hpsi^0) + \varsigma)
\dq \,\dx
\nonumber 
\\
&\hspace{0.3in}- k\,\int_0^{t_n} \int_{\Omega \times D} M\,
\left[1 -\frac{\beta^L(\hpsikaLDp)}{\beta^L(\hpsikaLDp
+ \varsigma)}\right] \utkaLDp \cdot \nabx \hpsikaLDp \dq \,\dx\, \dt
\nonumber 
\\
&\hspace{0.3in}- k\,\int_0^{t_n} \int_{\Omega \times D} M\,
\left[1 -\frac{\beta^L(\hpsikaLDp)}{\beta^L(\hpsikaLDp
+ \varsigma)}\right] \sum_{i=1}^K
\left[(\nabxtt \utkaLDp)\,\qt_i\right]
\cdot \nabqi \hpsikaLDp\, \dq \,\dx\, \dt.
\label{eq:energy-u+psi}
\end{align}

Similarly to (\ref{trhoin}), we have on choosing $\eta=1$ in (\ref{rhonL}) that
\begin{align}
\int_{\Omega} \rhokaL^n \,\dx = \int_{\Omega} \rhokaL^{n-1}\,\dx=
\int_{\Omega} \rho^0 \,\dx, \qquad n=1,\ldots,N.
\label{rhonint}
\end{align}
Noting (\ref{rhonint}) and (\ref{upm}), we have that
\begin{align}
\left|\int_0^{t_n} \int_{\Omega}\rhokaLDp
\,\ft \cdot \utkaLDp \dx \,\dt\right|
\leq \frac{1}{2} \left[ \int_0^{t_n} \int_{\Omega}\rhokaLDp\,|\utkaLDp|^2\,\dx\,\dt
+ \int_{0}^{t_n} \|\ft\|_{L^\infty(\Omega)}^2 \dt \int_\Omega \rho^0 \dx  \right].
\nonumber\\
\label{fbd}
\end{align}
Next we recall from
\cite[(4.25)]{BS2011-fene} the bound
\begin{align}
\int_{\Omega \times D} M\, 
\mathcal{F}^L(\beta^L(\hpsi^0) + \varsigma) \dq \dx
\leq
 \frac{3\varsigma}{2}\, 
 |\Omega|
+ \int_{\Omega \times D} M\, 
\mathcal{F}(\hpsi^0 + \varsigma) \dq \dx.
\label{FLbLbd}
\end{align}
Let $\bt := (b_1,\ldots,b_K)$, recall (\ref{inidata}), and $b :=|\bt|_1 := b_1 +\cdots + b_K$;
then 
we can bound the magnitude of the last term on the right-hand side of \eqref{eq:energy-u+psi},
on noting (\ref{eq:FL2a}) and (\ref{Mint1}), by
\begin{align}
&\frac{k\,a_0}{8 \lambda} \left(\int_0^{t_n} \int_{\Omega \times D}
M [[\mathcal{F}^L]''(\hpsikaLDp + \varsigma)]
\, |\nabq \hpsikaLDp|^2\, \dq\, \dx\, \dt\right)
\nonumber \\
&\hspace{2.5in}
+ \varsigma\,\frac{2k\,\lambda\, b}{a_0}
\left(\int_0^{t_n} \int_{\Omega}  |\nabxtt \utkaLDp|^2 \dx \,\dt\right),
\label{dragbd}
\end{align}
%
%
see \cite[(4.20)]{BS2010-hookean} for the details.
Similarly, the second to last term on the right-hand side of \eqref{eq:energy-u+psi}
can be bounded by
\begin{align}
&\frac{k\,\varepsilon}{2} \left(\int_0^{t_n} \int_{\Omega \times D}
M [[\mathcal{F}^L]''(\hpsikaLDp + \varsigma)]
\, |\nabx \hpsikaLDp|^2\, \dq\, \dx\, \dt\right)
+ \varsigma\,\frac{k}{2\,\varepsilon}
\left(\int_0^{t_n} \int_{\Omega}  |\utkaLDp|^2 \dx \,\dt\right).
\label{convbd}
\end{align}
Noting (\ref{eq:energy-u+psi})--(\ref{convbd}), 
and using \eqref{eq:FL2b} to bound the expression $[\mathcal{F}^L]''(\hpsikaLDp + \varsigma)$
 from below by $\mathcal{F}''(\hpsikaLDp + \varsigma)= (\hpsikaLDp
 + \varsigma)^{-1}$ and
\eqref{eq:FL2c} to bound $\mathcal{F}^L(\hpsikaLDp + \varsigma)$
by $\mathcal{F}(\hpsikaLDp+\varsigma)$ from below yields, for 
$n = 1,\ldots, N$, that
\begin{align}\label{eq:energy-u+psi1}
&\frac{1}{2}\,\displaystyle
\int_{\Omega} \rhokaLDp(t_n)
\,|\utkaLDp(t_n)|^2 \,\dx
+ \frac{1}{2\Delta t}\,\int_{0}^{t_n} \int_{\Omega} \rhokaLDm\,
|\utkaLDp-\utkaLDm|^2 \,
\dx\,\dt
\nonumber \\
& \qquad
+ \int_{\Omega} \Pk(\rhokaLDp(t_n))  \dx
+k\,\int_{\Omega \times D} M\,
\mathcal{F}(\hpsikaLDp(t_n) + \varsigma) \dq \,\dx
\nonumber\\
&\qquad
+ \alpha\,\kappa
\int_{0}^{t_n}
\left[ \|\nabx[(\rhokaL^{[\Delta t]})^2]\|_{L^2(\Omega)}^2 + \frac{4}{\Gamma}\,\|\nabx
[(\rhokaL^{[\Delta t]})^{\frac{\Gamma}{2}}]\|_{L^2(\Omega)}^2 \right]
\dt
\nonumber \\
&\qquad
+  \mu^S c_0\, \int_0^{t_n} \|\utkaLDp \|_{H^1(\Omega)}^2 \dt
+\frac{k}{2 \,\Delta t\, L}\int_0^{t_n}\int_{\Omega \times D}
M\,
(\hpsikaLDp - \hpsikaLDm)^2 \dq\, \dx\, \dt
\nonumber\\
&\qquad + \frac{k}{2}\,\int_0^{t_n} \int_{\Omega \times D}
\frac{M}{\hpsikaLDp + \varsigma}\,
\left[\frac{a_0}{4\,\lambda}\,
|\nabq \hpsikaLDp |^2
+ \varepsilon\, |\nabx \hpsikaLDp |^2
\right]
\,\dq \,\dx\, \dt
\nonumber \\
& \qquad + \mathfrak{z}\,\|\vrhokaLDp(t_n)\|_{L^2(\Omega)}^2
+ \mathfrak{z}\int_0^{t_n} \left[ \|\vrhokaLDp-\vrhokaLDm\|_{L^2(\Omega)}^2
+ 2\,\varepsilon
\,\|\nabx \vrhokaLDp\|_{L^2(\Omega)}^2 \right] \dt
\nonumber \\
&\hspace{0.2in}
\leq
\frac{1}{2}\,\displaystyle
\int_{\Omega} \rho^0\,|\ut^{0}|^2 \dx
+ \int_{\Omega} \Pk(\rho^0)  \dx
+ \mathfrak{z}\,
\|\varrho^0\|_{L^2(\Omega)}^2
\nonumber\\
&\hspace{0.3in}
+\frac{1}{2}  \int_{0}^{t_n} \|\ft\|_{L^\infty(\Omega)}^2 \dt \int_\Omega \rho^0 \dx
+k\,\int_{\Omega \times D} M\, 
\mathcal{F}(\hpsi^0 + \varsigma)
\dq \,\dx + \frac{3k\,\varsigma}{2}\,|\Omega|
\nonumber \\
&\hspace{0.3in}
+ 2k\,\varsigma\,\max\left\{\frac{\lambda\, b}{a_0},\frac{1}{4\varepsilon}\right\}\,
\int_0^{t_n} \|\utkaLDp\|_{H^1(\Omega)}^2 \dt
+ \frac{1}{2} \int_0^{t_n} \int_{\Omega}\rhokaLDp\,|\utkaLDp|^2\,\dx\,\dt.
\end{align}
Passing to the limit $\varsigma \rightarrow 0_+$ in (\ref{eq:energy-u+psi1}),
and then applying a \red{discrete} Gronwall inequality yields, for $n=1,\ldots, N$, that
\begin{align}\label{eq:energy-u+psi-final2}
&\frac{1}{2}\,\displaystyle
\int_{\Omega} \rhokaLDp(t_n)
\,|\utkaLDp(t_n)|^2 \,\dx
+ \frac{1}{2\Delta t}\,\int_{0}^{t_n} \int_{\Omega} \rhokaLDm\,
|\utkaLDp-\utkaLDm|^2 \,
\dx\,\dt
\nonumber \\
& \qquad
+ \int_{\Omega} \Pk(\rhokaLDp(t_n))  \dx
+k\,\int_{\Omega \times D} M\,
\mathcal{F}(\hpsikaLDp(t_n)) \dq \,\dx
\nonumber\\
&\qquad
+ \alpha\,\kappa
\int_{0}^{t_n}
\left[ \|\nabx[(\rhokaL^{[\Delta t]})^2]\|_{L^2(\Omega)}^2 + \frac{4}{\Gamma}\,\|\nabx
[(\rhokaL^{[\Delta t]})^{\frac{\Gamma}{2}}]\|_{L^2(\Omega)}^2 \right]
\dt
\nonumber \\
&\qquad
+  \mu^S c_0\, \int_0^{t_n} \|\utkaLDp \|_{H^1(\Omega)}^2 \dt
+\frac{k}{2 \,\Delta t\, L}\int_0^{t_n}\int_{\Omega \times D}
M\,
(\hpsikaLDp - \hpsikaLDm)^2 \dq\, \dx\, \dt
\nonumber\\
&\qquad + k\,\int_0^{t_n} \int_{\Omega \times D}
M\,
\left[\frac{a_0}{2\lambda}\,
\left|\nabq \sqrt{\hpsikaLDp} \right|^2
+ 2\varepsilon\, \left|\nabx \sqrt{\hpsikaLDp} \right|^2
\right]
\,\dq \,\dx\, \dt
\nonumber \\
& \qquad + \mathfrak{z}\,\|\vrhokaLDp(t_n)\|_{L^2(\Omega)}^2
+ \mathfrak{z}\int_0^{t_n} \left[ \|\vrhokaLDp-\vrhokaLDm\|_{L^2(\Omega)}^2
+ 2\,\varepsilon
\,\|\nabx \vrhokaLDp\|_{L^2(\Omega)}^2 \right] \dt
\nonumber \\
&\hspace{0.2in}
\leq {\rm e}^{t_n}\biggl[
\frac{1}{2}\,\displaystyle
\int_{\Omega} \rho^0\,|\ut^{0}|^2 \dx
+ \int_{\Omega} \Pk(\rho^0)  \dx
+k\,\int_{\Omega \times D} M\, 
\mathcal{F}(\hpsi^0)
\dq \,\dx
\nonumber\\
&\hspace{1in}
+ 
\mathfrak{z}\,
\|\varrho^0\|_{L^2(\Omega)}^2
+\frac{1}{2}  \int_{0}^{t_n} \|\ft\|_{L^\infty(\Omega)}^2 \dt \int_\Omega \rho^0 \dx
\biggr]
\nonumber \\
&\hspace{0.2in}
\leq {\rm e}^{t_n}\biggl[
\frac{1}{2}\,\displaystyle
\int_{\Omega} \rho^0\,|\ut_{0}|^2 \dx
+ \int_{\Omega} \Pk(\rho^0)  \dx
+k\,\int_{\Omega \times D} M\, 
\mathcal{F}(\hpsi_0)
\dq \,\dx
\nonumber\\
&\hspace{1in}
+
\mathfrak{z}\,\red{\int_{\Omega} \left(\int_D M \,\hpsi_0\dq\right)^2 \dx}
+
\frac{1}{2}  \int_{0}^{t_n} \|\ft\|_{L^\infty(\Omega)}^2 \dt \int_\Omega \rho^0 \dx
\biggr]
\nonumber\\
&\hspace{0.2in}\leq C,
\end{align}
where $C$ is a positive constant, independent of the parameters
$\Delta t$, $L$, $\alpha$ and $\kappa$.
Here, we have noted (\ref{idatabd}), (\ref{inidata-1}) and (\ref{vrho0}) for the penultimate
inequality in (\ref{eq:energy-u+psi-final2}), and
(\ref{rho0conv}) for the final inequality.

Next we bound the extra stress term (\ref{eqCtt}). As we do not have a bound on
$\|M^{\frac{1}{2}} \hpsikaLDp\|_{L^2(\Omega\times D)}$ in (\ref{eq:energy-u+psi-final2}),
we will need a weaker bound than (\ref{eqCttbd}). First, we deduce from
(\ref{eqCtt}), (\ref{eqM})
and as $M=0$ on $\partial D$ that
\begin{align}
\Ctt_i(M\,\varphi) = - \int_{D} (\nabqi M) \,\qt_i^{\rm T}\,\varphi \, \dq
= \int_{D} M\,(\nabqi \varphi) \,\qt_i^{\rm T} \, \dq
+ \left(\int_{D} M\,\varphi \dq\right) \,\Itt.
\label{Cittiden}
\end{align}
Hence, for $r \in [1,2)$, \red{on noting that $\nabqi \varphi = \nabqi (\sqrt{\varphi})^2 = 2 \sqrt{\varphi}\,\, \nabqi \sqrt{\varphi}$ for any
sufficiently smooth nonnegative function $\varphi$,} we have that
\begin{align}
\|\Ctt_i(M\,\varphi)\|_{L^r(\Omega)} 
&\leq C\left[
\int_{\Omega}  \left(\int_{D} M\,\varphi \dq\right)^{\frac{r}{2}}
\left(\int_{D} M\left|\nabqi \sqrt{\varphi}\right|^2 \dq\right)^{\frac{r}{2}}\dx +
\int_{\Omega} \left(\int_{D} M\,\varphi \dq\right)^r \dx \right]^{\frac{1}{r}}
\nonumber \\
& \leq C\left[ \|\nabqi \sqrt{\varphi}\|_{L^2_M(\Omega \times D)}\,
\left\|\int_{D} M\,\varphi \dq\right\|_{L^{\frac{r}{2-r}}(\Omega)}^{\frac{1}{2}}
+ \left\|\int_{D} M\,\varphi \dq\right\|_{L^r(\Omega)}
\right].
\label{Cttrbd}
\end{align}
Therefore, for $r \in [1,2)$ and $s \in [1,2]$, it follows that, \red{for any such function $\varphi$,}
\begin{align}
&\|\Ctt_i(M\,\varphi)\|_{L^s(0,T;L^r(\Omega))}
\nonumber \\
& \quad \leq
C\biggl[ \|\nabqi \sqrt{\varphi}\|_{L^2(0,T;L^2_M(\Omega \times D))}\,
\left\|\int_{D} M\,\varphi \dq\right\|_{L^{\upsilon}(0,T;L^{\frac{r}{2-r}}(\Omega))}^{\frac{1}{2}}
+ \left\|\int_{D} M\,\varphi \dq\right\|_{L^s(0,T;L^r(\Omega))}
\biggr],
\nonumber \\
&
\label{Cttrsbd}
\end{align}
where $\upsilon = \frac{s}{2-s}$ if $s \in [1,2)$ and $\upsilon =\infty$ if $s=2$.
We deduce from (\ref{Cttrsbd}) and (\ref{eq:energy-u+psi-final2})
that, for $i=1,\ldots,K$,
\begin{align}
\|\Ctt_i(M\,\hpsikaLDp)\|_{L^s(0,T;L^r(\Omega))}
&\leq C \qquad \mbox{if} \quad
\|\vrhokaLDp \|_{L^{\upsilon}(0,T;L^{\frac{r}{2-r}}(\Omega))} \leq C,
\label{Cttrsbdvr}
\end{align}
where $r \in [1,2)$, $s \in [1,2]$ and
$\upsilon = \frac{s}{2-s}$ if $s \in [1,2)$ and $\upsilon =\infty$ if $s=2$.

It follows from (\ref{eq:energy-u+psi-final2}) and (\ref{eqinterp}) that
\begin{align}
\|\vrhokaLDp\|_{L^{\frac{2}{\vartheta}}(0,T;L^\upsilon(\Omega))}
\leq C\,\|\vrhokaLDp\|_{L^{\infty}(0,T;L^2(\Omega))}^{1-\vartheta}\,
\|\vrhokaLDp\|_{L^{2}(0,T;H^1(\Omega))}^{\vartheta}\leq C,
\label{vrhokaLvbd}
\end{align}
where $\vartheta=\frac{(\upsilon-2)d}{2\upsilon}$, and $\upsilon \in (2,\infty)$ if $d=2$
and $\upsilon \in (2,6]$ if $d=3$.
For example, we have that
\begin{align}
\|\vrhokaLDp\|_{L^{\infty}(0,T;L^2(\Omega))}+
\|\vrhokaLDp\|_{L^{\frac{2(d+2)}{d}}(\Omega_T)}
+
\|\vrhokaLDp\|_{L^{4}(0,T;L^{\frac{2d}{d-1}}(\Omega))}
 \leq C,
\label{vrhokaLvvbd}
\end{align}
and hence we deduce from (\ref{Cttrsbdvr}) and (\ref{tau1})
that
\begin{subequations}
\begin{align}
\|\Ctt_i(M\,\hpsikaLDp)\|_{L^2(0,T;L^{\frac{4}{3}}(\Omega))}
+\|\Ctt_i(M\,\hpsikaLDp)\|_{L^{\frac{4(d+2)}{3d+4}}(\Omega_T)}
&\leq C, \qquad i=1,\ldots,K,
\label{Ctt43bd} \\
\|\tautt_1(M\,\hpsikaLDp)\|_{L^2(0,T;L^{\frac{4}{3}}(\Omega))}+
\|\tautt_1(M\,\hpsikaLDp)\|_{L^{\frac{4(d+2)}{3d+4}}(\Omega_T)}
&\leq C,
\label{tautt43bd}
\end{align}
\end{subequations}
where $C$ is independent of $\Delta t$, $L$, $\alpha$ and $\kappa$.

\begin{remark}\label{z0rem}
We note from {\rm(\ref{eq:energy-u+psi-final2})} and
{\rm(\ref{vrhokaLvbd})} that if $\mathfrak{z}=0$, then
the bounds {\rm(\ref{vrhokaLvvbd})} and {\rm(\ref{Ctt43bd},b)}
no longer hold.
In this case, we have only the following weaker bounds.

Similarly to {\rm(\ref{rhonint})}, we have on choosing $\varphi = 1$ in {\rm(\ref{psiG})}, and
noting {\rm(\ref{upm})} and {\rm(\ref{psi0conv})}, that, for a.a.\ $t \in (0,T)$,
\begin{align}
\int_\Omega \vrhokaLDp \,\dx = \int_{\Omega \times D} M\,\hpsikaLDp\,\dq\,\dx
= \int_{\Omega \times D} M\,\hpsi^0 \dq \, \dx
\leq C.
\label{vrhokaLDint}
\end{align}
Next we deduce from {\rm(\ref{eq:energy-u+psi-final2})} and {\rm(\ref{vrhokaLDint})} that
\begin{align}
\|\nabx \vrhokaLDp\|_{L^2(0,T;L^1(\Omega))}^2
&= 4\,\left\|\int_D M \,\sqrt{\hpsikaLDp}\,\nabx \sqrt{\hpsikaLDp}\,\dq \right\|_{L^2(0,T;L^1(\Omega))}^2 &
\nonumber \\
&\leq 4\, \|\vrhokaLDp\|_{L^\infty(0,T;L^1(\Omega))}\,
\left\|\nabx \sqrt{\hpsikaLDp}\right\|_{L^2(0,T;L^2_M(\Omega \times D))}^2
\leq C.
\label{vrhoW1sbd}
\end{align}
It follows from Sobolev embedding, {\rm(\ref{vrhoW1sbd})} and {\rm(\ref{vrhokaLDint})} that
\begin{align}
\|\vrhokaLDp\|_{L^2(0,T;L^{\frac{d}{d-1}}(\Omega))}
&\leq C\, \|\vrhokaLDp\|_{L^2(0,T;W^{1,1}(\Omega))} \leq C.
\label{vrhoLsbd}
\end{align}
Therefore, we obtain from {\rm(\ref{Cttrsbdvr})}, {\rm(\ref{vrhoLsbd})} and {\rm(\ref{tau1})} that
\begin{align}
\|\tautt_1(M\,\hpsikaLDp)\|_{L^{\frac{4}{3}}(0,T;L^{\frac{2d}{2d-1}}(\Omega))}
\leq C,
\label{tautt1red}
\end{align}
where $C$ is independent of $\Delta t$, $L$, $\alpha$ and $\kappa$.
\end{remark}

\subsection{$L,\,\Delta t$-independent bounds on the spatial and temporal derivatives of
$\rhokaL^{[\Delta t]}$}
\label{sec:time-rho}

In addition to the bounds on $\rhokaL^{[\Delta t]}$ and $\rhokaLDp$ in (\ref{eq:energy-u+psi-final2}),
we establish further relevant bounds here.
Similarly to (\ref{rhodbd}), on choosing, for any $s\in (0,T]$, $\eta(\cdot,t) = \chi_{[0,s]}\,
[\rhokaL^{[\Delta t]}(\cdot,t)]^{\vartheta-1}$, for $\vartheta=2$ and $\frac{\Gamma}{2}$,
in (\ref{eqrhocon}), we obtain, on noting 
(\ref{eq:energy-u+psi-final2}), that
\begin{align}
&\frac{1}{\vartheta}\,\|\rhokaL^{[\Delta t]}(\cdot,s)\|_{L^\vartheta(\Omega)}^\vartheta
+ \frac{4\alpha(\vartheta-1)}{\vartheta^2}\,
\int_{0}^{s} \|\nabx[ (\rhokaL^{[\Delta t]})^{\frac{\vartheta}{2}}]\|_{L^2(\Omega)}^2 \,\dt
\nonumber \\
& \quad = \frac{1}{\vartheta}\left[ \|\rho^0\|_{L^\vartheta(\Omega)}^\vartheta
+ (\vartheta-1)\,\int_0^s \int_{\Omega} 
\utkaLDp \cdot
\nabx [(\rhokaL^{[\Delta t]})^\vartheta] \dx\, \dt \right]
\nonumber \\
& \quad \leq \frac{1}{\vartheta}\left[\|\rho^0\|_{L^\vartheta(\Omega)}^\vartheta
+ \int_0^s \|\utkaLDp\|_{L^2(\Omega)}^2 \,\dt+
\frac{(\vartheta-1)^2}{4}\,\int_0^s 
\|\nabx [(\rhokaL^{[\Delta t]})^\vartheta]\|_{L^2(\Omega)}^2 \dt
\right]
\nonumber \\
& \quad \leq C,
\label{rhobd}
\end{align}
where $C$ is independent of $\Delta t$ and $L$.
Similarly to (\ref{PkrhokaLdc}), it follows from
a  Poincar\'e inequality,  
(\ref{eq:energy-u+psi-final2}) and (\ref{rhobd}) for
$\vartheta=\frac{\Gamma}{2}$ that
\begin{align}
&\|\rhokaL^{[\Delta t]}\|_{L^{\Gamma}(0,T;L^\Gamma(\Omega))}^\Gamma
= \|(\rhokaL^{[\Delta t]})^{\frac{\Gamma}{2}}\|_{L^{2}(0,T;L^2(\Omega))}^2
\nonumber \\[2mm]
& \hspace{0.6in} \leq C\left[
\|\nabx[(\rhokaL^{[\Delta t]})^{\frac{\Gamma}{2}}]\|_{L^{2}(0,T;L^2(\Omega))}^2
+ \|\rhokaL^{[\Delta t]}\|_{L^{\infty}(0,T;L^{\frac{\Gamma}{2}}(\Omega))}^\Gamma \right]
\leq C.
\label{PkrhokaLc}
\end{align}
Similarly to (\ref{rhodH-1a}),
we obtain from (\ref{cdef}), (\ref{eqinterp}), (\ref{eq:energy-u+psi-final2}) and (\ref{rhobd})
for $\vartheta=\frac{\Gamma}{2}$ 
that
\begin{align}
\left| \int_0^T c(\utkaLDp)(\rhokaL^{[\Delta t]},\eta) \dt\right|
&\leq C \,\|\eta\|_{L^2(0,T;H^{1}(\Omega))}
+ C \left| \int_0^T \|\utkaLDp\|_{H^1(\Omega)}\,
\|\eta\|_{H^{1}(\Omega)} \,\dt \right|
\nonumber \\
& 
\leq C \left[ 1 + \|\utkaLDp\|_{L^2(0,T;H^1(\Omega)}\right]
\,\|\eta\|_{L^2(0,T;H^{1}(\Omega))}
\nonumber \\
&
\leq C\,
\,\|\eta\|_{L^2(0,T;H^{1}(\Omega))}
\qquad \forall \eta \in L^2(0,T;H^{1}(\Omega)).
\label{rhoH-1a}
\end{align}
Hence, similarly to (\ref{rhodall}), we deduce from  
(\ref{rhobd}) for $\vartheta=2$ and $\frac{\Gamma}{2}$,
(\ref{rhoH-1a}),
on noting
(\ref{eqrhocon}) and (\ref{cdef}),
(\ref{eq:energy-u+psi-final2}) 
and
(\ref{PkrhokaLc}) that 
\begin{align}
&
\|\rhokaL^{[\Delta t]}\|_{L^\infty(0,T;L^{\frac{\Gamma}{2}}(\Omega))}
+
\|\rhokaL^{[\Delta t]}\|_{L^2(0,T;H^1(\Omega))}^2
\nonumber \\
& \hspace{1in}
+
\|\rhokaL^{[\Delta t]}\|_{H^1(0,T;H^{1}(\Omega)')}^2
+\|(\rhokaL^{[\Delta t]})^{\frac{\Gamma}{2}}\|_{L^2(0,T;H^1(\Omega))}^2
\leq C,
\label{rhoall}
\end{align}
where $C$ is independent of $\Delta t$ and $L$.
Similarly to (\ref{rhodL3}),
we deduce from  (\ref{eqinterp}) and the last bound in (\ref{rhoall}) that
\begin{align}
\|\rhokaL^{[\Delta t]}\|_{L^\Gamma(0,T;L^{3\Gamma}(\Omega))}^\Gamma =
\|(\rhokaL^{[\Delta t]})^{\frac{\Gamma}{2}}\|_{L^2(0,T;L^6(\Omega))}^2
\leq C\,\|(\rhokaL^{[\Delta t]})^{\frac{\Gamma}{2}}\|_{L^2(0,T;H^1(\Omega))}^2\leq C.
\label{rhoL3}
\end{align}
Next, similarly to (\ref{rhodL43}),
it follows from (\ref{eqLinterp}), 
the first bound in (\ref{rhoall})
and (\ref{rhoL3}) that
\begin{align}
\|\rhokaL^{[\Delta t]}\|_{L^{\frac{4\Gamma}{3}}(0,T;L^{\frac{4\Gamma}{3}}(\Omega))}^{\frac{4\Gamma}{3}}
\leq
\|\rhokaL^{[\Delta t]}\|_{L^{\infty}(0,T;L^{\frac{\Gamma}{2}}(\Omega))}^{\frac{\Gamma}{3}}
\|\rhokaL^{[\Delta t]}\|_{L^{\Gamma}(0,T;L^{3\Gamma}(\Omega))}^{\Gamma}
\leq C,
\label{rhoL43}
\end{align}
where $C$ is independent of $\Delta t$ and $L$.
Finally, it follows from (\ref{pkaL}), (\ref{pkdef}) and (\ref{rhoL43}) that
\begin{align}
\|\pkaL^{\{\Delta t\}}\|_{L^{\frac{4}{3}}
(\Omega_T)}^{\frac{4}{3}}
& \leq \|\pk(\rhokaL^{[\Delta t]})\|_{L^{\frac{4}{3}}
(\Omega_T)}^{\frac{4}{3}}
\leq C\,\|\rhokaL^{[\Delta t]}\|_{L^{\frac{4\Gamma}{3}}
(\Omega_T)}^{\frac{4\Gamma}{3}}
\leq C,
\label{pkaLbd}
\end{align}
where $C$ is independent of $\Delta t$ and $L$.

\subsection{Passing to the limit $\Delta t \rightarrow 0_+$ $(L \rightarrow \infty)$
in the continuity equation (\ref{eqrhocon})}
\label{rhoLcon}

As noted after (\ref{eq:energy-psi-summ1}), we shall assume that
\begin{equation}\label{LT}
\Delta t = o(L^{-1})\qquad \mbox{as $\Delta t \rightarrow 0_+$ $(L \rightarrow \infty)$}.
\end{equation}
Requiring, for example, that $0<\Delta t \leq C_0/(L\,\log L)$, $L > 1$,
with an arbitrary (but fixed)
constant $C_0$ will suffice to ensure that \eqref{LT} holds.
We have the following convergence results.
\begin{lemma}
\label{rhoLconv}
There exists a subsequence (not indicated) of $\{(\rhokaL^{[\Delta t]},\utkaLDp,\hpsikaLDp)
\}_{\Delta t >0}$,
and functions
\begin{align}
\rhoka \in L^2(0,T;H^1(\Omega))\cap H^1(0,T;H^1(\Omega)')\cap 
C_w([0,T];L^{\frac{\Gamma}{2}}(\Omega))\cap
L^{\frac{4\Gamma}{3}}_{\geq 0}(\Omega_T)
\label{rhokareg}
\end{align}
with $\rhoka(\cdot,0) = \rho^0(\cdot)$
and $\utka \in L^2(0,T;\Ht^1_0(\Omega))$
such that, as $\Delta t \rightarrow 0_+$ $(L \rightarrow \infty)$,
\begin{subequations}
\begin{alignat}{3}
\rhokaL^{[\Delta t]} &\rightarrow  \rhoka
\quad &&\mbox{weakly in } L^2(0,T;H^1(\Omega)), \quad
&&\mbox{weakly in } H^1(0,T;H^{1}(\Omega)'),
\label{rhoLwcon}\\
\rhokaL^{[\Delta t]} &\rightarrow  \rhoka
\quad &&
\mbox{in } C_w([0,T];L^{\frac{\Gamma}{2}}(\Omega)),\quad &&\red{\mbox{weakly in } L^{\frac{4\Gamma}{3}}(\Omega_T),}
\label{rhoLwconstar}\\
\left(\rhokaL^{[\Delta t]}\right)^{\frac{\Gamma}{2}} &\rightarrow  \rhoka^{\frac{\Gamma}{2}}
\quad &&\mbox{weakly in } L^2(0,T;H^1(\Omega)),
&&
\label{rhoLpwcon}\\
\rhokaL^{[\Delta t]} &\rightarrow  \rhoka \quad
&&\mbox{strongly in } L^2(0,T;L^r(\Omega)),\quad
&&\mbox{strongly in } L^{\upsilon}(\Omega_T),
\label{rhoLscon}
\end{alignat}
\end{subequations}
where $r \in [1,\infty)$ if $d=2$ and $r \in [1,6)$ if $d=3$, and $\upsilon \in [1,\frac{4\Gamma}{3})$;
\begin{subequations}
\begin{alignat}{2}
\pk(\rhokaL^{[\Delta t]}) & \rightarrow \pk(\rhoka)
\quad &&\mbox{strongly in } L^s(\Omega_T),
\label{pkaLscon} \\
\pkaL^{\{\Delta t\}} & \rightarrow \pk(\rhoka)
\quad &&\mbox{weakly in } L^s(\Omega_T),
\label{pkaLwcon}
\end{alignat}
\end{subequations}
where $s \in (1,\tfrac{4}{3})$; and
\begin{alignat}{2}
\utkaLDp &\rightarrow \utka \qquad &&\mbox{weakly in } L^2(0,T;\Ht^1_0(\Omega)).
\label{uLwconH1}
\end{alignat}

Moreover, we have that
\begin{align}
&\displaystyle\int_{0}^{T}\left\langle \frac{\partial \rhoka}{\partial t}\,,\eta
\right\rangle_{H^1(\Omega)} \dd t
+ \int_0^T \int_\Omega \left( \alpha \,\nabx \rhoka -\rhoka \,\utka \right)
\cdot \nabx \eta
\,\dx \,\dt =0
\nonumber \\
& \hspace{3.4in}
\qquad
\forall \eta \in L^2(0,T;H^1(\Omega)).
\label{eqrhoka}
\end{align}
\end{lemma}
\begin{proof}
The convergence results (\ref{rhoLwcon},b) follow immediately from (\ref{rhoall}), (\ref{Cwcoma},b)
and (\ref{rhoL43}).
The strong convergence results (\ref{rhoLscon}) follow from (\ref{rhoLwcon}), (\ref{compact1}),
(\ref{rhoL43}) and the interpolation result (\ref{eqLinterp}).
The weak convergence result (\ref{rhoLpwcon}) is then a consequence of (\ref{rhoall}) and (\ref{rhoLscon}).
\red{Therefore, we have the desired result (\ref{rhokareg}).
As $L^2(0,T;H^1(\Omega))\cap H^1(0,T;H^1(\Omega)') \hookrightarrow C([0,T];L^2(\Omega))$,
we obtain that $\rhoka(\cdot,0)=\rho^0(\cdot)$.}

Next, it follows from (\ref{pkdef}), for $s \in [1,\frac{4}{3}]$, that
\begin{align}
\|\pk(\rhoka)-\pk(\rhokaL^{[\Delta t]})\|_{L^s(\Omega_T)}
\leq C\,\left[\|\rhoka\|_{L^{s\Gamma}(\Omega_T)}^{\Gamma-1}
+ \|\rhokaL^{[\Delta t]}\|_{L^{s\Gamma}(\Omega_T)}^{\Gamma-1}\right]
 \|\rhoka-\rhokaL^{[\Delta t]}\|_{L^{s\Gamma}(\Omega_T)}.
\label{pkaLcon1}
\end{align}
Hence, the result (\ref{pkaLscon}) follows from (\ref{pkaLcon1}) and (\ref{rhoLscon}).

It follows from (\ref{pkaL}) and (\ref{pkaLscon}) that, for $s \in (1,\frac{4}{3})$,
\begin{align}
\int_{\Omega_T} \pkaL^{\{\Delta t\}}\, \eta \,\dx \,\dt =
\int_{\Omega_T} \pk(\rhokaL^{[\Delta t]})\,
\eta^{\{\Delta t\}} \,\dx \,\dt   \qquad \forall \eta \in L^{\frac{s}{s-1}}(\Omega_T),
\label{timshift}
\end{align}
where
\begin{align}
\eta^{\{\Delta t\}}(\cdot,t) := \frac{1}{\Delta t} \int_{t_{n-1}}^{t_n} \eta(\cdot,t') \,{\rm d}t', \qquad
t \in (t_{n-1},t_n], \qquad n=1,\ldots,N.
\label{timav}
\end{align}
We note that
\begin{align}
\lim_{\Delta t \rightarrow 0_+} \|\eta-\eta^{\{\Delta t\}}\|_{L^r(\Omega_T)} =0 \qquad \forall
\eta \in L^{r}(\Omega_T), \qquad r\in [1,\infty).
\label{timden}
\end{align}
Therefore, the desired result (\ref{pkaLwcon}) follows from (\ref{timshift}), (\ref{pkaLscon})
and (\ref{timden}).
Finally, the weak convergence result (\ref{uLwconH1}) follows immediately from
(\ref{eq:energy-u+psi-final2}).

It is now a simple matter to pass to the limit $\Delta t \rightarrow 0_+$ $(L \rightarrow \infty)$
for the subsequence in (\ref{eqrhocon}), on noting (\ref{rhoLwcon}--c) and (\ref{uLwconH1}),
to obtain (\ref{eqrhoka}).
\end{proof}

In order to pass to the limit $\Delta t \rightarrow 0_+$ $(L \rightarrow \infty)$
in the momentum equation (\ref{equncon}), we will need a
strong convergence result for $\nabx \rhokaL^{[\Delta t]}$.
First, it follows from
(\ref{upm}), \eqref{idatabd},
and \eqref{eq:energy-u+psi-final2} 
that
\begin{align}\label{u-t-3}
\|\nabxtt\utkaLDm\|^2_{L^2(0,T;L^2(\Omega))} &= \Delta t\,
\|\nabxtt\ut^0\|^2 + \int_{\Delta t}^T \|\nabxtt \utkaLDm\|^2 \dt
\nonumber\\
&\leq \int_\Omega \rho^0 |\ut_0|^2 \dx + \int_0^{T-\Delta t} \|\nabxtt \utkaLDp\|^2 \dt
\leq C; 
\end{align}
hence, we obtain from (\ref{eqinterp}), a  Poincar\'e inequality,
\eqref{eq:energy-u+psi-final2}, (\ref{u-t-3}) and (\ref{ulin},b) that
\begin{align}
\|\utkaLDpm\|_{L^2(0,T;L^6(\Omega))}
\leq \|\utkaLDpm\|_{L^2(0,T;H^1(\Omega))}
\leq \|\nabx \utkaLDpm\|_{L^2(0,T;L^2(\Omega))}
\leq C,
\label{utkaLL6}
\end{align}
where $C$ is independent
of $\Delta t$, $L$, $\alpha$ and $\kappa$.

\begin{lemma}\label{rhonabxsclem}
There exists a $C \in \mathbb R_{>0}$, independent of $\Delta t$ and $L$, such that
\begin{subequations}
\begin{align}
&\|\rhokaLDpm\|_{L^\infty(0,T;L^{\Gamma}(\Omega))} +
\left\|\sqrt{\rhokaLDwpm}\,\utkaLDwpm\right\|_{L^\infty(0,T;L^2(\Omega))}
\nonumber \\
& \qquad
+ \|\rhokaLDwpm\,\utkaLDwpm\|_{L^\infty(0,T;L^{\frac{2\Gamma}{\Gamma+1}}(\Omega))}
+ \left\|\sqrt{\rhokaLDwpm}\,\utkaLDwpm\right\|_{L^2(0,T;L^{\frac{6\Gamma}{\Gamma+3}}(\Omega))}
\leq C,
\label{kaLLinfL1} \\
&\|\rhokaL^{[\Delta t]}\,\utkaLDp\|_{L^\infty(0,T;L^1(\Omega))}
+\|\rhokaL^{[\Delta t]}\,\utkaLDp\|_{L^2(0,T;L^{\frac{6\Gamma}{\Gamma+12}}(\Omega))}
\nonumber \\
& \qquad
+ \left\|\sqrt{\rhokaL^{[\Delta t]}}\,\utkaLDp\right\|_{L^2(0,T;L^{\frac{6\Gamma}{\Gamma+6}}(\Omega))}
+ \|\rhokaL^{[\Delta t]}\,\utkaLDp\|_{L^\upsilon(\Omega_T)}
\leq C,
\label{kaLLrLs}\\
& \|\nabx \rhokaL^{[\Delta t]}\|_{L^\upsilon(\Omega_T)} \leq C,
\label{nabxkaLLrLs}
\end{align}
\end{subequations}
where $\upsilon=\frac{8\Gamma-12}{3\Gamma} \geq \frac{13}{6}$ as $\Gamma \geq 8$.

Hence, in addition to (\ref{rhokareg}), $\rhoka \in L^\infty(0,T;L^\Gamma(\Omega))$
and for a further subsequence of the subsequence of Lemma \ref{rhoLconv}, it follows that,
as $\Delta t \rightarrow 0_+$ $(L \rightarrow \infty)$,
\begin{subequations}
\begin{alignat}{3}
\nabx \rhokaL^{[\Delta t]} &\rightarrow  \nabx \rhoka \quad
&& \mbox{weakly in } L^{\upsilon}(\Omega_T),
\qquad
&&\mbox{strongly in } L^{2}(\Omega_T),
\label{rhoLxscon} \\
\rhokaLDp &\rightarrow \rhoka
\quad
&&\mbox{weakly-$\star$ in } L^\infty(0,T;L^{\Gamma}(\Omega)),\qquad
&&\mbox{strongly in } L^{2}(0,T;H^1(\Omega)'),
\label{rhoLDpcon}
\end{alignat}
and, for any nonnegative $\eta \in C[0,T],$ 
\begin{align}\label{Ppka}
&\int_0^T \left(\int_{\Omega} \Pk(\rhoka)\,\dx\right) \eta\,\dt
\leq \liminf_{\Delta t \rightarrow 0_+\, (L \rightarrow \infty)}
\int_0^T \left( \int_{\Omega} \Pk(\rhokaLDp)\, \dx\right) \eta\,\dt.
\end{align}
\end{subequations}
\end{lemma}
\begin{proof}
The first two bounds in (\ref{kaLLinfL1}) follow immediately from (\ref{ulin},b),
(\ref{eq:energy-u+psi-final2}), (\ref{Pdef}), $(\ref{rho0conv})$ and (\ref{idatabd}).
The third bound in (\ref{kaLLinfL1})
follows immediately from the first two on noting that
\begin{align}
\|\rhokaLDwpm\,\utkaLDwpm\|_{L^{\frac{2\Gamma}{\Gamma+1}}(\Omega)}
\leq \left\|\sqrt{\rhokaLDwpm}\,\right\|_{L^{2\Gamma}(\Omega)}\,
\left\|\sqrt{\rhokaLDwpm}\,\utkaLDwpm \right\|_{L^{2}(\Omega)}.
\label{LinfLr}
\end{align}
It follows from the first bound in (\ref{kaLLinfL1}), (\ref{eqinterp}) and
(\ref{utkaLL6})
that
\begin{align}
\left\|\sqrt{\rhokaLDwpm}\,\utkaLDwpm\right\|_{L^2(0,T;L^{\frac{6\Gamma}{\Gamma+3}}(\Omega))}
&\leq \left\|\sqrt{\rhokaLDwpm}\,\right\|_{L^\infty(0,T;L^{2\Gamma}(\Omega))}
\,\|\utkaLDwpm\|_{L^2(0,T;L^{6}(\Omega))} \leq C,
\label{L2Lsab}
\end{align}
and hence the fourth bound in (\ref{kaLLinfL1}).

Next we note, for any $\eta \in L^2(0,T;H^1(\Omega))$ and for a.a.\ $s \in (t_{n-1},t_n]$, that
\begin{align}
&\left|\int_\Omega \left[\rhokaL^{[\Delta t]}(\xt,s)-\rhokaLDp(\xt,s)\right] \eta(\xt,s) \,\dx
\right|
\nonumber \\
&\hspace{0.3in} =
\left|  \int_{s}^{t_n}\! \int_\Omega  \eta(\xt,s)\,
\frac{\partial \rhokaL^{[\Delta t]}}{\partial t}(\xt,t)\,\dx\,\dt \right|
\leq \|\eta(\cdot,s)\|_{H^1(\Omega)}
\int_{t_{n-1}}^{t_n}
\left \|\frac{\partial \rhokaL^{[\Delta t]}}{\partial t}
\right\|_{H^1(\Omega)'}\,\dt.
\label{LinfL1a}
\end{align}
It follows from (\ref{LinfL1a}) with $\eta=|\utkaLDp|$, (\ref{kaLLinfL1}), (\ref{rhoall}) and
(\ref{eq:energy-u+psi-final2}) that
\begin{align}
\|\rhokaL^{[\Delta t]}\,\utkaLDp\|_{L^\infty(0,T;L^1(\Omega))}
& \leq \|\rhokaLDp\,\utkaLDp\|_{L^\infty(0,T;L^1(\Omega))}
\nonumber \\
& \qquad
+ \frac{1}{2} \left[\left\|\frac{\partial \rhokaL^{[\Delta t]}}{\partial t}\right\|_{L^2(0,T,H^1(\Omega)')}^2
+ \|\utkaLDp\|_{L^2(0,T;H^1(\Omega))}^2\right]
\leq C,
\label{LinfL1b}
\end{align}
and hence the first desired bound in (\ref{kaLLrLs}).
Similarly to (\ref{L2Lsab}),
it follows from (\ref{rhoall}), (\ref{eqinterp}) and
(\ref{utkaLL6})
that
\begin{subequations}
\begin{align}
\|\rhokaL^{[\Delta t]}\,\utkaLDp\|_{L^2(0,T;L^{\frac{6\Gamma}{\Gamma+12}}(\Omega))}
&\leq \|\rhokaL^{[\Delta t]}\|_{L^\infty(0,T;L^{\frac{\Gamma}{2}}(\Omega))}
\,\|\utkaLDp\|_{L^2(0,T;L^{6}(\Omega))},
\label{L2Lsa}\\
\left\|\sqrt{\rhokaL^{[\Delta t]}}\,\utkaLDp\right\|_{L^2(0,T;L^{\frac{6\Gamma}{\Gamma+6}}(\Omega))}
&\leq \left\|\sqrt{\rhokaL^{[\Delta t]}}\,\right\|_{L^\infty(0,T;L^{\Gamma}(\Omega))}
\,\|\utkaLDp\|_{L^2(0,T;L^{6}(\Omega))},
\label{LsLsaa}
\end{align}
\end{subequations}
and hence the second and third bounds in (\ref{kaLLrLs}).
It follows from (\ref{eqLinterp}) with $\upsilon=\frac{8\Gamma-12}{3\Gamma}$, $r=1$ and $s=\frac{6\Gamma}{\Gamma+12}$
that $\upsilon\,\vartheta=2$ \red{(with $\vartheta \in (0,1)$ for $\Gamma \geq 8$)} and so
\begin{align}
\|\rhokaL^{[\Delta t]}\,\utkaLDp\|_{L^\upsilon(\Omega_T)}
\leq \|\rhokaL^{[\Delta t]}\,\utkaLDp\|_{L^\infty(0,T;L^1(\Omega))}^{1-\vartheta}
\,\|\rhokaL^{[\Delta t]}\,\utkaLDp\|_{L^2(0,T;L^s(\Omega))}^{\vartheta}.
\label{Lup}
\end{align}
Thus, (\ref{Lup}) and the first two bounds in (\ref{kaLLrLs}) yield the \red{fourth} bound in
(\ref{kaLLrLs}). On noting this bound and recalling from (\ref{inidata})
that $\partial \Omega \in C^{2,\theta}$, $\theta \in (0,1)$,
and $\rho^0 \in L^\infty(\Omega)$ satisfying (\ref{rho0conv}),
we can now apply the parabolic regularity result, Lemma 7.38 in
Novotn\'{y} \& Stra\v{s}kraba \cite{NovStras} \arxiv{(or Lemma \ref{Le-G-2} in Appendix \ref{sec:App-G})}, to
(\ref{eqrhocon}) to
obtain that the solution $\rhokaL^{[\Delta t]}$ satisfies the bound (\ref{nabxkaLLrLs}).

The first desired result in (\ref{rhoLxscon}) follows immediately from (\ref{nabxkaLLrLs}).
Next we obtain from (\ref{rhobd}) for $\vartheta=2$ that, for any $s \in (0,T]$,
\begin{align}
&\frac{1}{2}\,\|\rhokaL^{[\Delta t]}(\cdot,s)\|_{L^2(\Omega)}^2
+ \alpha\,
\int_{0}^{s} \|\nabx \rhokaL^{[\Delta t]}\|_{L^2(\Omega)}^2 \,\dt
\nonumber \\ & \hspace{1in}
= \frac{1}{2}\left[ \|\rho^0\|_{L^2(\Omega)}^2
- \int_0^s \int_{\Omega}
(\nabx \cdot \utkaLDp)\,
(\rhokaL^{[\Delta t]})^2 \dx\, \dt \right].
\label{rhoLsc1}
\end{align}
Integrating (\ref{rhoLsc1}) over $s \in (0,T)$, and performing integration by parts, yields that
\begin{align}
&\frac{1}{2}\,\|\rhokaL^{[\Delta t]}\|_{L^2(\Omega_T)}^2
+ \alpha\,
\int_{0}^{T} (T-t)\, \|\nabx \rhokaL^{[\Delta t]}\|_{L^2(\Omega)}^2 \,\dt
\nonumber \\ & \hspace{1in}
= \frac{1}{2}\left[ T\,\|\rho^0\|_{L^2(\Omega)}^2
- \int_0^T (T-t) \int_{\Omega}
(\nabx \cdot \utkaLDp)\,
(\rhokaL^{[\Delta t]})^2 \dx\, \dt \right].
\label{rhoLsc2}
\end{align}
Similarly, on choosing for any $s\in (0,T]$, $\eta(\cdot,t) = \chi_{[0,s]}\,
\rhoka(\cdot,t)$ in (\ref{eqrhoka}), and integrating over $s \in (0,T)$ yields that
\begin{align}
&\frac{1}{2}\,\|\rhoka\|_{L^2(\Omega_T)}^2
+ \alpha\,
\int_{0}^{T} (T-t)\, \|\nabx \rhoka\|_{L^2(\Omega)}^2 \,\dt
\nonumber \\ & \hspace{1in}
= \frac{1}{2}\left[ T\,\|\rho^0\|_{L^2(\Omega)}^2
- \int_0^T (T-t) \int_{\Omega}
(\nabx \cdot \utka)\,
(\rhoka)^2 \dx\, \dt \right].
\label{rhoLsc3}
\end{align}
We deduce from (\ref{rhoLsc2}), (\ref{rhoLsc3}), (\ref{rhoLscon}) \red{(with $\upsilon = 4$)} and (\ref{uLwconH1}) that
\begin{align}
\lim_{\Delta t \rightarrow 0_+\, (L \rightarrow \infty)}
\int_{0}^{T} (T-t)\, \|\nabx \rhokaL^{[\Delta t]}\|_{L^2(\Omega)}^2 \,\dt
= \int_{0}^{T} (T-t)\, \|\nabx \rhoka\|_{L^2(\Omega)}^2 \,\dt.
\label{rhoLsc4}
\end{align}
\red{By applying the elementary identity $|\undertilde{a}-\undertilde{b}|^2=|\undertilde{a}|^2 - |\undertilde{b}|^2 - 2 (\undertilde{a}-\undertilde{b})\cdot \undertilde{b}$
with $\undertilde{a}=\nabx\rhokaL^{[\Delta t]}$ and $\undertilde{b}=\nabx\rhokaL$,} it follows from (\ref{rhoLsc4}) and (\ref{rhoLwcon}) that
\begin{subequations}
\begin{align}
\lim_{\Delta t \rightarrow 0_+\, (L \rightarrow \infty)}
\int_{0}^{T} (T-t)\, \|\nabx (\rhokaL-\rhokaL^{[\Delta t]})\|_{L^2(\Omega)}^2 \,\dt = 0,
\label{wscon}
\end{align}
and hence, for a.a.\ $t \in (0,T)$,
\begin{align}
\|\nabx (\rhokaL-\rhokaL^{[\Delta t]})(\cdot, t)\|_{L^2(\Omega)}^2 \rightarrow
0, \qquad \mbox{as } \Delta t \rightarrow 0_+\; (L \rightarrow \infty).
\label{nabxaacon}
\end{align}
\end{subequations}
Therefore, we obtain the second desired result (\ref{rhoLxscon}) from
(\ref{nabxaacon}), (\ref{nabxkaLLrLs}) and Vitali's convergence theorem.
\red{The details of the argument are as follows. With $\upsilon\geq \frac{13}{6}>2$, the bound \eqref{nabxkaLLrLs} implies that $|\nabx\rhokaL^{[\Delta t]}|^2$ is equi-integrable in $L^1(\Omega_T)$, i.e.,  $\nabx\rhokaL^{[\Delta t]}$ is 2-equi-integrable. Further, thanks to \eqref{wscon}, a subsequence of $\nabx\rhokaL^{[\Delta t]}$ is
a.e. convergent on $\Omega_T$ (cf. Theorem 2.20 (iii) in \cite{Fonseca-Leoni}), and thus by Egoroff's theorem (cf. Theorem 2.22 in \cite{Fonseca-Leoni}) it also converges in measure. Hence, by Vitali's convergence theorem (cf. Theorem 2.24 in \cite{Fonseca-Leoni}, with $p=2$,) we have strong convergence of the subsequence (not indicated).}

\red{The first stated convergence result in \eqref{rhoLDpcon} follows directly from the first bound in \eqref{kaLLinfL1}.}
Next, it follows from (\ref{LinfL1a}) and (\ref{rhoall}) that
\begin{align}
\|\rhokaL^{[\Delta t]}-\rhokaLDp\|_{L^2(0,T;H^1(\Omega)')}^2
\leq (\Delta t)^2
\left\|\frac{\partial \rhokaL^{[\Delta t]}}{\partial t}
\right\|_{L^2(0,T;H^1(\Omega)')}^2 \leq C\,(\Delta t)^2.
\label{DiffH1d}
\end{align}
Hence the desired convergence results (\ref{rhoLDpcon}) follow immediately from (\ref{DiffH1d}) \red{and
(\ref{rhoLscon}) with $r>\frac{2d}{d+2}$ (to ensure that $L^r(\Omega) \subset H^1(\Omega)'$).} 

Finally, it follows for any nonnegative $\eta \in C[0,T$],
on noting  the convexity of $\Pk(\cdot)$, that
\begin{align}
\int_0^T \left(\int_\Omega \Pk(\rhokaLDp)\,\dx \right) \eta \,\dt
& \geq \int_0^T \left(\int_\Omega \left[ \Pk(\rhoka)
+ \Pk'(\rhoka)\,(\rhokaLDp-\rhoka)
\right]\dx \right) \eta \,\dt.
\label{Pkbel}
\end{align}
This yields the desired result (\ref{Ppka}) on noting (\ref{rhoLDpcon}) \red{and that $\Pk'(\rhoka) \in L^1(0,T;L^{\frac{\Gamma}{\Gamma-1}}(\Omega))$ by \eqref{rhokareg}}.
\end{proof}

The following result is also required to pass to the limit
$\Delta t \rightarrow 0_+$ $(L \rightarrow \infty)$ in the momentum equation (\ref{equncon}).

\begin{lemma}\label{badlem}
There exists a $C \in {\mathbb R}_{> 0}$, independent of $\Delta t$ and $L$, such that
\begin{align}
\|\nabx\cdot(\rhokaL^{[\Delta t]}\,\utkaLDp)\|_{L^s(\Omega_T)}+
\left\|\frac{\partial \rhokaL^{[\Delta t]}}{\partial t}\right\|_{L^s(\Omega_T)}
+ \|\Delta_x\,\rhokaL^{[\Delta t]}\|_{L^s(\Omega_T)} \leq C,
\label{Laprho}
\end{align}
where $s=\frac{8\Gamma-12}{7 \Gamma-6} \geq \frac{26}{25}$ as $\Gamma \geq 8$.
In addition, we have that
\begin{subequations}
\begin{align}
\lim_{\Delta t \rightarrow 0_+ \ (L \rightarrow \infty)}
\|\rhokaL^{[\Delta t]}\,\utkaLDp -
\rhokaLDm\,\utkaLDm \|_{L^2(\Omega_T)}=0,
\label{bada}\\
\lim_{\Delta t \rightarrow 0_+ \ (L \rightarrow \infty)}
\|\rhokaLDp\,\utkaLDp -
\rhokaLDm\,\utkaLDm \|_{L^2(\Omega_T)}=0.
\label{badb}
\end{align}
\end{subequations}
\end{lemma}
\begin{proof}
It follows from (\ref{PkrhokaLc}) and (\ref{nabxkaLLrLs}) 
that
\begin{align}
\|\rhokaL^{[\Delta t]}\|_{L^{\frac{8\Gamma-12}{3\Gamma}}(\Omega_T)}
+ \|\nabx\rhokaL^{[\Delta t]}\|_{L^{\frac{8\Gamma-12}{3\Gamma}}(\Omega_T)} \leq C,
\label{badex1a}
\end{align}
where we have noted that $\frac{8\Gamma-12}{3\Gamma} < \Gamma$, with $\Gamma \geq 8$.
Hence, (\ref{badex1a}) and (\ref{utkaLL6}) yield that
\begin{align}
\|\,|\nabx \rhokaL^{[\Delta t]}|\,|\nabxtt \utkaLDp|\,\|_{L^s(\Omega_T)}
\leq \|\nabx \rhokaL^{[\Delta t]}\|_{L^{\frac{8\Gamma-12}{3\Gamma}}(\Omega_T)}
\,\|\nabxtt \utkaLDp\|_{L^{2}(\Omega_T)} \leq C,
\label{badex1}
\end{align}
where $s=\frac{8\Gamma-12}{7\Gamma-6}$. \red{From the first bound in \eqref{badex1a} and \eqref{badex1}, we obtain the first bound in (\ref{Laprho}).}
As $\partial \Omega \in C^{2,\theta}$, $\theta \in (0,1)$,
it follows from (\ref{projrho0}), (\ref{rho0conv})
and elliptic regularity that, for all $r \in [1,\infty)$,
\begin{align}
\|\rho^0\|_{W^{2,r}(\Omega)} \leq C(\alpha)\,\|\rho_0-\rho^0\|_{L^\infty(\Omega)}
\leq C(\alpha) \qquad \mbox{and} \qquad \nabx \rho^0 \cdot \nt = 0 \mbox{ on } \partial \Omega.
\label{rho0reg}
\end{align}
On noting (\ref{rho0reg}) and the first bound in (\ref{Laprho}),
we can now apply \red{to (\ref{eqrhocon}) } the parabolic regularity result, Lemma 7.37 in
Novotn\'{y} \& Stra\v{s}kraba \cite{NovStras} \arxiv{(or Lemma \ref{Le-G-1} in Appendix \ref{sec:App-G})},  to
obtain that the solution $\rhokaL^{[\Delta t]}$ satisfies the
last two bounds in (\ref{Laprho}).

Next, we note that
\begin{align}
&\|\rhokaL^{[\Delta t]}\,\utkaLDp -
\rhokaLDm\,\utkaLDm \|_{L^2(\Omega_T)}^2
= \int_0^T \int_{\Omega} |\rhokaL^{[\Delta t]}\,\utkaLDp -
\rhokaLDm\,\utkaLDm|^2 \,\dx \,\dt
\nonumber \\
&\qquad = \int_0^T \int_{\Omega} \,\biggl|
\sqrt{\rhokaL^{[\Delta t]}\,\rhokaLDm}\,
\left(\utkaLDp-\utkaLDm\right)
\nonumber \\
& \qquad \qquad + \left(\sqrt{\rhokaL^{[\Delta t]}}-\sqrt{\rhokaLDm}\right)\,
\left(\sqrt{\rhokaL^{[\Delta t]}}\,\utkaLDp+\sqrt{\rhokaLDm}\,\utkaLDm\right) \biggr|^2
 \dx \,\dt
\nonumber \\
&\qquad \leq 2\,\int_0^T \int_{\Omega} \rhokaL^{[\Delta t]}\,\rhokaLDm\,
|\utkaLDp-\utkaLDm|^2\,\dx\,\dt
\nonumber \\
& \qquad \qquad +
2\,\int_0^T \int_{\Omega}
\left|\sqrt{\rhokaL^{[\Delta t]}}-\sqrt{\rhokaLDm}\right|^2\,
\left|\sqrt{\rhokaL^{[\Delta t]}}\,\utkaLDp+\sqrt{\rhokaLDm}\,\utkaLDm\right|^2
 \dx \,\dt
\nonumber \\
&\qquad \leq
2\,\int_0^T \int_{\Omega} \rhokaL^{[\Delta t]}\,\rhokaLDm\,
|\utkaLDp-\utkaLDm|^2\,\dx\,\dt
\nonumber \\
& \qquad \qquad +
2\,\int_0^T \int_{\Omega}
\left|\rhokaL^{[\Delta t]}-\rhokaLDm\right|\,
\left|\sqrt{\rhokaL^{[\Delta t]}}\,\utkaLDp+\sqrt{\rhokaLDm}\,\utkaLDm\right|^2
 \dx \,\dt.
\label{bad1}
\end{align}
It follows from (\ref{rhoall}) and (\ref{eq:energy-u+psi-final2}) that
\begin{align}
&\left\|\sqrt{\rhokaL^{[\Delta t]}}\,\sqrt{\rhokaLDm}\,
\left(\utkaLDp-\utkaLDm\right)\right\|_{L^2(0,T;L^{\frac{2\Gamma}{\Gamma+2}}(\Omega))}
\nonumber \\
& \qquad \leq \left\|\sqrt{\rhokaL^{[\Delta t]}}
\,\right\|_{L^\infty(0,T;L^{\Gamma}(\Omega))}\,
\left\|\sqrt{\rhokaLDm}\,
\left(\utkaLDp-\utkaLDm\right)\right\|_{L^2(\Omega_T)} \leq C\,(\Delta t)^{\frac{1}{2}}.
\label{bad2}
\end{align}
Similarly, it
 follows from \red{(\ref{utkaLL6}), (\ref{rhoall}) and (\ref{kaLLinfL1})} and  that
\begin{align}
&\left\|\sqrt{\rhokaL^{[\Delta t]}}\,\sqrt{\rhokaLDm}\,
\left(\utkaLDp-\utkaLDm\right)\right\|_{L^2(0,T;L^{\frac{6\Gamma}{\Gamma+9}}(\Omega))}
\nonumber \\
& \qquad \leq \left\|\sqrt{\rhokaL^{[\Delta t]}\,\rhokaLDm}
\right\|_{L^\infty(0,T;L^{\frac{2\Gamma}{3}}(\Omega))}\,
\|\utkaLDp-\utkaLDm\|_{L^2(0,T;L^6(\Omega))}
\nonumber \\
& \qquad \leq
C\,
\left\|\sqrt{\rhokaL^{[\Delta t]}}
\right\|_{L^\infty(0,T;L^{\Gamma}(\Omega))}\,
\left\|\sqrt{\rhokaLDm}
\right\|_{L^\infty(0,T;L^{2\Gamma}(\Omega))}
\leq C.
\label{bad3}
\end{align}
We deduce from (\ref{bad2}), (\ref{bad3}) and (\ref{eqLinterp}) as $\frac{6\Gamma}{\Gamma+9}>2>
\frac{2\Gamma}{\Gamma+2}$ that
\begin{align}
\lim_{\Delta t \rightarrow 0_+\;(L \rightarrow \infty)}
\left\|\sqrt{\rhokaL^{[\Delta t]}}\,\sqrt{\rhokaLDm}\,
\left(\utkaLDp-\utkaLDm\right)\right\|_{L^2(\Omega_T)} =0,
\label{bad4}
\end{align}
and so the first term on the right-hand side of (\ref{bad1})
converges to zero, as $\Delta t \rightarrow 0_+\;(L \rightarrow \infty)$.

Next, we deal with the second term on the right-hand side of (\ref{bad1}).
It follows from (\ref{Laprho}) that
\begin{align}
\|\rhokaL^{[\Delta t]}-\rhokaLDm\|_{L^s(\Omega_T)} \leq \Delta t \,
\left\|\frac{\partial \rhokaL^{[\Delta t]}}{\partial t}\right\|_{L^s(\Omega_T)}
\leq C\,\Delta t,
\label{bad6}
\end{align}
and from (\ref{rhoall}) and (\ref{kaLLinfL1}) that
\begin{align}
\|\rhokaL^{[\Delta t]}-\rhokaLDm\|_{L^\infty(0,T;L^{\frac{\Gamma}{2}}(\Omega))} \leq C.
\label{bad7}
\end{align}
Hence, the bounds (\ref{bad6}) and (\ref{bad7}) yield, on noting (\ref{eqLinterp}) and
as $\Gamma \geq 8$, that
\begin{align}
\lim_{\Delta t \rightarrow 0_+\;(L \rightarrow \infty)}
\|\rhokaL^{[\Delta t]}-\rhokaLDm\|_{L^{\upsilon}(0,T;L^{r}(\Omega))} =0,
\qquad \mbox{ for any } \upsilon \in [1,\infty), \; r \in [1,4).
\label{bad8}
\end{align}
As $\frac{6\Gamma}{\Gamma+3} \geq \frac{48}{11}> 4$ \red{for $\Gamma \geq 8$}, it follows from
\red{(\ref{eqLinterp}) and (\ref{kaLLinfL1})} that
\begin{align}
&\left\|\sqrt{\rhokaLDwpm}\,\utkaLDwpm
\right\|_{L^{3}(0,T;L^{3}(\Omega))}
\nonumber \\
& \hspace{1in} \leq
\left\|\sqrt{\rhokaLDwpm}\,\utkaLDwpm
\right\|_{L^{\infty}(0,T;L^{2}(\Omega))}^{\frac{1}{3}}\,\left\|\sqrt{\rhokaLDwpm}\,\utkaLDwpm
\right\|_{L^{2}(0,T;L^{4}(\Omega))}^{\frac{2}{3}} \leq C.
\label{bad9}
\end{align}
In addition, (\ref{kaLLrLs}) and (\ref{utkaLL6}) yield that
\begin{align}
&\left\|
\sqrt{\rhokaL^{[\Delta t]}}\,\utkaLDp \right\|_{L^{4}(0,T;L^{\frac{12}{7}}(\Omega))}
\leq
\left\|\rhokaL^{[\Delta t]}\,\utkaLDp
\right\|_{L^{\infty}(0,T;L^{1}(\Omega))}^{\frac{1}{2}}\,\left\|\utkaLDp
\right\|_{L^{2}(0,T;L^{6}(\Omega))}^{\frac{1}{2}} \leq C.
\label{bad10}
\end{align}
As $\frac{6\Gamma}{\Gamma+6}\geq \frac{24}{7}$ \red{for $\Gamma \geq 8$}, it then follows from
\red{(\ref{eqLinterp}), (\ref{bad10}) and (\ref{kaLLrLs})} that
\begin{align}
&\left\|
\sqrt{\rhokaL^{[\Delta t]}}\,\utkaLDp \right\|_{L^{\frac{28}{13}}(0,T;L^{3}(\Omega))}
\nonumber \\
& \hspace{1in}\leq
\left( \int_0^T
\left\|
\sqrt{\rhokaL^{[\Delta t]}}\,\utkaLDp \right\|_{L^{\frac{12}{7}}(\Omega)}^{\frac{4}{13}}
\,\left\|
\sqrt{\rhokaL^{[\Delta t]}}\,\utkaLDp \right\|_{L^{\frac{24}{7}}(\Omega)}^{\frac{24}{13}}\,\dt\right)^{\frac{13}{28}}
\nonumber \\
& \hspace{1in} \leq
\left\|
\sqrt{\rhokaL^{[\Delta t]}}\,\utkaLDp \right\|_{L^{4}(0,T;L^{\frac{12}{7}}(\Omega))}^{\frac{1}{7}}
\,\left\|
\sqrt{\rhokaL^{[\Delta t]}}\,\utkaLDp \right\|_{L^{2}(0,T;L^{\frac{24}{7}}(\Omega))}^{\frac{6}{7}}
\leq C.
\label{bad11}
\end{align}
Combining (\ref{bad8}), (\ref{bad9}) and (\ref{bad11})
yields that
the second term on the right-hand side of (\ref{bad1})
converges to zero, as $\Delta t \rightarrow 0_+\;(L \rightarrow \infty)$.
Therefore, we have the desired result (\ref{bada}).

We now adapt the argument above to prove the desired result (\ref{badb}).
The bounds (\ref{bad1})--(\ref{bad4}) remain true with $\rhokaL^{[\Delta t]}$
replaced by $\rhokaLDp$. Similarly, (\ref{bad6}) remains true
with $\rhokaL^{[\Delta t]}$ on the left-hand side of the inequality
replaced by $\rhokaLDp$. Hence, the bounds (\ref{bad7}) and (\ref{bad8})
remain true with $\rhokaL^{[\Delta t]}$
replaced by $\rhokaLDp$. Therefore, on combining all these modified bounds with
(\ref{bad9}), we obtain the desired result (\ref{badb}).
\end{proof}

\subsection{$L$, $\Delta t$-independent bounds on the time-derivatives of $\mtkaLD$
and $\hpsikaLD$}
\label{utsec}
%
%
On noting from (\ref{ulin},b) that
\begin{align} \mtkaLD &= \frac{t-t_{n-1}}{\Delta t}\, \mtkaLDp + \frac{t_n-t}{\Delta t}\, \mtkaLDm,
\quad t \in (t_{n-1},t_n],\quad
n=1,\dots,N,
\label{u-t-2}
\end{align}
an elementary calculation yields, for any $s \in [1,\infty]$, that
\begin{align}\label{u-t-1}
\int_0^T \|
\mtkaLD\|_{L^s(\Omega)}^2 \dt 
&\leq \frac{1}{2}\int_0^T\left(\|
\mtkaLDp\|_{L^s(\Omega)}^2
+ \|
\mtkaLDm\|_{L^s(\Omega)}^2\right)\dt.
\end{align}
In order to pass to the limit $\Delta t \rightarrow 0_+$ $(L \rightarrow \infty)$
in the momentum equation (\ref{equncon}), we require  the following result.
\begin{lemma}\label{mtkaLlem}
There exists a $C \in \mathbb{R}_{>0}$, independent
of $\Delta t$ and $L$, such that
\begin{subequations}
\begin{align}
&\|\mtkaLDpm\|_{L^\infty(0,T;L^{\frac{2\Gamma}{\Gamma+1}}(\Omega))}
+\|\mtkaLDpm\|_{L^2(0,T;L^{\frac{6\Gamma}{\Gamma+6}}(\Omega))}
+ \|\mtkaLDpm\|_{L^\upsilon(\Omega_T)}
\leq C,
\label{mtkaLLnu}\\
&\|\mtkaLDpm \otimes \utkaLDpm \|_{L^2(0,T;L^{\frac{6\Gamma}{4\Gamma+3}}(\Omega))}
\leq C,
\label{mtutkaLD}
\end{align}
\end{subequations}
where $\upsilon=\frac{10\Gamma-6}{3(\Gamma+1)} \geq \frac{74}{27}$ as $\Gamma \geq 8$.
\end{lemma}
\begin{proof}
The first bound in (\ref{mtkaLLnu}) for $\mtkaLDwpm$ follows immediately from the third bound
in (\ref{kaLLinfL1}).
The corresponding bound for $\mtkaLD$ is a direct consequence of (\ref{u-t-2}).
Similarly to (\ref{L2Lsa}),
it follows from (\ref{kaLLinfL1}) and
(\ref{utkaLL6})
that
\begin{align}
\|\rhokaLDwpm\,\utkaLDwpm\|_{L^2(0,T;L^{\frac{6\Gamma}{\Gamma+6}}(\Omega))}
&\leq \|\rhokaLDwpm\|_{L^\infty(0,T;L^{\Gamma}(\Omega))}
\,\|\utkaLDwpm\|_{L^2(0,T;L^{6}(\Omega))}
\leq C,
\label{mtkaLL2Lsa}
\end{align}
and hence the second bound in (\ref{mtkaLLnu}) for $\mtkaLDwpm$.
The corresponding bound for $\mtkaLD$ is then a direct consequence of (\ref{u-t-1}).
Similarly to (\ref{Lup}),
it follows from (\ref{eqLinterp}) with $\upsilon=\frac{10\Gamma-6}{3(\Gamma+1)}$,
$r=\frac{2\Gamma}{\Gamma+1}$ and $s=\frac{6\Gamma}{\Gamma+6}$
that $\upsilon\,\vartheta=2$ \red{(with $\vartheta \in (0,1)$ thanks to $\Gamma \geq 8$),} and so
\begin{align}
\|\mtkaLDpm\|_{L^\upsilon(\Omega_T)}
\leq \|\mtkaLDpm\|_{L^\infty(0,T;L^r(\Omega))}^{1-\vartheta}
\,\|\mtkaLDpm\|_{L^2(0,T;L^s(\Omega))}^{\vartheta}.
\label{mtkaLLup}
\end{align}
Hence (\ref{mtkaLLup}) and the first two bounds in (\ref{mtkaLLnu}) yield the third bound in
(\ref{mtkaLLnu}).
Finally, combining the first bound in (\ref{mtkaLLnu}) and
(\ref{utkaLL6}) yields the bound (\ref{mtutkaLD}).
\end{proof}

Next, we need to bound the time-derivative of $\mtkaLD$ independently of $\Delta t$ and $L$.
It follows from (\ref{pkaLbd}) that we will need to choose
at least $\wt \in L^4(0,T;\Wt^{1,4}_0(\Omega))$ in (\ref{equncon}).
We now rewrite the time-derivative of $\rhokaL^{\Delta t}$ in \eqref{equncon}
using \eqref{eqrhocon}.
Adopting the notation (\ref{timav}), we have for any $\wt \in L^4(0,T;\Wt^{1,4}_0(\Omega))$
that $\utkaLDp \cdot \wt^{\{\Delta t\}} \in L^2(0,T;H^1(\Omega))$, and so
(\ref{equncon}) yields that
\begin{align} -\displaystyle\int_{0}^{T}\!\! \int_\Omega
\frac{\partial \rhokaL^{\Delta t}}{\partial t}\,\utkaLDp \cdot \wt \,\dx\,\dt
&= - \displaystyle\int_{0}^{T} \left \langle
\frac{\partial \rhokaLD}{\partial t},\utkaLDp \cdot \wt^{\{\Delta t\}}
\right \rangle_{H^1(\Omega)}
\dt \nonumber \\
& = \displaystyle\int_{0}^{T}\!\! \int_\Omega \left(
\alpha \,\nabx \rhokaL^{[\Delta t]} - \rhokaL^{[\Delta t]}\,\utkaLDp
\right) \cdot \nabx (\utkaLDp \cdot \wt^{\{\Delta t\}})
 \,\dx\, \dt \nonumber \\
&= \alpha \displaystyle\int_{0}^{T}\!\! \int_\Omega
\nabx \rhokaL^{[\Delta t]} \cdot \nabx (\utkaLDp \cdot \wt^{\{\Delta t\}})
 \,\dx\, \dt \nonumber \\
& \qquad - \displaystyle\int_{0}^{T}\!\! \int_\Omega
\rhokaL^{[\Delta t]}\,\utkaLDp \cdot \nabx (\utkaLDp \cdot \wt)
 \,\dx\, \dt
\nonumber \\
& \qquad + \displaystyle\int_{0}^{T}\!\! \int_\Omega
\rhokaL^{[\Delta t]}\,\utkaLDp \cdot \nabx (\utkaLDp \cdot (\wt-\wt^{\{\Delta t\}}))
 \,\dx\, \dt.
\label{equncona}
\end{align}
Next we note that, for all $\wt \in L^4(0,T;\Wt_0^{1,4}(\Omega))$,
\begin{align}
&\displaystyle\int_{0}^{T}\!\! \int_\Omega
\rhokaL^{[\Delta t]}\,\utkaLDp \cdot \nabx (\utkaLDp \cdot \wt)
 \,\dx\, \dt
\nonumber \\
& \hspace{1in} =
\displaystyle\int_{0}^{T}\!\! \int_\Omega
\mtkaLDm \cdot \nabx (\utkaLDp \cdot \wt)
 \,\dx\, \dt \nonumber \\
& \hspace{1.5in}
+ \displaystyle\int_{0}^{T}\!\! \int_\Omega
\left( \rhokaL^{[\Delta t]}\,\utkaLDp -
\rhokaLDm\,\utkaLDm \right)
\cdot \nabx (\utkaLDp \cdot \wt)
 \,\dx\, \dt.
\label{equnconb}
\end{align}
Therefore, on combining (\ref{equncon}), (\ref{equncona}) and (\ref{equnconb}),
one can rewrite (\ref{equncon}) as
\begin{align}
&\displaystyle\int_{0}^{T}\!\! \int_\Omega \left[
\frac{\partial \mtkaLD}{\partial t} \cdot \wt
+ \frac{\alpha}{2}
\,\nabx \rhokaL^{[\Delta t]} \cdot \nabx (\utkaLDp \cdot \wt^{\{\Delta t\}})
- \left[(\mtkaLDm \cdot \nabx) \wt  \right]\cdot\,\utkaLDp
\right]\!
\dx\, \dt
\nonumber \\
& \qquad \qquad
+\displaystyle\int_{0}^{T}\!\! \int_\Omega
\Stt(\utkaLDp)
:\nabxtt \wt
\,\dx\, \dt
- \int_{0}^T \!\!\int_{\Omega}
\pkaL^{\{\Delta t\}} \,\nabx \cdot \wt \,\dx \,\dt
\nonumber \\
&\qquad \qquad
- \frac{1}{2} \displaystyle\int_{0}^{T}\!\! \int_\Omega
\left( \rhokaL^{[\Delta t]}\,\utkaLDp -
\rhokaLDm\,\utkaLDm \right)
\cdot \nabx (\utkaLDp \cdot \wt)
 \,\dx\, \dt
\nonumber \\
& \qquad \qquad + \frac{1}{2}\displaystyle\int_{0}^{T}\!\! \int_\Omega
\rhokaL^{[\Delta t]}\,\utkaLDp \cdot \nabx (\utkaLDp \cdot (\wt-\wt^{\{\Delta t\}}))
 \,\dx\, \dt
\nonumber \\
&\qquad =\int_{0}^T
\int_{\Omega}  \left[\rhokaL^{\Delta t,+}\,
\ft^{\{\Delta t\}} \cdot \wt
-\tautt_1(M\,\hpsikaLDp)
: \nabxtt
\wt\right] \dx \, \dt
\nonumber
\\
& \qquad \qquad
- 2\, \mathfrak{z} \,
\int_0^T \int_\Omega
\left( \int_D M\,\beta^L(\hpsikaLDp) \,\dq \right)
\nabx 
\vrhokaLDp
\cdot \wt \,\dx\,\dt
\qquad
\forall \wt \in L^4(0,T;\Wt_0^{1,4}(\Omega)).
\nonumber\\
\label{equnconadj}
\end{align}

We have the following result.
\begin{lemma}
\label{mtkaLDtdlem}
There exists a $C \in \mathbb{R}_{>0}$, independent
of $\Delta t$ and $L$, such that
\begin{align}
\left\|\frac{\partial \mtkaLD}{\partial t}\right\|_{L^s(0,T;W_0^{1,4}(\Omega)')} \leq C,
\label{mtkaLDdtbd}
\end{align}
where $s=\frac{8\Gamma-12}{7\Gamma-6} \geq \frac{26}{25}$ as $\Gamma \geq 8$.
\end{lemma}
\begin{proof}
Let $\upsilon = \frac{8\Gamma -12}{3\Gamma} \geq \frac{13}{6}$ as in Lemma \ref{rhonabxsclem}.
Then for $s'=\frac{s}{s-1}=\frac{8\Gamma-12}{\Gamma-6} \red{>} 8$, we have that
\begin{align}
\frac{1}{s'}+\frac{1}{\upsilon}+ \frac{1}{2}=1.
\label{frac}
\end{align}
It follows from (\ref{equnconadj}), (\ref{frac}), $W^{1,4}(\Omega) \hookrightarrow L^\infty(\Omega)$,
$H^1(\Omega) \hookrightarrow L^6(\Omega)$,
(\ref{kaLLinfL1}--c), (\ref{utkaLL6}), (\ref{mtkaLLnu}), on noting that
$\frac{10\Gamma-6}{3(\Gamma+1)}\geq \frac{8\Gamma-12}{7\Gamma-6}$,
(\ref{pkaLbd}), (\ref{tautt43bd}),
(\ref{betaLa}),  (\ref{eq:energy-u+psi-final2}) and (\ref{fnbd}) that,
for any $\wt \in L^{s'}(0,T;\Wt_0^{1,4}(\Omega))$,
\begin{align}
&\left|\displaystyle\int_{0}^{T}\!\! \int_\Omega
\frac{\partial \mtkaLD}{\partial t} \cdot \wt \,\dx\,\dt\right| \nonumber \\
& \quad \leq C\,\|\nabx \rhokaL^{[\Delta t]}\|_{L^{\upsilon}(\Omega_T)}
\,\|\utkaLDp\|_{L^2(0,T;H^1(\Omega))}\,\|\wt^{\{\Delta t\}}\|_{L^{s'}(0,T;W^{1,4}(\Omega))}
\nonumber \\
& \qquad +
C\,\left[\|\rhokaL^{[\Delta t]}\,\utkaLDp\|_{L^{\upsilon}(\Omega_T)}
+ \|\mtkaLDm\|_{L^{\upsilon}(\Omega_T)} + 1
\right]
\,\|\utkaLDp\|_{L^2(0,T;H^1(\Omega))}\,\|\wt\|_{L^{s'}(0,T;W^{1,4}(\Omega))}
\nonumber \\
& \qquad + C\,\|\rhokaL^{[\Delta t]}\,\utkaLDp\|_{L^{\upsilon}(\Omega_T)}
\,\|\utkaLDp\|_{L^2(0,T;H^1(\Omega))}
\,\|\wt^{\{\Delta t\}}\|_{L^{s'}(0,T;W^{1,4}(\Omega))}
\nonumber \\
& \qquad +
\|\pkaL^{\{\Delta t\}}\|_{L^{\frac{4}{3}}(\Omega_T)}\, \|\wt\|_{L^{4}(0,T;W^{1,4}(\Omega))}
+ C\,\|\tautt_1(M\,\hpsikaLDp)\|_{L^2(0,T;L^{\frac{4}{3}}(\Omega))}\,
\|\wt\|_{L^{2}(0,T;W^{1,4}(\Omega))}
\nonumber \\
& \qquad +\|\vrhokaLDp\|_{L^\infty(0,T;L^2(\Omega))}
\,\|\vrhokaLDp\|_{L^2(0,T;H^1(\Omega))}\,\|\wt\|_{L^2(0,T;L^\infty(\Omega))}
\nonumber \\
&\qquad + \|\rhokaL^{\Delta t,+}\|_{L^\infty(0,T;L^2(\Omega))}\,
\|\ft^{\{\Delta t\}}\|_{L^2(0,T;L^\infty(\Omega))}\, \|\wt\|_{L^{2}(\Omega_T)}
\nonumber \\
& \quad \leq 
C\,\left[ 
\|\wt\|_{L^{s'}(0,T;W^{1,4}(\Omega))}
+
\|\wt^{\{\Delta t\}}\|_{L^{s'}(0,T;W^{1,4}(\Omega))}\right].
\nonumber \\
\label{equnconadjbd}
\end{align}
The desired result (\ref{mtkaLDdtbd}) then follows immediately from (\ref{equnconadjbd})
on noting, as $s'\geq 4$, that, for any $\wt \in L^{s'}(0,T;\Wt_0^{1,4}(\Omega))$,
\begin{align}
\|\wt^{\{\Delta t\}}\|_{L^{s'}(0,T;W^{1,4}(\Omega))}^{s'}
& = \Delta t \,\sum_{n=1}^N  \,\left\|\frac{1}{\Delta t} \int_{t_{n-1}}^{t_n} \wt \,\dt\right\|_{W^{1,4}(\Omega)}^{s'}
\nonumber \\
&\leq C\,\Delta t\,\sum_{n=1}^N  \left(\frac{1}{\Delta t} \,
\int_{t_{n-1}}^{t_n} \|\wt\|_{W^{1,4}(\Omega)}^{4}\,\dt\right)^{\frac{s'}{4}}
\leq C\,
\|\wt\|_{L^{s'}(0,T;W^{1,4}(\Omega))}^{s'}.
\label{wtimeav}
\end{align}
\end{proof}

Next, we bound the time derivative of $\hpsikaLD$.

\begin{lemma}\label{hpsikaLDdtlem}
There exists a $C \in {\mathbb R}_{>0}$, independent of $\Delta t$ and $L$, such that
\begin{align}
\left\|M\,\frac{\partial \hpsikaLD}{\partial t}
\right\|_{L^2(0,T;H^s(\Omega\times D)')} \leq C,
\label{hpsikaLDdtbd}
\end{align}
where $s > 1+ \frac{1}{2}(K+1)d$.
\end{lemma}
\begin{proof}
It follows from (\ref{eqpsincon}), (\ref{eq:energy-u+psi-final2}), (\ref{betaLa})
and \red{(\ref{vrhokaLvvbd})}
that, for any $\varphi \in L^2(0,T;W^{1,\infty}(\Omega\times D))$,
\begin{align}
\label{eqpsinconbd}
&\left|\int_{0}^T \!\!\int_{\Omega \times D}
M\,\frac{ \partial \hpsikaLD}{\partial t}
\varphi \dq \dx \dt\right|
\nonumber \\
& \; \leq
2\epsilon \left|\int_{0}^T \!\!\int_{\Omega \times D} M
\,\sqrt{\hpsikaLDp}\,
\nabx \sqrt{\hpsikaLDp} \cdot\, \nabx
\varphi
\dq \dx \dt \right|
\nonumber \\
& \quad \quad +
\frac{1}{2\,\lambda}
\left|
\,\sum_{i=1}^K
 \,\sum_{j=1}^K A_{ij}\,
 \int_{0}^T
 \!\!\int_{\Omega \times D}
 M\,\sqrt{\hpsikaLDp}\,\nabqj \sqrt{\hpsikaLDp}
\cdot\, \nabqi
\varphi
\dq \dx \dt\right|
\nonumber \\
& \quad \quad
+ \left|\int_{0}^T \!\!\int_{\Omega \times D} M \,\utkaLDp\,
\beta^L(\hpsikaLDp)\cdot\, \nabx
\varphi
\dq \dx \dt \right|
\nonumber \\
&
\quad \quad
+ \left| \int_{0}^T \!\!\int_{\Omega \times D} M\,\sum_{i=1}^K
\left[\sigtt(\utkaLDp)
\,\qt_i\right]
\beta^L(\hpsikaLDp) \,\cdot\, \nabqi
\varphi
\,\dq \dx \dt\right|
\nonumber \\
& \;
\leq C\,\max\left\{1,\|\vrhokaLDp\|_{L^\infty(0,T;L^2(\Omega))}\right\}
\biggl[ \|\nabx \sqrt{\hpsikaLDp}\|_{L^2(0,T;L^2_M(\Omega\times D))}
\nonumber \\
& \quad \quad
+
\|\nabq \sqrt{\hpsikaLDp}\|_{L^2(0,T;L^2_M(\Omega\times D))}
+ \|\utkaLDp\|_{L^2(0,T;H^1(\Omega))}\biggr]
\,\|\varphi\|_{L^2(0,T;W^{1,\infty}(\Omega \times D))}
\nonumber \\
& \; \leq C
\,\|\varphi\|_{L^2(0,T;W^{1,\infty}(\Omega \times D))}.
\end{align}
The desired result (\ref{hpsikaLDdtbd}) then follows on noting that $H^s(\Omega \times D) \hookrightarrow
W^{1,\infty}(\Omega \times D)$ for the stated bound on $s$.
\end{proof}

\begin{remark}\label{rem-z}
\red{We note that allowing $\mathfrak{z}=0$ would impact on the proof of Lemma \ref{mtkaLDtdlem}. As is already clear from the
formal energy inequality \eqref{energy-id4}, by setting $\mathfrak{z} = 0$ one looses control over the $L^\infty(0,T;L^2(\Omega))
\cap L^2(0,T; H^1(\Omega))$ norm of $\varrho^{\Delta t,+}_{\kappa,\alpha,L}$; instead, one can only control weaker norms
of $\varrho^{\Delta t,+}_{\kappa,\alpha,L}$, leading to \eqref{tautt1red} in place of \eqref{tautt43bd}.
While this weaker control is not
sufficient to prove \eqref{mtkaLDdtbd} as stated, one can prove a weaker result by replacing the $\Wt^{1,4}(\Omega)$
norm on the test function $\wt$ by the $\Wt^{1,2d}(\Omega)$ norm throughout the proof.
Admitting $\mathfrak{z}=0$ in Lemma \ref{hpsikaLDdtlem}, on the other hand, results in unsurmountable difficulties:
the proof of the lemma cannot be completed without an $L^r(0,T;L^2(\Omega))$ norm bound on $\varrho^{\Delta t,+}_{\kappa,\alpha,L}$, with $r>2$ at least; in particular, we are unable to prove \eqref{hpsikaLDdtbd}, or a weaker result on the time-difference of $\hpsikaLD$, when $\mathfrak{z}=0$. Remark \ref{z0rem},
and equation \eqref{vrhoLsbd} in particular, indicate that an $L^r(0,T;L^2(\Omega))$ norm bound on $\varrho^{\Delta t,+}_{\kappa,\alpha,L}$, with $r>2$, is unlikely to hold without requiring $\mathfrak{z}>0$, regardless of the choice of $\mathfrak{k}$ in \eqref{tau1a}.}
\end{remark}

\subsection{Passing to the limit $\Delta t \rightarrow 0_+$ $(L\rightarrow \infty)$ in the
momentum equation (\ref{equncon}) and the {F}okker--{P}lanck equation (\ref{eqpsincon})}

We have the following convergence results.

\begin{lemma}\label{mtkaLDconlem}
We have that
\begin{align}
\mtka:= \rhoka\,\utka \in \Lt^\upsilon(\Omega_T) \cap W^{1,s}(0,T;\Wt_0^{1,4}(\Omega)')
\cap 
C_w([0,T];L^{\frac{2\Gamma}{\Gamma+1}}(\Omega)),
\label{mtkareg}
\end{align}
where $\upsilon = \frac{10\Gamma-6}{3(\Gamma+1)}$ and $s=\frac{8\Gamma-12}{7\Gamma-6}$,
with $\Gamma \geq 8$.
\red{In addition,} for a further subsequence of the subsequences of Lemmas \ref{rhoLconv} and \ref{rhonabxsclem}
it follows that, as $\Delta t \rightarrow 0_+$ $(L \rightarrow \infty)$,
\begin{subequations}
\begin{alignat}{2}
\mtkaLDpm &\rightarrow \mtka  \qquad&&\mbox{weakly in } \Lt^\upsilon(\Omega_T),
\label{mtkaLwcon}\\
\mtkaLD &\rightarrow \mtka
\qquad &&\mbox{weakly in } W^{1,s}(0,T;\Wt_0^{1,4}(\Omega)'),
\label{mtkaLdtwcon}\\
\mtkaLDpm &\rightarrow \mtka  \qquad&&\mbox{strongly in } L^2(0,T;\Ht^{1}(\Omega)'),
\label{mtkaLscon}\\
\mtkaLDpm \otimes \utkaLDp &\rightarrow \mtka\otimes \utka
\qquad &&\mbox{weakly in } L^{2}(0,T;\Ltt^{\frac{6\Gamma}{4\Gamma+3}}(\Omega)),
\label{mtutkaLDpmcon}\\
\mtkaLDpm & \rightarrow \mtka  \qquad&&
\mbox{in } C_w([0,T];L^{\frac{2\Gamma}{\Gamma+1}}(\Omega)).
\label{mtkaLwscon}
\end{alignat}
\end{subequations}
\end{lemma}
\begin{proof}
The weak convergence result (\ref{mtkaLwcon}) for
some limit function $\mtka \in \Lt^\upsilon(\Omega_T)$\red{, which is the common limit of
$\mtkaL^{\Delta t}$, $\mtkaL^{\Delta t,+}$ and $\mtkaL^{\Delta t,-}$,} follows immediately from (\ref{mtkaLLnu}),
(\ref{badb}), (\ref{mtkaLn}) and (\ref{u-t-2}).
The weak convergence result (\ref{mtkaLdtwcon})
and the strong convergence result (\ref{mtkaLscon})
for $\mtkaLD$ follow immediately for $\mtka \in \Lt^\upsilon(\Omega_T)\cap
W^{1,s}(0,T;\Wt_0^{1,4}(\Omega)')$
from (\ref{mtkaLLnu}), (\ref{mtkaLDdtbd}) and (\ref{compact1}),
on noting that  $L^\upsilon(\Omega)$ is compactly embedded in $H^1(\Omega)'$, \red{which is in turn continuously embedded in $W^{1,4}(\Omega)'$}.
The corresponding result (\ref{mtkaLscon}) for $\mtkaLDwpm$ then follows from (\ref{mtkaLscon}) for $\mtkaLD$,
and (\ref{badb}).
It follows from (\ref{rhoLDpcon}), (\ref{uLwconH1}) and (\ref{mtkaLscon})
for $\mtkaLDp$ that $\mtka=\rhoka\,\utka$, and hence (\ref{mtkareg}).
The result (\ref{mtutkaLDpmcon}) follows immediately from
(\ref{mtutkaLD}), (\ref{mtkaLscon}) and (\ref{uLwconH1}). Finally \eqref{mtkaLwscon}
follows from the bound on the first term in \eqref{mtkaLLnu}, (\ref{mtkaLDdtbd}) and (\ref{Cwcoma},b).
\end{proof}

Next, noting (\ref{ulin},b), a simple calculation yields that
[see (6.32)--(6.34) in \cite{BS2010} for details]:
\begin{align}
\int_0^T\int_{\Omega \times D} M\big|\nabx \sqrt{\hpsikaLD}\big|^2\dq\dx \dt
&\leq 2 \int_0^T \int_{\Omega \times D} M\,
\left[ |\nabx\sqrt{\hpsikaLDp}|^2
+ |\nabx\sqrt{\hpsikaLDm}|^2 \right] \dq \dx \dt,
\label{nabxDTpm}
\end{align}
and an analogous result with $\nabx$ replaced by $\nabq$.
\red{Then, the bound \eqref{eq:energy-u+psi-final2},
on noting
(\ref{inidata-1}), (\ref{nabxDTpm})
and the convexity of ${\mathcal F}$,}
imply the existence
of a $C \in \mathbb{R}_{>0}$,
independent of $\Delta t$ and $L$, such that:
\begin{align}\label{eq:energy-u+psi-final5}
&
\mbox{ess.sup}_{t \in [0,T]}
\int_{\Omega \times D}  M\, \mathcal{F}(\hpsikaLDpm(t)) \dq \dx
 + \frac{1}{\Delta t\, L}
\int_0^T \!\! \int_{\Omega \times D}\!\! M\, (\hpsikaLDp - \hpsikaLDm)^2
\dq \dx \dd t
\nonumber \\
&\; + \int_0^T\!\! \int_{\Omega \times D} M\,
\big|\nabx \sqrt{\hpsikaLDpm} \big|^2 \dq \dx \dd t
+\, \int_0^T\!\! \int_{\Omega \times D}M\,\big|\nabq \sqrt{\hpsikaLDpm}\big|^2
\,\dq \dx \dd t \leq C.
\end{align}

\begin{lemma}
\label{convfinal}
%
%
%
For a further subsequence of the subsequences of Lemmas \ref{rhoLconv}, \ref{rhonabxsclem}
and \ref{mtkaLDconlem},
there exists a function
\begin{subequations}
\begin{align}
\label{hpsika}
&\hpsika \in L^{\upsilon}(0,T;Z_1
)\cap
H^1(0,T; M^{-1}(H^{s}(\Omega \times D))'),
\end{align}
where $\upsilon \in [1,\infty)$ and $s>1+\frac{1}{2}(K+1)d$,
with finite relative entropy and Fisher information,
\begin{align}
\mathcal{F}(\hpsika) \in L^\infty(0,T;L^1_M(\Omega \times D)) \qquad \mbox{and} \qquad
\sqrt{\hpsika} \in L^2(0,T;H^1_M(\Omega \times D)),
\label{Fisent}
\end{align}
%
\end{subequations}
such that,
as $\Delta t \rightarrow 0_+$ $(L \rightarrow \infty)$,
\begin{subequations}
\begin{alignat}{2}
M^{\frac{1}{2}}\,\nabx \sqrt{\hpsikaLDpm} &\rightarrow M^{\frac{1}{2}}\,\nabx \sqrt{\hpsika}
&&\qquad \mbox{weakly in } L^{2}(0,T;\Lt^2(\Omega\times D)), \label{psiwconH1a}\\
M^{\frac{1}{2}}\,\nabq \sqrt{\hpsikaLDpm} &\rightarrow M^{\frac{1}{2}}\,\nabq \sqrt{\hpsika}
&&\qquad \mbox{weakly in } L^{2}(0,T;\Lt^2(\Omega\times D)), \label{psiwconH1xa}\\
M\,\frac{\partial \hpsikaLD}{\partial t} &\rightarrow M\,\frac{\partial \hpsika}{\partial t}
&&\qquad \mbox{weakly in } L^{2}(0,T;H^s(\Omega\times D)'), \label{psidtwcon}\\
\hpsikaLDpm & \rightarrow
\hpsika &&\qquad \mbox{strongly in } L^\upsilon(0,T;L^1_M(\Omega\times D)),\label{psisconL1}\\
\beta^L(\hpsikaLDpm) & \rightarrow
\hpsika &&\qquad \mbox{strongly in } L^\upsilon(0,T;L^1_M(\Omega\times D)),\label{psisconL1beta}\\
\tautt(M\,\hpsikaLDp) & \rightarrow \tautt(M\,\hpsika)
&&\qquad \mbox{strongly in } \Ltt^r(\Omega_T),\label{tausconLs}
\end{alignat}
where  $r \in [1,\frac{4(d+2)}{3d+4})$; and,
for a.a.\ $t \in (0,T)$,
\begin{align}\label{fatou-app}
&\int_{\Omega \times D} M(\qt)\, \mathcal{F}(\hpsika(\xt,\qt,t))\dq \dx
\leq \liminf_{\Delta t \rightarrow 0_+\, (L\rightarrow \infty)}
\int_{\Omega \times D} M(\qt)\, \mathcal{F}(\hpsikaLDp(\xt,\qt,t)) \dq \dx.
\end{align}
\end{subequations}

In addition, we have that
\begin{align}
\vrhoka := \int_D M\, \hpsika\, \dq \in L^\infty(0,T;L^2(\Omega)) \cap L^2(0,T;H^1(\Omega)),
\label{vrhokareg}
\end{align}
and, as $\Delta t \rightarrow 0_+$ $(L \rightarrow \infty)$,
\begin{subequations}
\begin{align}
\vrhokaLDp &\rightarrow \vrhoka \quad
\mbox{weakly-$\star$ in } L^\infty(0,T;L^2(\Omega)),
\quad \mbox{weakly in } L^2(0,T;H^1(\Omega)),
\label{vrhokaLDpH1con} \\
\vrhokaLDp,\,\int_D M\,\beta^L(\hpsikaLDp)\,\dq &\rightarrow \vrhoka
\quad \mbox{strongly in } L^{\frac{5\varsigma}{3(\varsigma-1)}}(0,T;L^\varsigma(\Omega)),
\label{bvrhokaLDpL2con}
\end{align}
\end{subequations}
for any  $\varsigma \in (1,6)$.
\end{lemma}
\begin{proof}
In order to prove the strong convergence result (\ref{psisconL1}),
we will apply Dubinski\u{\i}'s compactness result (\ref{Dubinskii}) with
\red{
${\mathfrak X}=L^1_M(\Omega \times D)$, ${\mathfrak X_1}=M^{-1}\,H^s(\Omega \times D)'$
and
\begin{align}
{\mathfrak M} = \{ \varphi \in Z_1 : \int_{\Omega \times D} M \,\left[
|\nabq \sqrt{\varphi}|^2 + |\nabx \sqrt{\varphi}|^2 \right] \dq \dx <\infty \}.
\label{calM}
\end{align}
See Section 5 in \cite{BS2011-fene}
for the proof of the compactness of the embedding
$\mathfrak{M} \hookrightarrow \mathfrak{X}$,
and the continuity of the embedding $\mathfrak{X} \hookrightarrow \mathfrak{X}_1$.}
Hence, the desired result (\ref{psisconL1}) for $\hpsikaLD$ and $\upsilon=1$ follows from
(\ref{Dubinskii}) with $\varsigma_0=\varsigma_1=2$,
and the stated choices of \red{${\mathfrak M}$, ${\mathfrak X}$ and ${\mathfrak X}_1$}
above, on noting 
(\ref{eq:energy-u+psi-final5}) and (\ref{hpsikaLDdtbd}).
The desired result (\ref{psisconL1}) for $\hpsikaLDwpm$ and $\upsilon=1$
then follows from (\ref{psisconL1}) for $\hpsikaLD$ and $\upsilon=1$, (\ref{u-t-2}) with
$\mtkaLDpm$ replaced by $\hpsikaLDpm$, the second bound in
(\ref{eq:energy-u+psi-final5}) and (\ref{LT}).
The desired result (\ref{psisconL1}) for $\upsilon\in (1,\infty)$
then follows from (\ref{psisconL1}) for $\upsilon=1$,
the first bound in (\ref{eq:energy-u+psi-final5}) \red{(note that $\mathcal{F}(s) \geq [s-{\rm e}+1]_{+}$  for all $s \geq 0$, since $\mathcal{F}(s) \geq 0$ and, by convexity of $\mathcal{F}$, $\mathcal{F}(s) \geq \mathcal{F}({\rm e}) + (s-{\rm e})\mathcal{F}'({\rm e})=s - {\rm e}+1$)} and
an interpolation result, see Lemma 5.1 in \cite{BS2011-fene}.
The weak convergence result (\ref{psidtwcon}) follows immediately from
(\ref{hpsikaLDdtbd}).
The weak convergence results (\ref{psiwconH1a},b) follow immediately from
the last two terms in (\ref{eq:energy-u+psi-final5}),
on noting an argument similar to that in the proof of Lemma 3.3 in
\cite{BS2011-fene} in order to identify the limit.
The result (\ref{psisconL1beta}) follows from (\ref{psisconL1})
and the Lipschitz continuity of $\beta^L$, see (5.8) in
\cite{BS-2012-density-JDE} for details.
The result (\ref{fatou-app}) follows from (\ref{psisconL1}) and Fatou's lemma,
see (6.46) in \cite{BS2011-fene} for details.
In addition, the convergence results (\ref{psiwconH1a}--d,g)
yield the  desired results (\ref{hpsika},b).

The results (\ref{vrhokaLDpH1con}) for some limit function $\vrhoka$
follow immediately from the bounds on $\vrhokaLDp$ in (\ref{eq:energy-u+psi-final2}).
The fact that $\vrhoka= \int_D M\,\hpsika \dq$ follows from (\ref{vrhokaLn}), (\ref{upm})
and (\ref{psisconL1}),
and hence the desired result (\ref{vrhokareg}).
The strong convergence results (\ref{bvrhokaLDpL2con}) follow from noting the embedding
$H^1(\Omega) \hookrightarrow L^6(\Omega)$, (\ref{betaLa}), (\ref{psisconL1},e)
and (\ref{eqLinterp}).

Finally, we need to prove (\ref{tausconLs}).
Similarly to (\ref{Cttrsbdvr})--(\ref{Ctt43bd},b), we deduce from (\ref{Fisent}), (\ref{vrhokareg}),
(\ref{eqinterp}) and (\ref{Cttrsbd}) that
\begin{align}\label{tautt43bd-0}
\|\rhoka\|_{L^{\frac{2(d+2)}{d}}(\Omega_T)}+
\|\tautt_1(M\,\hpsika)\|_{L^{\frac{4(d+2)}{3d+4}}(\Omega_T)}
&\leq C.
\end{align}
On recalling (\ref{tau1}) and (\ref{Cittiden}),  we have that
\begin{align}
\tautt_1(M\,\varphi)
= k   \int_{D} M\,
\left(\sum_{i=1}^K \nabqi \varphi 
\otimes \qt_i\right) 
\dq - k \left(\int_{D} M 
\,\varphi
\dq\right) \Itt.
\label{tau1ibp}
\end{align}
Let $D_0 \subset \overline{D_0} \subset D$ be an arbitrary
Lipschitz subdomain of $D$, then \eqref{tau1ibp} yields that
\begin{align}\label{tau1limit2}
&\int_{\Omega_T} |\tautt_1(M \hpsikaLDp)
- \tautt_1(M \hpsika)
| \dx \dt
\nonumber \\
&\hspace{1in} \leq k\, |\bt|_1^{\frac{1}{2}} \int_0^T \int_{\Omega} \int_{D\setminus D_0} M
\,\sum_{i=1}^K \left( |\nabq \hpsikaLDp
| + |\nabq\hpsika
|\right) \dq \dx \dt \nonumber\\
&\hspace{1.5in} + k \int_0^T \int_{\Omega} \left| \int_{D_0} M
\,\sum_{i=1}^K \qt_i
\otimes \nabqi (\hpsikaLDp 
- \hpsika 
) \dq \right|\dx \dt \nonumber \\
&\hspace{1.5in} + k \,d^{\frac{1}{2}} \int_0^T \int_{\Omega \times D} M
\,|\hpsikaLDp 
- \hpsika 
|\dq \dx \dt =: {\tt T}_1 + {\tt T}_2 + {\tt T}_3, 
\end{align}
where we have recalled (\ref{inidata}).
%
Further, 
we deduce from (\ref{psiwconH1xa},d) that
\[M\,\nabqi \hpsikaLDp= 2M\, \sqrt{\hpsikaLDp}\, \nabqi\sqrt{\hpsikaLDp} \rightarrow
2M \,\sqrt{\hpsika}\, \nabqi\sqrt{\hpsika} = M\, \nabqi \hpsika, \quad i=1,\dots,K,\]
weakly in $L^1(0,T; \Lt^1(\Omega \times D)) = \Lt^1(\Omega_T \times D)$
as $\Delta t \rightarrow 0_+$ $(L\rightarrow \infty)$.
By the Dunford--Pettis theorem the sequence $\{M\,\nabq \hpsikaLDp\}_{\Delta t>0}$ is therefore
equi-integrable in $\Lt^1(\Omega_T \times D)$; hence, for any $\delta>0$
there exists a $\delta_0=\delta_0(\delta)$ such that for any set $D_0
\subset D$ with $T \,|\Omega|\, |D\setminus D_0|< \delta_0$,
\[ k\, |\bt|_1^{\frac{1}{2}} \int_0^T \int_{\Omega} \int_{D\setminus D_0} M
\,\sum_{i=1}^K \left( |\nabq \hpsikaLDp
| + |\nabq\hpsika
|\right) \dq \dx \dt < \frac{\delta}{3}.\]
We therefore select $D_0 \subset \overline{D_0} \subset D$ to be a Lipschitz subdomain of $D$ such that
$T \,|\Omega|\, |D\setminus D_0|< \delta_0$, which implies that $0<{\tt T}_1 < \frac{\delta}{3}$; that,
now, fixes our choice of $D_0$.

Next, 
we bound ${\tt T}_2$.
By performing partial integration over $D_0$, we have that
\begin{align}
{\tt T}_2&= k \int_0^T \int_{\Omega} \left| \int_{D_0} M
\,\sum_{i=1}^K \qt_i \otimes \nabqi (\hpsikaLDp 
- \hpsika
) \dq \right|\dx \dt \nonumber
\\
&\leq k \int_0^T \int_{\Omega} \left| - \int_{D_0}
\sum_{i=1}^K \left(\nabqi M 
\otimes  \qt_i\right)
(\hpsikaLDp
- \hpsika
)
\dq \right|\dx \dt \nonumber\\
&\qquad + k \int_0^T \int_{\Omega} \left| - K \left[\int_{D_0} M
\,(\hpsikaLDp
- \hpsika
) \dq \right] \Itt \right|\dx \dt \nonumber\\
&\qquad + k \int_0^T \int_{\Omega} \left|\int_{\partial D_0}
\sum_{i=1}^K M 
\,(\nt_i \otimes  \qt_i) (\hpsikaLDp 
- \hpsika
) \,{\rm d}\sigma(\qt) \right|\dx \dt, \nonumber
\end{align}
where the $d$-component column vector $\nt_i$ is the $i$th component of the
$Kd$-component unit outward (column) normal vector
$\nt = (\nt_1^{\rm T},\dots, \nt_K^{\rm T})^{\rm T}$ to the boundary $\partial D_0$ of $D_0$.
As the closure of the Lipschitz subdomain $D_0$ is a strict subset of the open set $D$,
we have, on noting (\ref{eqM}) and (\ref{growth1},b),
that $\sup_{\qt \in D_0} \big( \frac{1}{M(\qt)}|\nabq M(\qt)|\big)
\leq C(\delta_0)<\infty$. Hence,
\begin{align}
{\tt T}_2 &\leq k\, \int_0^T \int_{\Omega}\int_{D_0} \left[
|\bt|_1^{\frac{1}{2}}\,|\nabq M 
| + K\,d^{\frac{1}{2}} \,M
\right]
|\hpsikaLDp
- \hpsika
|\dq \dx \dt \nonumber
\\
& \qquad \qquad + k \,|\bt|_1^{\frac{1}{2}} \int_0^T \int_{\Omega}
\int_{\partial D_0} M
\, |\hpsikaLDp
- \hpsika
| \,{\rm d}\sigma(\qt)\dx \dt \nonumber\\
& \leq  k\,\left( |\bt|_1^{\frac{1}{2}} \,C(\delta_0)+ K\, d^{\frac{1}{2}}\right)\,
\int_0^T \int_{\Omega} \int_{D_0} M
\,|\hpsikaLDp
- \hpsika
| \dq \dx \dt \nonumber\\
&\qquad \qquad  + k \,|\bt|_1^{\frac{1}{2}}\,\|M\|_{L^\infty(D)}\,
\int_0^T \int_{\Omega} \int_{\partial D_0}
|\hpsikaLDp
- \hpsika
| \,{\rm d}\sigma(\qt)\dx \dt \nonumber\\
&=:{\tt T}_{21} + {\tt T}_{22}.
\label{T21T22}
\end{align}
Thus, thanks to \eqref{psisconL1} with $\upsilon=1$, there exists a $\Delta t_0$ such that for all
$\Delta t \leq \Delta t_0$,
we have that $0 < {\tt T}_{21} < \frac{\delta}{6}$ and $0 < {\tt T}_{3} < \frac{\delta}{3}$.

Finally, we shall show that, for $\Delta t_0$ sufficiently small, also $0<{\tt T}_{22}< \frac{\delta}{6}$.
In the process of doing so we shall repeatedly use the following result.
As the closure of $D_0$ is a compact subset of $D$, we have from (\ref{growth1}) that
\begin{equation}\label{tau1limit0}
\sup_{\qt \in D_0} [M(\qt)]^{-1} \leq C(D_0)< \infty.
\end{equation}
We begin by noting that (\ref{eq:energy-u+psi-final5}) and (\ref{tau1limit0})
imply that
$\{ \sqrt{\hpsikaLDp}\}_{\Delta t>0}$ is a bounded sequence in $L^2(0,T;H^1(\Omega \times D_0))$;
hence, by Sobolev embedding,
it is also a bounded sequence in $L^2(0,T;$ $L^{\frac{2(K+1)d}{(K+1)d-2}}(\Omega \times D_0))$.
Further, by \eqref{eq:energy-u+psi-final5} and (\ref{tau1limit0}),
$\{\sqrt{\hpsikaLDp}\}_{\Delta t >0}$ is a bounded sequence in
$L^\infty(0,T;L^2(\Omega \times D_0))$. 
It then follows from (\ref{eqLinterp}) that $\{ \sqrt{\hpsikaLDp}\}_{\Delta t >0}$
is a bounded sequence in $L^{\frac{2((K+1)d +2)}{(K+1)d}}(0,T; L^{\frac{2((K+1)d +2)}{(K+1)d}}
(\Omega \times D_0))$; thus,
\begin{align}\label{tau1limit3}
\int_0^T \int_{\Omega \times D_0} |\hpsikaLDp
|^{\frac{(K+1)d +2}{(K+1)d}}
\dq \dx \dt 
\leq C(D_0),
\end{align}
where the constant $C(D_0)$ is independent of $\Delta t$ and $L$. Now, for any $s \in (1,2)$,
we have by H\"older's inequality, 
(\ref{eq:energy-u+psi-final5}), (\ref{tau1limit0})
and the inequality $(a^{\frac{s}{2}}
+ b^{\frac{s}{2}}) \leq 2^{1-\frac{s}{2}} (a+b)^{\frac{s}{2}}$ with $a,\,b\geq 0$,
which follows from the concavity of
the function $x \in [0,\infty) \mapsto x^{\frac{s}{2}} \in [0,\infty)$, that
\begin{align*}
&\int_0^T \int_{\Omega \times D_0} \left( |\nabx \hpsikaLDp
|^s
+ |\nabq \hpsikaLDp
|^s \right) \dq \dx \dt\\
&\hspace{1in}= 2^s \int_0^T \int_{\Omega \times D_0} |\hpsikaLDp
|^{\frac{s}{2}}
\left( \left|\nabx \sqrt{\hpsikaLDp
}\,\right|^s + \left|\nabq\sqrt{\hpsikaLDp
}\,
\right|^s\right) \dq \dx \dt\\
&\hspace{1in}\leq 2^{\frac{s}{2}+1} \left(\int_0^T \int_{\Omega \times D_0} |\hpsikaLDp
|^{\frac{s}{2-s}}
\dq \dx \dt\right)^{\frac{2-s}{2}}\nonumber\\
&\hspace{1.5in}\times \left(\int_0^T \int_{\Omega \times D_0} \left|\nabx \sqrt{\hpsikaLDp
}\,\right|^2
+ \left|\nabq\sqrt{\hpsikaLDp
}\,\right|^2\dq \dx \dt\right)^{\frac{s}{2}}\\
&\hspace{1in}\leq C \left(\int_0^T \int_{\Omega \times D_0} |\hpsikaLDp
|^{\frac{s}{2-s}}\dq \dx \dt\right)^{\frac{2-s}{2}}.
\end{align*}
Comparing this with \eqref{tau1limit3} indicates that $s \in (1,2)$ should be chosen so that
\[  \frac{s}{2-s} \leq    \frac{(K+1)d +2}{(K+1)d}. \]
The largest such $s$ is $s= \frac{(K+1)d+2}{(K+1)d+1}$; using this value of $s$, we then deduce
on noting \eqref{tau1limit3} that
\begin{align}\label{tau1limit4}
&\|\hpsikaLDp\|_{L^s(0,T;W^{1,s}(\Omega \times D_0))} \leq C(D_0).
\end{align}
Note further that, thanks to \eqref{tau1limit0} and \eqref{psisconL1},
we have for any $\upsilon \in [1,\infty)$ that
\begin{align}\label{tau1limit5}
\|\hpsikaLDp- \hpsika\|_{L^\upsilon(0,T;L^{1}(\Omega \times D_0))} \rightarrow 0,
\quad
\mbox{as $\Delta t \rightarrow 0_+$ $(L \rightarrow  \infty)$.}
\end{align}

We shall now use \eqref{tau1limit4} and \eqref{tau1limit5} to show that ${\tt T}_{22}$ converges to $0$
as $\Delta t \rightarrow 0_+$ $(L \rightarrow \infty)$.
To this end, we shall make use of the following sharp trace inequality, established recently by
Auchmuty (cf. Theorem 6.3 inequality (6.3) in \cite{auchmuty}): suppose that $\mathcal{O}$ is a
bounded Lipschitz domain and let $r\in (1,2)$; then, the following inequality holds for all
$\varphi \in W^{1,r}(\mathcal{O})$:
\begin{equation}
\int_{\partial \mathcal{O}} |\varphi|^{2-\frac{1}{r}} \dd \sigma \leq \frac{|\partial \mathcal{O}|}{|\mathcal{O}|}
\|\varphi\|^{2-\frac{1}{r}}_{L^{2-\frac{1}{r}}(\mathcal{O})} + \left(2- \frac{1}{r}\right)
k_{\mathcal O}\, \| \varphi\|_{L^{1}(\mathcal{O})}^{1-\frac{1}{r}}\,
\|\nabla \varphi\|_{L^r(\mathcal{O})},
\label{aucha}
\end{equation}
where $k_{\mathcal{O}}$ is a positive constant, which depends on $\mathcal{O}$ only.
We deduce from (\ref{aucha}) and (\ref{eqLinterp}) for any $r \in (1,2)$ that
\begin{equation}
\int_{\partial \mathcal{O}} |\varphi|^{2-\frac{1}{r}} \dd \sigma \leq
C(\mathcal O,r)\, \| \varphi\|_{L^{1}(\mathcal{O})}^{1-\frac{1}{r}}\,
\|\varphi\|_{W^{1,r}(\mathcal{O})} \qquad \forall \varphi \in W^{1,r}(\mathcal{O}).
\label{auchb}
\end{equation}
We apply (\ref{auchb}) with
%
%
$\mathcal{O}=D_0$, integrate the resulting inequality over
$(0,T)\times \Omega$ and apply H\"older's inequality; this yields
for any $r\in(1,2)$ and for all $\varphi \in L^r(0,T;W^{1,r}(\Omega \times D_0))$ that
\begin{align}\label{auch1}
&\int_0^T \int_{\Omega \times \partial D_0} |\varphi|^{2-\frac{1}{r}} \dd \sigma(\qt) \dx \dt
\leq
C(D_0,r)\,
\|\varphi\|_{L^1(0,T;L^1(\Omega\times D_0))}^{1-\frac{1}{r}}\,
\|\varphi\|_{L^r(0,T;W^{1,r}(\Omega\times D_0))}.
\end{align}
Motivated by the bound \eqref{tau1limit4}, we fix
\[ r=s= \frac{(K+1)d + 2}{(K+1)d + 1} \in (1,2)\]
in \eqref{auch1}.
It follows from (\ref{auch1}), (\ref{tau1limit4}) and (\ref{tau1limit5}) that
\begin{align}\label{auch3}
&\int_0^T \int_{\Omega \times \partial D_0}
|\hpsikaLDp
- \hpsika
|^{2-\frac{1}{s}} \dd \sigma(\qt) \dx \dt
\rightarrow 0,\qquad
\mbox{as $\Delta t \rightarrow 0_+$ $(L\rightarrow \infty)$.}
\end{align}
Since $2-\frac{1}{s}>1$, it follows from \eqref{auch3} that ${\tt T}_{22}$ converges to $0$, as
$\Delta t \rightarrow 0_+$ $(L\rightarrow \infty)$.
We thus deduce that there exists a 
$\Delta t_0$
such that for all $\Delta t  \leq \Delta t_0$, 
we have that $0 < {\tt T}_{22} < \frac{\delta}{6}$.
Finally, by recalling the inequalities \eqref{tau1limit2} and (\ref{T21T22})
and the bounds on ${\tt T_1},\, {\tt T}_{21},\, {\tt T}_{22}$ and ${\tt T}_3$,
it then follows that for each $\delta>0$
there exists a 
$\Delta t_0$
such that for all $\Delta t \leq \Delta t_0$,
we have that
\begin{align*}
&\int_{\Omega_T} |\tautt_1(M \hpsikaLDp)
- \tautt_1(M \hpsika)
| \dx \dt  < \delta.
\end{align*}
Thus we have proved
that
\[ \tautt_1(M \hpsikaLDp) \rightarrow \tautt_1(M \hpsika)\qquad \mbox{strongly in $L^1(\Omega_T)$,\quad
as
$ \Delta t \rightarrow 0_+$ $(L \rightarrow \infty)$}.\]
This, together with (\ref{tautt43bd}), \eqref{tautt43bd-0} and (\ref{eqLinterp}), implies that, as
$\Delta t \rightarrow 0_+$ $(L \rightarrow \infty)$,
%
\begin{align}\label{tau1limit6}
\tautt_1(M \hpsikaLDp) \rightarrow \tautt_1(M \hpsika)\qquad \mbox{strongly in $L^r(\Omega_T)$ for all
$r \in \left[1, \frac{4(d+2)}{3d+4}\right)$.}
\end{align}
We note further that, according to \eqref{bvrhokaLDpL2con} with $\varsigma=2$,
$\vrhokaLDp \rightarrow \vrhoka$ strongly in $L^{\frac{10}{3}}(0,T;L^2(\Omega))$
as $\Delta t \rightarrow 0_+$ $(L \rightarrow \infty)$;
therefore $(\vrhokaLDp)^2 \rightarrow \vrhoka^2$ strongly in $L^{\frac{5}{3}}(0,T;L^1(\Omega))$,
and thus strongly in $L^1(0,T;L^1(\Omega))=L^1(\Omega_T)$, as
$\Delta t \rightarrow 0_+$ $(L \rightarrow \infty)$.
Also, by \eqref{vrhokaLvvbd}, since $\frac{8(d+2)}{3d+4} < \frac{2(d+2)}{d}$ for
$d \in \{2,3\}$, we have
that $\{(\vrhokaLDp)^2\}_{\Delta t>0}$ is a bounded sequence in $L^\frac{4(d+2)}{3d+4}(\Omega_T)$;
consequently from (\ref{tautt43bd-0}) and (\ref{eqLinterp}),
$(\vrhokaLDp)^2 \rightarrow (\vrhoka)^2$ strongly in $L^r(\Omega_T)$
for all $r \in \left[1, \frac{4(d+2)}{3d+4}\right)$.
Combining this with \eqref{tau1limit6} we deduce \eqref{tausconLs}
thanks to \eqref{tautot}.
\end{proof}

We are now ready to pass to the limit with
$\Delta t \rightarrow 0_+$
$(L\rightarrow \infty)$ in (\ref{eqrhocon}--c)
to prove the existence of a weak
solution to the regularized problem (P$_{\kappa,\alpha}$).

\begin{theorem}
\label{5-convfinal}
The triple $(\rhoka,\utka,\hpsika)$,
defined as in Lemmas \ref{rhoLconv} and \ref{convfinal},
is a global weak solution to problem (P$_{\kappa,\alpha}$),
in the sense that
\begin{subequations}
\begin{align}
&\displaystyle\int_{0}^{T}\left\langle \frac{\partial \rhoka}{\partial t}\,,\eta
\right\rangle_{H^1(\Omega)} \dd t
+ \int_0^T \int_\Omega \left( \alpha \,\nabx \rhoka -\rhoka \,\utka \right)
\cdot \nabx \eta
\,\dx \,\dt =0
\nonumber \\
& \hspace{3in}
\qquad
~\hfill \forall \eta \in L^2(0,T;H^1(\Omega)),
\label{eqrhokarep}
\end{align}
with $\rhoka(\cdot,0) = \rho^0(\cdot)$,
\begin{align}
&\displaystyle\int_{0}^{T} \left \langle
\frac{\partial (\rhoka\,\utka)}{\partial t}, \wt \right \rangle_{W^{1,4}_0(\Omega)}
\dt
+ \frac{\alpha}{2} \displaystyle\int_{0}^{T}\!\! \int_\Omega
\nabx \rhoka \cdot \nabx ( \utka \cdot \wt)
\,\dx\,\dt
\nonumber \\
& \qquad \qquad
+
\displaystyle\int_{0}^{T}\!\! \int_\Omega
\left[
\Stt(\utka) - \rhoka\,\utka \otimes \utka
-\pk(\rhoka)\,\Itt
\right]
:\nabxtt \wt
\,\dx\, \dt
\nonumber \\
&\qquad =\int_{0}^T
 \int_{\Omega} \left[ \rhoka\,
\ft \cdot \wt
-\left(\tautt_1 (M\,\hpsika) - \mathfrak{z}\,\vrhoka^2\,\Itt\right)
: \nabxtt
\wt \right] \dx \, \dt
\qquad
\forall \wt \in L^r(0,T;\Wt^{1,4}_0(\Omega)),
\nonumber \\
\label{equtka}
\end{align}
~\vspace{-8mm}

\noindent
with $(\rhoka\,\utka)(\cdot,0) = (\rho^0\ut_0)(\cdot)$ and
$r = \frac{8\Gamma-12}{\Gamma-6}$, $\Gamma \geq 8$, and
\begin{align}
\label{eqhpsika}
&\int_{0}^T \left \langle
M\,\frac{ \partial \hpsika}{\partial t},
\varphi \right \rangle_{H^s(\Omega \times D)} \dt
+
\frac{1}{4\,\lambda}
\,\sum_{i=1}^K
 \,\sum_{j=1}^K A_{ij}
\int_{0}^T \!\!\int_{\Omega \times D}
M\,
 \nabqj \hpsika
\cdot\, \nabqi
\varphi\,
\dq \,\dx \,\dt
\nonumber \\
& \quad 
+ \int_{0}^T \!\!\int_{\Omega \times D} M \left[
\epsilon\,
\nabx \hpsika
- \utka\,\hpsika \right]\cdot\, \nabx
\varphi
\,\dq \,\dx \,\dt
\nonumber \\
&
\quad
- \int_{0}^T \!\!\int_{\Omega \times D} M\,\sum_{i=1}^K
\left[\sigtt(\utka)
\,\qt_i\right]
\hpsika \,\cdot\, \nabqi
\varphi
\,\dq \,\dx\, \dt = 0
\qquad
\forall \varphi \in L^2(0,T;H^s(\Omega \times D)),
\nonumber\\
\end{align}
\end{subequations}
~\vspace{-8mm}

\noindent
with $\hpsika(\cdot,0) = \hpsi_0(\cdot)$  and $s > 1+ \frac{1}{2}\,(K+1)\,d$.


In addition, the weak solution $(\rhoka,\utka,\hpsika)$
satisfies, for
a.a.\ $t' \in (0,T)$, \red{the following energy inequality:}
{\color{black}{
\begin{align}\label{5-eq:energyest}
&\frac{1}{2}\,\displaystyle
\int_{\Omega} \rhoka(t')
\,|\utka(t')|^2 \,\dx
+ \int_{\Omega} \Pk(\rhoka(t'))  \dx
+k\,\int_{\Omega \times D} M\,
\mathcal{F}(\hpsika(t')) \dq \,\dx
\nonumber\\
&\quad
+ \alpha\,\kappa
\int_{0}^{t'}
\left[ \|\nabx(\rhoka^2)\|_{L^2(\Omega)}^2 + \frac{4}{\Gamma}\,\|\nabx
(\rhoka^{\frac{\Gamma}{2}})\|_{L^2(\Omega)}^2 \right]
\dt
+  \mu^S c_0\, \int_0^{t'} \|\utka\|_{H^1(\Omega)}^2 \dt
\nonumber \\
&\quad + k\,\int_0^{t'} \int_{\Omega \times D}
M\,
\left[\frac{a_0}{2\lambda}\,
\left|\nabq \sqrt{\hpsika} \right|^2
+ 2\varepsilon\, \left|\nabx \sqrt{\hpsika} \right|^2
\right]
\,\dq \,\dx\, \dt
\nonumber \\
& \quad + \mathfrak{z}\,\|\vrhoka(t')\|_{L^2(\Omega)}^2
+ 2\, \mathfrak{z}\, \varepsilon\,\int_0^{t'} \|\nabx \vrhoka\|_{L^2(\Omega)}^2 \dt
\nonumber\\
&\leq {\rm e}^{t'}\biggl[
\frac{1}{2}\,\displaystyle
\int_{\Omega} \rho^0\,|\ut_{0}|^2 \dx
+ \int_{\Omega} \Pk(\rho^0)  \dx
+k\,\int_{\Omega \times D} M\, 
\mathcal{F}(\hpsi_0)
\dq \,\dx
\nonumber\\
&\hspace{1in}
+
\mathfrak{z}\,\int_{\Omega} \left(\int_D M \,\hpsi_0\dq\right)^2 \dx
+
\frac{1}{2}  \int_{0}^{t'} \|\ft\|_{L^\infty(\Omega)}^2 \dt \int_\Omega \rho^0 \dx
\biggr]
\nonumber\\
&
\leq C,
\end{align}
}}
where $C \in {\mathbb R}_{>0}$ is independent of $\alpha$ and $\kappa$.
\end{theorem}

\begin{proof}
The limit equation (\ref{eqrhokarep}) has already been established
in Lemma \ref{rhoLconv}, see (\ref{eqrhoka}).

We now pass to the limit $\Delta t \rightarrow 0_+$ $(L \rightarrow \infty)$,
subject to (\ref{LT}), for the subsequence of $\{(\rhokaL^{[\Delta t]},\utkaLDp,
\hpsikaLDp)\}_{\Delta t >0}$ of Lemma \ref{convfinal}
in (\ref{equnconadj}) initially for a fixed test function
$\wt \in \Ct^\infty_0(\Omega_T)$.
We consider first the five terms on the left-hand side of
(\ref{equnconadj}).
On noting (\ref{mtkaLdtwcon},d), (\ref{rhoLxscon}), (\ref{uLwconH1}), (\ref{nabxkaLLrLs})
and (\ref{timden}),
we obtain the first two terms on the left-hand side of
(\ref{equtka}) and the $\rhoka\,\utka \otimes \utka$ term from the first term on the
left-hand side of (\ref{equnconadj}).
The second and third terms on the left-hand side of (\ref{equnconadj})
give rise to the remaining terms on the left-hand side of (\ref{equtka}),
on noting (\ref{uLwconH1}) and (\ref{pkaLwcon}).
The fourth and fifth terms on the left-hand side of (\ref{equnconadj})
converge to zero, on noting (\ref{bada}), (\ref{utkaLL6}), (\ref{kaLLrLs})
and (\ref{timden}).

We now consider the two terms on the right-hand side of (\ref{equnconadj}).
The first term gives rise to the $\ft$ and $\tautt_1$ contributions
on the right-hand side of (\ref{equtka}),
on noting (\ref{rhoLDpcon}), (\ref{fncon}) and (\ref{tausconLs}).
The second term on the right-hand side of (\ref{equnconadj})
converges to the $\varrho_{\kappa,\alpha}^2$ term on the right-hand side of (\ref{equtka}),
on noting (\ref{vrhokaLDpH1con},b) and performing integration by parts.
Therefore, we have obtained (\ref{equtka}) for any $\wt \in \Ct^\infty_0(\Omega_T)$.
The desired result (\ref{equtka}) for any $\wt \in L^r(0,T;\Wt^{1,4}_0(\Omega))$
then follows from the denseness of $\Ct^\infty_0(\Omega_T)$ in $L^r(0,T;\Wt^{1,4}_0(\Omega))$,
(\ref{mtkareg}), (\ref{rhoLxscon},b), (\ref{uLwconH1}), (\ref{frac}), (\ref{mtutkaLDpmcon}),
(\ref{rhokareg}), (\ref{Pdef}), (\ref{inidata}), 
(\ref{tausconLs}), and finally (\ref{vrhokareg}), which with (\ref{eqinterp}) yields, similarly to
(\ref{vrhokaLvbd}) and (\ref{vrhokaLvvbd}),
that $\vrhoka \in L^{4}(0,T;L^{\frac{2d}{d-1}}(\Omega))$.

Similarly, we now pass to the limit $\Delta t \rightarrow 0_+$ $(L \rightarrow \infty)$
for the subsequence
in (\ref{eqpsincon}) initially for a fixed test function
$\varphi \in C([0,T];C^\infty(\overline{\Omega\times D}))$.
The first term of
(\ref{eqpsincon}) converges to the first term of (\ref{eqhpsika}),
on noting (\ref{psidtwcon}).
For the second term of (\ref{eqpsincon}), we note that
\begin{align}
&\int_{0}^T \int_{\Omega \times D} M\,\nabqi \hpsikaLDp \cdot \nabqj \varphi \dq\,\dx\,\dt
\nonumber \\
& \hspace{1in} = 2 \int_{0}^T \int_{\Omega \times D} M\,
\left( \sqrt{\hpsikaLDp} - \sqrt{\hpsika}\right)
\nabqi \sqrt{\hpsikaLDp} \cdot \nabqj \varphi \dq\,\dx\,\dt
\nonumber \\
& \hspace{1.4in} + 2 \int_{0}^T \int_{\Omega \times D} M\,
\sqrt{\hpsika}\,
\nabqi \sqrt{\hpsikaLDp} \cdot \nabqj \varphi \dq\,\dx\,\dt
=: {\tt T}_1 + {\tt T}_2.
\label{diffqiqjcon1}
\end{align}
Next, on noting (\ref{eq:energy-u+psi-final5}) and that
$|\sqrt{c_1}-\sqrt{c_2}| \leq \sqrt{|c_1-c_2|}$ for all $c_1,\, c_2 \in {\mathbb R}_{\geq 0}$,
we have that
\begin{align}
|{\tt T}_1| & \leq C\,\left\|\sqrt{\hpsikaLDp}
- \sqrt{\hpsika}\right\|_{L^2(0,T;L^2_M(\Omega \times D))}
\, \|\nabqj \varphi\|_{L^\infty(0,T;L^\infty(\Omega \times D))} \nonumber \\
& \leq C\,\|\hpsikaLDp - \hpsika\|_{L^1(0,T;L^1_M(\Omega \times D))}^{\frac{1}{2}}
\, \|\nabqj \varphi\|_{L^\infty(0,T;L^\infty(\Omega \times D))},
\label{diffqiqjcon2}
\end{align}
and so (\ref{psisconL1}) yields that ${\tt T}_1$ converges to zero
as $\Delta t \rightarrow 0_+$.
Similarly, as $M^{\frac{1}{2}}\,
\sqrt{\hpsika}\,\nabqj \varphi \in L^2(0,T;\Lt^2(\Omega \times D))$,
it follows from (\ref{psiwconH1xa}) that, as $\Delta t \rightarrow 0_+$,
\begin{align}
{\tt T}_2 \rightarrow 2 \int_{0}^T \int_{\Omega \times D} M\,
\sqrt{\hpsika}\,
\nabqi \sqrt{\hpsika} \cdot \nabqj \varphi \dq\,\dx\,\dt
= \int_{0}^T \int_{\Omega \times D} M\,
\nabqi \hpsika \cdot \nabqj \varphi \dq\,\dx\,\dt.
\label{diffqiqjcon3}
\end{align}
Hence the second term in (\ref{eqpsincon}) converges to the second term in
(\ref{eqhpsika}).
For the fourth term in (\ref{eqpsincon}), we note that
\begin{align}
&\int_{0}^T \!\!\int_{\Omega \times D} M\,
\left[\sigtt(\utkaLDp)
\,\qt_i\right]
\beta^L(\hpsikaLDp) \,\cdot\, \nabqi
\varphi
\,\dq \dx \dt
\nonumber \\
& \hspace{1in} = \int_{0}^T \!\!\int_{\Omega \times D} M\,
\left[\nabxtt \utkaLDp
\,\qt_i\right]
\left( \beta^L(\hpsikaLDp)-\hpsika \right) \,\cdot\, \nabqi
\varphi
\,\dq \dx \dt
\nonumber \\
& \hspace{1.5in}
+ \int_{0}^T \!\!\int_{\Omega \times D} M\,
\left[\nabxtt \utkaLDp
\,\qt_i\right]
\hpsika \,\cdot\, \nabqi
\varphi
\,\dq \dx \dt =: {\tt T}_3 + {\tt T}_4.
\label{diffqiqjcon4}
\end{align}
Next, on noting (\ref{utkaLL6}), (\ref{eqLinterp}), (\ref{betaLa}),
(\ref{vrhokaLDint}) and
(\ref{vrhokareg}),
we have that
\begin{align}
|{\tt T}_3| & \leq C\,\bigg\|\int_D M\,|\beta^L(\hpsikaLDp) - \hpsika|\,\dq \bigg\|_{L^2(0,T;L^2(\Omega))}
\, \|\nabqi \varphi\|_{L^\infty(0,T;L^\infty(\Omega \times D))} \nonumber \\
& \leq C\,\|\beta^L(\hpsikaLDp) - \hpsika\|_{L^2(0,T;L^1_M(\Omega \times D))}^{\frac{2}{5}}
\,\|\vrhokaLDp+
\vrhoka\|_{L^2(0,T;L^6(\Omega))}^{\frac{3}{5}}
\,
\|\nabqi \varphi\|_{L^\infty(0,T;L^\infty(\Omega))},
\nonumber \\
\label{diffqiqjcon5}
\end{align}
and so (\ref{psisconL1beta}) and (\ref{vrhokaLvbd}) \red{(with $\upsilon =6$ and $\vartheta=\frac{d}{3}$, $d=2,3$,)}
 yield that ${\tt T}_3$ converges to zero
as $\Delta t \rightarrow 0_+$.
Similarly, as (\ref{vrhokareg}) yields that $\displaystyle \int_D \,M\,
\hpsika\,\qt_i \otimes \nabqi \varphi \,\dq \in L^2(\Omega_T)$,
it follows from (\ref{uLwconH1}) that, as $\Delta t \rightarrow 0_+$,
\begin{align}
{\tt T}_4 \rightarrow
\int_{0}^T \!\!\int_{\Omega \times D} M\,
\big[(\nabxtt \utka)
\,\qt_i\big]
\hpsika \,\cdot\, \nabqi
\varphi
\,\dq \dx \dt .
\label{diffqiqjcon6}
\end{align}
Hence the last term in (\ref{eqpsincon}) converges to the last term in
(\ref{eqhpsika}).
Similarly to the second and last terms,
the third term in (\ref{eqpsincon}) converges to the third term in
(\ref{eqhpsika}).
Therefore, we have obtained (\ref{eqhpsika}) for any $\varphi \in
 C([0,T];C^\infty(\overline{\Omega\times D}))$.
The desired result (\ref{eqhpsika}), for any \red{$\varphi \in L^2(0,T;H^s(\Omega\times D))$},
then follows from the denseness of \red{the function space}
$C([0,T];C^\infty(\overline{\Omega\times D}))$
in \red{$L^2(0,T;H^s(\Omega\times D))$}, $H^s(\Omega \times D) \hookrightarrow
W^{1,\infty}(\Omega \times D)$, (\ref{hpsika},b), (\ref{vrhokareg}) and (\ref{uLwconH1}).

Next we shall verify the attainment of the respective initial data by $\rhoka \,\utka$ and $\hpsika$.
We have already established that
$\rhoka \,\utka\in C_w([0,T];L^{\frac{2\Gamma}{\Gamma+1}}(\Omega))$,
see (\ref{mtkareg}).
That $\hpsika\in C_{w}([0,T];$ $L^1_M(\Omega \times D))$ follows from
$\mathcal{F}(\hpsika) \in L^\infty(0,T; L^1_M(\Omega
\times D))$ and $\hpsika \in H^1(0,T; M^{-1}(H^s(\Omega \times D))')$ (cf. \eqref{Fisent} and \eqref{hpsikaLDdtbd})
with $s>1 + \frac{1}{2}(K+1)d$,
by Lemma \ref{lemma-strauss}(b) on taking
$\mathfrak{X}:=L^\Phi_M(\Omega \times D)$, the Maxwellian weighted Orlicz space with
Young's function $\Phi(r) = \mathcal{F}(1+|r|)$ (cf. Kufner, John \& Fu\v{c}ik \cite{KJF},
Sec. 3.18.2) whose separable
predual $\mathfrak{E}:=E^\Psi_M(\Omega \times D)$ has Young's function
$\Psi(r) = \exp|r| - |r| - 1$, and $\mathfrak{Y} := M^{-1}(H^s(\Omega \times D))'$
whose predual with respect to the duality pairing $\langle M \cdot , \cdot
\rangle_{H^s(\Omega \times D)}$ is $\mathfrak{F}:=H^s(\Omega \times D)$, with
$s> 1 + \frac{1}{2}(K+1)d$, and noting the embedding $C_{w\ast}([0;T]; L^\Phi_M(\Omega \times D))
\hookrightarrow C_{w}([0,T]; L^1_M(\Omega \times D))$.
The last embedding and that $\mathfrak{F} \hookrightarrow \mathfrak{E}$ are proved by adapting
Def. 3.6.1. and Thm. 3.2.3 in Kufner, John \& Fu\v{c}ik \cite{KJF} to the measure
$M(\qt)\dq \dx$ to show that $L^\infty(\Omega \times D) \hookrightarrow L^\Xi_M(\Omega \times D)$
for any Young's function $\Xi$, and then adapting Theorem 3.17.7 {\em ibid.}
to deduce that $\mathfrak{F} \hookrightarrow L^\infty(\Omega \times D)
\hookrightarrow E^\Psi_M(\Omega \times D)= \mathfrak{E}$.

We are now ready to prove that
$\rhoka \utka$ and $\hpsika$ satisfy the initial conditions
$(\rhoka\utka)(\cdot,0)=(\rho^0\ut_0)(\cdot)$ and
$\hpsika(\cdot,\cdot,0) = \hpsi_0(\cdot,\cdot)$ in the sense of
$C_w([0,T];\Lt^{\frac{2\Gamma}{\Gamma+1}}(\Omega))$ and
$C_w([0,T]; L^1_M(\Omega \times D))$, respectively.
The desired result for $\rhoka \utka$ follows immediately from
(\ref{mtkaLwscon}) and (\ref{ut0conv}).
We now consider $\hpsika$.
According to (\ref{hpsikaLDdtbd}), there exists a $C \in \mathbb{R}_{>0}$, independent
of $\Delta t$ and $L$, such that
\begin{align*}
\left|\int_0^T\int_{\Omega \times D}M\,\frac{\partial \hpsikaLD}{\partial t}
\,\varphi\dq  \dx \dt \right| \leq C\, \|\varphi\|_{L^{2}(0,T;H^s(\Omega\times D))}\qquad
\forall \varphi \in L^{2}(0,T;H^s(\Omega \times D)),
\end{align*}
where $s>1+\frac{1}{2}(K+1)d$. Choosing, in particular
\[ \varphi(\xt,t) = \phi(\xt) \left(1 - \frac{t}{\delta}\right)_+,\qquad \phi \in
H^s(\Omega\times D),\quad
0< \delta < T,\]
integrating by parts with respect to $t$ and using that $\hpsikaLD(\cdot,\cdot,0) =
\hpsi^0(\cdot,\cdot)$, we have
that
\begin{align*}
\left|\frac{1}{\delta}\int_0^\delta \int_{\Omega \times D} M\,\hpsikaLD
\,\phi \dq \dx \dt - \int_{\Omega \times D} M\,\hpsi^0 \,\phi \dq \dx \right|
\leq C\, \delta^{\frac{1}{2}}
\,\|\phi\|_{H^s(\Omega\times D)}
\qquad
\forall \phi \in H^s(\Omega \times D).
\end{align*}
For $\delta \in (0,T)$ and $\phi$ fixed,
we now pass to the limit $\Delta t \rightarrow 0_+$ ($L \rightarrow \infty$)
in this inequality using \eqref{psisconL1} and (\ref{psi0conv})
to deduce that
\begin{align*}
\left|\frac{1}{\delta} \int_0^\delta  \int_{\Omega \times D} M\,\hpsika
\,\phi \dq \dx \dt - \int_{\Omega \times D} M\,\hpsi_0 \,\phi \dq \dx \right|
\leq C \,\delta^{\frac{1}{2}}
\,\|\phi\|_{H^{s}(\Omega \times D)}
\qquad
\forall \phi \in H^s(\Omega\times D),
\end{align*}
where we have recalled that
$H^s(\Omega \times D) \hookrightarrow
W^{1,\infty}(\Omega \times D)$.
Thus, noting
the weak continuity result $\hpsika \in C_w([0,T];L^1_M(\Omega\times D))$
established above, it follows on passing to the limit $\delta \rightarrow 0_+$ that
$$\int_{\Omega \times D} M\,\hpsika(0)\,\phi \dq \dx = \int_{\Omega \times D}
M\,\hpsi_0\,\phi  \dq \dx
\qquad  \forall \phi \in H^s(\Omega \times D).$$
Hence, we have $\hpsika(\cdot,\cdot,0) = \hpsi_0$ in $L^1_M(\Omega \times D)$.

\smallskip

{\color{black}
It remains to prove the inequality (\ref{5-eq:energyest}). For $t' \in (0,T]$ fixed, let
$n=n(t', \Delta t)$ be a positive integer such that $0 \leq (n-1)\Delta t < t' \leq n\Delta t \leq T$.
It follows from (\ref{eq:energy-u+psi-final2}) and (\ref{upm}), on noting that the interval $(0,t']$ is contained in $(0,t_n]$, that
\begin{align}
&\frac{1}{2}\,\displaystyle
\int_{\Omega} \rhokaLDp(t')
\,|\utkaLDp(t')|^2 \,\dx
+ \int_{\Omega} \Pk(\rhokaLDp(t'))  \dx
+k\,\int_{\Omega \times D} M\,
\mathcal{F}(\hpsikaLDp(t')) \dq \,\dx
\nonumber\\
&\qquad
+ \alpha\,\kappa
\int_{0}^{t'}
\left[ \|\nabx[(\rhokaL^{[\Delta t]})^2]\|_{L^2(\Omega)}^2 + \frac{4}{\Gamma}\,\|\nabx
[(\rhokaL^{[\Delta t]})^{\frac{\Gamma}{2}}]\|_{L^2(\Omega)}^2 \right]
\dt
+  \mu^S c_0\, \int_0^{t'} \|\utkaLDp \|_{H^1(\Omega)}^2 \dt
\nonumber \\
&\qquad + k\,\int_0^{t'} \int_{\Omega \times D}
M\,
\left[\frac{a_0}{2\lambda}\,
\left|\nabq \sqrt{\hpsikaLDp} \right|^2
+ 2\varepsilon\, \left|\nabx \sqrt{\hpsikaLDp} \right|^2
\right]
\,\dq \,\dx\, \dt
\nonumber \\
& \qquad + \mathfrak{z}\,\|\vrhokaLDp(t')\|_{L^2(\Omega)}^2
+ 2\mathfrak{z}\,\varepsilon \int_0^{t'}
\|\nabx \vrhokaLDp\|_{L^2(\Omega)}^2  \dt
\nonumber\\
&\leq {\rm e}^{t_n}\biggl[
\frac{1}{2}\,\displaystyle
\int_{\Omega} \rho^0\,|\ut_{0}|^2 \dx
+ \int_{\Omega} \Pk(\rho^0)  \dx
+k\,\int_{\Omega \times D} M\, 
\mathcal{F}(\hpsi_0)
\dq \,\dx
\nonumber\\
&\hspace{1in}
+
\mathfrak{z}\,\int_{\Omega} \left(\int_D M \,\hpsi_0\dq\right)^2 \dx
+
\frac{1}{2}  \int_{0}^{t_n} \|\ft\|_{L^\infty(\Omega)}^2 \dt \int_\Omega \rho^0 \dx
\biggr]
\nonumber\\
&\leq C, \nonumber\\
\label{energyka1}
\end{align}

\vspace{-4mm}

\noindent
where $C \in {\mathbb R}_{>0}$ is independent of $\alpha$ and $\kappa$. Clearly $n=n(t',\Delta t) \geq \frac{t'}{\Delta t} \rightarrow \infty$
as $\Delta t \rightarrow 0_+$. Since $t' \in (t_{n-1},t_n]$ and $t_n - t_{n-1}=\Delta t$, we deduce that as $\Delta t \rightarrow 0_+$
(and, hence, $n=n(t',\Delta t) \rightarrow \infty$) both $t_{n-1}$ and $t_n$ converge to $t'$; hence
\begin{equation}\label{e-f-conv}
{\rm e}^{t_n} \rightarrow {\rm e}^{t'}\quad \mbox{and}\quad
\int_0^{t_n} \|\ft\|^2_{L^\infty(\Omega)}\dt \rightarrow \int_0^{t'}\|\ft\|^2_{L^\infty(\Omega)}\dt,\quad \mbox{as}\quad \Delta t \rightarrow 0_+.
\end{equation}
}

We multiply \eqref{energyka1} by any nonnegative $\eta \in C_0^\infty(0,T)$, integrate over $(0,T)$, and pass to the limit $\Delta t \rightarrow 0_+$ (and $L \rightarrow \infty$) in the resulting inequality.
It then follows from (\ref{mtutkaLDpmcon}), (\ref{Ppka}), (\ref{fatou-app}); weak lower-semicontinuity,
via the weak convergence results (\ref{rhoLwcon},c),
(\ref{uLwconH1}), (\ref{psiwconH1a},b) and (\ref{vrhokaLDpH1con}); and \eqref{e-f-conv},
that we obtain the inequality (\ref{5-eq:energyest}) multiplied by $\eta$
and integrated over $(0,T)$.
The desired result (\ref{5-eq:energyest}) then follows from the well-known variant of du Bois-Reymond's lemma
according to which, if $\phi \in L^1(0,T)$, then
\begin{align}
\int_{0}^T \phi \, \eta \dt \geq 0 \quad \forall \eta \in C^\infty_0(0,T) \mbox{ with }
\eta \geq 0 \mbox{ on } (0,T) \quad \Rightarrow \quad \phi \geq 0 \quad \red{\mbox{ a.e. in $(0,T)$}}.
\label{weaknneg}
\end{align}
\end{proof}

\section{Existence of a solution to (P$_{\kappa}$)}
\label{sec:Pkappa}
\setcounter{equation}{0}

It follows from the bounds on $\vrhoka$ in (\ref{5-eq:energyest}),
similarly to (\ref{vrhokaLvbd}) and (\ref{vrhokaLvvbd}),
that
\begin{align}
\|\vrhoka\|_{L^{\infty}(0,T;L^2(\Omega))}+
\|\vrhoka\|_{L^{\frac{2(d+2)}{d}}(\Omega_T)}
+\|\vrhoka\|_{L^{4}(0,T;L^{\frac{2d}{d-1}}(\Omega))}
\leq C,
\label{vrhokabd}
\end{align}
where throughout this section  $C$ is a generic positive constant, independent of $\alpha$.
Hence, we deduce from (\ref{vrhokabd}), (\ref{Cttrsbd}) and (\ref{5-eq:energyest}),
similarly to (\ref{tautt43bd}),
that
\begin{align}
\|\tautt_1(M\,\hpsika)\|_{L^2(0,T;L^{\frac{4}{3}}(\Omega))}
+\|\tautt_1(M\,\hpsika)\|_{L^{\frac{4(d+2)}{3d+4}}(\Omega_T)}
&\leq C.
\label{tautt1ka}
\end{align}
Similarly to (\ref{hpsikaLDdtbd}), it follows from
(\ref{vrhokabd}) and (\ref{5-eq:energyest}) that
\begin{align}
\left\|M\,\frac{\partial \hpsika}{\partial t}
\right\|_{L^2(0,T;H^s(\Omega\times D)')} \leq C,
\label{hpsikadtbd}
\end{align}
where $s > 1+ \frac{1}{2}(K+1)d$.
We have the following analogue of Lemma \ref{convfinal}.

\begin{lemma}
\label{hpsikconv}
There exist functions
\begin{subequations}
\begin{align}
\label{hpsik}
&\utk \in L^2(0,T;\Ht^1_0(\Omega))\qquad \mbox{and} \qquad  \hpsik \in L^{\upsilon}(0,T;Z_1)\cap
H^1(0,T; M^{-1}(H^{s}(\Omega \times D))'),
\end{align}
where  $\upsilon \in [1,\infty)$ and $s>1+\frac{1}{2}(K+1)d$,
with finite relative entropy and Fisher information,
\begin{align}
\mathcal{F}(\hpsik) \in L^\infty(0,T;L^1_M(\Omega \times D)) \qquad \mbox{and} \qquad
\sqrt{\hpsik} \in L^2(0,T;H^1_M(\Omega \times D)),
\label{Fisentk}
\end{align}
\end{subequations}
and a subsequence of $\{(\rhoka,\,\utka,\,\hpsika)\}_{\alpha > 0}$
such that,
as $\alpha \rightarrow 0_+$,
\begin{align}
\utka &\rightarrow \utk
\qquad \mbox{weakly in } L^{2}(0,T;\Ht_0^1(\Omega)), \label{uwconH1k}
\end{align}
and
\begin{subequations}
\begin{alignat}{2}
M^{\frac{1}{2}}\,\nabx \sqrt{\hpsika} &\rightarrow M^{\frac{1}{2}}\,\nabx \sqrt{\hpsik}
&&\qquad \mbox{weakly in } L^{2}(0,T;\Lt^2(\Omega\times D)), \label{psiwconH1ak}\\
M^{\frac{1}{2}}\,\nabq \sqrt{\hpsika} &\rightarrow M^{\frac{1}{2}}\,\nabq \sqrt{\hpsik}
&&\qquad \mbox{weakly in } L^{2}(0,T;\Lt^2(\Omega\times D)), \label{psiwconH1xak}\\
M\,\frac{\partial \hpsika}{\partial t} &\rightarrow M\,\frac{\partial \hpsik}{\partial t}
&&\qquad \mbox{weakly in } L^{2}(0,T;H^s(\Omega\times D)'), \label{psidtwconk}\\
\hpsika & \rightarrow
\hpsik &&\qquad \mbox{strongly in } L^\upsilon(0,T;L^1_M(\Omega\times D)),\label{psisconL1k}\\
\tautt(M\,\hpsika) & \rightarrow \tautt(M\,\hpsik)
&&\qquad \mbox{strongly in } \Ltt^r(\Omega_T),\label{tausconLsk}
\end{alignat}
where $r \in [1,\frac{4(d+2)}{3d+4})$, and,
for a.a.\ $t \in (0,T)$,
\begin{align}\label{fatou-appk}
&\int_{\Omega \times D} M(\qt)\, \mathcal{F}(\hpsik(\xt,\qt,t))\dq \dx
\leq \liminf_{\alpha \rightarrow 0_+}
\int_{\Omega \times D} M(\qt)\, \mathcal{F}(\hpsika(\xt,\qt,t)) \dq \dx.
\end{align}
\end{subequations}

In addition, we have that
\begin{align}
\vrhok := \int_D M\, \hpsik\, \dq \in L^\infty(0,T;L^2(\Omega)) \cap L^2(0,T;H^1(\Omega)),
\label{vrhokreg}
\end{align}
and, as $\alpha \rightarrow 0_+$,
\begin{subequations}
\begin{alignat}{2}
\vrhoka &\rightarrow \vrhok \qquad
&&\mbox{weakly-$\star$ in } L^\infty(0,T;L^2(\Omega)),
\qquad \mbox{weakly in } L^2(0,T;H^1(\Omega)),
\label{vrhokaH1con} \\
\vrhoka &\rightarrow \vrhok
\qquad &&\mbox{strongly in } L^{\frac{5\varsigma}{3(\varsigma-1)}}(0,T;L^\varsigma(\Omega)),
\label{vrhokaL2con}
\end{alignat}
\end{subequations}
for any $\varsigma \in (1,6)$.
\end{lemma}
\begin{proof} The convergence result (\ref{uwconH1k}) and the first result in (\ref{hpsik})
follow immediately from the bound on $\utka$ in (\ref{5-eq:energyest}).
The remainder of the results follow from the bounds on $\hpsika$ and $\vrhoka$ in (\ref{5-eq:energyest})
in the same way as the results of Lemma \ref{convfinal}.
\end{proof}

We have the following analogue of
Lemmas \ref{rhonabxsclem} and \ref{mtkaLlem}.

\begin{lemma}\label{rhokalem}
Let $\Gamma \geq 8$; then,
there exists a $C \in \mathbb R_{>0}$, independent of $\alpha$, such that
\begin{subequations}
\begin{align}
&\|\rhoka\|_{L^\infty(0,T;L^{\Gamma}(\Omega))}
+ \|\utka\|_{L^2(0,T;H^1(\Omega))}
+
\left\|\sqrt{\rhoka}\,\utka\right\|_{L^\infty(0,T;L^2(\Omega))}
\nonumber \\
& \qquad
+ \|\rhoka\,\utka\|_{L^\infty(0,T;L^{\frac{2\Gamma}{\Gamma+1}}(\Omega))}
+ \|\rhoka\,\utka\|_{L^2(0,T;L^{\frac{6\Gamma}{\Gamma+6}}(\Omega))}
\nonumber \\
&\qquad
+ \|\rhoka\,\utka\|_{L^{\frac{10\Gamma-6}{3(\Gamma+1)}}(\Omega_T)}
+ \left\|\rhoka\,|\utka|^2\right\|_{L^2(0,T;L^{\frac{6\Gamma}{4\Gamma+3}}(\Omega))}
\leq C,
\label{mtkabd} \\
& \sqrt{\alpha}\,\,\|\nabx \rhoka\|_{L^2(\Omega_T)}
+ \alpha\,\|\nabx \rhoka\|_{L^{\frac{10\Gamma-6}{3(\Gamma+1)}}(\Omega_T)}
+ \alpha\,\|(\nabx \rhoka \cdot \nabx) \utka  \|_{L^{\frac{5\Gamma-3}{4\Gamma}}(\Omega_T)}
\nonumber \\
&  \qquad +\alpha\,\|\nabx \rhoka \otimes \utka  \|_{L^{\frac{5\Gamma-3}{4\Gamma}}(\Omega_T)}
\leq C,
\label{nabxrhoka}\\
&
\left\|\frac{\partial\rhoka}{\partial t}\right\|_{L^2(0,T;H^{1}(\Omega)')} \leq C.
\label{dtrhoka}
\end{align}
\end{subequations}

Hence, there exists a function $\rhok \in 
C_w([0,T];L^{\Gamma}_{\geq 0}(\Omega))
\cap H^1(0,T;H^1(\Omega)')$,
and for a further subsequence of the subsequence of Lemma \ref{hpsikconv}, it follows that,
as $\alpha \rightarrow 0_+$,
\begin{subequations}
\begin{alignat}{2}
\rhoka &\rightarrow \rhok
\qquad
&&
\mbox{in } 
C_w([0,T];L^{\Gamma}(\Omega)),
\qquad \mbox{weakly in } H^1(0,T;H^1(\Omega)'),
\label{rhokawcon}
\\
\rhoka &\rightarrow \rhok
\qquad
&&\mbox{strongly in } L^2(0,T;H^1(\Omega)'),
\label{rhokascon}
\\
\alpha\,\nabx \rhoka &\rightarrow  \zerot \qquad
&&\mbox{strongly in } \Lt^{r}(\Omega_T), \qquad r \in [1,\textstyle\frac{10\Gamma-6}{3(\Gamma+1)}),
\label{rhokaxwcon0} \\
\alpha\, (\nabx \rhoka \cdot \nabx) \utka &\rightarrow
\zerot
\quad
&&\mbox{weakly in } \Lt^{\frac{5\Gamma-3}{4\Gamma}}(\Omega_T),
\label{alpha0a}\\
\alpha\, \nabx \rhoka \otimes \utka &\rightarrow
\zerott
\quad
&&\mbox{weakly in } \Ltt^{\frac{5\Gamma-3}{4\Gamma}}(\Omega_T),
\quad \mbox{strongly in } L^1(0,T;\Ltt^{\frac{3}{2}}(\Omega)),
\label{alpha0b}
\end{alignat}
and, for any nonnegative $\eta \in C[0,T],$
\begin{align}\label{Ppk}
&\int_0^T \left(\int_{\Omega} \Pk(\rhok)\,\dx\right) \eta\,\dt
\leq \liminf_{\alpha \rightarrow 0_+}
\int_0^T \left( \int_{\Omega} \Pk(\rhoka)\, \dx\right) \eta\,\dt.
\end{align}
\end{subequations}
\end{lemma}
\begin{proof}
The first three bounds in (\ref{mtkabd}) follow immediately from
(\ref{5-eq:energyest}). The last four bounds in (\ref{mtkabd})
follow, similarly to (\ref{mtkaLLnu},b), from the first two
bounds in (\ref{mtkabd}).

On recalling (\ref{rhokareg}), we choose $\eta= \rhoka$ in (\ref{eqrhoka}) to obtain,
on noting (\ref{mtkabd}), that
\begin{align}
\frac{1}{2}\,\|\rhoka(\cdot,T)\|_{L^2(\Omega)}^2
+ \alpha\,\|\nabx \rhoka\|_{L^2(\Omega_T)}^2
&= \frac{1}{2}\,
\left[\|\rho^0\|_{L^2(\Omega)}^2
- \int_0^T \int_\Omega (\nabx \cdot \utka)\,\rhoka^2 \dx\right]
\nonumber \\
&\leq C\,\left[1+ \|\utka\|_{L^2(0,T;H^1(\Omega))}\,\|\rhoka\|_{L^4(0,T;L^4(\Omega))}^2\right]
\leq C.
\label{alpnabx}
\end{align}
Hence the first bound (\ref{nabxrhoka}).
On noting the sixth bound in (\ref{mtkabd}) and recalling from (\ref{inidata})
that $\partial \Omega \in C^{2,\theta}$, $\theta \in (0,1)$,
and $\rho^0 \in L^\infty(\Omega)$ satisfying (\ref{rho0conv}),
we can now apply the parabolic regularity result,
Lemma 7.38 in Novotn\'{y} \& Stra\v{s}kraba \cite{NovStras} \arxiv{(or Lemma \ref{Le-G-2} in Appendix \ref{sec:App-G})}, to
(\ref{eqrhokarep}) to
obtain that the solution $\rhoka$ satisfies the second bound in (\ref{nabxrhoka}).
The third and fourth bounds in (\ref{nabxrhoka}) follow from the second bounds
in (\ref{mtkabd},b), similarly to (\ref{badex1}).
The bound (\ref{dtrhoka}) follows immediately from (\ref{eqrhokarep}),
the fifth bound in (\ref{mtkabd}) and the first bound in (\ref{nabxrhoka}).

The convergence results (\ref{rhokawcon},b) follow immediately from the first bound in
(\ref{mtkabd},c), 
(\ref{Cwcoma},b) and (\ref{compact1}).
The first two bounds in (\ref{nabxrhoka}) and (\ref{eqLinterp})
yield the desired result (\ref{rhokaxwcon0}).
The convergence result (\ref{alpha0a}) and the first result in (\ref{alpha0b})
follow from the final two bounds in (\ref{nabxrhoka}),
the second bound in (\ref{mtkabd})
and (\ref{rhokaxwcon0}).
The second result in (\ref{alpha0b}) also follows from
the second bound in (\ref{mtkabd}) and
(\ref{rhokaxwcon0}) with $r=2$
on noting that, for all $\eta_1 \in L^2(\Omega_T)$ and $\eta_2 \in L^2(0,T;H^1(\Omega))$,
\begin{align}
\|\eta_1\,\eta_2\|_{L^1(0,T;L^{\frac{3}{2}}(\Omega))}
\leq \|\eta_1\|_{L^2(\Omega_T)}\,\|\eta_2\|_{L^2(0,T;L^6(\Omega))}
\leq C\,\|\eta_1\|_{L^2(\Omega_T)}\,\|\eta_2\|_{L^2(0,T;H^1(\Omega))}.
\label{L1L32}
\end{align}
Finally, the result (\ref{Ppk}) follows, similarly to (\ref{Pkbel}), from (\ref{rhokawcon})
and the convexity of $\Pk$.
\end{proof}

Next, we set
$\upsilon = \frac{10\Gamma -6}{3(\Gamma+1)} \geq \frac{74}{27}$, which appears in Lemma \ref{rhokalem},
and $s'=\frac{5\Gamma-3}{\Gamma-3}\geq 5$ so that (\ref{frac}) holds.
Then, similarly to (\ref{equnconadjbd}),
it follows from (\ref{equtka}), (\ref{frac}), $W^{1,4}(\Omega) \hookrightarrow L^\infty(\Omega)$,
$H^{1}(\Omega) \hookrightarrow L^6(\Omega)$,
(\ref{mtkabd},b), (\ref{vrhokabd}), 
 on noting that $\frac{2d}{d-1} > \frac{8}{3}$,
(\ref{tautt1ka}),
and (\ref{inidata}) that,
for any $\wt \in L^{s'}(0,T;\Wt_0^{1,4}(\Omega))$,
\begin{align}
&\left|\displaystyle\int_{0}^{T} \left\langle
\frac{\partial (\rhoka\,\utka)}{\partial t},  \wt \right\rangle_{W^{1,4}_0(\Omega)}\dt
-\int_{0}^T \int_{\Omega} \pk(\rhoka)\,\nabx \cdot \wt
\dx\,\dt
\right| \nonumber \\
& \quad \leq C\left[ \alpha\,\|\nabx \rhoka\|_{L^{\upsilon}(\Omega_T)} +
\|\rhoka\,\utka\|_{L^{\upsilon}(\Omega_T)} + 1\right]
\|\utka\|_{L^2(0,T;H^1(\Omega))}\,\|\wt\|_{L^{s'}(0,T;W^{1,4}(\Omega))}
\nonumber \\
& \qquad +
C\left[\|\tautt_1(M\,\hpsika)\|_{L^2(0,T;L^{\frac{4}{3}}(\Omega))}\,
+\|\vrhoka\|_{L^4(0,T;L^{\frac{2d}{d-1}}(\Omega))}^2 \right]
\|\wt\|_{L^{2}(0,T;W^{1,4}(\Omega))}
\nonumber \\
&\qquad + \|\rhoka\|_{L^\infty(0,T;L^2(\Omega))}\,
\|\ft\|_{L^2(0,T;L^\infty(\Omega))}\, \|\wt\|_{L^{2}(\Omega_T)}
\nonumber \\
& \quad \leq
C\,\|\wt\|_{L^{s'}(0,T;W^{1,4}(\Omega))}.
\nonumber \\
&
\label{equtkabd}
\end{align}

For $r,\,s \in (1,\infty)$, let
\begin{align}
L^r_0(\Omega)&:= \{ \zeta \in L^r(\Omega) : \displaystyle \int_{\Omega} \zeta \,\dx=0\},
\quad \Et^{r,s}(\Omega):= \{ \wt \in \Lt^r(\Omega) : \nabx \cdot \wt \in L^s(\Omega)\}
\nonumber \\
\mbox{and} \quad \Et^{r,s}_0(\Omega)&:= \{ \wt \in \Et^{r,s}(\Omega) : \wt \cdot \nt = 0
\mbox{ on } \partial \Omega\}.
\label{Etrs}
\end{align}
\red{The equality $\wt \cdot \nt = 0 \mbox{ on } \partial \Omega$ should be understood in the sense of traces of Sobolev functions, with equality in
$W^{1-\frac{\upsilon}{\upsilon'}, \upsilon'}(\partial\Omega)'$, where $\frac{1}{\upsilon} + \frac{1}{\upsilon'}=1$ and $\upsilon=\min\{r,s\}$; cf. Lemma 3.10 in \cite{NovStras}.}

We now introduce the Bogovski\u{\i} operator $\calBt:
L^r_0(\Omega)\rightarrow \Wt^{1,r}_0(\Omega)$, $r \in (1,\infty)$, such that
\begin{align}
\int_{\Omega} \left(\nabx \cdot \calBt(\zeta) - \zeta\right) \eta \,\dx =0 \qquad
\forall \eta \in L^{\frac{r}{r-1}}(\Omega);
\label{Bop}
\end{align}
which satisfies
\begin{subequations}
\begin{alignat}{2}
\|\calBt(\zeta) \|_{W^{1,r}(\Omega)} &\leq C\, \|\zeta\|_{L^r(\Omega)} \qquad &&\forall \zeta \in L^r_0(\Omega),
\label{Bop1}\\
\|\calBt(\nabx\cdot\wt) \|_{L^r(\Omega)} &\leq C\, \|\wt\|_{L^r(\Omega)} \qquad
&&\forall \wt \in \Et^{r,s}_0(\Omega),
\label{Bop2}
\end{alignat}
\end{subequations}
see Lemma 3.17 in Novotn\'{y} \& Stra\v{s}kraba \cite{NovStras} \arxiv{(or Lemma \ref{Le-C-1} in Appendix \ref{sec:App-C})}.

\begin{lemma}\label{BogLem}
There exists a $C(\alpha) \in \mathbb{R}_{>0}$ such that
\begin{align}
\|\nabx\cdot(\rhoka\,\utka)\|_{L^s(\Omega_T)}+
\left\|\frac{\partial \rhoka}{\partial t}\right\|_{L^s(\Omega_T)}
+ \|\Delta_x\,\rhoka\|_{L^s(\Omega_T)} &\leq C(\alpha),
\label{Laprhoka}
\end{align}
where $s=\frac{5\Gamma-3}{4\Gamma}$.
In addition, there exists a $C \in \mathbb{R}_{>0}$, independent of $\alpha$, such that
\begin{align}
\|\rhoka\|_{L^{\Gamma+1}(\Omega_T)} &\leq C.
\label{Bogpkabd}
\end{align}
\end{lemma}
\begin{proof}
We prove (\ref{Laprhoka}), similarly to (\ref{Laprho}).
The first bound in (\ref{Laprhoka}) follows from (\ref{nabxrhoka}).
On noting (\ref{rho0reg}) and the first bound in (\ref{Laprhoka}),
we can now apply the parabolic regularity result, Lemma 7.37 in
Novotn\'{y} \& Stra\v{s}kraba \cite{NovStras} \arxiv{(or Lemma \ref{Le-G-1} in Appendix \ref{sec:App-G})}, to
(\ref{eqrhokarep}) to
obtain that the solution $\rhoka$ satisfies the
last two bounds in (\ref{Laprho}).
It follows from (\ref{Laprhoka}), (\ref{mtkabd},b) and (\ref{eqrhokarep}) that
\begin{align}
&\frac{\partial \rhoka}{\partial t} = \nabx \cdot
\left( \alpha\,\nabx \rhoka - \rhoka\,\utka \right)
\in L^s(\Omega_T)\nonumber \\
\qquad \mbox{with} \qquad  &\left(\alpha\,\nabx \rhoka - \rhoka\,\utka \right)\cdot \nt =0
\quad {\red{\mbox{ on } \partial \Omega \times (0,T),}}
\label{Bop4}
\end{align}
and hence, on recalling (\ref{Etrs}) and that $s=\frac{5\Gamma-3}{4\Gamma} <
\frac{10\Gamma-6}
{3(\Gamma+1)}=r$, we have that
\begin{align}
\alpha\,\nabx \rhoka - \rhoka\,\utka \in L^s(0,T;\Et^{r,s}_0(\Omega)).
\label{Bop5}
\end{align}
Therefore (\ref{Bop4}), (\ref{Bop5}), (\ref{Bop2}) and (\ref{mtkabd},b) yield that
\begin{align}
\bigg\|\calBt\bigg(\frac{\partial \rhoka}{\partial t}\bigg)\bigg\|_{L^s(0,T;L^r(\Omega))} &=
\|\calBt(\nabx \cdot (\alpha\,\nabx \rhoka - \rhoka\,\utka))\|_{L^s(0,T;L^r(\Omega))}
\nonumber \\
&\leq C\,\|\alpha\,\nabx \rhoka - \rhoka\,\utka\|_{L^s(0,T;L^r(\Omega))}
\leq C.
\label{Bop6}
\end{align}

On recalling the notation used in (\ref{PkrhokaLdc}),
then, similarly to (\ref{rhonint}), we obtain, on choosing $\eta =1$ in (\ref{eqrhokarep})
and noting (\ref{rho0conv}),
that, for all $t \in [0,T]$,
\begin{align}
0 \leq \mint \rhoka(t) = \mint \rho^0 \leq \|\rho_0\|_{L^\infty(\Omega)}.
\label{mintrhokabd}
\end{align}
We now choose $\wt = \eta\, \calBt((I-{\mint})\rhoka)$
in (\ref{equtkabd}), where 
$\eta \in C^\infty_0(0,T)$, to obtain, on noting (\ref{pkdef}), (\ref{pgamma}),
(\ref{mintrhokabd}), (\ref{Bop1}), (\ref{Bop6}) and (\ref{mtkabd}) that,
for $s'=\frac{5\Gamma-3}{\Gamma-3}$,
\begin{align}
&\left| \int_0^T \eta \int_\Omega \left( c_p\,\rhoka^{\gamma+1} + \kappa \left( \rhoka^{5}
+ \rhoka^{\Gamma+1} \right) \right) \dx \,\dt \right|
\nonumber \\
&\qquad
\leq C\,\|\eta\|_{L^{\infty}(0,T)}\,\left[
\|\mint \rhoka \|_{L^\infty(0,T)}\,\|\pk(\rhoka)\|_{L^1(\Omega_T)}
+\|\calBt((I-\mint)\rhoka)\|_{L^{s'}(0,T;W^{1,4}(\Omega))}
\right] \nonumber \\
& \qquad \qquad + \left| \int_0^T \int_\Omega \rhoka\,\utka \cdot \left[ \frac{{\rm d} \eta}
{{\rm dt}}\, \calBt((I-\mint)\rhoka)
+ \eta \,\calBt\bigg(\frac{\partial \rhoka}{\partial t}\bigg) \right] \dx \,\dt \right|
\nonumber \\
&\qquad
\leq C\,\|\eta\|_{L^{\infty}(0,T)}\,\left[\|\rhoka\|_{L^\Gamma(\Omega_T)}^{\Gamma}
+\|\rhoka\|_{L^{s'}(0,T;L^4(\Omega))}
\right] \nonumber \\
& \qquad \qquad + \|\rhoka\,\utka\|_{L^\infty(0,T:L^{\frac{2\Gamma}{\Gamma+1}}(\Omega))}\,
\left\|\frac{{\rm d} \eta}
{{\rm dt}}\right\|_{L^1(0,T)}\,
\|\calBt((I-\mint)\rhoka)\|_{L^{\infty}(0,T;L^{\frac{2\Gamma}{\Gamma-1}}(\Omega))}\,
\nonumber \\
& \qquad \qquad + \|\rhoka\,\utka\|_{L^\infty(0,T:L^{\frac{2\Gamma}{\Gamma+1}}(\Omega))}\,
\|\eta\|_{L^{\infty}(0,T)} \,\left\|\calBt\left(\frac{\partial \rhoka}{\partial t}\right)
\right\|_{L^1(0,T;L^{\frac{2\Gamma}{\Gamma-1}}(\Omega))}
\nonumber \\
&\qquad \leq
C\,\left[\|\eta\|_{L^{\infty}(0,T)}+
\left\|\frac{{\rm d} \eta}
{{\rm dt}}\right\|_{L^1(0,T)}\right].
\nonumber \\
\label{Bop3}
\end{align}
We now consider (\ref{Bop3}) with $\eta=\eta_m \in C^\infty_0(0,T)$, $m \in {\mathbb N}$, where
$\eta_m \in [0,1]$ with $\eta_m(t)=1$ for $t \in [\frac{1}{m},T-\frac{1}{m}]$
and $\|\frac{{\rm d}\eta_m}{{\rm d} t}\|_{L^\infty(0,T)} \leq 2m$ yielding
$\|\frac{{\rm d}\eta_m}{{\rm d} t}\|_{L^1(0,T)} \leq 4$.
As $\eta_m \rightarrow 1$ pointwise in $(0,T)$, as $m \rightarrow \infty$,
we obtain the desired result
(\ref{Bogpkabd}).
\end{proof}

We have the following analogue of Lemmas
 \ref{mtkaLDtdlem} and \ref{mtkaLDconlem}.

\begin{lemma}\label{mtkalem}
Let $\Gamma\geq 8$; then,
there exists a $C \in \mathbb R_{>0}$, independent of $\alpha$, such that
\begin{align}
\left\|\frac{\partial(\rhoka\,\utka)}{\partial t}
\right\|_{L^{\frac{\Gamma+1}{\Gamma}}(0,T;W^{1,\Gamma+1}_0(\Omega)')} \leq C.
\label{dtmtka}
\end{align}

Hence, for a further subsequence of the subsequence of Lemma \ref{rhokalem}, it follows that,
as $\alpha \rightarrow 0_+$,
\begin{subequations}
\begin{alignat}{2}
\rhoka\,\utka &\rightarrow \rhok\,\utk
\quad
&&\mbox{weakly in } \Lt^{\frac{10\Gamma-6}{3(\Gamma+1)}}(\Omega_T),
\quad \mbox{weakly in $W^{1,\frac{\Gamma+1}{\Gamma}}(0,T;\Wt^{1,\Gamma+1}_0(\Omega)')$},
\label{mtkawcon}
\\
\rhoka\,\utka &\rightarrow \rhok\,\utk
\quad &&\mbox{in }C_w([0,T];L^{\frac{2\Gamma}{\Gamma+1}}(\Omega)),
\quad \mbox{strongly in } L^2(0,T;\Ht^1(\Omega)'),
\label{mtkascon} \\
\rhoka\,\utka \otimes \utka &\rightarrow \rhok\,\utk \otimes \utk
\quad
&&\mbox{weakly in } L^2(0,T;\Ltt^{\frac{6\Gamma}{4\Gamma+3}}(\Omega)),
\label{mtkautkawcon}
\\
\rhoka &\rightarrow \rhok
\quad
&&\mbox{weakly in } L^{\Gamma+1}(\Omega_T),
\label{rhokaconp}\\
\pk(\rhoka) &\rightarrow \overline{\pk(\rhok)}
\quad
&&\mbox{weakly in } L^{\frac{\Gamma+1}{\Gamma}}(\Omega_T),
\qquad
\label{pkkacon}
\end{alignat}
\end{subequations}
where
$\overline{\pk(\rhok)} 
\in L^{\frac{\Gamma+1}{\Gamma}}_{\geq 0}(\Omega_T)$
remains to be identified.
\end{lemma}
\begin{proof}
We deduce from (\ref{equtkabd}), (\ref{Bogpkabd}) and as $4<s'<\Gamma+1$ that
\begin{align}
&\left|\displaystyle\int_{0}^{T} \left\langle
\frac{\partial (\rhoka\,\utka)}{\partial t},  \wt \right\rangle_{W^{1,\Gamma+1}_0(\Omega)}
\!\!\!\dt
\right| \leq C\,\|\wt\|_{L^{\Gamma+1}(0,T;W^{1,\Gamma+1}(\Omega))}
\qquad \forall \wt \in L^{\Gamma+1}(0,T;\Wt^{1,\Gamma+1}_0(\Omega)),
\label{dtmtkapr}
\end{align}
and hence the desired result (\ref{dtmtka}).

The results (\ref{mtkawcon}--c) follow similarly to (\ref{mtkaLwcon}--e)
from (\ref{mtkabd}), (\ref{dtmtka}), (\ref{Cwcoma},b), (\ref{compact1}), (\ref{rhokascon})
and (\ref{uwconH1k}).
The results (\ref{rhokaconp}--f) follow immediately from (\ref{Bogpkabd}),
(\ref{pkdef}) and (\ref{pgamma}).
\end{proof}

We have the following analogue of Theorem \ref{5-convfinal}.

\begin{lemma}
\label{Pkexistslem}
The triple $(\rhok,\utk,\hpsik)$,
defined as in Lemmas \ref{hpsikconv} and \ref{rhokalem},
satisfies
\begin{subequations}
\begin{align}
&\displaystyle\int_{0}^{T}\left\langle \frac{\partial \rhok}{\partial t}\,,\eta
\right\rangle_{H^1(\Omega)} \dd t
- \int_0^T \int_\Omega \rhok \,\utk
\cdot \nabx \eta
\,\dx \,\dt =0
\qquad
\forall \eta \in L^2(0,T;H^1(\Omega)),
\label{eqrhok}
\end{align}
with $\rhok(\cdot,0) = \rho_0(\cdot)$,
\begin{align}
&\displaystyle\int_{0}^{T} \left \langle
\frac{\partial (\rhok\,\utk)}{\partial t}, \wt \right \rangle_{W^{1,\Gamma+1}_0(\Omega)}
\!\!\!\dt
+
\displaystyle\int_{0}^{T}\!\! \int_\Omega
\left[
\Stt(\utk) - \rhok\,\utk \otimes \utk
-\overline{\pk(\rhok)}\,\Itt
\right]
:\nabxtt \wt
\,\dx\, \dt
\nonumber \\
&\qquad =\int_{0}^T
 \int_{\Omega} \left[ \rhok\,
\ft \cdot \wt
-\left(\tautt_1 (M\,\hpsik) - \mathfrak{z}\,\vrhok^2\,\Itt\right)
: \nabxtt
\wt \right] \dx \, \dt
\nonumber \\
& \hspace{3.3in}
\quad
\forall \wt \in L^{\Gamma+1}(0,T;\Wt^{1,\Gamma+1}_0(\Omega)),
\label{equtk}
\end{align}
with $(\rhok\,\utk)(\cdot,0) = (\rho_0\,\ut_0)(\cdot)$, and
\begin{align}
\label{eqhpsik}
&\int_{0}^T \left \langle
M\,\frac{ \partial \hpsik}{\partial t},
\varphi \right \rangle_{H^s(\Omega \times D)} \dt
+
\frac{1}{4\,\lambda}
\,\sum_{i=1}^K
 \,\sum_{j=1}^K A_{ij}
\int_{0}^T \!\!\int_{\Omega \times D}
M\,
 \nabqj \hpsik
\cdot\, \nabqi
\varphi\,
\dq \,\dx \,\dt
\nonumber \\
& \qquad
+ \int_{0}^T \!\!\int_{\Omega \times D} M \left[
\epsilon\,
\nabx \hpsik
- \utk\,\hpsik \right]\cdot\, \nabx
\varphi
\,\dq \,\dx \,\dt
\nonumber \\
&
\qquad
- \int_{0}^T \!\!\int_{\Omega \times D} M\,\sum_{i=1}^K
\left[\sigtt(\utk)
\,\qt_i\right]
\hpsik \,\cdot\, \nabqi
\varphi
\,\dq \,\dx\, \dt = 0
\quad
\forall \varphi \in L^2(0,T;H^s(\Omega \times D)),
\end{align}
\end{subequations}
with $\hpsik(\cdot,0) = \hpsi_0(\cdot)$ and $s > 1 + \frac{1}{2}(K+1)d$.

In addition, the triple $(\rhok,\utk,\hpsik)$
satisfies, for
a.a.\ $t' \in (0,T)$,
{\color{black}
\begin{align}\label{Pkenergy}
&\frac{1}{2}\,\displaystyle
\int_{\Omega} \rhok(t')
\,|\utk(t')|^2 \,\dx
+ \int_{\Omega} \Pk(\rhok(t'))  \dx
+k\,\int_{\Omega \times D} M\,
\mathcal{F}(\hpsik(t')) \dq \,\dx
\nonumber\\
&\quad
+  \mu^S c_0\, \int_0^{t'} \|\utk\|_{H^1(\Omega)}^2 \dt
+ k\,\int_0^{t'} \int_{\Omega \times D}
M\,
\left[\frac{a_0}{2\lambda}\,
\left|\nabq \sqrt{\hpsik} \right|^2
+ 2\varepsilon\, \left|\nabx \sqrt{\hpsik} \right|^2
\right]
\,\dq \,\dx\, \dt
\nonumber \\
& \quad + \mathfrak{z}\,\|\vrhok(t')\|_{L^2(\Omega)}^2
+ 2\, \mathfrak{z}\, \varepsilon\,\int_0^{t'} \|\nabx \vrhok\|_{L^2(\Omega)}^2 \dt
\nonumber\\
&\leq {\rm e}^{t'}\biggl[
\frac{1}{2}\,\displaystyle
\int_{\Omega} \rho_0\,|\ut_{0}|^2 \dx
+ \int_{\Omega} \Pk(\rho_0)  \dx
+k\,\int_{\Omega \times D} M\, 
\mathcal{F}(\hpsi_0)
\dq \,\dx
\nonumber\\
&\hspace{1in}
+
\mathfrak{z}\,\int_{\Omega} \left(\int_D M \,\hpsi_0\dq\right)^2 \dx
+
\frac{1}{2}  \int_{0}^{t'} \|\ft\|_{L^\infty(\Omega)}^2 \dt \int_\Omega \rho_0 \dx
\biggr]
\nonumber\\
&\leq C,
\end{align}
}
where $C \in {\mathbb R}_{>0}$ is independent of $\kappa$.
\end{lemma}

\begin{proof}
Passing to the limit $\alpha \rightarrow 0_+$ for the subsequence of Lemma \ref{mtkalem}
in (\ref{eqrhokarep}) yields (\ref{eqrhok}) subject to the stated initial condition,
on noting (\ref{rhokawcon},c), (\ref{mtkawcon}) and (\ref{rho0convL2}).

Similarly to the proof of (\ref{equtka}), passing to the limit $\alpha \rightarrow 0_+$
for the subsequence of Lemma \ref{mtkalem}
in (\ref{equtka}) yields (\ref{equtk}) subject to the stated initial condition, on noting
(\ref{uwconH1k}), (\ref{rhokawcon},d,e), (\ref{mtkawcon}--c,e), 
(\ref{tausconLsk}), (\ref{vrhokaL2con}) and (\ref{rho0convL2}).
Similarly to the proof of (\ref{eqhpsika}), passing to the limit $\alpha \rightarrow 0_+$
for the subsequence of Lemma \ref{mtkalem}
in (\ref{eqhpsika}) yields (\ref{eqhpsik}) subject to the stated initial condition, on noting
(\ref{uwconH1k}), (\ref{psiwconH1ak}--d),
(\ref{vrhokreg})
and (\ref{eqinterp}).
Similarly to the proof of (\ref{5-eq:energyest}),
we deduce (\ref{Pkenergy}) from (\ref{5-eq:energyest})
using the results  (\ref{mtkautkawcon}), (\ref{Ppk}), (\ref{psiwconH1ak},b,f),
(\ref{uwconH1k}), (\ref{vrhokaH1con}) and \red{(\ref{rho0convLp})}.
\end{proof}

Finally, to obtain the complete analogue of Theorem \ref{5-convfinal},
we have to identify $\overline{\pk(\rhok)}$, which appears in
(\ref{equtk}) and (\ref{pkkacon}), by establishing
that $\overline{\pk(\rhok)}=\pk(\rhok)$.
Due to the presence of the extra stress term in the momentum equation,
we require a modification of the effective viscous flux compactness result, Proposition 7.36 in
Novotn\'{y} \& Stra\v{s}kraba \cite{NovStras}.
Such results require pseudodifferential operators identified via the Fourier transform
${\mathfrak F}$. We briefly recall the key ideas, and refer to Section 4.4.1
in  \cite{NovStras} for the details.

With
\begin{align}
{\mathfrak S}({\mathbb R}^d) :=
\left\{ \eta \in C^\infty({\mathbb R}^d) : \sup_{\xt \in {\mathbb R}^d}
\left|x_1^{\varsigma_1} \cdots x_d^{\varsigma_d}  \,
\frac{\Dlxn \eta}{\Dlxd}\right|
\leq C(|\varsigmat|,|\lambdat|) \quad \forall \varsigmat,\,\lambdat \in {\mathbb N}^d \right\},
\label{Sfrak}
\end{align}
the space of smooth rapidly decreasing (complex-valued) functions, we introduce
the Fourier transform
${\mathfrak F} : {\mathfrak S}({\mathbb R}^d) \rightarrow {\mathfrak S}({\mathbb R}^d)$,
and its inverse
${\mathfrak F}^{-1} : {\mathfrak S}({\mathbb R}^d) \rightarrow {\mathfrak S}({\mathbb R}^d)$,
defined by
\begin{align}
[{\mathfrak F}(\eta)](\yt) = \frac{1}{(2\pi)^{\frac{d}{2}}}
\int_{{\mathbb R}^d} {\rm e}^{-i\,{\small \xt}\, \cdot\, {\small \yt}}
\,\eta(\xt) \,\dx \qquad \mbox{and} \qquad
\label{FTran}
[{\mathfrak F}^{-1}(\eta)](\xt) = \frac{1}{(2\pi)^{\frac{d}{2}}} \int_{{\mathbb R}^d}
{\rm e}^{i\,{\small \xt}\, \cdot\, {\small \yt}}
\,\eta(\yt) \,\dy.
\end{align}
These are extended to
${\mathfrak F},\, {\mathfrak F}^{-1}:
{\mathfrak S}({\mathbb R}^d)' \rightarrow {\mathfrak S}({\mathbb R}^d)'$,
where ${\mathfrak S}({\mathbb R}^d)'$, the dual of ${\mathfrak S}({\mathbb R}^d)$,
is the space of tempered distributions, via
\begin{align}
\langle {\mathfrak F}(\eta),\xi \rangle_{{\mathfrak S}({\mathbb R}^d)}
= \langle \eta, {\mathfrak F}(\xi)\rangle_{{\mathfrak S}({\mathbb R}^d)}
\quad \mbox{and}
\quad
\langle {\mathfrak F}^{-1}(\eta),\xi \rangle_{{\mathfrak S}({\mathbb R}^d)}
= \langle \eta, {\mathfrak F}^{-1}(\xi)\rangle_{{\mathfrak S}({\mathbb R}^d)}
\quad \forall \xi \in {\mathfrak S}({\mathbb R}^d).
\label{Ftext}
\end{align}

We now introduce the inverse divergence operator
${\mathcal A}_j : {\mathfrak S}({\mathbb R}^d) \rightarrow
{\mathfrak S}({\mathbb R}^d)'$, $j=1,\ldots,d$, such that
\begin{align}
{\mathcal A}_j(\eta)= - {\mathfrak F}^{-1}\left( \frac{i\,y_j}{|\yt|^2}\,
[{\mathfrak F}(\eta)](\yt) \right).
\label{Adef}
\end{align}
It follows from Theorems 1.55 and 1.57 in \cite{NovStras} \arxiv{(or Lemmas \ref{Le-B-1} and \ref{Le-B-2} in Appendix \ref{sec:App-B})}
and Sobolev embedding
that, for $j=1,\ldots d$,
\begin{subequations}
\begin{align}
\|\nabx{\mathcal A}_j(\eta)\|_{L^r({\mathbb R}^d)} &\leq C(r)\,
\|\eta\|_{L^r({\mathbb R}^d)}
\qquad \forall \eta \in {\mathfrak S}({\mathbb R}^d), \quad r\in (1,\infty),
\label{Abd}\\
\|{\mathcal A}_j(\eta)\|_{L^{\frac{dr}{d-r}}({\mathbb R}^d)} &\leq C(r)\,
\|\eta\|_{L^r({\mathbb R}^d)}
\qquad \forall \eta \in {\mathfrak S}({\mathbb R}^d), \quad r\in (1,d).
\label{Abd1}
\end{align}
\end{subequations}
Hence, we deduce from (\ref{Abd},b) that ${\mathcal A}_j$ can be extended
to ${\mathcal A}_j : L^r({\mathbb R}^d) \rightarrow D^{1,r}({\mathbb R}^d)$ for $r\in (1,\infty)$,
$j=1,\ldots,d$, where $D^{1,r}({\mathbb R}^d)$ is a homogeneous Sobolev space;
see Section 1.3.6 in \cite{NovStras} \arxiv{(or Appendix \ref{sec:App-A} here)}.
In addition, \red{by duality},
${\mathcal A}_j$ can be extended
to ${\mathcal A}_j : D^{1,r}({\mathbb R}^d)' \rightarrow D^{1,r}({\mathbb R}^d)$ for $r\in (1,\infty)$,
$j=1,\ldots,d$, see (4.4.4) in \cite{NovStras}.
Moreover, as ${\mathcal A}_j(\eta)$ is real, for a real-valued function $\eta$,
and from the Parseval--Plancherel formula we have, for
all $\eta \in L^r({\mathbb R}^d)$ and
$\xi \in L^{\frac{r}{r-1}}({\mathbb R}^d)$, $r\in (1,\infty)$,
having compact support that
\begin{align}
\int_{\mathbb R^d}
{\mathcal A}_j(\eta)\,\xi \,\dx &=
-\int_{{\mathbb R}^d} \eta\,{\mathcal A}_j(\xi)\,\dx,
\qquad j= 1, \ldots,d.
\label{Aint}
\end{align}
Finally, we introduce the so-called Riesz operator ${\mathcal R}_{kj}: L^r({\mathbb R}^d)
\rightarrow L^r({\mathbb R}^d)$, $r \in (1,\infty)$,
defined by
\begin{align}
{\mathcal R}_{kj}(\eta) = \frac{\partial}{\partial x_k} {\mathcal A}_j(\eta),
\qquad j,\,k = 1,\ldots,d.
\label{Rjkdef}
\end{align}
We note for
all $\eta \in L^r({\mathbb R}^d)$ and
$\xi \in L^{\frac{r}{r-1}}({\mathbb R}^d)$, $r\in (1,\infty)$, that
\begin{subequations}
\begin{align}
\sum_{j=1}^d {\mathcal R}_{jj}(\eta)&=
\sum_{j=1}^d \frac{\partial}{\partial x_j} {\mathcal A}_j(\eta)=\eta,
\label{Rjjiden}\\
{\mathcal R}_{kj}(\eta)={\mathcal R}_{jk}(\eta)\qquad
\mbox{and} \qquad
&\int_{{\mathbb R}^d} {\mathcal R}_{jk}(\eta)\,\xi \,\dx =
\int_{{\mathbb R}^d} \eta\,{\mathcal R}_{jk}(\xi)\,\dx,
\qquad j,\,k = 1, \ldots,d.
\label{Rint}
\end{align}
\end{subequations}
Below we \red{use} the notation $\undertilde{\mathcal A}(\cdot)$ and $\underdtilde{\mathcal R}(\cdot)$
with components ${\mathcal A}_i(\cdot)$ and ${\mathcal R}_{ij}(\cdot)$,
$ i,j=1,\ldots,d$, respectively. \red{We shall adopt the convention that whenever any of these
operators is applied to a function or a distribution that has been defined on $\Omega$ only, it is tacitly understood that the function
or distribution in
question has been extended by $0$ from $\Omega$ to the whole of $\mathbb{R}^d$.}

We now have the following modification of Proposition 7.36 in
Novotn\'{y} \& Stra\v{s}kraba \cite{NovStras},
which is adequate for our purposes.

\begin{lemma}\label{Prop736}
Given $\{(g_n,\ut_n,\mt_n,p_n,\tautt_n,f_n,\Ft_n)\}_{n \in {\mathbb N}}$,
we assume for any $\zeta \in C^\infty_0(\Omega)$ that, as $n \rightarrow \infty$,
\begin{subequations}
\begin{alignat}{2}
g_n &\rightarrow g \qquad && \mbox{in } C_w([0,T];L^q(\Omega)), \qquad \quad
 \mbox{weakly $($-$\star)$ in }L^\omega(\Omega_T),
\label{gncon}\\
\ut_n &\rightarrow \ut \qquad && \mbox{weakly in }L^2(0,T;\Ht^1_0(\Omega)),
\label{utncon}\\
\mt_n &\rightarrow \mt \qquad && \mbox{in } C_w([0,T];\Lt^z(\Omega)),
\label{mtncon} \\
p_n &\rightarrow p \qquad && \mbox{weakly in }L^r(\Omega_T),
\label{pncon}\\
\tautt_n &\rightarrow \tautt \qquad && \mbox{strongly in }L^1(0,T;\Ltt^{\frac{q}{q-1}}(\Omega)),
\label{tauncon}\\
f_n &\rightarrow f \qquad && \mbox{weakly in }L^2(0,T;H^1(\Omega)'),
\label{fncona}\\
\undertilde{\mathcal A}(\zeta\,f_n)
&\rightarrow \undertilde{\mathcal A}(\zeta\,f)\qquad
&&\mbox{strongly in }L^2(0,T;\Lt^{\frac{z}{z-1}}(\Omega)),
\label{fnconb}\\
\Ft_n &\rightarrow \Ft \qquad && \mbox{weakly in }\Lt^s(\Omega_T),
\label{Ftcon1} 
\end{alignat}
\end{subequations}
where $q \in (d,\infty)$, $r,\,s \in (1,\infty)$, $\omega \in \big[\!\max\{2,\frac{r}{r-1}\},\infty\big]$
and $z \in \big(\frac{6q}{5q-6},\infty\big)$.

In addition, suppose that
\begin{subequations}
\begin{align}
&\frac{\partial g_n}{\partial t} + \nabx \cdot (\ut_n\,g_n) = f_n \qquad
\mbox{in }C^\infty_0(\Omega_T)',
\label{gnpde} \\
&\frac{\partial \mt_n}{\partial t}+ \nabx \cdot (\mt_n \otimes \ut_n) - \mu\,\Delta_x
\,\ut_n - (\mu+ \lambda)\, \nabx\,(\nabx \cdot \ut_n) + \nabx\, p_n
\nonumber \\
&\hspace{2.7in}
= \Ft_n 
+ \nabx \cdot \tautt_n \qquad
\mbox{in }\Ct^\infty_0(\Omega_T)'.
\label{mnpde}
\end{align}
\end{subequations}
Then it follows 
that, for any $\zeta \in C^\infty_0(\Omega)$ and $\eta \in C^\infty_0(0,T)$,
\begin{align}
&\lim_{n \rightarrow \infty}
\int_0^T \!\! \!\eta \left(\int_{\Omega} \zeta \,g_n\,
[p_n-(2\mu+\lambda)\, \nabx\cdot \ut_n]
\dx \right) \dt
=
\int_0^T \!\! \!\eta \left(\int_{\Omega} \zeta \,g\,
[p-(2\mu+\lambda)\, \nabx\cdot \ut]
\dx \right) \dt.
\label{EFVd}
\end{align}
\end{lemma}
\begin{proof}
We adapt the proof of Proposition 7.36 in \cite{NovStras}, by just pointing out the key differences.
As $q>d$, then $q^\star$, the Sobolev conjugate of $q$ in the notation (1.3.64) of \cite{NovStras},
is such that $q^\star=\infty$.
Hence our restrictions on $r,s,\omega$ and $z$ satisfy the restrictions of Proposition 7.36 in
\cite{NovStras}.
With any $\widetilde{\zeta} \in C^\infty_0(\Omega)$, it follows from (\ref{gnpde}) and
properties (\ref{Aint})--(\ref{Rjjiden},b) of ${\mathcal A}_j$ and ${\mathcal R}_{kj}$  that,
for $i=1,\ldots,d$,
\begin{align}
&\frac{\partial}{\partial t}{\mathcal A}_i(\widetilde{\zeta}\,g_n)
+ \sum_{j=1}^d {\mathcal R}_{ij}(\widetilde{\zeta}\,g_n\,u^j_n)
= {\mathcal A}_i(\widetilde{\zeta}\,f_n)
+ 
{\mathcal A}_i(g_n\,
\ut_n\cdot \nabx\,\widetilde{\zeta})
\qquad
\mbox{in }C^\infty_0(\Omega_T)',
\qquad
\label{Agnpde}
\end{align}
where we adopt the notation $u^j_n$ for the $j^{\rm th}$ component of $\ut_n$.
With any $\zeta,\,\widetilde{\zeta} \in C^\infty_0(\Omega)$ and $\eta \in C^\infty_0(0,T)$,
we now consider $\eta\,\zeta\,\undertilde{\mathcal A}(\widetilde{\zeta}\,g_n)$ as a test function
for (\ref{mnpde}).
It follows from (\ref{gncon}) and (\ref{Abd},b) that
$\undertilde{\mathcal A}(\widetilde{\zeta}\,g_n) \in L^\infty(0,T;\Wt^{1,q}(\Omega))\cap
L^\omega(0,T;\Wt^{1,\omega}(\Omega))$, \red{and hence}
$\undertilde{\mathcal A}(\widetilde{\zeta}\,g_n) \in \Lt^\infty(\Omega_T)$
as $q>d$.
Similarly to (\ref{mtkabd}), $$g_n\,\ut_n \in L^2(0,T;\Lt^{\frac{6q}{q+6}}(\Omega))$$
and $$\mt_n\otimes\ut_n \in
L^2(0,T;\Ltt^{\frac{6z}{z+6}}(\Omega)).$$
As $z \in (\frac{6q}{5q-6},\infty)$,
and therefore $\frac{z}{z-1} \in (1,\frac{6q}{q+6})$
and $\frac{6z}{z+6} \in (\frac{q}{q-1},6)$, it follows from (\ref{Agnpde}),
(\ref{Abd},b), (\ref{Rjkdef}) and (\ref{fnconb}) that
$$\frac{\partial}{\partial t}{\undertilde{\mathcal A}}(\widetilde{\zeta}\,g_n) \in
L^2(0,T;\Lt^{\frac{z}{z-1}}(\Omega)).$$
Noting the above and (\ref{utncon}--e,h), we see that
$\eta\,\zeta\,\undertilde{\mathcal A}(\widetilde{\zeta}\,g_n)$ is a valid test function
for (\ref{mnpde}),
and we obtain, on using integration by parts several times
and properties  (\ref{Aint})--(\ref{Rjjiden},b) of ${\mathcal A}_j$ and ${\mathcal R}_{kj}$, that
\begin{align}
&\int_0^T \!\! \!\eta \left(\int_{\Omega} \zeta \,\widetilde{\zeta}\,g_n\,
[p_n-(2\mu+\lambda)\, \nabx\cdot \ut_n]
\dx \right) \dt
\nonumber \\
&\;=
\mu \int_0^T \!\! \!\eta \left(\int_{\Omega}
\left(\nabxtt \ut_n : \undertilde{\mathcal A}(\widetilde{\zeta}\,g_n)
\otimes \nabx\,\zeta
- \ut_n \otimes  \nabx\,\zeta: \underdtilde{\mathcal R}(\widetilde{\zeta}\,g_n)
+ \widetilde{\zeta}\,g_n\,
\ut_n \cdot  \nabx\,\zeta
\right)\dx\right)\dt
\nonumber \\
& \quad + \int_0^T \!\! \!\eta \left(\int_{\Omega} \left(
\left(\tautt_n - \mt_n \otimes \ut_n\right) :
\undertilde{\mathcal A}(\widetilde{\zeta}\,g_n) \otimes \nabx \,\zeta
-[p_n-(\mu+\lambda)\,\nabx\cdot \ut_n]\,
\undertilde{\mathcal A}(\widetilde{\zeta}\,g_n)\cdot\nabx \,\zeta
\right)\dx \right)\dt
\nonumber \\
& \quad +
\int_0^T \!\! \!\eta \left(\int_{\Omega} \zeta\,\left(
\left(\tautt_n - \mt_n \otimes \ut_n\right) : \underdtilde{\mathcal R}(\widetilde{\zeta}\,g_n)-
\Ft_n
\cdot \undertilde{{\mathcal A}}(\widetilde{\zeta}\,g_n)
\right)\,
\dx \right)\dt
\nonumber \\
& \quad -\int_0^T \!\! \frac{{\rm d}\eta}{{\rm d}t} \left(\int_{\Omega} \zeta\,
\mt_n \cdot \undertilde{\mathcal A}(\widetilde{\zeta}\,g_n) \,\dx\right)\dt
-\int_0^T \!\! \!\eta \left(\int_{\Omega} \zeta\,
\mt_n \cdot \left[ \undertilde{\mathcal A}(\widetilde{\zeta}\,f_n)
+ \undertilde{\mathcal A}(g_n\,\ut_n\cdot \nabx \,\widetilde{\zeta})
\right]
\dx\right)\dt
\nonumber\\
&\quad + \int_0^T \!\! \!\eta
\left(\int_{\Omega} \widetilde{\zeta}\,
\sum_{i=1}^d \sum_{j=1}^d
g_n \,u_n^j\,{\mathcal R}_{ij}(\zeta\,m_n^i)
\dx \right)\dt.
\nonumber \\
& \label{EVFa}
\end{align}
The equation (\ref{EVFa}) is exactly the same as (7.5.12) in \cite{NovStras},
except for the extra $\tautt_n$ terms 
and the change of notation.

We will just concentrate on the terms involving $\tautt_n$, 
as the other terms are dealt with as in \cite{NovStras}.
It follows from (\ref{gncon}), (\ref{Abd},b),
see (7.5.18)--(7.5.20) in \cite{NovStras} for the details, that
\begin{subequations}
\begin{alignat}{3}
\undertilde{\mathcal A}(\widetilde{\zeta}\,g_n) &\rightarrow
\undertilde{\mathcal A}(\widetilde{\zeta}\,g)\quad
&&\mbox{weakly in }L^{\infty}(0,T;\Wt^{1,q}(\Omega)),\quad &&\mbox{strongly in }
\Lt^{\upsilon}(\Omega_T),
\label{Agn}\\
\underdtilde{\mathcal R}(\widetilde{\zeta}\,g_n) &\rightarrow
\underdtilde{\mathcal R}(\widetilde{\zeta}\,g)\quad
&&\mbox{weakly in }L^{\infty}(0,T;\Ltt^{q}(\Omega)),\quad &&\mbox{strongly in }
L^{\upsilon}(0,T;\Htt^{-1}(\Omega)),
\label{Rgn}
\end{alignat}
\end{subequations}
where $\upsilon \in [1,\infty)$. 
It follows from (\ref{tauncon}) and (\ref{Agn},b) that, as $n \rightarrow \infty$,
\begin{align}
&\int_0^T \!\! \!\eta \left(\int_{\Omega} \tautt_n : \left[
\undertilde{\mathcal A}(\widetilde{\zeta}\,g_n) \otimes \nabx \,\zeta
+\underdtilde{\mathcal R}(\widetilde{\zeta}\,g_n)\,\zeta
\right]
\dx \right) \dt \nonumber \\
& \hspace{1in} \rightarrow
\int_0^T \!\! \!\eta \left(\int_{\Omega} \tautt : \left[
\undertilde{\mathcal A}(\widetilde{\zeta}\,g) \otimes \nabx \,\zeta
+\underdtilde{\mathcal R}(\widetilde{\zeta}\,g)\,\zeta
\right]
\dx \right) \dt.
\label{tauncona}
\end{align}
Combining (\ref{tauncona}) with the convergence, as $n \rightarrow \infty$,
of other terms in (\ref{EVFa}) as in the proof of Proposition 7.36 in \cite{NovStras},
which involves the use of the crucial Commutator lemma
(Lemma 4.25 in \cite{NovStras}, \arxiv{or Lemma \ref{Le-D-3} in Appendix \ref{sec:App-D}}),
we obtain that
\begin{align}
&\lim_{n \rightarrow \infty}
\int_0^T \!\! \!\eta \left(\int_{\Omega} \zeta \,\widetilde{\zeta}\,g_n\,
[p_n-(2\mu+\lambda)\, \nabx\cdot \ut_n]
\dx \right) \dt
\nonumber \\
&\;=
\mu \int_0^T \!\! \!\eta \left(\int_{\Omega}
\left(\nabxtt \ut : \undertilde{\mathcal A}(\widetilde{\zeta}\,g)
\otimes \nabx\,\zeta
- \ut \otimes  \nabx\,\zeta: \underdtilde{\mathcal R}(\widetilde{\zeta}\,g)
+ \widetilde{\zeta}\,g\,
\ut \cdot  \nabx\,\zeta
\right)\dx\right)\dt
\nonumber \\
& \quad + \int_0^T \!\! \!\eta \left(\int_{\Omega} \left(
\left(\tautt - \mt \otimes \ut\right) :
\undertilde{\mathcal A}(\widetilde{\zeta}\,g) \otimes \nabx \,\zeta
-[p-(\mu+\lambda)\,\nabx\cdot \ut]\,
\undertilde{\mathcal A}(\widetilde{\zeta}\,g)\cdot\nabx \,\zeta
\right)\dx \right)\dt
\nonumber \\
& \quad +
\int_0^T \!\! \!\eta \left(\int_{\Omega} \zeta\,\left(
\left(\tautt - \mt \otimes \ut\right) : \underdtilde{\mathcal R}(\widetilde{\zeta}\,g)-
\Ft
\cdot \undertilde{{\mathcal A}}(\widetilde{\zeta}\,g)
\right)\,
\dx \right)\dt
\nonumber \\
& \quad -\int_0^T \!\! \frac{{\rm d}\eta}{{\rm d}t} \left(\int_{\Omega} \zeta\,
\mt \cdot \undertilde{\mathcal A}(\widetilde{\zeta}\,g) \,\dx\right)\dt
-\int_0^T \!\! \!\eta \left(\int_{\Omega} \zeta\,
\mt \cdot \left[ \undertilde{\mathcal A}(\widetilde{\zeta}\,f)
+ \undertilde{\mathcal A}(g\,\ut\cdot \nabx \,\widetilde{\zeta})
\right]
\dx\right)\dt
\nonumber\\
&\quad + \int_0^T \!\! \!\eta
\left(\int_{\Omega} \widetilde{\zeta}\,
\sum_{i=1}^d \sum_{j=1}^d
g \,u^j\,{\mathcal R}_{ij}(\zeta\,m^i)
\dx \right)\dt,
\label{EVFb}
\end{align}
which is exactly the same as (7.5.25) in \cite{NovStras},
except for the extra $\tautt$ terms and the change of notation.

In addition, the equations (\ref{gnpde},b) are exactly the same as in (7.5.7)--(7.5.8)
in \cite{NovStras} except for the extra $\tautt_n$ term.
One can use (\ref{gncon}--h) to pass to the limit $n \rightarrow \infty$
in (\ref{gnpde},b) to obtain
\begin{subequations}
\begin{alignat}{2}
\frac{\partial g}{\partial t} + \nabx \cdot (\ut\,g) &= f \quad
&&\mbox{in }C^\infty_0(\Omega_T)',
\label{gpde} \\
\frac{\partial \mt}{\partial t}+ \nabx \cdot (\mt \otimes \ut) - \mu\,\Delta_x
\,\ut - (\mu+ \lambda)\, \nabx\,(\nabx \cdot \ut) + \nabx\, p
&= \Ft
+ \nabx \cdot \tautt\quad
&&\mbox{in }\Ct^\infty_0(\Omega_T)',
\nonumber \\
&
\label{mpde}
\end{alignat}
\end{subequations}
see \cite{NovStras} for details. Clearly, the $\tautt_n$ term in (\ref{mnpde})
is easily dealt with using (\ref{tauncon}).
Similarly to (\ref{Agnpde}), we deduce that $\eta\,\zeta\,\undertilde{\mathcal{A}}(\widetilde{\zeta}
\,g)$ is a valid test function for (\ref{mpde}),
for any $\zeta,\,\widetilde{\zeta} \in C^\infty_0(\Omega)$ and $\eta \in C^\infty_0(0,T)$,
and we obtain (\ref{EVFa}) without the subscript $n$.
Combining this with (\ref{EVFb}), we deduce that,
for any $\zeta,\,\widetilde{\zeta} \in C^\infty_0(\Omega)$ and $\eta \in C^\infty_0(0,T)$,
\begin{align}
&\lim_{n \rightarrow \infty}
\int_0^T \!\! \!\eta \left(\int_{\Omega} \zeta \,\widetilde{\zeta}\,g_n\,
[p_n-(2\mu+\lambda)\, \nabx\cdot \ut_n]
\dx \right) \dt
=
\int_0^T \!\! \!\eta \left(\int_{\Omega} \zeta \,\widetilde{\zeta}\,g\,
[p-(2\mu+\lambda)\, \nabx\cdot \ut]
\dx \right) \dt.
\label{EFVc}
\end{align}
Hence \red{we arrive at (\ref{EFVd}) by taking $\tilde{\zeta} \in C^\infty_0(\Omega)$ such that $\tilde{\zeta} \equiv 1$ on the support
of $\zeta$}.
\end{proof}

We need also the following variation of Lemma \ref{Prop736}
for later use in Section \ref{sec:P}.
\begin{corollary}\label{Cor736}
The results of Lemma \ref{Prop736} hold with the assumptions (\ref{fncona},g) replaced by
\begin{alignat}{2}
f_n &\rightarrow f 
\qquad
&&\mbox{weakly in } 
L^2(\Omega_T), \qquad
\mbox{as }n \rightarrow \infty.
\label{fnconbcor}
\end{alignat}
\end{corollary}
\begin{proof}
One can still pass to the limit $n \rightarrow \infty$
in (\ref{gnpde}) to obtain (\ref{gpde}) using (\ref{fnconbcor}) in place of
(\ref{fncona},g). We deduce from (\ref{fnconbcor}), (\ref{Abd1}) and (\ref{Aint}) that
\begin{align}
\undertilde{\mathcal A}(\zeta\,f_n)
&\rightarrow \undertilde{\mathcal A}(\zeta\,f)
\qquad
\mbox{weakly in } L^2(0,T;\Lt^6(\Omega)),
\qquad
\mbox{as }n \rightarrow \infty.
\label{Afnconbcor}
\end{align}
As $\frac{z}{z-1} < \frac{6q}{q+6} < 6$,
(\ref{Afnconbcor}) ensures that one can still conclude from (\ref{Agnpde})
that $\eta\,\zeta\,\undertilde{\mathcal A}(\widetilde{\zeta}\,g_n)$ is a valid test function
for (\ref{mnpde}).
Similarly, one can deduce that
$\eta\,\zeta\,\undertilde{\mathcal A}(\widetilde{\zeta}\,g)$ is a valid test function
for (\ref{mpde}).
The only other place \red{where} (\ref{fncona},g) are used in the proof of Lemma
\ref{Prop736} is in dealing with the term involving $f_n$ on the right-hand side of (\ref{EVFa});
that is, the term
\begin{align}
-\int_0^T \!\! \!\eta \left(\int_{\Omega} \zeta\,
\mt_n \cdot \undertilde{\mathcal A}(\widetilde{\zeta}\,f_n)
\dx\right)\dt
=
\int_0^T \!\! \!\eta \left(\int_{\Omega}
\widetilde{\zeta}\,f_n \, \sum_{i=1}^d
{\mathcal A}_i(\zeta\,
m_n^i)
\dx\right)\dt,
\label{EFVacor}
\end{align}
where we have noted (\ref{Aint}).
Similarly to (\ref{Agn}), it follows from (\ref{mtncon}), (\ref{Abd},b) and Sobolev embedding,
as $z > \frac{6}{5}$, that
\begin{align}
\undertilde{\mathcal A}(\widetilde{\zeta}\,m_n) &\rightarrow
\undertilde{\mathcal A}(\widetilde{\zeta}\,m)\quad
&&\mbox{weakly in }L^{\infty}(0,T;\Wt^{1,z}(\Omega)),\quad &&\mbox{strongly in }
L^{\upsilon}(0,T;\Lt^3(\Omega)),
\label{Amncor}
\end{align}
where $\upsilon \in [1,\infty)$.
Therefore, (\ref{Amncor}), (\ref{fnconbcor}) and (\ref{Aint}) imply
that we can pass to the limit $n \rightarrow \infty$ in (\ref{EFVacor}) to obtain
the term involving $f$ on the right-hand side of (\ref{EVFb}).
\end{proof}

In order to identify $\overline{\pk(\rhok)}$ in (\ref{equtk}) and (\ref{pkkacon}),
we now apply Lemma \ref{Prop736} with
(\ref{gnpde},b) being (\ref{eqrhokarep},b)
so that $\mu=\frac{\mu^S}{2}$ and $\lambda= \mu^B-\frac{\mu^S}{d}$,
$g_n=\rhoka$, $\ut_n=\utka$, $\mt_n = \rhoka\,\utka$, $p_n=\pk(\rhoka)$,
$\tautt_n = \tautt(M\,\hpsika)+\frac{\alpha}{2}\,( \utka \otimes \nabx\, \rhoka)$,
$f_n = \alpha\,\Delta_x\,\rhoka$ and
$\Ft_n = \rhoka\,\ft - \frac{\alpha}{2}\,(\nabx\, \rhoka \cdot \nabx)\,\utka$.
With $\{(\rhoka,\utka,\hpsika)\}_{\alpha>0}$ being the subsequence (not indicated)
of Lemma \ref{mtkalem}, we have that
(\ref{gncon}--d) hold with
$g=\rhok$, $\ut=\utk$, $\mt = \rhok\,\utk$ and $p=\overline{\pk(\rhok)}$,
and $q=\Gamma$, $\omega=\Gamma+1$, $z=\frac{2\Gamma}{\Gamma+1}$ and
$r=\frac{\Gamma+1}{\Gamma}$ on recalling
(\ref{rhokawcon}), (\ref{mtkascon},d,e) and  (\ref{uwconH1k}).
We note that $\omega=\Gamma +1 = \frac{r}{r-1} > 2$
and, as $\Gamma \geq 8$, $z=\frac{2\Gamma}{\Gamma+1} \geq \frac{16}{9} > \frac{24}{17}
\geq \frac{6\Gamma}{5\Gamma-6}=\frac{6q}{5q-6}$. Hence, the constraints on
$q,r,\omega$ and $z$ hold.
The results (\ref{tauncon},h) hold
with $\tautt = \tautt(M\,\hpsik)$, $\Ft= \rhok\,\ft$ and $s=\frac{5\Gamma-3}{4\Gamma}$,
on recalling  (\ref{tausconLsk}),
(\ref{alpha0a},e) and (\ref{rhokaconp}),
and  noting that $\frac{4(d+2)}{3d+4} \geq \frac{20}{13}
> \frac{8}{7} \geq \frac{q}{q-1}$.
Finally, the results (\ref{fncona},g) hold with $f=0$
on recalling (\ref{rhokaxwcon0}) and the properties of
$\undertilde{{\mathcal A}}$ and $\underdtilde{{\mathcal R}}$,
on noting that $\frac{z}{z-1} = \frac{2\Gamma}{\Gamma-1} < \frac{10\Gamma-6}{3(\Gamma+1)}$,
see (7.9.21) in \cite{NovStras} for details.
Hence, we obtain from (\ref{EFVd}) for the subsequence of Lemma
\ref{mtkalem} that, for all $\zeta \in C^\infty_0(\Omega)$ and $\eta \in C^\infty_0(0,T)$,
\begin{align}
&\lim_{\alpha \rightarrow 0_+}
\int_0^T \!\! \!\eta \left(\int_{\Omega} \zeta \,\rhoka\,
[\pk(\rhoka)-\mu^\star \, \nabx\cdot \utka]
\dx \right) \dt
\nonumber \\
& \hspace{2.5in}
=
\int_0^T \!\! \!\eta \left(\int_{\Omega} \zeta \,\rhok\,
[\overline{\pk(\rhok)}-\mu^\star\, \nabx\cdot \utk]
\dx \right) \dt,
\label{EVF2}
\end{align}
where $\mu^\star:=\frac{(d-1)}{d}\,\mu^S+\mu^B$.
The first two bounds in (\ref{mtkabd}) yield that
\begin{align}
\|\rhoka\,\nabx\cdot \utka\|_{L^2(0,T;L^{\frac{2\Gamma}{\Gamma+2}}(\Omega))} \leq C,
\label{rhokautkabd}
\end{align}
and hence there exists a $\overline{\rhok\,\nabx\cdot \utk}
\in L^2(0,T;L^{\frac{2\Gamma}{\Gamma+2}}(\Omega))$
such that for a subsequence (not indicated)
\begin{alignat}{2}
\rhoka\,\nabx\cdot \utka &\rightarrow
\overline{\rhok\,\nabx\cdot \utk}
\qquad &&\mbox{weakly in }L^{2}(0,T;L^{\frac{2\Gamma}{\Gamma+2}}(\Omega)),
\qquad \mbox{ as } \alpha \rightarrow 0_+.
\label{rhokautka}
\end{alignat}
It follows from the monotonicity 
of $\pk(\cdot)$ 
that
\begin{align}
\rhoka \,\pk(\rhoka) &=
(\rhoka -\rhok)\,(\pk(\rhoka)-\pk(\rhok))
+ (\rhoka -\rhok)\,\pk(\rhok) + \rhok\,\pk(\rhoka)
\nonumber \\
& \geq
(\rhoka -\rhok)\,\pk(\rhok) +
\rhok\,\pk(\rhoka) 
\hspace{1.65in} \mbox{a.e. in }\Omega_T.
\label{pkmono}
\end{align}
We deduce from (\ref{EVF2}), (\ref{pkmono}), (\ref{rhokautka}) and (\ref{rhokaconp},e)
that for all nonnegative
$\zeta \in C^\infty_0(\Omega)$ and $\eta \in C^\infty_0(0,T)$,
\begin{align}
&\int_0^T \!\! \!\eta \left(\int_{\Omega} \zeta \,
[\overline{\rhok\,\nabx\cdot \utk}-
\rhok\, \nabx\cdot \utk]
\dx \right) \dt \geq 0 \nonumber \\
& \hspace{2.3in}\quad \Rightarrow \quad \overline{\rhok\,\nabx\cdot \utk} \geq
\rhok\, \nabx\cdot \utk \quad
\mbox{a.e. in }\Omega_T,
\label{rhokutklim}
\end{align}
where we have noted (\ref{weaknneg}) with $(0,T)$ replaced by $\Omega_T$
for the final implication.

Next, we introduce $\mathfrak{L}(s) = s\,{\rm log}\, s$ for $s \in [0,\infty)$.
On recalling (\ref{rho0conv},b), we have for a subsequence (not indicated),
via Lebesgue's dominated convergence theorem, that
\begin{align}
\lim_{\alpha \rightarrow 0_+}
\int_{\Omega} \mathfrak{L}(\rho^0) \,\dx  = \int_{\Omega} \mathfrak{L}(\rho_0)\,\dx.
\label{rho0lncon}
\end{align}
We can now follow the discussion in Section 7.9.3 in
Novotn\'{y} \& Stra\v{s}kraba \cite{NovStras} to deduce that
$\overline{\pk(\rhok)}=\pk(\rhok)$. For the benefit of the reader,
we briefly outline the argument.
One deduces from (\ref{eqrhok}) as $\rhok \in C_w([0,T];L^{\Gamma}_{\geq 0}(\Omega))$,
via renormalization and noting that $s\,\mathfrak{L}'(s)-\mathfrak{L}(s)=s$
for $s \in [0,\infty)$,  that,
for any $t' \in (0,T]$,
\begin{align}
\int_{\Omega} \left[\mathfrak{L}(\rhok)(t') -\mathfrak{L}(\rho_0)\right] \dx =
- \int_0^{t'} \int_\Omega \rhok\,\nabx\cdot \utk \,\dx\,\dt.
\label{lnrhok}
\end{align}
On noting (\ref{Laprhoka}), one can choose, similarly to (\ref{z-terms}),
$\eta = \chi_{[0,t']}\,[{\rm log}\, (\rhoka +\varsigma) - 1]$, where
$\varsigma \in {\mathbb R}_{>0}$, in (\ref{eqrhokarep}), and  on passing to the limit
$\varsigma \rightarrow 0_+$ obtain that, for any $t' \in (0,T]$,
\begin{align}
\int_{\Omega} \left[\mathfrak{L}(\rhoka)(t') -\mathfrak{L}(\rho^0)\right] \dx \leq
- \int_0^{t'} \int_\Omega \rhoka\,\nabx\cdot \utka \,\dx\,\dt.
\label{lnrhoka}
\end{align}
Subtracting (\ref{lnrhok}) from (\ref{lnrhoka}),
and passing to the limit $\alpha \rightarrow 0_+$,
one deduces from
(\ref{rho0lncon}), (\ref{rhokautka}) and (\ref{rhokutklim}) that,
for any $t' \in (0,T]$,
\begin{align}
\int_{\Omega}
\left[\overline{\mathfrak{L}(\rhok)(t')} -\mathfrak{L}(\rhok)(t')\right] \dx &\leq
\int_0^{t'} \int_\Omega \left[
\rhok\,\nabx\cdot \utk - \overline{\rhok\,\nabx\cdot \utk}
\right] \dx \,\dt \leq 0,
\label{EVF2a}
\end{align}
where, on noting (\ref{rhokawcon}),
\begin{align}
\mathfrak{L}(\rhoka)(t') \rightarrow  \overline{\mathfrak{L}(\rhok)(t')}
\qquad \mbox{weakly in } L^{r}(\Omega), \qquad \mbox{for any } r \in [1,\Gamma),
\qquad\mbox{as $\alpha \rightarrow 0_+$}.
\label{rhokalncon}
\end{align}
As $\mathfrak{L}(s)$ is continuous and convex for $s\in [0,\infty)$,
it follows from (\ref{rhokawcon}) and
(\ref{rhokalncon}), see e.g.\ Corollary 3.33 in \cite{NovStras} \arxiv{(or Lemma \ref{Le-D-1} in Appendix \ref{sec:App-D})}, that
$\overline{\mathfrak{L}(\rhok)(t')} \geq \mathfrak{L}(\rhok)(t')$ a.e.\ in $\Omega$
for any $t' \in (0,T]$.
Hence, 
we deduce from (\ref{EVF2a}) that
$\overline{\mathfrak{L}(\rhok)(t')} = \mathfrak{L}(\rhok)(t')$ a.e.\ in $\Omega$
for any $t' \in (0,T]$.
Therefore, on applying Lemma 3.34 in \cite{NovStras} \arxiv{(or Lemma \ref{Le-D-2} in Appendix \ref{sec:App-D})}, we conclude from the above,
(\ref{rhokalncon}) and (\ref{rhokawcon}) that
$\rhoka(t) \rightarrow \rhok(t) \mbox{ strongly in }L^1(\Omega)$ for any $t \in (0,T]$,
as $\alpha \rightarrow 0_+$.
It immediately follows from this, 
(\ref{Bogpkabd}),
(\ref{eqLinterp}) and (\ref{pkkacon}),
on possibly extracting a further subsequence (not indicated), that,
as $\alpha \rightarrow 0_+$,
\begin{subequations}
\begin{alignat}{3}
\rhoka &\rightarrow \rhok \qquad &&\mbox{strongly in } L^r(\Omega_T), \qquad
&&\mbox{for any }r \in [1,\Gamma+1),
\label{rhokascon1}\\
\pk(\rhoka) &\rightarrow \pk(\rhok) \qquad &&\mbox{weakly in } L^{\frac{\Gamma+1}{\Gamma}}(\Omega_T),
\qquad&&\mbox{that is, }  \overline{\pk(\rhok)}=\pk(\rhok).
\label{pkaiden1}
\end{alignat}
\end{subequations}
Finally, we have the following complete analogue of Theorem \ref{5-convfinal}.

\begin{theorem}
\label{Pkexistslemfinal}
The triple $(\rhok,\utk,\hpsik)$,
defined as in Lemmas \ref{hpsikconv} and \ref{rhokalem},
is a global weak solution to problem (P$_{\kappa}$),
in the sense that
(\ref{eqrhok},c), with their initial conditions, hold and
\begin{align}
&\displaystyle\int_{0}^{T} \left \langle
\frac{\partial (\rhok\,\utk)}{\partial t}, \wt \right \rangle_{W^{1,\Gamma+1}_0(\Omega)}
\!\!\!\dt
+
\displaystyle\int_{0}^{T}\!\! \int_\Omega
\left[
\Stt(\utk) - \rhok\,\utk \otimes \utk
-\pk(\rhok)\,\Itt
\right]
:\nabxtt \wt
\,\dx\, \dt
\nonumber \\
&\qquad =\int_{0}^T
 \int_{\Omega} \left[ \rhok\,
\ft \cdot \wt
-\left(\tautt_1 (M\,\hpsik) - \mathfrak{z}\,\vrhok^2\,\Itt\right)
: \nabxtt
\wt \right] \dx \, \dt
\nonumber \\
& \hspace{3.3in}
\quad
\forall \wt \in L^{\Gamma+1}(0,T;\Wt^{1,\Gamma+1}_0(\Omega)),
\label{equtkfinal}
\end{align}
with $(\rhok\,\utk)(\cdot,0) = (\rho_0\,\ut_0)(\cdot)$.
In addition, the weak solution $(\rhok,\utk,\hpsik)$
satisfies (\ref{Pkenergy}).
\end{theorem}
\begin{proof}
The results (\ref{eqrhok},c) and (\ref{Pkenergy}) have already been established
in Lemma \ref{Pkexistslem}.
Equation (\ref{equtkfinal}) was established in Lemma \ref{Pkexistslem} with
$\pk(\rhok)$ replaced by $\overline{\pk(\rhok)}$, see (\ref{equtk}).
The desired result (\ref{equtkfinal}) then follows immediately from (\ref{equtk})
and (\ref{pkaiden1}).
\end{proof}

\section{Existence of a solution to (P)}
\label{sec:P}
\setcounter{equation}{0}

It follows from the bounds on $\vrhok$ in (\ref{Pkenergy}),
similarly to (\ref{vrhokaLvbd}) and (\ref{vrhokaLvvbd}),
that
\begin{align}
\|\vrhok\|_{L^{\infty}(0,T;L^2(\Omega))}+
\|\vrhok\|_{L^{\frac{2(d+2)}{d}}(\Omega_T)}
+\|\vrhok\|_{L^{2}(0,T;L^{6}(\Omega))}
+\|\vrhok\|_{L^{4}(0,T;L^{\frac{2d}{d-1}}(\Omega))}
\leq C,
\label{vrhokbd}
\end{align}
where throughout this section  $C$ is a generic positive constant, independent of $\kappa$.
Hence, we deduce from (\ref{vrhokbd}), (\ref{Cttrsbd}) and (\ref{Pkenergy}),
similarly to (\ref{tautt43bd}),
that
\begin{align}
\|\tautt_1(M\,\hpsik)\|_{L^2(0,T;L^{\frac{4}{3}}(\Omega))}
+\|\tautt_1(M\,\hpsik)\|_{L^{\frac{4(d+2)}{3d+4}}(\Omega_T)}
+\|\tautt_1(M\,\hpsik)\|_{L^{\frac{4}{3}}(0,T;L^{\frac{12}{7}}(\Omega))}
&\leq C.
\label{tautt1k}
\end{align}
Similarly to (\ref{hpsikaLDdtbd}), it follows from
(\ref{vrhokbd}) and (\ref{Pkenergy}) that
\begin{align}
\left\|M\,\frac{\partial \hpsik}{\partial t}
\right\|_{L^2(0,T;H^s(\Omega\times D)')} \leq C,
\label{hpsikdtbd}
\end{align}
where $s > 1+ \frac{1}{2}(K+1)d$.
We have the following analogue of Lemma \ref{hpsikconv}.

\begin{lemma}
\label{hpsiconv}
There exist functions
\begin{subequations}
\begin{align}
\label{hpsi}
&\ut \in L^2(0,T;\Ht^1_0(\Omega))\qquad \mbox{and} \qquad  \hpsi \in L^{\upsilon}(0,T;Z_1)\cap
H^1(0,T; M^{-1}(H^{s}(\Omega \times D))'),
\end{align}
where  $\upsilon \in [1,\infty)$ and $s>1+\frac{1}{2}(K+1)d$,
with finite relative entropy and Fisher information,
\begin{align}
\mathcal{F}(\hpsi) \in L^\infty(0,T;L^1_M(\Omega \times D)) \qquad \mbox{and} \qquad
\sqrt{\hpsi} \in L^2(0,T;H^1_M(\Omega \times D)),
\label{FisentP}
\end{align}
\end{subequations}
and a subsequence of $\{(\rhok,\,\utk,\,\hpsik)\}_{\kappa > 0}$
such that,
as $\kappa \rightarrow 0_+$,
\begin{align}
\utk &\rightarrow \ut
\qquad \mbox{weakly in } L^{2}(0,T;\Ht^1_0(\Omega)), \label{uwconH1b}
\end{align}
and
\begin{subequations}
\begin{alignat}{2}
M^{\frac{1}{2}}\,\nabx \sqrt{\hpsik} &\rightarrow M^{\frac{1}{2}}\,\nabx \sqrt{\hpsi}
&&\qquad \mbox{weakly in } L^{2}(0,T;\Lt^2(\Omega\times D)), \label{hpsiwconH1}\\
M^{\frac{1}{2}}\,\nabq \sqrt{\hpsik} &\rightarrow M^{\frac{1}{2}}\,\nabq \sqrt{\hpsi}
&&\qquad \mbox{weakly in } L^{2}(0,T;\Lt^2(\Omega\times D)), \label{hpsiwconH1x}\\
M\,\frac{\partial \hpsik}{\partial t} &\rightarrow M\,\frac{\partial \hpsi}{\partial t}
&&\qquad \mbox{weakly in } L^{2}(0,T;H^s(\Omega\times D)'), \label{hpsidtwcon}\\
\hpsik & \rightarrow
\hpsi &&\qquad \mbox{strongly in } L^\upsilon(0,T;L^1_M(\Omega\times D)),\label{hpsisconL1}\\
\tautt(M\,\hpsik) & \rightarrow \tautt(M\,\hpsi)
&&\qquad \mbox{strongly in } \Ltt^r(\Omega_T),\label{tausconLr}
\end{alignat}
where $r \in [1,\frac{4(d+2)}{3d+4})$, and,
for a.a.\ $t \in (0,T)$,
\begin{align}\label{Pfatou-app}
&\int_{\Omega \times D} M(\qt)\, \mathcal{F}(\hpsi(\xt,\qt,t))\dq \dx
\leq \liminf_{\kappa \rightarrow 0_+}
\int_{\Omega \times D} M(\qt)\, \mathcal{F}(\hpsik(\xt,\qt,t)) \dq \dx.
\end{align}
\end{subequations}

In addition, we have that
\begin{align}
\vrho := \int_D M\, \hpsi\, \dq \in L^\infty(0,T;L^2(\Omega)) \cap L^2(0,T;H^1(\Omega)),
\label{vrhoreg}
\end{align}
and, as $\kappa \rightarrow 0_+$,
\begin{subequations}
\begin{alignat}{2}
\vrhok &\rightarrow \vrho \qquad
&&\mbox{weakly-$\star$ in } L^\infty(0,T;L^2(\Omega)),
\qquad \mbox{weakly in } L^2(0,T;H^1(\Omega)),
\label{vrhokH1con} \\
\vrhok &\rightarrow \vrho
\qquad &&\mbox{strongly in } L^{\frac{5\varsigma}{3(\varsigma-1)}}(0,T;L^\varsigma(\Omega)),
\label{vrhokL2con}
\end{alignat}
\end{subequations}
for any $\varsigma \in (1,6)$.
\end{lemma}
\begin{proof}
The convergence result (\ref{uwconH1b}) and the first result in (\ref{hpsi})
follow immediately from the bound on $\utk$ in (\ref{Pkenergy}).
The remainder of the results follow from the bounds on $\hpsik$ and $\vrhok$ in (\ref{Pkenergy})
in the same way as the results of Lemma \ref{convfinal}.
\end{proof}

We have the following analogue of
Lemma \ref{rhokalem}.

\begin{lemma}\label{rhoklem}
Let $\Gamma \geq 8$; then, there exists a $C \in \mathbb R_{>0}$, independent of $\kappa$,
such that, for any $\gamma > \frac{3}{2}$ as in (\ref{pgamma}),
\begin{subequations}
\begin{align}
&\|\rhok\|_{L^\infty(0,T;L^{\gamma}(\Omega))}
+ \|\utk\|_{L^2(0,T;H^1(\Omega))}
+ \kappa^{\frac{1}{\Gamma}}\,\|\rhok\|_{L^\infty(0,T;L^{\Gamma}(\Omega))}
+ \left\|\sqrt{\rhok}\,\utk\right\|_{L^\infty(0,T;L^2(\Omega))}
\nonumber \\
& \qquad
+ \|\rhok\,\utk\|_{L^\infty(0,T;L^{\frac{2\gamma}{\gamma+1}}(\Omega))}
+ \|\rhok\,\utk\|_{L^2(0,T;L^{\frac{6\gamma}{\gamma+6}}(\Omega))}
\nonumber \\
&\qquad
+ \|\rhok\,\utk\|_{L^{\frac{10\gamma-6}{3(\gamma+1)}}(\Omega_T)}
+ \left\|\rhok\,|\utk|^2\right\|_{L^2(0,T;L^{\frac{6\gamma}{4\gamma+3}}(\Omega))}
\leq C,
\label{mtkbd} \\
&
\left\|\frac{\partial\rhok}{\partial t}\right\|_{L^2(0,T;W^{1,6}(\Omega)')}
\leq C.
\label{dtrhok}
\end{align}
\end{subequations}

Hence, there exists a function $\rho \in 
C_w([0,T];L^{\gamma}_{\geq 0}(\Omega))
\cap H^1(0,T;W^{1,6}(\Omega)')$,
and for a further subsequence of the subsequence of Lemma \ref{hpsiconv}, it follows that,
as $\kappa \rightarrow 0_+$,
\begin{subequations}
\begin{alignat}{2}
\rhok &\rightarrow \rho
\qquad
&&
\mbox{in } C_w([0,T];L^{\gamma}(\Omega))
\qquad \mbox{weakly in } H^1(0,T;W^{1,6}(\Omega)'),
\label{rhokwcon}
\\
\rhok &\rightarrow \rho
\qquad
&&\mbox{strongly in } L^2(0,T;H^1(\Omega)'),
\label{rhokscon}
\end{alignat}
and, for any nonnegative $\eta \in C[0,T],$
\begin{align}\label{Pp}
&\int_0^T \left(\int_{\Omega} P(\rho)\,\dx\right) \eta\,\dt
\leq \liminf_{\kappa \rightarrow 0_+}
\int_0^T \left( \int_{\Omega} P(\rhok)\, \dx\right) \eta\,\dt.
\end{align}
\end{subequations}
\end{lemma}
\begin{proof}
The first four bounds in (\ref{mtkbd}) follow immediately from
(\ref{Pkenergy}). The last four bounds in (\ref{mtkbd})
follow, similarly to (\ref{mtkaLLnu},b), from the first two
bounds in (\ref{mtkbd}).
The bound (\ref{dtrhok}) follows immediately from (\ref{eqrhok}),
the sixth bound in (\ref{mtkbd}), on noting that $\frac{6\gamma}{\gamma+6} > \frac{6}{5}$
as $\gamma > \frac{3}{2}$.

The convergence results (\ref{rhokwcon},b) follow immediately from
(\ref{mtkbd},b), (\ref{Cwcoma},b) and (\ref{compact1}).
The result (\ref{Pp}) follows, similarly to (\ref{Pkbel}), from (\ref{rhokwcon})
and the convexity of $P$.
\end{proof}

Similarly to (\ref{equtkabd}),
it follows from (\ref{equtk}),
(\ref{mtkbd}), (\ref{vrhokbd}),
(\ref{tautt1k}), (\ref{inidata}), on noting that $\gamma >\frac{3}{2}$,
and (\ref{eqinterp}) that,
for any $\wt \in L^{\infty}(0,T;\Wt_0^{1,r}(\Omega))$ with $r=\max\{\Gamma+1,\upsilon\}$
and $\upsilon =\max\{\frac{3\gamma}{2\gamma-3},\frac{12}{5}\}$,
\begin{align}
&\left|\displaystyle\int_{0}^{T} \left\langle
\frac{\partial (\rhok\,\utk)}{\partial t},  \wt \right\rangle_{W^{1,\Gamma+1}_0(\Omega)}\dt
-\int_{0}^T \int_{\Omega} \pk(\rhok)\,\nabx \cdot \wt
\dx\,\dt
\right| \nonumber \\
& \qquad \leq C\,
\|\rhok\|_{L^\infty(0,T;L^\gamma(\Omega))}\,\|\utk\|_{L^2(0,T;L^6(\Omega))}^2\,
\,\|\wt\|_{L^{\infty}(0,T;W^{1,\frac{3\gamma}{2\gamma-3}}(\Omega))}
\nonumber \\
&\qquad \qquad + C\,
\|\utk\|_{L^2(0,T;H^1(\Omega))}
\,\|\wt\|_{L^{2}(0,T;H^1(\Omega))}
\nonumber \\
& \qquad \qquad +
C\left[\|\tautt_1(M\,\hpsik)\|_{L^{\frac{4}{3}}(0,T;L^{\frac{12}{7}}(\Omega))}
+\|\vrhok\|_{L^2(0,T;L^6(\Omega))}^2 \right]
\|\nabttx\wt\|_{L^\infty(0,T;L^{\frac{12}{5}}(\Omega))}
\nonumber \\
&\qquad \qquad + \|\rhok\|_{L^\infty(0,T;L^\gamma(\Omega))}\,
\|\ft\|_{L^2(0,T;L^\infty(\Omega))}\, \|\wt\|_{L^{2}(0,T;L^3(\Omega))}
\nonumber \\
& \qquad \leq
C\,\|\wt\|_{L^{\infty}(0,T;W^{1,\upsilon}(\Omega))}.
\label{equtkbd}
\end{align}
We deduce from (\ref{equtkbd}) with $\wt = \eta \, \vt$, where $\eta \in C^{\infty}_0(0,T)$
and $\vt \in L^{\infty}(0,T;\Wt^{1,\upsilon}_0(\Omega))\cap
H^1(0,T;\Lt^{\frac{6\gamma}{5\gamma-6}}(\Omega))$
with $r=\max\{\Gamma+1,\upsilon\}$
and
$\upsilon=
\max\{\frac{3\gamma}{2\gamma-3},\frac{12}{5}\}$,
on noting
(\ref{mtkbd}) and (\ref{eqinterp}) as $\frac{2\gamma}{\gamma-1} < 6$
for $\gamma > \frac{3}{2}$, that
\begin{align}
&\left|\displaystyle
\int_{0}^T \eta \int_{\Omega} \pk(\rhok)\,\nabx \cdot \vt\,
\dx\,\dt
\right| \nonumber \\
& \quad \leq
\left|\int_{0}^{T} \int_{\Omega} \rhok\,\utk \cdot\frac{\partial(\eta\,\vt)}{\partial t}\,
\dx\,\dt\right|
+ C\,\|\eta\|_{L^\infty(0,T)}\,\|\vt\|_{L^{\infty}(I;W^{1,\upsilon}(\Omega))}
\nonumber \\
& \quad \leq C\,\|\rhok\,\utk\|_{L^\infty(0,T;L^{\frac{2\gamma}{\gamma+1}}(\Omega))}
\,\left\|\frac{{\rm d}\eta}{{\rm d} t}\right\|_{L^1(0,T)}
\,\|\vt\|_{L^\infty(I;L^{\frac{2\gamma}{\gamma-1}}(\Omega))}
\nonumber \\
& \quad \quad + C\,\|\eta\|_{L^\infty(0,T)}\left[
\|\rhok\,\utk\|_{L^2(0,T;L^{\frac{6\gamma}{\gamma+6}}(\Omega))}
\left\|\frac{\partial \vt}{\partial t}\right\|_{L^2(I;L^{\frac{6\gamma}{5\gamma-6}}(\Omega))}
+ \|\vt\|_{L^{\infty}(I;W^{1,\upsilon}(\Omega))} \right]
\nonumber \\
& \quad \leq C\left[
\left\|\frac{{\rm d}\eta}{{\rm d}t}\right\|_{L^1(0,T)}+
\|\eta\|_{L^\infty(0,T)}\right] \left[\|\vt\|_{L^{\infty}(I;W^{1,\upsilon}(\Omega))}
+
\left\|\frac{\partial \vt}{\partial t}\right
\|_{L^{2}(I;L^{\frac{6\gamma}{5\gamma-6}}(\Omega))}\right],
\label{equtkbd1}
\end{align}
where $I={\rm supp}(\eta) \subset (0,T)$.
With $\upsilon=\upsilon(\gamma)$ thus defined,
let
\begin{align}
\vartheta(\gamma) := \frac{\gamma}{\upsilon(\gamma)} = \left\{ \begin{array}{ll}
\frac{2\gamma - 3}{3} \quad &\mbox{for } \frac{3}{2} < \gamma \leq 4, \\
\frac{5}{12}\gamma \quad &\mbox{for } 4 \leq \gamma.
\end{array}
\right.
\label{vthetag}
\end{align}
With $\vartheta(\gamma) \in \mathbb{R}_{>0}$ defined as above and $\ell \in \mathbb{N}$,
we now introduce $\mathfrak{b} : \mathbb{R}_{\geq 0}
\rightarrow \mathbb{R}_{\geq 0}$ and $\mathfrak{b}_\ell : \mathbb{R}_{\geq 0}
\rightarrow \mathbb{R}_{\geq 0}$ such that
\begin{align}
\mathfrak{b}(s):= s^\vartheta  \qquad \mbox{and} \qquad
\mathfrak{b}_{\ell}(s):= \left\{ \begin{array}{ll}
\mathfrak{b}(s) \quad &\mbox{for } 0 \leq s \leq \ell,\\
\mathfrak{b}(\ell) \quad &\mbox{for } \ell \leq s.
\end{array} \right.
\label{bbell}
\end{align}
We note from (\ref{bbell}), (\ref{mtkbd}) and (\ref{eqinterp})
that, 
for $\upsilon(\gamma):=\max\{\frac{3\gamma}{2\gamma-3},
\frac{12}{5}\}$ as in (\ref{equtkbd1}) and $\vartheta(\gamma)$ as in (\ref{vthetag}),
\red{we have that}
\begin{subequations}
\begin{align}
\|\mathfrak{b}_\ell(\rhok)\|_{L^\infty(0,T;L^{\upsilon}(\Omega))}
&\leq \|\rhok^\vartheta\|_{L^\infty(0,T;L^\upsilon(\Omega))}
\leq \|\rhok\|_{L^\infty(0,T;L^\gamma(\Omega))}^\vartheta \leq C,
\label{Bttheta} \\
\|\mathfrak{b}_\ell(\rhok)\,\utk\|_{L^2(0,T;L^{\frac{6\gamma}{\gamma+6\vartheta}}(\Omega))}
&\leq \|\rhok\|_{L^\infty(0,T;L^\gamma(\Omega))}^\vartheta
\, \|\utk\|_{L^2(0,T;L^6(\Omega))} \leq C,
\label{Btutktheta} \\
\|\mathfrak{b}_\ell(\rhok)\,\nabx \cdot \utk\|_{L^2(0,T;L^{\frac{2\gamma}{\gamma+2\vartheta}}(\Omega))}
&\leq \|\rhok\|_{L^\infty(0,T;L^\gamma(\Omega))}^\vartheta
\, \|\utk\|_{L^2(0,T;H^1(\Omega))} \leq C,
\label{Btdivuttheta}
\end{align}
\end{subequations}
where $C\in {\mathbb R}_{>0}$ is independent of $\kappa, \vartheta$ and $\ell$.

As $\Gamma > 2$, it follows from (\ref{eqrhok}), on extending $\rhok$ and $\utk$ from $\Omega$
to ${\mathbb R}^d$ by zero, that
\begin{align}
\frac{\partial \rhok}{\partial t}
+ \nabx \cdot ( \rhok\,\utk)=0
\qquad \mbox{in } C^\infty_0({\mathbb R}^d \times (0,T))',
\label{prolong}
\end{align}
see Lemmas 6.8 in Novotn\'{y} \& Stra\v{s}kraba \cite{NovStras} \arxiv{(or Lemma \ref{Le-F-1} in Appendix \ref{sec:App-F})}.
Applying Lemma 6.11 in  \cite{NovStras} \arxiv{(or Lemma \ref{Le-F-3} in Appendix \ref{sec:App-F})} to (\ref{prolong}),
we have the renormalised equation, for any $\ell \in {\mathbb N}$,
\begin{align}
\frac{\partial \mathfrak{b}_\ell(\rhok)}{\partial t}
+ \nabx \cdot ( \mathfrak{b}_\ell(\rhok)\,\utk)
+ (\rhok\,(\mathfrak{b}_\ell)_+'(\rhok)-\mathfrak{b}_\ell(\rhok)\,)\,\nabx\cdot \utk=0
\qquad \mbox{in } C^\infty_0({\mathbb R}^d \times (0,T))',
\label{renormcon}
\end{align}
where $(\mathfrak{b}_\ell)_+'(\cdot)$ is the right-derivative
of $\mathfrak{b}_\ell(\cdot)$ satisfying
\begin{align}
(\mathfrak{b}_\ell)_+'(s 
) = \left\{
\begin{array}{ll}
\mathfrak{b}'(s
)
\quad& \mbox{for } 0 \leq s < \ell, \\
0 \quad &
\mbox{for } \ell \leq s.
\end{array}
\right.
\label{rderiv}
\end{align}
For any $\delta \in (0,\frac{T}{2})$, we now introduce the Friedrichs mollifier,
with respect to the time variable, $\mathcal{S}_\delta : L^1(0,T;L^q(\Omega))
\rightarrow  C^\infty(\delta,T-\delta;L^q(\Omega))$, $q\in [1,\infty]$,
\begin{align}
\mathcal{S}_\delta(\eta)(\xt,t) = \frac{1}{\delta} \int_0^T 
\omega\left(\frac{t-s}{\delta}\right)\eta(\xt,s)\,{\rm d}s  \qquad \mbox{a.e.\ in } \Omega
\times (\delta,T-\delta),
\label{Sddef}
\end{align}
where $\omega \in C^\infty_0({\mathbb R})$, $\omega \geq 0$, ${\rm supp}(\omega) \subset (-1,1)$
and $\int_{{\mathbb R}} \omega \,{\rm d}s=1$.
It follows from
(\ref{renormcon}) and (\ref{Sddef})  that
\begin{align}
&\frac{\partial {\mathcal S}_\delta(\mathfrak{b}_\ell(\rhok))}{\partial t}
+ \nabx \cdot {\mathcal S}_\delta( \mathfrak{b}_\ell(\rhok)\,\utk)
+ {\mathcal S}_\delta([\rhok\,(\mathfrak{b}_\ell)_+'(\rhok)-\mathfrak{b}_\ell(\rhok)\,]\,\nabx\cdot \utk)=0
\nonumber \\
& \hspace{3.5in}
\qquad \mbox{in } C^\infty_0({\mathbb R}^d \times (\delta,T-\delta))'.
\label{Sdrenormcon}
\end{align}
In addition, it follows from (\ref{Sddef}), (\ref{bbell}),
(\ref{mtkbd}), (\ref{eqinterp}), (\ref{rderiv}) and (\ref{Sdrenormcon}) that
\begin{align}
&{\mathcal S}_\delta( \mathfrak{b}_\ell(\rhok)) \in
C^\infty(\delta,T-\delta;L^\infty({\mathbb R}^d)),
\quad
{\mathcal S}_\delta( \mathfrak{b}_\ell(\rhok)\,\utk) \in C^\infty(\delta,T-\delta;\Lt^6({\mathbb R}^d)),
\nonumber \\
& {\mathcal S}_\delta([\rhok\,(\mathfrak{b}_\ell)_+'(\rhok)-\mathfrak{b}_\ell(\rhok)\,]\,\nabx\cdot \utk),
\;\;\;\;\,
\nabx \cdot [{\mathcal S}_\delta( \mathfrak{b}_\ell(\rhok)\,\utk)]
\in C^\infty(\delta,T-\delta;L^2({\mathbb R}^d)).
\label{Sdreg}
\end{align}
One can deduce from $\utk \in L^2(0,T;\Ht^1_0(\Omega))$ and (\ref{Sdreg}) that
\begin{align}
{\mathcal S}_\delta( \mathfrak{b}_\ell(\rhok)\,\utk) \in C^\infty(\delta,T-\delta;\Et^{6,2}_0(\Omega)),
\label{SdEreg}
\end{align}
where we recall (\ref{Etrs}).
We note from (\ref{Bop1}),  (\ref{Sddef}), (\ref{bbell}) and (\ref{Bttheta})
that $\calBt([{\mathcal S}_\delta(\mathfrak{b}_\ell(\rhok))]) \in
L^\infty(\delta,T-\delta;\Wt^{1,r}_0(\Omega))$,
$r \in [1,\infty)$, and, for $\upsilon(\gamma):=\max\{\frac{3\gamma}{2\gamma-3},
\frac{12}{5}\}$ as in (\ref{equtkbd1}),
that
\begin{subequations}
\begin{align}
\|\calBt((I-\mint)[{\mathcal S}_\delta(\mathfrak{b}_\ell(\rhok))])
\|_{L^\infty(\delta,T-\delta;W^{1,\upsilon}(\Omega))}
&\leq C\,\|{\mathcal S}_\delta(\mathfrak{b}_\ell(\rhok))
\|_{L^\infty(\delta,T-\delta;L^\upsilon(\Omega))}
\nonumber \\
&\leq C\,\|\mathfrak{b}_\ell(\rhok)\|_{L^\infty(0,T;L^\upsilon(\Omega))} \leq C,
\label{SdBttheta}
\end{align}
and from (\ref{Sdrenormcon}), Sobolev embedding, (\ref{SdEreg}), (\ref{Bop1},b),
(\ref{Sddef}), (\ref{bbell}), (\ref{rderiv}) and (\ref{Btutktheta},c)
with $\vartheta$ as in (\ref{vthetag})
that
\begin{align}
&\left\|\frac{\partial}{\partial t}\calBt((I-\mint)[{\mathcal S}_\delta(\mathfrak{b}_\ell(\rhok))])\right
\|_{L^{2}(\delta,T-\delta;L^{\frac{6\gamma}{\gamma+6\vartheta}}(\Omega))}
\nonumber \\
& \quad \leq
\|\calBt(\nabx\cdot[{\mathcal S}_\delta(\mathfrak{b}_\ell(\rhok)\,\utk)])
\|_{L^{2}(\delta,T-\delta;L^{\frac{6\gamma}{\gamma+6\vartheta}}(\Omega))}
\nonumber \\
& \qquad \qquad +
\|\calBt((I-\mint)[{\mathcal S}_\delta(
[\rhok\,(\mathfrak{b}_\ell)_+'(\rhok)-\mathfrak{b}_\ell(\rhok)\,]\,\nabx\cdot \utk
)])
\|_{L^{2}(\delta,T-\delta;L^{\frac{6\gamma}{\gamma+6\vartheta}}(\Omega))}
\nonumber \\
& \quad \leq
\|\calBt(\nabx\cdot[{\mathcal S}_\delta(\mathfrak{b}_\ell(\rhok)\,\utk)])
\|_{L^{2}(\delta,T-\delta;L^{\frac{6\gamma}{\gamma+6\vartheta}}(\Omega))}
\nonumber \\
& \qquad \qquad +
\|\calBt((I-\mint)[{\mathcal S}_\delta(
[\rhok\,(\mathfrak{b}_\ell)_+'(\rhok)-\mathfrak{b}_\ell(\rhok)\,]\,\nabx\cdot \utk
)])
\|_{L^{2}(\delta,T-\delta;W^{1,\frac{2\gamma}{\gamma+2\vartheta}}(\Omega))}
\nonumber \\
& \quad \leq
C\left[\|
\mathfrak{b}_\ell(\rhok)\,\utk 
\|_{L^{2}(
0,T;L^{\frac{6\gamma}{\gamma+6\vartheta}}(\Omega))}
+\|
[\rhok\,(\mathfrak{b}_\ell)_+'(\rhok)-\mathfrak{b}_\ell(\rhok)\,]\,\nabx\cdot \utk
\|_{L^{2}(
0,T;L^{\frac{2\gamma}{\gamma+2\vartheta}}(\Omega))}
\right]
\nonumber \\
& \quad \leq
C\left[\|
\mathfrak{b}_\ell(\rhok)\,\utk 
\|_{L^{2}(
0,T;L^{\frac{6\gamma}{\gamma+6\vartheta}}(\Omega))}
+\|
\mathfrak{b}_\ell(\rhok)\,\nabx\cdot \utk
\|_{L^{2}(
0,T;L^{\frac{2\gamma}{\gamma+2\vartheta}}(\Omega))}
\right]
\leq C,\nonumber \\
\label{dtSdBt}
\end{align}
\end{subequations}
where $C \in {\mathbb R}_{>0}$ in (\ref{SdBttheta},b) is independent of
$\kappa, \vartheta, \ell$ and $\delta$.
We now have the following analogue of Lemma \ref{BogLem}.

\begin{lemma}\label{BogLema}
With $\vartheta(\gamma)$ as defined in (\ref{vthetag}), we have that
\begin{align}
\|\rhok\|_{L^{\gamma+\vartheta}(\Omega_T)} +
\kappa^{\frac{1}{4+\vartheta}}\,\|\rhok\|_{L^{4+\vartheta}(\Omega_T)}
+ \kappa^{\frac{1}{\Gamma+\vartheta}}\,\|\rhok\|_{L^{\Gamma+\vartheta}(\Omega_T)}
\leq C.
\label{Bogpkbd}
\end{align}
\end{lemma}
\begin{proof}
For any $\ell \in {\mathbb N}$ and $\delta \in (0,\frac{T}{2})$, we
choose $\vt = \calBt((I-\mint)[{\mathcal S}_\delta(\mathfrak{b}_\ell(\rhok))])
\in L^\infty(\delta,T-\delta,\Wt^{1,r}_0(\Omega))$, any $r \in [1,\infty)$,
and $\eta \in C^\infty_0(0,T)$, with ${\rm supp}(\eta) \subset (\delta,T-\delta)$,
in (\ref{equtkbd1}) to obtain, on noting (\ref{SdBttheta},b), $\frac{6\gamma}{5\gamma-6}
\leq \frac{6\gamma}{\gamma+6\vartheta}$ as $\vartheta \leq \frac{2 \gamma}{3}-1$,
(\ref{mtkbd}) and (\ref{pkdef}), that
\begin{align}
&\left|\displaystyle
\int_{0}^T \eta \int_{\Omega} \pk(\rhok)\,{\mathcal S}_\delta(\mathfrak{b}_\ell(\rhok))
\dx\,\dt
\right| \nonumber \\
& \quad \leq C\left[
\left\|\frac{{\rm d}\eta}{{\rm d}t}\right\|_{L^1(0,T)}+
\|\eta\|_{L^\infty(0,T)}\right]
\left[1+ 
\|\pk(\rhok)\|_{L^1(\Omega_T)}
\,\|\mint {\mathcal S}_\delta(\mathfrak{b}_\ell(\rhok))\|_{L^\infty(\delta,T-\delta)} \right]
\nonumber \\
& \quad \leq C\left[
\left\|\frac{{\rm d}\eta}{{\rm d}t}\right\|_{L^1(0,T)}+
\|\eta\|_{L^\infty(0,T)}\right].
\label{equtkbd2}
\end{align}

We now consider (\ref{equtkbd2}) with $\eta=\eta_{\ell} \in C^\infty_0(0,T)$ with
supp$(\eta_{\ell}) \subset (\frac{1}{\ell},T-\frac{1}{\ell})$,
$\ell \in {\mathbb N}$ with $\ell > \frac{4}{T}$, where
$\eta_{\ell} \in [0,1]$ with $\eta_\ell(t)=1$ for $t \in [\frac{2}{\ell},T-\frac{2}{\ell}]$
and $\|\frac{{\rm d}\eta_\ell}{{\rm d} t}\|_{L^\infty(0,T)} \leq 2\ell$ yielding
$\|\frac{{\rm d}\eta_\ell}{{\rm d} t}\|_{L^1(0,T)} \leq 4$.
For a fixed $\ell$, we now let $\delta \rightarrow 0$ in (\ref{equtkbd2})
and using the standard convergence properties of mollifiers we obtain that
\begin{align}
&\left|
\displaystyle
\int_{0}^T \eta_\ell \int_{\Omega} \pk(\rhok)\,\mathfrak{b}_\ell(\rhok)
\dx\,\dt
\right|
\leq C,
\label{equtkbd3}
\end{align}
where $C \in \mathbb{R}$ is independent of $\ell$ and  $\kappa$.
Letting $\ell \rightarrow \infty$ in (\ref{equtkbd3}), and noting that
$\eta_\ell \rightarrow 1$ pointwise in $(0,T)$, $\mathfrak{b}_\ell(\rhok)
\rightarrow \mathfrak{b}(\rhok)=\rhok^\vartheta$ pointwise in $\Omega_T$
and Fatou's lemma, we obtain that
\begin{align}
&
\displaystyle
\int_{0}^T \int_{\Omega} \pk(\rhok)\,\rhok^\vartheta
\dx\,\dt
\leq C,
\label{equtkbd4}
\end{align}
where $C \in \mathbb{R}$ is independent of $\kappa$.
Hence the desired result (\ref{Bogpkbd}) follows from (\ref{equtkbd4}),
(\ref{pkdef}) and (\ref{pgamma}).
\end{proof}

Similarly to (\ref{equtkbd}),
it follows from (\ref{equtk}),
(\ref{mtkbd}), (\ref{vrhokbd}),
(\ref{tautt1k})
and (\ref{inidata}), on noting $\gamma >\frac{3}{2}$, that,
for any $\wt \in L^{\Gamma+1}(0,T;\Wt_0^{1,\upsilon}(\Omega))$ with
$\upsilon=\max\{\Gamma+1,\frac{6\gamma}{2\gamma-3}\}$,
\begin{align}
&\left|\displaystyle\int_{0}^{T} \left\langle
\frac{\partial (\rhok\,\utk)}{\partial t},  \wt \right\rangle_{W^{1,\Gamma+1}_0(\Omega)}\dt
-\int_{0}^T \int_{\Omega} \pk(\rhok)\,\nabx \cdot \wt
\dx\,\dt
\right| \nonumber \\
& \quad \leq C\left[
\|\rhok\,|\utk|^2\|_{L^{2}(0,T;L^{\frac{6\gamma}{4\gamma+3}}(\Omega))} +
\|\utk\|_{L^2(0,T;H^1(\Omega))}
\right]
\,\|\wt\|_{L^{2}(0,T;W^{1,\frac{6\gamma}{2\gamma-3}}(\Omega))}
\nonumber \\
& \qquad +
C\left[\|\tautt_1(M\,\hpsik)\|_{L^2(0,T;L^{\frac{4}{3}}(\Omega))}\,
+\|\vrhok\|_{L^4(0,T;L^{\frac{2d}{d-1}}(\Omega))}^2 \right]
\|\nabxtt\wt\|_{L^2(0,T;L^4(\Omega))}
\nonumber \\
&\qquad + \|\rhok\|_{L^\infty(0,T;L^\gamma(\Omega))}\,
\|\ft\|_{L^2(0,T;L^\infty(\Omega))}\, \|\wt\|_{L^{2}(0,T;L^3(\Omega))}
\nonumber \\
& \quad \leq
C\,\|\wt\|_{L^{2}(0,T;W^{1,s}(\Omega))},
\label{equtkbda}
\end{align}
where $s=\max\{4,\frac{6\gamma}{2\gamma-3}\}$.
We now have the following analogue of Lemma
 \ref{mtkalem}.

\begin{lemma}\label{mtklem}
There exists a $C \in \mathbb R_{>0}$, independent of $\kappa$, such that
\begin{align}
\left\|\frac{\partial(\rhok\,\utk)}{\partial t}
\right\|_{L^{\frac{\Gamma+\vartheta}{\Gamma}}(0,T;W^{1,r}_0(\Omega)')} \leq C,
\label{dtmtk}
\end{align}
where 
$\vartheta(\gamma)$ is defined as in (\ref{vthetag}),
$r= \max\{s,\frac{\Gamma+\vartheta}{\vartheta}
\}$ and
$s= \max\{4,\frac{6\gamma}{2\gamma-3}\}$.

Hence, for a further subsequence of the subsequence of Lemma \ref{rhoklem}, it follows that,
as $\kappa \rightarrow 0_+$,
\begin{subequations}
\begin{alignat}{3}
\rhok\,\utk &\rightarrow \rho\,\ut
\quad
&&\mbox{weakly in } \Lt^{\frac{10\gamma-6}{3(\gamma+1)}}(\Omega_T),
\quad &&\mbox{weakly in } W^{1,\frac{\Gamma+\vartheta}{\Gamma}}(0,T;\Wt^{1,r}_0(\Omega)'),
\label{mtkwcon}
\\
\rhok\,\utk &\rightarrow \rho\,\ut
\quad
&&\mbox{in } C_w([0,T];\Lt^{\frac{2\gamma}{\gamma+1}}(\Omega)), \quad
&&\mbox{strongly in } L^2(0,T;\Ht^1(\Omega)'),
\label{mtkscona} \\
\rhok\,\utk &\rightarrow \rho\,\ut
\quad
&&\mbox{weakly in } L^2(0,T;\Lt^{\frac{6\gamma}{\gamma+6}}(\Omega)),
\label{mtkscon} \\
\rhok\,\utk \otimes \utk &\rightarrow \rho\,\ut \otimes \ut
\quad
&&\mbox{weakly in } L^2(0,T;\Ltt^{\frac{6\gamma}{4\gamma+3}}(\Omega)),
\label{mtkutkwcon}  \\
\rhok &\rightarrow \rho
&&\mbox{weakly in } L^{{\gamma+\vartheta}}(\Omega_T),
\label{rhokconp} \\
\rhok^\gamma &\rightarrow \overline{\rho^\gamma}
\quad
&&\mbox{weakly in } L^{\frac{\gamma+\vartheta}{\gamma}}(\Omega_T),
\label{pkkcon} \\
\kappa\,(\rhok^4+\rhok^\Gamma) &\rightarrow 0
\quad
&&\mbox{weakly in } L^{\frac{\Gamma+\vartheta}{\Gamma}}(\Omega_T),
\label{pkkconzero}
\end{alignat}
\end{subequations}
where $\overline{\rho^\gamma} \in L^{\frac{\gamma+\vartheta}{\gamma}}_{\geq 0}(\Omega_T)$
remains to be identified.
\end{lemma}
\begin{proof}
It follows from (\ref{equtkbda}), (\ref{pkdef}) and (\ref{Bogpkbd}) that,
for all $\wt \in L^{\frac{\Gamma+\vartheta}{\vartheta}}(0,T;\Wt^{1,r}_0(\Omega))$,
\begin{align}
&\left|\displaystyle\int_{0}^{T} \left\langle
\frac{\partial (\rhok\,\utk)}{\partial t},  \wt \right\rangle_{W^{1,r}_0(\Omega)}\dt
\right| \nonumber \\
& \hspace{1in} \leq
C\,\|\wt\|_{L^{2}(0,T;W^{1,s}(\Omega))}
+ \|\pk(\rhok)\|_{L^{\frac{\Gamma+\vartheta}{\Gamma}}(\Omega_T)}
\,\|\wt\|_{L^{\frac{\Gamma+\vartheta}{\vartheta}}(0,T;
W^{1,\frac{\Gamma+\vartheta}{\vartheta}}(\Omega))}
\nonumber \\
& \hspace{1in} \leq
C\,\|\wt\|_{L^{\frac{\Gamma+\vartheta}{\vartheta}}(0,T;
W^{1,r}(\Omega))},
\label{equtkbde}
\end{align}
where we have noted from (\ref{vthetag}) that $
\frac{\Gamma+\vartheta}{\vartheta} \geq \frac{\gamma+\vartheta}{\vartheta} \geq 2$.
The desired result (\ref{dtmtk}) then follows from (\ref{equtkbde}).

The results (\ref{mtkwcon}--d) follow similarly to (\ref{mtkaLwcon}--e)
from (\ref{mtkbd}), (\ref{dtmtk}), (\ref{Cwcoma},b), (\ref{compact1}), (\ref{rhokscon}),
and (\ref{uwconH1b}).
The results (\ref{rhokconp}--g) follow immediately from
(\ref{Bogpkbd}) and (\ref{pkdef}).
\end{proof}

We now have the following analogue of Lemma \ref{Pkexistslem}.

\begin{lemma}
\label{Pexistslem}
The triple $(\rho,\ut,\hpsi)$,
defined as in Lemmas \ref{hpsiconv} and \ref{rhoklem},
satisfies
\begin{subequations}
\begin{align}
&\displaystyle\int_{0}^{T}\left\langle \frac{\partial \rho}{\partial t}\,,\eta
\right\rangle_{W^{1,6}(\Omega)} \dd t
- \int_0^T \int_\Omega \rho \,\ut
\cdot \nabx \eta
\,\dx \,\dt =0
\qquad
\forall \eta \in L^2(0,T;W^{1,6}(\Omega)),
\label{eqrho}
\end{align}
with $\rho(\cdot,0) = \rho_0(\cdot)$,
\begin{align}
&\displaystyle\int_{0}^{T} \left \langle
\frac{\partial (\rho\,\ut)}{\partial t}, \wt \right \rangle_{W^{1,r}_0(\Omega)}
\dt
+
\displaystyle\int_{0}^{T}\!\! \int_\Omega
\left[
\Stt(\ut) - \rho\,\ut \otimes \ut - c_p\,\overline{\rho^\gamma}\,\Itt\right]
:\nabxtt \wt
\,\dx\, \dt
\nonumber \\
&\quad =\int_{0}^T
 \int_{\Omega} \left[ \rho\,
\ft \cdot \wt
-\left(\tautt_1 (M\,\hpsi) - \mathfrak{z}\,\vrho^2\,\Itt\right)
: \nabxtt
\wt \right] \dx \, \dt
\quad
\forall \wt \in L^{\frac{\gamma+\vartheta}{\vartheta}}(0,T;\Wt^{1,r}_0(\Omega)),
\label{equt}
\end{align}
with $(\rho\,\ut)(\cdot,0) = (\rho_0\,\ut_0)(\cdot)$,
$\vartheta(\gamma)$ defined as in (\ref{vthetag}) and
$r= \max\{4,\frac{6\gamma}{2\gamma-3}\}$, and
\begin{align}
\label{eqhpsi}
&\int_{0}^T \left \langle
M\,\frac{ \partial \hpsi}{\partial t},
\varphi \right \rangle_{H^s(\Omega \times D)} \dt
+
\frac{1}{4\,\lambda}
\,\sum_{i=1}^K
 \,\sum_{j=1}^K A_{ij}
\int_{0}^T \!\!\int_{\Omega \times D}
M\,
 \nabqj \hpsi
\cdot\, \nabqi
\varphi\,
\dq \,\dx \,\dt
\nonumber \\
& \qquad
+ \int_{0}^T \!\!\int_{\Omega \times D} M \left[
\epsilon\,
\nabx \hpsi
- \ut\,\hpsi \right]\cdot\, \nabx
\varphi
\,\dq \,\dx \,\dt
\nonumber \\
&
\qquad
- \int_{0}^T \!\!\int_{\Omega \times D} M\,\sum_{i=1}^K
\left[\sigtt(\ut)
\,\qt_i\right]
\hpsi \,\cdot\, \nabqi
\varphi
\,\dq \,\dx\, \dt = 0
\quad
\forall \varphi \in L^2(0,T;H^s(\Omega \times D)),
\end{align}
\end{subequations}
with $\hpsi(\cdot,0) = \hpsi_0(\cdot)$  and $s > 1 + \frac{1}{2}(K+1)d$.

In addition, the triple $(\rho,\ut,\hpsi)$
satisfies, for
a.a.\ $t' \in (0,T)$,
{\color{black}
\begin{align}\label{Penergy}
&\frac{1}{2}\,\displaystyle
\int_{\Omega} \rho(t')
\,|\ut(t')|^2 \,\dx
+ \int_{\Omega} P(\rho(t'))  \dx
+k\,\int_{\Omega \times D} M\,
\mathcal{F}(\hpsi(t')) \dq \,\dx
\nonumber\\
&\quad
+  \mu^S c_0\, \int_0^{t'} \|\ut\|_{H^1(\Omega)}^2 \dt
+ k\,\int_0^{t'} \int_{\Omega \times D}
M\,
\left[\frac{a_0}{2\lambda}\,
\left|\nabq \sqrt{\hpsi} \right|^2
+ 2\varepsilon\, \left|\nabx \sqrt{\hpsi} \right|^2
\right]
\,\dq \,\dx\, \dt
\nonumber \\
& \quad + \mathfrak{z}\,\|\vrho(t')\|_{L^2(\Omega)}^2
+ 2\, \mathfrak{z}\, \varepsilon\,\int_0^{t'} \|\nabx \vrho\|_{L^2(\Omega)}^2 \dt
\nonumber\\
&\leq {\rm e}^{t'}\biggl[
\frac{1}{2}\,\displaystyle
\int_{\Omega} \rho_0\,|\ut_{0}|^2 \dx
+ \int_{\Omega} P(\rho_0)  \dx
+k\,\int_{\Omega \times D} M\, 
\mathcal{F}(\hpsi_0)
\dq \,\dx
\nonumber\\
&\hspace{1in}
+
\mathfrak{z}\,\int_{\Omega} \left(\int_D M \,\hpsi_0\dq\right)^2 \dx
+
\frac{1}{2}  \int_{0}^{t'} \|\ft\|_{L^\infty(\Omega)}^2 \dt \int_\Omega \rho_0 \dx
\biggr].
\end{align}
}
\end{lemma}

\begin{proof}
Passing to the limit $\kappa \rightarrow 0_+$ for the subsequence of Lemma \ref{mtklem}
in (\ref{eqrhok}) yields (\ref{eqrho}) subject to the stated initial condition,
on noting (\ref{rhokwcon}),
(\ref{mtkscon}) and that  $\frac{6\gamma}{\gamma+6} > \frac{6}{5}$ as $\gamma > \frac{3}{2}$.

Similarly to the proof of (\ref{equtka}), passing to the limit $\kappa \rightarrow 0_+$
for the subsequence of Lemma \ref{mtklem} in (\ref{equtkfinal}) for any
$\wt \in C^\infty_0(\Omega_T)$
yields (\ref{equt}) for any
$\wt \in C^\infty_0(\Omega_T)$ subject to the stated initial condition,
on noting
(\ref{uwconH1b}), (\ref{rhokwcon}), (\ref{mtkwcon},b,d,f,g),
(\ref{tausconLr}) and (\ref{vrhokL2con}).
The desired result (\ref{equt}) for any
$\wt \in L^{\frac{\gamma+\vartheta}{\vartheta}}(0,T;\Wt^{1,r}_0(\Omega))$
then follows from (\ref{equtkbda}), (\ref{pkkcon})
and noting from (\ref{vthetag}) that $r \geq \frac{\gamma+\vartheta}{\vartheta} \geq 2$.
Similarly to the proof of (\ref{eqhpsika}), passing to the limit $\kappa \rightarrow 0_+$
for the subsequence of Lemma \ref{mtklem}
in (\ref{eqhpsik}) yields (\ref{eqhpsi}) \red{subject to the stated} initial condition, on noting
(\ref{uwconH1b}), (\ref{hpsiwconH1}--d),
(\ref{vrhoreg})
and (\ref{eqinterp}).
Similarly to the proof of (\ref{5-eq:energyest}),
we deduce (\ref{Penergy}) from (\ref{Pkenergy})
using the results  (\ref{mtkutkwcon}), (\ref{Pp}), (\ref{hpsiwconH1},b,f),
(\ref{uwconH1b}) and (\ref{vrhokH1con}).
\end{proof}

We need to identify $\overline{\rho^\gamma}$ in (\ref{equt}) and (\ref{pkkcon}).
Similarly to (\ref{bbell}),
with $\ell \in \mathbb{N}$,
we now introduce $\mathfrak{t} : \mathbb{R}_{\geq 0}
\rightarrow \mathbb{R}_{\geq 0}$ and $\mathfrak{t}_\ell : \mathbb{R}_{\geq 0}
\rightarrow \mathbb{R}_{\geq 0}$ such that
\begin{align}
\mathfrak{t}(s):= s  \qquad \mbox{and} \qquad
\mathfrak{t}_{\ell}(s):= \left\{ \begin{array}{ll}
\mathfrak{t}(s) \quad &\mbox{for } 0 \leq s \leq \ell,\\
\mathfrak{t}(\ell) \quad &\mbox{for } \ell \leq s.
\end{array} \right.
\label{ttell}
\end{align}
Then, similarly to (\ref{renormcon}), we have the renormalised equation, for any
$\ell \in \mathbb{N}$,
\begin{align}
\frac{\partial \mathfrak{t}_\ell(\rhok)}{\partial t}
+ \nabx \cdot ( \mathfrak{t}_\ell(\rhok)\,\utk)
+ (\rhok\,(\mathfrak{t}_\ell)_+'(\rhok)-\mathfrak{t}_\ell(\rhok)\,)\,\nabx\cdot \utk=0
\qquad \mbox{in } C^\infty_0({\mathbb R}^d \times (0,T))',
\label{renormcont}
\end{align}
where $(\mathfrak{t}_\ell)_+'(\cdot)$ is defined similarly to (\ref{rderiv}).
It follows from (\ref{ttell}),  (\ref{Pkenergy}) and (\ref{renormcont}) that,
for any fixed $\ell \in {\mathbb N}$,
\begin{align}
\|\mathfrak{t}_\ell(\rhok)\|_{L^\infty(\Omega_T)} +
\|
(\rhok\,({\mathfrak t}_{\ell})'_+(\rhok)-{\mathfrak t}_{\ell}(\rhok))
\,\nabx\cdot\utk\|_{L^2(\Omega_T)}+
\left\|\frac{\partial \mathfrak{t}_\ell(\rhok)}{\partial t}
\right\|_{L^2(0,T;H^1(\Omega)')}
\leq C(\ell).
\label{ttellbd}
\end{align}

In order to identify $\overline{\rho^\gamma}$ in (\ref{equt}) and (\ref{pkkcon}),
we now apply Corollary \ref{Cor736} with
(\ref{gnpde},b) being (\ref{renormcont}) and (\ref{equtkfinal})
so that $\mu=\frac{\mu^S}{2}$ and $\lambda= \mu^B-\frac{\mu^S}{d}$,
$g_n=\mathfrak{t}_{\ell}(\rhok)$ for a fixed $\ell \in {\mathbb N}$,
$\ut_n=\utk$, $\mt_n = \rhok\,\utk$, $p_n=\pk(\rhok)$,
$\tautt_n = \tautt(M\,\hpsik)$,
$f_n = -(\rhok\,({\mathfrak t}_{\ell})'_+(\rhok)-{\mathfrak t}_{\ell}(\rhok))
\,\nabx\cdot\utk$ and
$\Ft_n = \rhok\,\ft$.
With $\{(\rhok,\utk,\hpsik)\}_{\kappa>0}$ being the subsequence (not indicated)
of Lemma \ref{mtklem}, we have that
(\ref{gncon}--d) hold with
$g=\overline{\mathfrak{t}_{\ell}(\rho)}$,
$\ut=\ut$, $\mt = \rho\,\ut$ and $p=c_p\,\overline{\rho^\gamma}$,
and $q<\infty$, $\omega=\infty$, $z=\frac{2\gamma}{\gamma+1}$ and
$r=\frac{\Gamma+\vartheta}{\Gamma}$ on recalling
(\ref{ttellbd}), (\ref{Cwcoma},b), (\ref{uwconH1b}), (\ref{mtkscona}) and (\ref{pkkcon},g).
We note that
$z=\frac{2\gamma}{\gamma+1} > \frac{6}{5}$ as $\gamma >\frac{3}{2}$.
Hence, the constraints on
$q,r,\omega$ and $z$ hold.
The results (\ref{tauncon},h) hold
with $\tautt = \tautt(M\,\hpsi)$, $\Ft= \rho\,\ft$ and $s=\gamma+\vartheta$,
on recalling  (\ref{tausconLr}) and (\ref{rhokconp}),
and  noting that $\frac{4(d+2)}{3d+4} \geq \frac{q}{q-1}$.
Finally, the result (\ref{fnconbcor}) holds with
$f=-\overline{(\rho\,({\mathfrak t}_{\ell})'_+(\rho)-{\mathfrak t}_{\ell}(\rho))
\,\nabx\cdot\ut}$
on recalling (\ref{ttellbd}).
Hence, we obtain from (\ref{EFVd}) for the subsequence of Lemma \ref{mtklem} that,
for any fixed $\ell \in \mathbb{N}$
and for all $\zeta \in C^\infty_0(\Omega)$ and $\eta \in C^\infty_0(0,T)$,
\begin{align}
&\lim_{\kappa \rightarrow 0_+}
\int_0^T \!\! \!\eta \left(\int_{\Omega} \zeta \,\mathfrak{t}_\ell(\rhok)\,
[\pk(\rhok)-\mu^\star \, \nabx\cdot \utk]
\dx \right) \dt
\nonumber \\
& \hspace{2.5in}
=
\int_0^T \!\! \!\eta \left(\int_{\Omega} \zeta \,\overline{\mathfrak{t}_{\ell}(\rho)}\,
[c_p\,\overline{\rho^\gamma}-\mu^\star\, \nabx\cdot \ut]
\dx \right) \dt,
\label{EVF2k}
\end{align}
where $\mu^\star:=\frac{(d-1)}{d}\,\mu^S+\mu^B$.

We deduce from (\ref{EVF2k}), (\ref{ttell}), (\ref{pkdef}), (\ref{pgamma}), (\ref{pkkcon},g)
and (\ref{uwconH1b}) that, for any fixed $\ell \in \mathbb{N}$,
\begin{align}
c_p \left[ \overline{\mathfrak{t}_{\ell}(\rho)\,\rho^\gamma}
- \overline{\mathfrak{t}_{\ell}(\rho)}\,\overline{\rho^\gamma}
\right] = \mu^\star
\left[
\overline{\mathfrak{t}_{\ell}(\rho)\,\nabx\cdot \ut}
- \overline{\mathfrak{t}_{\ell}(\rho)}\,\nabx \cdot \ut
\right] \qquad \mbox{a.e. in } \Omega_T,
\label{EVF3k}
\end{align}
where, as $\kappa \rightarrow 0_+$,
\begin{subequations}
\begin{alignat}{2}
\mathfrak{t}_{\ell}(\rhok)\,\rhok^\gamma &\rightarrow
\overline{\mathfrak{t}_{\ell}(\rho)\,\rho^\gamma} \qquad &&\mbox{weakly in }
L^{\frac{\gamma+\vartheta}{\gamma}}(\Omega_T),
\label{EVF4kb}\\
 \qquad \mbox{and}
\qquad
\mathfrak{t}_{\ell}(\rhok)\,\nabx \cdot \utk &\rightarrow
\overline{\mathfrak{t}_{\ell}(\rho)\,\nabx \cdot \ut} \qquad &&\mbox{weakly in }
L^{2}(\Omega_T).
\label{EVF4kc}
\end{alignat}
\end{subequations}
We can now follow the discussion in Sections 7.10.2--7.10.5 in
Novotn\'{y} \& Stra\v{s}kraba \cite{NovStras} to deduce that
$\overline{p(\rho)}=p(\rho)$. For the benefit of the reader,
we briefly outline the argument. First, it follows from (\ref{pkkcon}) and (\ref{EVF4kb})
that, for any fixed $\ell \in \mathbb{N}$,
\begin{align}
&\int_{\Omega_T}
\left[ \overline{\mathfrak{t}_{\ell}(\rho)\,\rho^\gamma}
- \overline{\mathfrak{t}_{\ell}(\rho)}\,\overline{\rho^\gamma}
\right] \dx \,\dt \nonumber \\
& \qquad = \lim_{\kappa \rightarrow 0_+} \left[
\int_{\Omega_T} (\mathfrak{t}_{\ell}(\rhok)-
\mathfrak{t}_{\ell}(\rho))\,(\rhok^\gamma -\rho^\gamma)\, \dx \,\dt
+
\int_{\Omega_T} (\mathfrak{t}_{\ell}(\rho)-
\overline{\mathfrak{t}_{\ell}(\rho)})\,(\overline{\rho^\gamma} -\rho^\gamma)\, \dx \,\dt
\right]
\nonumber \\
& \qquad \geq \limsup_{\kappa \rightarrow 0_+}
\int_{\Omega_T} |\mathfrak{t}_{\ell}(\rhok)-
\mathfrak{t}_{\ell}(\rho)|^{\gamma+1} \dx \,\dt,
\label{EVF5k}
\end{align}
where we have noted that the second term on the second line is nonnegative
as $\mathfrak{t}_\ell(s)$ is concave and $s^\gamma$ is convex for $s \in [0,\infty)$.

We now introduce $\mathfrak{L}_\ell : {\mathbb R}_{\geq 0} \rightarrow  {\mathbb R}_{\geq 0}$,
$\ell \in {\mathbb N}$, such that $s\,\mathfrak{L}_\ell'(s)-\mathfrak{L}_\ell(s)=\mathfrak{t}_\ell(s)$
for all $s \in [0,\infty)$, so that
\begin{align}
\mathfrak{L}_{\ell}(s):= \left\{ \begin{array}{ll}
\mathfrak{L}(s):=s\,{\rm log}\, s \quad &\mbox{for } 0 \leq s \leq \ell,\\
s\,{\rm log}\, \ell + s-\ell \quad &\mbox{for } \ell \leq s.
\end{array} \right.
\label{LLell}
\end{align}
Similarly to (\ref{lnrhok}) and (\ref{renormcon}),
one deduces, via renormalization, from (\ref{eqrho}) and (\ref{eqrhok}) that, for any fixed
$\ell \in \mathbb{N}$ and
for any $t' \in (0,T]$,
\begin{subequations}
\begin{align}
\int_{\Omega} \left[\mathfrak{L}_{\ell}(\rho)(t') -\mathfrak{L}_{\ell}(\rho_0)\right] \,\dx &= -
\int_0^{t'} \int_\Omega \mathfrak{t}_\ell(\rho)\,\nabx\cdot \ut \,\dx\,\dt,
\label{lnrhokb} \\
\int_{\Omega} \left[\mathfrak{L}_{\ell}(\rhok)(t') -\mathfrak{L}_{\ell}(\rho_0)\right] \,\dx &= -
\int_0^{t'} \int_\Omega \mathfrak{t}_\ell(\rhok)\,\nabx\cdot \utk \,\dx\,\dt.
\label{lnrhokc}
\end{align}
\end{subequations}
Although establishing (\ref{lnrhokc}) is straightforward as $\rhok \in
C_w([0,T;L^\Gamma_{\geq 0}(\Omega))$, establishing (\ref{lnrhokb}) is not,
since $\rho \in
C_w([0,T;L^\gamma_{\geq 0}(\Omega))$, and so $\rho$ may not be in $L^2(\Omega_T)$
as $\gamma > \frac{3}{2}$.
Nevertheless, (\ref{lnrhokb}) can still be established, see Lemma 7.57 in
\cite{NovStras}. We note that our $\vartheta(\gamma)$, recall (\ref{vthetag}),
differs from the $\vartheta(\gamma)$ in \cite{NovStras} for $\gamma \geq 4$,
due to the presence of the extra stress term in the momentum equation for our polymer
model. However, as $\rho \in L^2(\Omega_T)$ for $\gamma \geq 4$, Lemma 7.57 in
\cite{NovStras} is not required for such $\gamma$.
Subtracting (\ref{lnrhokb}) from (\ref{lnrhokc}),
and passing to the limit $\kappa \rightarrow 0_+$,
one deduces from
 (\ref{EVF4kc}) that,
for any fixed $\ell \in \mathbb{N}$ and
for any $t' \in (0,T]$,
\begin{align}
\int_{\Omega}
\left[\overline{\mathfrak{L}_{\ell}(\rho)(t')} -\mathfrak{L}_{\ell}(\rho)(t')\right] \,\dx & =
\int_0^{t'} \int_\Omega \left[
\mathfrak{t}_\ell(\rho)\,\nabx \cdot \ut - \overline{\mathfrak{t}_\ell(\rho)\, \nabx \cdot \ut}
\right] \dx \,\dt,
\label{EVF2m}
\end{align}
where, on noting (\ref{rhokwcon}) and the convexity of $\mathfrak{L}_\ell$,
\begin{align}
\mathfrak{L}_{\ell}(\rhok)(t') \rightarrow  \overline{\mathfrak{L}_{\ell}(\rho)(t')}
\geq \mathfrak{L}_{\ell}(\rho)(t')
\qquad \mbox{weakly in } L^{\gamma}(\Omega), \qquad 
\mbox{as $\kappa \rightarrow 0_+$}.
\label{rhoklncon}
\end{align}
It follows from (\ref{EVF5k}), (\ref{EVF3k}), (\ref{EVF2m}), (\ref{rhoklncon})
and (\ref{hpsi}) that,
for any fixed $\ell \in \mathbb{N}$,
\begin{align}
\limsup_{\kappa \rightarrow 0_+}
\|\mathfrak{t}_{\ell}(\rho)-
\mathfrak{t}_{\ell}(\rhok)\|^{\gamma+1}_{L^{\gamma+1}(\Omega_T)}
&\leq \frac{\mu^\star}{c_p}
\int_{\Omega_T}
\left[
\overline{\mathfrak{t}_{\ell}(\rho)\,\nabx\cdot \ut}
- \overline{\mathfrak{t}_{\ell}(\rho)}\,\nabx \cdot \ut
\right] \dx \dt
\nonumber \\
&\leq \frac{\mu^\star}{c_p}
\int_{\Omega_T}
\left[\mathfrak{t}_{\ell}(\rho)
- \overline{\mathfrak{t}_{\ell}(\rho)}
\right] \,\nabx \cdot \ut \,\dx \dt
\nonumber \\
&\leq C\,
\|\mathfrak{t}_{\ell}(\rho)
- \overline{\mathfrak{t}_{\ell}(\rho)}\|_{L^2(\Omega_T)}
\nonumber \\
&\leq C\,\limsup_{\kappa \rightarrow 0_+}
\|\mathfrak{t}_{\ell}(\rho)
- \mathfrak{t}_{\ell}(\rhok)\|_{L^2(\Omega_T)}
\nonumber \\
&\leq C\,\limsup_{\kappa \rightarrow 0_+}
\|\mathfrak{t}_{\ell}(\rho)
- \mathfrak{t}_{\ell}(\rhok)\|_{L^{\gamma+1}(\Omega_T)},
\label{ttellk}
\end{align}
where $C\in {\mathbb R}_{>0}$ is independent of $\ell$ and $\kappa$.
It is easily deduced from (\ref{Bogpkbd}), (\ref{rhokconp}) and (\ref{ttell}) that,
for all $\ell \in \mathbb N$, $\kappa >0$ and $r \in [1,\gamma+\vartheta)$,
\begin{align}
\|\rhok-\mathfrak{t}_\ell(\rhok)\|_{L^r(\Omega_T)}
+\|\rho-\mathfrak{t}_\ell(\rho)\|_{L^r(\Omega_T)} +
\|\rho-\overline{\mathfrak{t}_\ell(\rho)}\|_{L^r(\Omega_T)}
\leq C\,\ell^{1-\frac{\gamma+\vartheta}{r}},
\label{telllbd}
\end{align}
where $C\in {\mathbb R}_{>0}$ is independent of $\ell$ and $\kappa$.
It follows from (\ref{ttellk}), (\ref{telllbd}) and (\ref{eqLinterp}) that
\begin{align}
\lim_{\ell \rightarrow \infty} \limsup_{\kappa \rightarrow 0_+}
\|\mathfrak{t}_{\ell}(\rho)-
\mathfrak{t}_{\ell}(\rhok)\|_{L^{\gamma+1}(\Omega_T)} =0.
\label{limsupell}
\end{align}
It immediately follows from (\ref{telllbd}), (\ref{limsupell}),
(\ref{Bogpkbd}),
(\ref{eqLinterp}) and (\ref{pkkcon}),
on possibly extracting a further subsequence (not indicated),
that, as $\kappa \rightarrow 0_+$,
\begin{subequations}
\begin{alignat}{3}
\rhok &\rightarrow \rho \qquad &&\mbox{strongly in } L^s(\Omega_T),
\qquad &&\mbox{for any } s \in [1,\gamma+\vartheta(\gamma)),
\label{rhokscon1}\\
\rhok^\gamma &\rightarrow \rho^\gamma \qquad &&\mbox{weakly in } L^{\frac{\gamma+\theta}{\gamma}}(\Omega_T),
\qquad &&\mbox{that is, } \overline{\rho^\gamma}=\rho^\gamma.
\label{pkiden1}
\end{alignat}
\end{subequations}

Finally, we have the analogue of Theorem \ref{Pkexistslemfinal}.

\begin{theorem}
\label{Pexistslemfinal}
The triple $(\rho,\ut,\hpsi)$,
defined as in Lemmas \ref{hpsiconv} and \ref{rhoklem},
is a global weak solution to problem (P),
in the sense that
(\ref{eqrho},c), with their initial conditions, hold and
\begin{align}
&\displaystyle\int_{0}^{T} \left \langle
\frac{\partial (\rho\,\ut)}{\partial t}, \wt \right \rangle_{W^{1,r}_0(\Omega)}
\!\!\!\dt
+
\displaystyle\int_{0}^{T}\!\! \int_\Omega
\left[
\Stt(\ut) - \rho\,\ut \otimes \ut
-c_p\,\rho^\gamma\,\Itt
\right]
:\nabxtt \wt
\,\dx\, \dt
\nonumber \\
&\qquad =\int_{0}^T
 \int_{\Omega} \left[ \rho\,
\ft \cdot \wt
-\left(\tautt_1 (M\,\hpsi) - \mathfrak{z}\,\vrho^2\,\Itt\right)
: \nabxtt
\wt \right] \dx \, \dt
\nonumber \\
& \hspace{3.3in}
\quad
\forall \wt \in L^{\frac{\gamma+\vartheta}{\vartheta}}(0,T;\Wt^{1,r}_0(\Omega)),
\label{equtfinal}
\end{align}
with $(\rho\,\ut)(\cdot,0) = (\rho_0\,\ut_0)(\cdot)$,
$\vartheta(\gamma)$ defined as in (\ref{vthetag}) and
$r= \max\{4,\frac{6\gamma}{2\gamma-3}\}$.
In addition, the weak solution $(\rho,\ut,\hpsi)$
satisfies (\ref{Penergy}).
\end{theorem}
\begin{proof}
The results (\ref{eqrho},c) and (\ref{Penergy}) have already been established
in Lemma \ref{Pexistslem}.
Equation (\ref{equtfinal}) was established in Lemma \ref{Pexistslem} with
$\rho^\gamma$ replaced by $\overline{\rho^\gamma}$, see (\ref{equt}).
The desired result (\ref{equtfinal}) then follows immediately from (\ref{equt})
and (\ref{pkiden1}).
\end{proof}

\section{Conclusions}

{\color{black}
We proved the existence of global-in-time weak solutions to a general class of models that arise
from the kinetic theory of dilute solutions of nonhomogeneous polymeric liquids, where the polymer
molecules are idealized as bead-spring chains with finitely extensible nonlinear elastic (FENE)
type spring potentials.
The class of models under consideration involved the unsteady, compressible, isentropic,
isothermal Navier--Stokes system, with constant shear and bulk viscosity coefficients,
in a bounded domain $\Omega$ in $\mathbb{R}^d$, $d = 2$ or $3$,
for the density $\rho$, the velocity $\ut$ and the pressure $p$ of the fluid, with an equation of state
of the form $p(\rho) = c_p \rho^\gamma$, where $c_p$ is a positive constant and $\gamma>\frac{3}{2}$.
Our analysis applies, without alterations, to some other familiar \textit{monotone} equations
of state, such as the (Kirkwood-modified) Tait equation of state (cf. Remark \ref{tait}).
As a starting point for the extension of this work to nonmonotone equations of state we refer
the reader to the work of Feireisl \cite{Feir2} in the context of the compressible Navier--Stokes
equations; see also Section 7.12.3 in \cite{NovStras}, which explains how, in the case of the compressible
Navier--Stokes equations at least, the existence proof for monotone equations of state can be extended,
with minor modifications, to the case of nonmonotone equations of state.

The right-hand side of the Navier--Stokes momentum equation included an elastic extra-stress
tensor, defined as the sum of the classical Kramers expression and a quadratic interaction term.
The elastic extra-stress tensor stems from the random movement of the polymer chains and involves the
associated probability density function that satisfies a Fokker--Planck-type parabolic equation,
a crucial feature of which is the presence of a centre-of-mass diffusion term.
Our analysis required no structural assumptions
on the drag term in the Fokker--Planck equation; in particular, the drag term need not
be corotational. With a nonnegative initial density
$\rho_0 \in L^\infty(\Omega)$ for the continuity equation;
a square-integrable initial velocity datum $\undertilde{u}_0$ for
the Navier--Stokes momentum equation; and
a nonnegative initial probability density function $\psi_0$
for the Fokker--Planck equation, which has finite
relative entropy with respect to the Maxwellian $M$ associated with the spring potential
in the model, we proved, {\em via} a limiting procedure on certain discretization and regularization parameters,
the existence of a global-in-time bounded-energy weak solution
$t \mapsto (\rho(t),\undertilde{u}(t), \psi(t))$
to the coupled Navier--Stokes--Fokker--Planck system, satisfying the initial condition
$(\rho(0),\undertilde{u}(0), \psi(0)) = (\rho_0,\undertilde{u}_0, \psi_0)$.

The representation $p_p = k(K+1)\varrho + \mathfrak{z} \varrho^2$ of the polymeric pressure $p_p$ used in the paper (cf.
Remark \ref{rem-1.3}) can be viewed as a quadratic truncation of a virial expansion of the polymeric pressure
in terms of the polymer number density $\varrho$. The presence of the quadratic interaction term $\mathfrak{z}\varrho^2$,
with $\mathfrak{z}>0$,  was found to be essential for the analysis, pointing at the necessity of
accounting for effects that are not captured by the classical Kramers expression, such as excluded volume effects.
}

\vspace{4mm}


\bibliographystyle{siam}

\bibliography{polyjwbesrefs}


\bigskip

\bigskip

\begin{appendices}

\begin{center}
\sc{Appendix}
\end{center}

\setcounter{equation}{0}
\renewcommand{\theequation}{\Alph{section}.\arabic{equation}}

\section{Homogeneous Sobolev spaces}\label{sec:App-A}
The linear space $C^\infty_0(\mathbb{R}^d)$, when equipped with the functional $|\cdot|_{1,r}$ defined by $|\eta|_{1,r}:=\|\nabx \eta\|_{L^r(\mathbb{R}^d)}$, is a normed linear space. We denote by $D^{1,r}(\mathbb{R}^r)$ the closure of $C^\infty_0(\mathbb{R}^d)$ in this norm, and will refer to $D^{1,r}(\mathbb{R}^d)$ as a \textit{homogeneous Sobolev space}.
The space $D^{1,r}(\mathbb{R}^d)$ is separable for $r \in [1,\infty)$ and reflexive for $r \in (1,\infty)$.
For $r \in [1,d)$ the space $D^{1,r}(\mathbb{R}^d)$ can be characterized as follows:
\[ D^{1,r}(\mathbb{R}^d) = \big\{ \eta \in C^\infty_0(\mathbb{R}^d)'\,:\, \eta \in L^{\frac{dr}{d-r}}(\mathbb{R}^d), \, \nabx \eta \in L^{r}(\mathbb{R}^d)\big\}.\]
More specifically, for $r \in [1,d)$, the following \textit{Sobolev inequality} holds:
\[ \|\eta\|_{L^{\frac{dr}{d-r}}(\mathbb{R}^d)} \leq C(r,d) \|\nabx \eta\|_{L^r(\mathbb{R}^d)} \qquad \forall \eta \in D^{1,r}(\mathbb{R}^d),\]
where $C(r,d)$ is a positive constant depending on $r$ and $d$ but independent of $\eta$.
For $r \in [d, \infty)$ the elements of $D^{1,r}(\mathbb{R}^d)$ are not distributions and can only be viewed as sets of equivalence classes.

\setcounter{equation}{0}
\renewcommand{\theequation}{\Alph{section}.\arabic{equation}}

\section{Fourier multipliers}\label{sec:App-B}
Let, as in \eqref{Sfrak},  $\mathfrak{S}(\mathbb{R}^d)$ signify the linear space of smooth \textit{rapidly decreasing (complex-valued)  functions}, equipped with the topology of locally uniform convergence, and denote by ${\mathfrak S}({\mathbb R}^d)'$ its dual space,
referred to as the space of \textit{tempered distributions}. On the space of rapidly decreasing functions the Fourier transform
${\mathfrak F} : {\mathfrak S}({\mathbb R}^d) \rightarrow {\mathfrak S}({\mathbb R}^d)$ and its inverse
${\mathfrak F}^{-1} : {\mathfrak S}({\mathbb R}^d) \rightarrow {\mathfrak S}({\mathbb R}^d)$ are defined by the formulae \eqref{FTran}, which
we extend to ${\mathfrak F},\, {\mathfrak F}^{-1}: {\mathfrak S}({\mathbb R}^d)' \rightarrow {\mathfrak S}({\mathbb R}^d)'$ via the
identities \eqref{Ftext}.
\begin{lemma}[Theorem 1.55 in \cite{NovStras}]\label{Le-B-1}
~
\begin{itemize}
\item[(i)] The operator $\mathfrak{F}$ defined by \eqref{FTran} can be extended by continuity to a bounded linear operator from $L^2(\mathbb{R}^d)$ to $L^2(\mathbb{R}^d)$. Furthermore, the following Parseval--Plancherel identity holds:
\[ \int_{\mathbb{R}^d} \overline{\eta(\xt)}\,\zeta(\xt) \dd \xt = \int_{\mathbb{R}^d} \overline{[\mathfrak{F}\eta](\yt)}\,\,[\mathfrak{F}(\zeta)](\yt) \dd \yt\qquad \forall \eta,\zeta \in L^2(\mathbb{R}^d).\]
\item[(ii)] For any $\eta \in \mathfrak{S}(\mathbb{R}^d)'$ and any multiindex $\undertilde{\alpha} \in \mathbb{N}^d$, one has the following identities:
\[  \frac{\partial^{|\undertilde\alpha|}\,\mathfrak{F}{\eta}}{\partial^{\alpha_1}_{y_1}\cdots \partial^{\alpha_d}_{y_d}} = \mathfrak{F}[(-i \xt)^{\undertilde\alpha} \eta],\qquad \mathfrak{F}\left(\frac{\partial^{|\undertilde\alpha|}\eta}{\partial^{\alpha_1}_{x_1}\cdots \partial^{\alpha_d}_{x_d}}\right) = (i \yt)^{\undertilde{\alpha}} \mathfrak{F}(\eta). \]
\item[(iii)] For any $\eta \in \mathfrak{S}(\mathbb{R}^d)'$ and $\zeta \in \mathfrak{S}(\mathbb{R}^d)$, the following convolution identity holds:
\[ \mathfrak{F}(\eta \ast \zeta) = \mathfrak{F}(\eta)\,\mathfrak{F}(\zeta).\]
\end{itemize}
\end{lemma}

\medskip

A bounded measurable function $m\,: \mathbb{R}^d \rightarrow \mathbb{R}$ is called a \textit{Fourier multiplier} of type $(r,s)$, with $r,s \in [1,\infty)$, if there exists a positive constant $C(r,s)$ such that
\[ \|\mathfrak{F}^{-1}(m\, \mathfrak{F}(\eta))\|_{L^s(\mathbb{R}^d)} \leq C(r,s)\|\eta\|_{L^r(\mathbb{R}^d)}\qquad \forall \eta \in \mathfrak{S}(\mathbb{R}^d).\]
The linear operator
\[T\,:\, \mathfrak{S}(\mathbb{R}^d) \subset L^r(\mathbb{R}^d) \rightarrow L^s(\mathbb{R}^d), \qquad T\eta = \mathfrak{F}^{-1}(m\,\mathfrak{F}(\eta)) = \mathfrak{F}^{-1}(m) \ast \eta,\]
with domain $\mathcal{D}(T)=\mathfrak{S}(\mathbb{R}^d)$ is a densely defined continuous linear operator from $L^r(\mathbb{R}^d)$ to $L^s(\mathbb{R}^d)$. Therefore, its closure, which is still denoted by $T$, is a continuous linear operator from $L^r(\mathbb{R}^d)$ into $L^s(\mathbb{R}^d)$; i.e.,
\[ \|\mathfrak{F}^{-1}(m\, \mathfrak{F}(\eta))\|_{L^s(\mathbb{R}^d)} \leq C(r,s)\|\eta\|_{L^r(\mathbb{R}^d)}\qquad \forall \eta \in L^r(\mathbb{R}^d).\]
\begin{lemma}[Lizorkin's multiplier theorem; Theorem 1.57 in \cite{NovStras}]\label{Le-B-2}
Let $r \in (1,\infty)$, $\beta \in \big[0,\frac{1}{r}\big)$, and let $m\,:\, \yt \in \mathbb{R}^d \mapsto m(\yt) \in \mathbb{R}$ be such that $m \in L^\infty(\mathbb{R}^d)$, and $m$ has the derivative
\[ \frac{\partial^d m}{\partial_{y_1} \ldots \partial_{y_d}}\]
as well as all lower-order partial derivatives, continuous on $\mathbb{R}^d\setminus\{0\}$. Suppose further that there exists a $B>0$ such that
\[ |y_1|^{\alpha_1 + \beta}\cdots |y_d|^{\alpha_d + \beta}\left|\frac{\partial^{|\undertilde\alpha|}}{\partial^{\alpha_1}_{y_1}\cdots \partial^{\alpha_d}_{y_d}} m(\yt)\right| \leq B\qquad
\forall \yt \in \mathbb{R}^d\setminus\{0\},\]
where, for each $i \in \{1, \ldots, d\}$,  $\alpha_i$ is $0$ or $1$. Then, $m$ is a Fourier multiplier of type $(r,s)$ with $\frac{1}{s} = \frac{1}{r}-\beta$. In particular if $\beta=0$, then $m$ is
a Fourier multiplier of type $(r,r)$.
\end{lemma}
A particularly simple case is when $r=2$: it follows from part (i) of Theorem \ref{Le-B-1} that any $m \in L^\infty(\mathbb{R}^d)$ is a Fourier multiplier of type $(2,2)$.

\setcounter{equation}{0}
\renewcommand{\theequation}{\Alph{section}.\arabic{equation}}

\section{The Bogovski\u{\i} operator}\label{sec:App-C}

\begin{lemma}[Lemma 3.17 in \cite{NovStras}]\label{Le-C-1}
Suppose that $\Omega$ is a bounded Lipschitz domain in $\mathbb{R}^d$. Then, there exists a linear operator $\calBt = (\mathcal{B}^1,\ldots,\mathcal{B}^d)$, called the Bogovski\u{\i} operator, with the following properties:
\begin{equation*}
\calBt \,:\, L^r_0(\Omega) \rightarrow \Wt^{1,r}_0(\Omega), \qquad 1 < r < \infty;
\end{equation*}
\begin{equation*}
\nabx\cdot \calBt(\eta) = \eta\quad \mbox{a.e. in $\Omega$}\qquad \forall \eta \in L^r_0(\Omega), \;1 < r < \infty;
\end{equation*}
\begin{equation*}
\|\nabxtt\calBt(\eta)\|_{L^r(\Omega)} \leq C(r,\Omega)\|\eta\|_{L^{r}(\Omega)}\qquad \forall \eta \in L^r_0(\Omega), \;1<r<\infty.
\end{equation*}
For $1 < r, s < \infty$, let
\[ \Et^{r,s}(\Omega) := \{ \undertilde{\zeta} \in \Lt^r(\Omega)\,:\, \nabx \cdot \undertilde{\zeta} \in L^s(\Omega)\},\]
equipped with the norm $\|\undertilde{\zeta}\|_{E^{r,s}(\Omega)} = \|\undertilde{\zeta}\|_{L^r(\Omega)} + \|\nabx\cdot \undertilde{\zeta}\|_{L^s(\Omega)}$,
and define $\Et^{r,s}_0(\Omega)$ as the closure of $\Ct^\infty_0(\Omega)$ in the norm of $\Et^{r,s}(\Omega)$.
If $\eta = \nabx\cdot \undertilde{\zeta}$, where $\undertilde{\zeta} \in \Et^{r,s}_0(\Omega)$ with some $r \in (1,\infty)$, then
\begin{equation*}
\|\calBt(\eta)\|_{L^r(\Omega)} \leq C(r,\Omega)\|\undertilde{\zeta}\|_{L^r(\Omega)}.
\end{equation*}
Finally, if $\eta \in C^\infty_0(\Omega)$ and ${\mint} \eta = 0$, then $\calBt(\eta) \in \Ct^\infty_0(\Omega)$.
\end{lemma}

\setcounter{equation}{0}
\renewcommand{\theequation}{\Alph{section}.\arabic{equation}}

\section{Convexity and weak convergence}\label{sec:App-D}

\begin{lemma}[Corollary 3.33 in \cite{NovStras}]\label{Le-D-1}
~
\begin{itemize}
\item[(i)]
Let $r \in [1,\infty)$ and suppose that $\Omega$ is a domain in $\mathbb{R}^d$. Let further $\mathcal{H}(\eta):=\int_K|\eta|^r \dx$, $\eta \in L^r(K)$, where $K$ is a measurable subset of $\Omega$. Then,
\[ \mathcal{H}(\eta) \leq \lim\mbox{\rm inf}_{n \rightarrow \infty} \mathcal{H}(\eta_n) \qquad \mbox{whenever $\eta_n \rightarrow \eta$ weakly in $L^r(\Omega)$}.\]
\item[(ii)] Let $r \in [1,\infty)$ and suppose that $\Omega$ is a domain in $\mathbb{R}^d$, $I$ is an interval in $\mathbb{R}$ and $f$ is a convex lower-semicontinuous (respectively, concave upper-semicontinuous) function defined on $I$. Suppose further that $\eta_n$ is a sequence of nonnegative functions from $L^r(\Omega)$ with
    values in $I$ such that
    \[ \eta_n \rightarrow \eta \qquad \mbox{weakly in $L^r(\Omega)$}\]
    and
    \[ f(\eta_n) \rightarrow \overline{f(\eta)} \qquad \mbox{weakly in $L^1(\Omega)$}.\]
    Then,
    \[ f(\eta) \leq \overline{f(\eta)}\quad \mbox{{\rm (}respectively, $f(\eta) \geq \overline{f(\eta)}$\rm{)}}\quad \mbox{a.e. in $\Omega$}.\]
\end{itemize}
\end{lemma}
\begin{lemma}[Lemma 3.34 in \cite{NovStras}]\label{Le-D-2}
Let $r \in (1,\infty)$, suppose that $\Omega$ is a bounded domain in $\mathbb{R}^d$ and $I$ is an interval in $\mathbb{R}$. Suppose further that $f\,:\,I \rightarrow \mathbb{R}$ is a strictly convex function. Let $\{\eta_n\}_{n\in \mathbb{N}}$ be a sequence of functions in $L^r(\Omega)$ with values in $I$. If $\eta_n \rightarrow \eta$ weakly in $L^r(\Omega)$ and
$f(\eta_n) \rightarrow f(\eta)$ weakly in $L^1(\Omega)$, then $\eta_n \rightarrow \eta$ strongly in $L^1(\Omega)$.
\end{lemma}

The following lemma, due to Feireisl, plays an important role in the proof of weak compactness of the effective viscous flux.

\begin{lemma}[Commutator lemma; Lemma 4.25 in \cite{NovStras}]\label{Le-D-3}
Suppose that $d \in \{2,3\}$, $r,s \in (1,\infty)$, $\frac{1}{r} + \frac{1}{s}=\frac{1}{\upsilon}<1$, and
\[ f_n \rightarrow f \quad \mbox{weakly in $L^r(\mathbb{R}^d)$},\]
\[ g_n \rightarrow g \quad \mbox{weakly in $L^s(\mathbb{R}^d)$}.\]
Consider the Riesz operators $\mathcal{R}_{kj}$, with $k,j=1,\dots, d$,  defined by \eqref{Rjkdef}. Then,
\[ f_n \mathcal{R}_{kj}(g_n) - g_n \mathcal{R}_{kj}(f_n) \rightarrow f \mathcal{R}_{kj}(g) - g \mathcal{R}_{kj}(f)\quad \mbox{weakly in $L^\upsilon(\mathbb{R}^d)$}.\]
\end{lemma}

\section{An Arzela--Ascoli type theorem in $C_{w}([0,T]; L^s(\Omega))$}\label{sec:App-E}
\begin{lemma}[Lemma 6.2 in \cite{NovStras}]\label{Le-E-1}
Let $r,s \in (1,\infty)$ and suppose that $\Omega$ is a bounded Lipschitz domain in $\mathbb{R}^d$, $d \geq 2$. Suppose further that
$\{g_n\}_{n \in \mathbb{N}}$ is a sequence of functions
defined on $[0,T]$ with values in $L^s(\Omega)$ such that $g_n \in C_w([0,T];L^s(\Omega))$, and $g_n$ is bounded in $C([0,T];W^{-1,r}(\Omega)) \cap L^\infty(0,T; L^s(\Omega))$ independent of $n$.
Then, there exists a subsequence (not indicated) such that
\begin{itemize}
\item[(i)]  $g_n \rightarrow g$ in $C_w([0,T]; L^s(\Omega))$;
\item[(ii)] If, in addition, $1 < r \leq \frac{d}{d-1}$ and $1<s<\infty$, or $\frac{d}{d-1} < r < \infty$ and $\frac{dr}{d+r}<s<\infty$, then $g_n \rightarrow g$ strongly in
$C([0,T]; W^{-1,r}(\Omega))$.
\end{itemize}
\end{lemma}

\setcounter{equation}{0}
\renewcommand{\theequation}{\Alph{section}.\arabic{equation}}

\section{The continuity equation: extension to $\mathbb{R}^d$ and renormalization}\label{sec:App-F}

\begin{lemma}[Lemma 6.8 in \cite{NovStras}]\label{Le-F-1}
Let $\Omega$ be a bounded Lipschitz domain in $\mathbb{R}^d$, $d \geq 2$, and define $\Omega_T:=\Omega \times (0,T)$. Suppose that $\rho \in L^2(\Omega_T)$, $\ut \in L^2(0,T; \Ht^{1}_0(\Omega))$ and $f \in L^1(\Omega_T)$ satisfy
\[ \frac{\partial\rho}{\partial t} + \nabx\cdot (\ut\, \rho) = f\qquad \mbox{in $C^\infty_0(\Omega_T)'$}.\]
Then, upon extending $(\rho, \ut, f)$ to $\mathbb{R}^d$ by $(0,\undertilde{0},0)$ outside $\Omega$, one has that
\[ \frac{\partial\rho}{\partial t} + \nabx\cdot (\ut\,\rho) = f\qquad \mbox{in $C^\infty_0(\mathbb{R}^d \times (0,T))'$}.\]
\end{lemma}

\bigskip

Next we collect some results of relevance for our analysis that concern the renormalization of the continuity equation.
Consider a function $b$ defined on $[0,\infty)$ such that
\begin{equation}\label{eq:F-1}
b \in C[0,\infty)\cap C^1(0,\infty), \quad |b'(s)| \leq C s^{-\lambda_0}\quad \mbox{for $s \in (0,1]$, where $C>0$ and $\lambda_0 <1$},
\end{equation}
and
\begin{equation}\label{eq:F-2}
|b'(s)| \leq C s^{\lambda_1}\quad \mbox{for $s \geq 1$, where $C>0$ and $\lambda_1> -1$}.
\end{equation}

\begin{lemma}[Lemma 6.9 in \cite{NovStras}]\label{Le-F-2}
Let $d \geq 2$, $\beta \in [2,\infty)$, and let $\lambda_1$ be a real number such that
\begin{equation}\label{eq:F-3}
\lambda_1 \leq \frac{\beta}{2} -1.
\end{equation}
Suppose further that
$\rho \in L^\beta(0,T; L^\beta_{\rm loc}(\mathbb{R}^d))$, $\rho \geq 0$ a.e. on $\mathbb{R}^d \times (0,T)$, $\ut \in L^2(0,T; \Ht^{1}_{\rm loc}(\mathbb{R}^d))$
and $f \in L^z(0,T; L^z_{\rm loc}(\mathbb{R}^d))$, where $z = \left[\frac{\beta}{\lambda_1}\right]'$ is the conjugate of $\frac{\beta}{\lambda_1}$
if $\lambda_1>0$ and $z=1$ if $\lambda_1\leq 0$, are such that
\[ \frac{\partial\rho}{\partial t} + \nabx \cdot (\ut\, \rho) = f \qquad \mbox{in $C^\infty_0(\mathbb{R}^d \times (0,T))'$}.\]
\begin{itemize}
\item[(i)] For any function $b \in C^1[0,\infty)$ satisfying \eqref{eq:F-2} and \eqref{eq:F-3} one has
\begin{equation}\label{eq:F-4}
 \frac{\partial b(\rho)}{\partial t} + \nabx\cdot (\ut\,  b(\rho)) + \left\{\rho\, b'(\rho) - b(\rho)\right\} \nabx \cdot \ut = f\, b'(\rho) \qquad \mbox{in $C^\infty_0(\mathbb{R}^d \times (0,T))'$.}
\end{equation}
\item[(ii)] If $f=0$, then \eqref{eq:F-4} holds for any $b$ satisfying \eqref{eq:F-1}--\eqref{eq:F-3}.
\end{itemize}
\end{lemma}

Next we state an extension of the previous result to a class of piecewise $C^1$ functions $b$. Specifically, for $k>0$, let
\[ b_k(s) = \left\{ \begin{array}{ll} b(s) & \mbox{if $s \in [0,k)$},\\
                                      b(k) & \mbox{if $s \in [k,\infty)$},
                     \end{array}\right. \qquad \mbox{where $b \in C^1[0,\infty)$}.\]
Let $(b_k)'_+$ denote the right-derivative of $b_k$.

\begin{lemma}[Lemma 6.11 in \cite{NovStras}]\label{Le-F-3}
Let $\beta$, $\rho$, $\ut$ satisfy the assumptions of Lemma \ref{Le-F-2} and let $f \in L^1_{\rm loc}(\mathbb{R}^d \times (0,T))$ be such that
\[ \frac{\partial\rho}{\partial t} + \nabx \cdot (\ut\,\rho) = f \qquad \mbox{in $C^\infty_0(\mathbb{R}^d \times (0,T))'$}.\]
Then, part (i) of Lemma \ref{Le-F-2} holds with $b$ replaced with $b_k$ and $b'$ replaced with $(b_k)'_+$, for every $k>0$. Furthermore, if $f=0$, then part (ii) of Lemma
\ref{Le-F-2} holds, with the regularity assumption on $b \in C^1[0,\infty)$ appearing in the definition of $b_k$ replaced with the relaxed assumption \eqref{eq:F-1}, for every $k>0$.
\end{lemma}

\setcounter{equation}{0}
\renewcommand{\theequation}{\Alph{section}.\arabic{equation}}

\section{Continuity equation with dissipation}\label{sec:App-G}
Suppose that $\Omega$ is a bounded domain in $\mathbb{R}^d$ and  $T>0$. We shall consider the following parabolic initial-boundary-value problem, with
homogeneous Neumann boundary condition:
\begin{subequations}
\begin{alignat}{2}
\frac{\partial z}{\partial t} - \alpha\, \Delta_x z &= h\qquad && \mbox{in $\Omega_T:=\Omega \times (0,T)$},\label{eq:G-1}\\
z(\xt,0) &= z_0(\xt) \qquad &&\mbox{for $\xt \in \Omega$},\label{eq:G-2}\\
\nabx z \cdot \nt &= 0 \qquad &&\mbox{on $\partial \Omega \times (0,T)$},\label{eq:G-3}
\end{alignat}
\end{subequations}
where $\alpha>0$, $z_0$ and $h$ are given functions defined on $\Omega$ and $\Omega_T$, respectively, and $z$ is an unknown function defined on $\Omega_T$.

\begin{lemma}[Lemma 7.37 in \cite{NovStras}]\label{Le-G-1}
Let $\theta \in (0,1]$, $r,s \in (1,\infty)$, and suppose that $\Omega$ is a bounded domain in $\mathbb{R}^d$. Suppose further that
\begin{equation*}\label{eq:G-4}
\Omega \in C^{2,\theta},\qquad \rho_0 \in \widetilde{W}^{2-\frac{2}{r},s}(\Omega), \qquad h \in L^r(0,T; L^s(\Omega)),
\end{equation*}
where $\widetilde{W}^{2-\frac{2}{r},s}(\Omega)$ is the completion of the space $\{\zeta \in C^\infty(\overline{\Omega})\,:\,
\nabx \zeta \cdot \nt|_{\partial \Omega} = 0\}$ in
$W^{2-\frac{2}{r},s}(\Omega)$. Then, there exists a unique $z \in L^r(0,T; W^{2,s}(\Omega))\cap C([0,T]; W^{2-\frac{2}{r},s}(\Omega))$, $\frac{\partial z}{\partial t}
\in L^r([0,T]; L^s(\Omega))$ satisfying equation \eqref{eq:G-1} a.e. on $\Omega_T$, equation \eqref{eq:G-2} a.e. on $\Omega$, equation \eqref{eq:G-3} in the sense of normal
traces a.e. on $[0,T]$, and which satisfies the estimate
\begin{eqnarray*}
&&\alpha^{1-\frac{1}{r}}\|z\|_{L^\infty(0,T;W^{2-\frac{2}{r},s}(\Omega))} + \left\|\frac{\partial z}{\partial t}\right\|_{L^r(0,T; L^s(\Omega))} + \alpha \|z\|_{L^r(0,T; W^{2,s}(\Omega))}\nonumber\\
&&\qquad\qquad \leq C(r,s,\Omega)\big[\alpha^{1-\frac{1}{r}} \|z_0\|_{W^{2-\frac{2}{r},s}(\Omega)} + \|h\|_{L^r(0,T;L^s(\Omega))}\big].
\end{eqnarray*}
\end{lemma}

We shall also require the following regularity result concerning the case when the right-hand side $h$ of \eqref{eq:G-1} is in divergence form, i.e., $h = \nabx \cdot \bt$.
\begin{lemma}[Lemma 7.38 in \cite{NovStras}]\label{Le-G-2}
Let $\theta \in (0,1]$, $r,s \in (1,\infty)$, and suppose that $\Omega$ is a bounded domain in $\mathbb{R}^d$. Suppose further that
\begin{equation*}\label{eq:G-6}
\Omega \in C^{2,\theta},\qquad z_0 \in L^{s}(\Omega), \qquad \bt \in L^r(0,T; L^s(\Omega)).
\end{equation*}
Then, there exists a unique $z \in L^r(0,T; W^{1,s}(\Omega))\cap C([0,T]; L^{s}(\Omega))$, such that
\begin{equation*} \frac{\dd}{\dd t} \int_\Omega z \, \eta \dx + \alpha \int_\Omega \nabx z \cdot \nabx \eta \dx = - \int_\Omega \bt \cdot \nabx \eta \dx\quad \mbox{in $(C^\infty_0[0,T])'$}\quad \forall \eta \in C^\infty(\overline{\Omega}),
\end{equation*}
$z(\xt,0) = z_0(\xt)$ for a.e. $\xt \in \Omega$, and
\begin{equation*} \alpha^{1-\frac{1}{r}}\|z\|_{L^\infty(0,T;L^s(\Omega))} + \alpha \|\nabx z\|_{L^r(0,T;L^s(\Omega))} \leq C(r,s,\Omega) \big[\alpha^{1-\frac{1}{r}}\|z_0\|_{L^s(\Omega)}
+ \|\bt\|_{L^r(0,T;L^s(\Omega))}\big].
\end{equation*}
\end{lemma}
\end{appendices}

\bigskip

\bigskip

\hfill{\footnotesize \textit{London \& Oxford}

\bigskip

\hfill{\footnotesize \textit{\today}}

\bigskip

\bigskip

\end{document}
